\documentclass[11pt]{preprint}
\usepackage[full]{textcomp}
\usepackage[osf]{newtxtext} 
\usepackage[cal=boondoxo]{mathalfa}
\usepackage{colortbl}
\usepackage{dsfont}
\usepackage{amsmath, amssymb, xfrac}
\usepackage[all,cmtip]{xy}
\usepackage{comment}
\usepackage{nicefrac}

\usepackage{amssymb}
\usepackage{mathtools}
\usepackage{hyperref}
\usepackage{breakurl}
\usepackage{mhenvs}
\usepackage{mhequ} 
\usepackage{mhsymb}
\newcommand{\mb}[1]{\mathbb{#1}}

\usepackage{booktabs}
\usepackage{stackengine}
\newcommand\oast{\stackMath\mathbin{\stackinset{c}{0ex}{c}{0ex}{\ast}{\bigcirc}}}
\usepackage{tikz}
\usepackage{pgf}
\usetikzlibrary{snakes}
\usepackage{tcolorbox}
\usepackage{mathrsfs}
\usepackage[utf8]{inputenc}
\usepackage{longtable}
\usepackage{wrapfig}
\usepackage{subcaption}
\usepackage{mathrsfs}
\usepackage{epsfig}
\usepackage{microtype}
\usepackage{comment}
\usepackage{wasysym}
\usepackage{centernot}
\newcommand{\bed}{\begin{definition}}
\newcommand{\eed}{\end{definition}}
\usepackage{bbm}
\usepackage{enumitem}
\usepackage{bm}
\usepackage{stackrel}
\usepackage{graphicx}
\usepackage{axodraw}
\usepackage{xspace}
\makeatletter
\newcommand{\globalcolor}[1]{%
  \color{#1}\global\let\default@color\current@color
}
\makeatother
\newcommand{\dsE}{\mathds{E}}
\newcommand{\dsV}{\mathds{V}}

\usetikzlibrary{calc}
\usetikzlibrary{decorations}
\usetikzlibrary{positioning}
\usetikzlibrary{shapes}
\usetikzlibrary{external}

\newcommand{\ds}[1]{\mathds{#1}}
\newcommand{\colb}[1]{{\color{blue} #1}}
\newcommand{\btau}{\colb{\tau}}
\newcommand{\bXi}{\colb{\Xi}}
\newcommand{\bcI}[1]{{\color{blue} \mathcal{I}(#1)}}

\newcommand{\dsG}{\ds{G}}

\definecolor{blush}{rgb}{0.87, 0.36, 0.51}
	\definecolor{brightcerulean}{rgb}{0.11, 0.67, 0.84}
	\definecolor{greenryb}{rgb}{0.4, 0.69, 0.2}

\newif\ifdark
\darkfalse

\ifdark
\definecolor{darkred}{rgb}{0.9,0.2,0.2}
\definecolor{darkblue}{rgb}{0.7,0.3,1}
\definecolor{darkgreen}{rgb}{0.1,0.9,0.1}
\definecolor{franck}{rgb}{0,0.8,1}
\definecolor{pagebackground}{rgb}{.15,.21,.18}
\definecolor{pageforeground}{rgb}{.84,.84,.85}
\pagecolor{pagebackground}
\AtBeginDocument{\globalcolor{pageforeground}}
\definecolor{symbols}{rgb}{0,0.7,1}
\colorlet{connection}{red!80!black}
\colorlet{boxcolor}{blue!50}

\else

\definecolor{darkred}{rgb}{0.7,0.1,0.1}
\definecolor{darkblue}{rgb}{0.4,0.1,0.8}
\definecolor{darkgreen}{rgb}{0.1,0.7,0.1}
\definecolor{franck}{rgb}{0,0,1}
\definecolor{pagebackground}{rgb}{1,1,1}
\definecolor{pageforeground}{rgb}{0,0,0}
\colorlet{symbols}{blue!90!black}
\colorlet{connection}{red!30!black}
\colorlet{boxcolor}{blue!50!black}

\fi

\def\slash{\leavevmode\unskip\kern0.18em/\penalty\exhyphenpenalty\kern0.18em}
\def\dash{\leavevmode\unskip\kern0.18em--\penalty\exhyphenpenalty\kern0.18em}

\DeclareMathAlphabet{\mathbbm}{U}{bbm}{m}{n}

\DeclareFontFamily{U}{BOONDOX-calo}{\skewchar\font=45 }
\DeclareFontShape{U}{BOONDOX-calo}{m}{n}{
  <-> s*[1.05] BOONDOX-r-calo}{}
\DeclareFontShape{U}{BOONDOX-calo}{b}{n}{
  <-> s*[1.05] BOONDOX-b-calo}{}
\DeclareMathAlphabet{\mcb}{U}{BOONDOX-calo}{m}{n}
\SetMathAlphabet{\mcb}{bold}{U}{BOONDOX-calo}{b}{n}

\setlist{noitemsep,topsep=4pt,leftmargin=1.5em}

\DeclareMathAlphabet{\mathbbm}{U}{bbm}{m}{n}

\DeclareMathAlphabet{\mcb}{U}{BOONDOX-calo}{m}{n}
\SetMathAlphabet{\mcb}{bold}{U}{BOONDOX-calo}{b}{n}
\DeclareFontFamily{U}{mathx}{\hyphenchar\font45}
\DeclareFontShape{U}{mathx}{m}{n}{
      <5> <6> <7> <8> <9> <10>
      <10.95> <12> <14.4> <17.28> <20.74> <24.88>
      mathx10
      }{}
\DeclareSymbolFont{mathx}{U}{mathx}{m}{n}
\DeclareMathSymbol{\bigtimes}{1}{mathx}{"91}

\def\s{\mathfrak{s}}

\def\emptyset{{\centernot\ocircle}}

\providecommand{\figures}{false}
{ \ifthenelse{\equal{\figures}{false}} {#1}{\[ {\rm Figure \ missing !} \]} }{}
\def\id{\mathrm{id}}

\newcommand{\be}{\begin{equation*}}
\newcommand{\ee}{\end{equation*}}

\def\CP{\mathcal{P}}

\def\CG{\mathcal{G}}

\def\CA{\mathcal{A}}

\def\CQ{\mathcal{Q}}

\def\CT{\mathcal{T}}

\def\K{\mathfrak{K}}

\tikzstyle{tinydots}=[dash pattern=on \pgflinewidth off \pgflinewidth]
\tikzstyle{superdense}=[dash pattern=on 4pt off 1pt]

\newcommand{\mcI}{\mathcal{I}}

\newcommand{\mc}[1]{\mathcal{#1}}

\newcommand{\beq}{\begin{equation}}
\newcommand{\eeq}{\end{equation}}

\usepackage{empheq}

\newcommand{\mbbE}{\mathbb{E}}

\newcommand{\mfT}{\mathfrak{T}}
\newcommand{\mfn}{\mathfrak{n}}

\newcommand{\mfe}{\mathfrak{e}}

\def\Labe{\mathfrak{e}}

\def\Labn{\mathfrak{n}}

\def\Deltam{\Delta^{\!-}}

\def\${|\!|\!|}

\newcounter{theorem}

\newtheorem{ex}[theorem]{Example}

\newenvironment{DIFnomarkup}{}{} 

\newtheorem{assumption}{Assumption}
 
\theorembodyfont{\rmfamily}
\newtheorem{example}[lemma]{Example}
\newtheorem{remarque}[lemma]{Remark}

\newcommand{\rrightarrow}{{\to\hskip -4.9mm\raise 1pt\hbox{$\to$}}}

\newfont{\indic}{bbmss12}

\def\Nabla_#1{\nabla_{\!#1}}

\makeatletter
\pgfdeclareshape{crosscircle}
{
  \inheritsavedanchors[from=circle] 
  \inheritanchorborder[from=circle]
  \inheritanchor[from=circle]{north}
  \inheritanchor[from=circle]{north west}
  \inheritanchor[from=circle]{north east}
  \inheritanchor[from=circle]{center}
  \inheritanchor[from=circle]{west}
  \inheritanchor[from=circle]{east}
  \inheritanchor[from=circle]{mid}
  \inheritanchor[from=circle]{mid west}
  \inheritanchor[from=circle]{mid east}
  \inheritanchor[from=circle]{base}
  \inheritanchor[from=circle]{base west}
  \inheritanchor[from=circle]{base east}
  \inheritanchor[from=circle]{south}
  \inheritanchor[from=circle]{south west}
  \inheritanchor[from=circle]{south east}
  \inheritbackgroundpath[from=circle]
  \foregroundpath{
    \centerpoint%
    \pgf@xc=\pgf@x%
    \pgf@yc=\pgf@y%
    \pgfutil@tempdima=\radius%
    \pgfmathsetlength{\pgf@xb}{\pgfkeysvalueof{/pgf/outer xsep}}%
    \pgfmathsetlength{\pgf@yb}{\pgfkeysvalueof{/pgf/outer ysep}}%
    \ifdim\pgf@xb<\pgf@yb%
      \advance\pgfutil@tempdima by-\pgf@yb%
    \else%
      \advance\pgfutil@tempdima by-\pgf@xb%
    \fi%
    \pgfpathmoveto{\pgfpointadd{\pgfqpoint{\pgf@xc}{\pgf@yc}}{\pgfqpoint{-0.707107\pgfutil@tempdima}{-0.707107\pgfutil@tempdima}}}
    \pgfpathlineto{\pgfpointadd{\pgfqpoint{\pgf@xc}{\pgf@yc}}{\pgfqpoint{0.707107\pgfutil@tempdima}{0.707107\pgfutil@tempdima}}}
    \pgfpathmoveto{\pgfpointadd{\pgfqpoint{\pgf@xc}{\pgf@yc}}{\pgfqpoint{-0.707107\pgfutil@tempdima}{0.707107\pgfutil@tempdima}}}
    \pgfpathlineto{\pgfpointadd{\pgfqpoint{\pgf@xc}{\pgf@yc}}{\pgfqpoint{0.707107\pgfutil@tempdima}{-0.707107\pgfutil@tempdima}}}
  }
}
\makeatother

\def\symbol#1{\textcolor{symbols}{#1}}

\def\decorate#1#2{
        \ifnum#2>0
    		\foreach \count in {1,...,#2}{
	       	let
				\p1 = (sourcenode.center),
                \p2 = (sourcenode.east),
				\n1 = {\x2-\x1},
				\n2 = {1mm},
				\n3 = {(1.3+0.6*(\count-1))*\n1},
				\n4 = {0.7*\n1}
			in 
        		node[rectangle,fill=symbols,rotate=30,inner sep=0pt,minimum width=0.2*\n2,minimum height=\n2] at ($(sourcenode.center) + (\n3,\n4)$) {}
				}
		\fi
        \ifnum#1>0
    		\foreach \count in {1,...,#1}{
	       	let
				\p1 = (sourcenode.center),
                \p2 = (sourcenode.east),
				\n1 = {\x2-\x1},
				\n2 = {1mm},
				\n3 = {(1.3+0.6*(\count-1))*\n1},
				\n4 = {0.7*\n1}
			in 
        		node[rectangle,fill=symbols,rotate=-30,inner sep=0pt,minimum width=0.2*\n2,minimum height=\n2] at ($(sourcenode.center) + (-\n3,\n4)$) {}
				}
		\fi
}

\tikzset{
    dectriangle/.style 2 args={
        triangle,
        alias=sourcenode,
        append after command={\decorate{#1}{#2}}
    },
    dectriangle/.default={0}{0},
}

\tikzset{
	cross/.style={path picture={ 
  		\draw[symbols]
			(path picture bounding box.south east) -- (path picture bounding box.north west) (path picture bounding box.south west) -- (path picture bounding box.north east);
		}},
root/.style={circle,fill=green!50!black,inner sep=0pt, minimum size=1.2mm},
        varus/.style={circle,fill=blue!20!white,draw=blue!50!white,inner sep=0pt,minimum size=3mm},
		kepsus/.style={semithick, ->},
		testf/.style{ultra thick, green!30!black, - >}
        dot/.style={circle,fill=pageforeground,inner sep=0pt, minimum size=1mm},
        dotred/.style={circle,fill=pageforeground!50!pagebackground,inner sep=0pt, minimum size=2mm},
        var/.style={circle,fill=pageforeground!10!pagebackground,draw=pageforeground,inner sep=0pt, minimum size=3mm},
        varu/.style={circle,fill=pageforeground!10!pagebackground,draw=pageforeground,inner sep=0pt, minimum size=1mm},
        bar/.style={circle,fill=pageforeground!10!pagebackground,draw=pageforeground,inner sep=0pt, minimum size=1.5mm},
        barg/.style={circle,fill=green,draw=green,inner sep=0pt, minimum size=1.5mm},
        bnode/.style={circle,fill=black,draw=black,inner sep=0pt, minimum size=0.5mm},
        barx/.style={crosscircle,fill=pageforeground!10!pagebackground,draw=pageforeground,inner sep=0pt, minimum size=1.5mm},
        barx2/.style={crosscircle,fill=pageforeground!50!pagebackground,draw=pageforeground,inner sep=0pt, minimum size=1.5mm},
        kernel/.style={semithick,shorten >=2pt,shorten <=2pt},
        kernels/.style={snake=zigzag,shorten >=2pt,shorten <=2pt,segment amplitude=1pt,segment length=4pt,line before snake=2pt,line after snake=5pt},
        rho/.style={densely dashed,semithick,shorten >=2pt,shorten <=2pt},
           testfcn/.style={dotted,semithick,shorten >=2pt,shorten <=3pt},
        renorm/.style={shape=circle,fill=pagebackground,inner sep=1pt},
        labl/.style={shape=rectangle,fill=pagebackground,inner sep=1pt},
        xic/.style={very thin,circle,draw=symbols,fill=symbols,inner sep=0pt,minimum size=1.2mm},
        g/.style={very thin,rectangle,draw=symbols,fill=symbols!10!pagebackground,inner sep=0pt,minimum width=2.5mm,minimum height=1.2mm},
        xi/.style={very thin,circle,draw=symbols,fill=symbols!10!pagebackground,inner sep=0pt,minimum size=1.2mm},
	xies/.style={very thin,rectangle,fill=green!50!black!25,draw=symbols,inner sep=0pt,minimum size=1.1mm},
	xiesf/.style={very thin,rectangle,fill=green!50!black,draw=symbols,inner sep=0pt,minimum size=1.1mm},
        xix/.style={very thin,crosscircle,fill=symbols!10!pagebackground,draw=symbols,inner sep=0pt,minimum size=1.2mm},
        X/.style={very thin,cross,rectangle,fill=pagebackground,draw=symbols,inner sep=0pt,minimum size=1.2mm},
	xib/.style={thin,circle,fill=symbols!10!pagebackground,draw=symbols,inner sep=0pt,minimum size=1.6mm},
	xie/.style={thin,circle,fill=green!50!black,draw=symbols,inner sep=0pt,minimum size=1.6mm},
	xid/.style={thin,circle,fill=symbols,draw=symbols,inner sep=0pt,minimum size=1.6mm},
	xibx/.style={thin,crosscircle,fill=symbols!10!pagebackground,draw=symbols,inner sep=0pt,minimum size=1.6mm},
	kernels2/.style={very thick,draw=connection,segment length=12pt},
	keps/.style={thin,draw=symbols,->},
	kepspr/.style={thick,draw=connection,->},
	krho/.style={thin,draw=symbols,superdense,->},
	krhopr/.style={thick,draw=connection,superdense},
	triangle/.style = { regular polygon, regular polygon sides=3},
	not/.style={thin,circle,draw=connection,fill=connection,inner sep=0pt,minimum size=0.5mm},
	diff/.style = {very thin,draw=symbols,triangle,fill=red!50!black,inner sep=0pt,minimum size=1.6mm},
	diff1/.style = {very thin,dectriangle={1}{0},fill=red!50!black,draw=symbols,inner sep=0pt,minimum size=1.6mm},
	diff2/.style = {very thin,dectriangle={1}{1},fill=red!50!black,draw=symbols,inner sep=0pt,minimum size=1.6mm},
		diffmini/.style = {very thin,rectangle,fill=black,draw=black,inner sep=0pt,minimum size=0.75mm},
	 kernelsmod/.style={very thick,draw=connection,segment length=12pt},
	 rec/.style = {very thin,rectangle,fill=black,draw=black,inner sep=0pt,minimum size=2mm},
	cerc/.style={very thin,circle,draw=black,fill=symbols,inner sep=0pt,minimum size=2mm},
	stars/.style={very thin,star,star points=6,star point ratio=0.5, draw=black,fill=red,inner sep=0pt,minimum size=0.7mm},
	>=stealth,
        }
        \tikzset{
root/.style={circle,fill=black!50,inner sep=0pt, minimum size=3mm},
        circ/.style={circle,fill=white,draw=black,very thin,inner sep=.5pt, minimum size=1.2mm},
        round1/.style={fill=white,outer sep = 0,inner sep=2pt,rounded corners=1mm,draw,text=black,thin,minimum size=1.2mm},
          circ1/.style={circle,fill=red!10,draw=red,very thin,inner sep=.5pt, minimum size=1.2mm},
        rect/.style={fill=white,outer sep = 0,inner sep=2pt,rectangle,draw,text=black,thin,minimum size=1.2mm},
        rect1/.style={fill=white,outer sep = 0,inner sep=2pt,rectangle,draw,text=black,thin,minimum size=1.2mm},
        round2/.style={fill=red!10,outer sep = 0,inner sep=2pt,rounded corners=1mm,draw,text=black,thin,minimum size=1.2mm},
       round3/.style={fill=blue!10,outer sep = 0,inner sep=2pt,rounded corners=1mm,draw,text=black,thin,minimum size=1.2mm}, 
        rect2/.style={fill=black!10,outer sep = 0,inner sep=2pt,rectangle,draw,text=black,thin,minimum size=1.2mm},
        dot/.style={circle,fill=black,inner sep=0pt, minimum size=1.2mm},
        dotred/.style={circle,fill=black!50,inner sep=0pt, minimum size=2mm},
        var/.style={circle,fill=black!10,draw=black,inner sep=0pt, minimum size=3mm},
        kernel/.style={semithick,shorten >=2pt,shorten <=2pt},
         diag/.style={thin,shorten >=4pt,shorten <=4pt},
        kernel1/.style={thick},
        kernels/.style={snake=zigzag,shorten >=2pt,shorten <=2pt,segment amplitude=1pt,segment length=4pt,line before snake=2pt,line after snake=5pt,},
		kernels1/.style={snake=zigzag,segment amplitude=0.5pt,segment length=2pt},
		rho1/.style={densely dotted,semithick},
        rho/.style={densely dashed,semithick,shorten >=2pt,shorten <=2pt},
           testfcn/.style={dotted,semithick,shorten >=2pt,shorten <=2pt},
           visible/.style={draw, circle, fill, inner sep=0.25ex},
        renorm/.style={shape=circle,fill=white,inner sep=1pt},
        labl/.style={shape=rectangle,fill=white,inner sep=1pt},
        xic/.style={very thin,circle,fill=symbols,draw=black,inner sep=0pt,minimum size=1.2mm},
        xi/.style={very thin,circle,fill=blue!10,draw=black,inner sep=0pt,minimum size=1.2mm},
	xib/.style={very thin,circle,fill=blue!10,draw=black,inner sep=0pt,minimum size=1.6mm},
	xie/.style={very thin,circle,fill=green!50!black,draw=black,inner sep=0pt,minimum size=1mm},
	xid/.style={very thin,circle,fill=symbols,draw=black,inner sep=0pt,minimum size=1.6mm},
	edgetype/.style={very thin,circle,draw=black,inner sep=0pt,minimum size=5mm},
	nodetype/.style={very thick,circle,draw=black,inner sep=0pt,minimum size=5mm},
	kernels2/.style={very thick,draw=connection,segment length=12pt},
clean/.style={thin,circle,fill=black,inner sep=0pt,minimum size=1mm},	not/.style={thin,circle,fill=symbols,draw=connection,fill=connection,inner sep=0pt,minimum size=0.8mm},
	>=stealth,
cumu2n/.style={inner sep=3pt},
cumu2/.style={draw=red!50,fill=red!20},
dot1/.style={circle,fill=black,inner sep=0pt, minimum size=1mm},
cumu3/.style={regular polygon, regular polygon sides=3,draw=red!50,rounded corners=3pt,fill=red!20,minimum size=5mm},
cumu4/.style={regular polygon, regular polygon sides=4,draw=red!50,rounded corners=3pt,fill=red!20,minimum size=7mm},
cumu5/.style={regular polygon, regular polygon sides=5,draw=red!50,rounded corners=3pt,fill=red!20,minimum size=5mm},
	xi/.style={circle,fill=symbols!10,draw=symbols,inner sep=0pt,minimum size=1.2mm},
	xix/.style={crosscircle,fill=symbols!10,draw=symbols,inner sep=0pt,minimum size=1.2mm},
	xib/.style={circle,fill=symbols!10,draw=symbols,inner sep=0pt,minimum size=1.6mm},
	xibx/.style={crosscircle,fill=symbols!10,draw=symbols,inner sep=0pt,minimum size=1.6mm},
	not/.style={circle,fill=symbols,draw=symbols,inner sep=0pt,minimum size=0.5mm},
	>=stealth,
	}

\makeatletter
\def\DeclareSymbol#1#2#3{%
	\expandafter\gdef\csname MH@symb@#1\endcsname{\tikzsetnextfilename{symbol#1}%
	\tikz[baseline=#2,scale=0.15,draw=symbols,line join=round]{#3}}%
	\expandafter\gdef\csname MH@symb@#1s\endcsname{\scalebox{0.75}{\tikzsetnextfilename{symbol#1}%
	\tikz[baseline=#2,scale=0.15,draw=symbols,line join=round]{#3}}}%
	\expandafter\gdef\csname MH@symb@#1ss\endcsname{\scalebox{0.65}{\tikzsetnextfilename{symbol#1}%
	\tikz[baseline=#2,scale=0.15,draw=symbols,line join=round]{#3}}}%
	}
\def\<#1>{\ifthenelse{\boolean{mmode}}{\mathchoice{\csname MH@symb@#1\endcsname}{\csname MH@symb@#1\endcsname}{\csname MH@symb@#1s\endcsname}{\csname MH@symb@#1ss\endcsname}}{\csname MH@symb@#1\endcsname}}
\makeatother

\DeclareSymbol{Xi22}{0.5}{\draw (0,0) node[xi] {} -- (-1,1) node[not] {} -- (0,2) node[xi] {};} 
\DeclareSymbol{Xi2}{-2}{\draw (-1,-0.25) node[xi] {} -- (0,1) node[xi] {};} 
\DeclareSymbol{Xi2b}{-2}{\draw (-1,-0.25) node[xic] {} -- (0,1) node[xic] {};} 
\DeclareSymbol{Xi2g}{-2}{\draw (-1,-0.25) node[xies] {} -- (0,1) node[xi] {};} 
\DeclareSymbol{Xi2g2}{-2}{\draw (-1,-0.25) node[xi] {} -- (0,1) node[xies] {};} 
\DeclareSymbol{cXi2}{-2}{\draw (0,-0.25) node[xi] {} -- (-1,1) node[xic] {};}
\DeclareSymbol{Xi3}{0}{\draw (0,0) node[xi] {} -- (-1,1) node[xi] {} -- (0,2) node[xi] {};}
\DeclareSymbol{XiIIXi}{0}{\draw (0,0) node[xi] {} -- (-1,1); \draw[kernels2] (-1,1) node[not] {} -- (0,2) node[xi] {};}

\DeclareSymbol{Xi4}{2}{\draw (0,0) node[xi] {} -- (-1,1) node[xi] {} -- (0,2) node[xi] {} -- (-1,3) node[xi] {};}
\DeclareSymbol{Xi4_1}{2}{\draw (0,0) node[xic] {} -- (-1,1) node[xic] {} -- (0,2) node[xi] {} -- (-1,3) node[xi] {};}
\DeclareSymbol{Xi4_2}{2}{\draw (0,0) node[xic] {} -- (-1,1) node[xi] {} -- (0,2) node[xi] {} -- (-1,3) node[xic] {};}
\DeclareSymbol{Xi2X}{-2}{\draw (0,-0.25) node[xi] {} -- (-1,1) node[xix] {};}
\DeclareSymbol{XXi2}{-2}{\draw (0,-0.25) node[xix] {} -- (-1,1) node[xi] {};}
\DeclareSymbol{IIXi}{0}{\draw (0,-0.25) node[not] {} -- (-1,1) node[xi] {} -- (0,2) node[xi] {};}
\DeclareSymbol{IXi^2}{-1}{\draw (-1,1) node[xi] {} -- (0,0) node[not] {} -- (1,1) node[xi] {};}
\DeclareSymbol{IIXi^2}{-4}{\draw (0,-1.5) node[not] {} -- (0,0);
\draw[kernels2] (-1,1) node[xi] {} -- (0,0) node[not] {} -- (1,1) node[xi] {};}
\DeclareSymbol{XiX}{-2.8}{\node[xibx] {};}
\DeclareSymbol{tauX}{-2.8}{ \node[X] {};}
\DeclareSymbol{Xi}{-2.8}{\node[xib] {};}

\DeclareSymbol{IXiX}{-1}{\draw (0,-0.25) node[not] {} -- (-1,1) node[xix] {};}
\DeclareSymbol{IXi3}{2}{\draw (0,-0.25) node[not] {} -- (-1,1) node[xi] {} -- (0,2) node[xi] {} -- (-1,3) node[xi] {};}
\DeclareSymbol{IXi}{-2}{\draw (0,-0.25) node[not] {} -- (-1,1) node[xi] {};}
\DeclareSymbol{XiI}{-2}{\draw (0,-0.25) node[xi] {} -- (-1,1) node[not] {};}

\DeclareSymbol{Xi4b}{0}{\draw(0,1.5) node[xi] {} -- (0,0); \draw (-1,1) node[xi] {} -- (0,0) node[xi] {} -- (1,1) node[xi] {};}
\DeclareSymbol{Xi4b'}{0}{\draw(0,1.5) node[xi] {} -- (0,-0.2); \draw (-1,1) node[xi] {} -- (0,-0.2) node[not] {} -- (1,1) node[xi] {};}
\DeclareSymbol{Xi4c}{0}{\draw (0,1) -- (0.8,2.2) node[xi] {};\draw (0,-0.25) node[xi] {} -- (0,1) node[xi] {} -- (-0.8,2.2) node[xi] {};}
\DeclareSymbol{Xi4d}{-4.5}{\draw (0,-1.5) node[not] {} -- (0,0); \draw (-1,1) node[xi] {} -- (0,0) node[xi] {} -- (1,1) node[xi] {};}
\DeclareSymbol{Xi4e}{0}{\draw (0,2) node[xi] {} -- (-1,1) node[xi] {} -- (0,0) node[xi] {} -- (1,1) node[xi] {};}
\DeclareSymbol{Xi4e'}{0}{\draw (0,2) node[xi] {} -- (-1,1) node[xi] {} -- (0,-0.2) node[not] {} -- (1,1) node[xi] {};}

\DeclareSymbol{Xitwo}
{0}{\draw[kernels2] (0,0) node[not] {} -- (-1,1) node[not] {}
-- (-2,2) node[not]{} -- (-3,3) node[xi]  {};
\draw[kernels2] (0,0) -- (1,1) node[xi] {};
\draw[kernels2] (-1,1) -- (0,2) node[xi] {};
\draw[kernels2] (-2,2) -- (-1,3) node[xi] {};}

\DeclareSymbol{IXitwo}
{0}{\draw (-.7,1.2) node[xi] {} -- (0,-0.2) -- (.7,1.2) node[xi] {};}
\DeclareSymbol{I1Xitwo}
{0}{\draw[kernels2] (0,0) node[not] {} -- (-1,1) node[xi] {};
\draw[kernels2] (0,0) -- (1,1) node[xi] {};}

\DeclareSymbol{I1Xitwobis}
{0}{\draw[kernels2] (0,0) node[not] {} -- (-1,1) node[xies] {};
\draw[kernels2] (0,0) -- (1,1) node[xies] {};}

\DeclareSymbol{I1Xitwog}
{0}{\draw[kernels2] (0,0) node[not] {} -- (-1,1) node[xies] {};
\draw[kernels2] (0,0) -- (1,1) node[xi] {};}

\DeclareSymbol{cI1Xitwo}
{0}{\draw[kernels2] (0,0) node[not] {} -- (-1,1) node[xic] {};
\draw[kernels2] (0,0) -- (1,1) node[xi] {};}

\DeclareSymbol{I1IXi3}{0}{\draw (0,0) node[xi] {} -- (-1,1) ; 
\draw[kernels2] (-1,1) node[not] {} -- (0,2) node[xi] {};
\draw[kernels2] (-1,1) node[not] {} -- (-2,2) node[xi] {};}

\DeclareSymbol{I1Xi3c}{-1}{\draw[kernels2](0,1.5) node[xi] {} -- (0,0) node[not] {}; \draw (-1,1) node[xi] {} -- (0,0) ; \draw[kernels2] (0,0) -- (1,1) node[xi] {};}

\DeclareSymbol{I1Xi3cbis}{-1}{\draw[kernels2](0,1.5) node[xies] {} -- (0,0) node[not] {}; \draw (-1,1) node[xies] {} -- (0,0) ; \draw[kernels2] (0,0) -- (1,1) node[xies] {};}

\DeclareSymbol{I1IXi3b}{0}{\draw[kernels2] (0,0) node[not] {} -- (-1,1) ; \draw[kernels2] (0,0)   -- (1,1) node[xi] {} ;
\draw (-1,1) node[xi] {} -- (0,2) node[xi] {};
}

\DeclareSymbol{I1IXi3c}{0}{\draw[kernels2] (0,0) node[not] {} -- (-1,1) ; \draw[kernels2] (0,0)   -- (1,1) node[xi] {} ;
\draw[kernels2] (-1,1) node[not] {} -- (0,2) node[xi] {};
\draw[kernels2] (-1,1) node[not] {} -- (-2,2) node[xi] {};}

\DeclareSymbol{I1IXi3cbis}{0}{\draw[kernels2] (0,0) node[not] {} -- (-1,1) ; \draw[kernels2] (0,0)   -- (1,1) node[xies] {} ;
\draw[kernels2] (-1,1) node[not] {} -- (0,2) node[xies] {};
\draw[kernels2] (-1,1) node[not] {} -- (-2,2) node[xies] {};}

\DeclareSymbol{I1Xi}{0}{\draw[kernels2] (0,0) node[not] {} -- (-1,1)  node[xi] {} ;}

\DeclareSymbol{I1Xi4a}{2}{\draw[kernels2] (0,0) node[not] {} -- (-1,1) ; \draw[kernels2] (0,0) node[not] {} -- (1,1) node[xi] {} ;
\draw (-1,1) node[xi] {} -- (0,2) node[xi] {} -- (-1,3) node[xi] {};}

\DeclareSymbol{cI1Xi4a}{2}{\draw[kernels2] (0,0) node[not] {} -- (-1,1) ; \draw[kernels2] (0,0) node[not] {} -- (1,1) node[xic] {} ;
\draw (-1,1) node[xic] {} -- (0,2) node[xi] {} -- (-1,3) node[xi] {};}

\DeclareSymbol{I1Xi4b}{2}{\draw (0,0) node[xi] {} -- (-1,1) node[xi] {} -- (0,2) ; \draw[kernels2] (0,2) node[not] {} -- (-1,3) node[xi] {};\draw[kernels2] (0,2)  -- (1,3) node[xi] {};
}

\DeclareSymbol{cI1Xi4b}{2}{\draw (0,0) node[xic] {} -- (-1,1) node[xic] {} -- (0,2) ; \draw[kernels2] (0,2) node[not] {} -- (-1,3) node[xi] {};\draw[kernels2] (0,2)  -- (1,3) node[xi] {};
}

\DeclareSymbol{I1Xi4c}{2}{\draw (0,0) node[xi] {} -- (-1,1) node[not] {}; \draw[kernels2] (-1,1) -- (0,2) ; 
\draw[kernels2] (-1,1) -- (-2,2) node[xi] {} ;
\draw (0,2) node[xi] {} -- (-1,3) node[xi] {};}

\DeclareSymbol{cI1Xi4c}{2}{\draw (0,0) node[xic] {} -- (-1,1) node[not] {}; \draw[kernels2] (-1,1) -- (0,2) ; 
\draw[kernels2] (-1,1) -- (-2,2) node[xic] {} ;
\draw (0,2) node[xi] {} -- (-1,3) node[xi] {};}

\DeclareSymbol{I1Xi4ab}{2}{\draw[kernels2] (0,0) node[not] {} -- (-1,1) ; \draw[kernels2] (0,0) node[not] {} -- (1,1) node[xi] {};\draw (-1,1) node[xi] {} -- (0,2) ; \draw[kernels2] (0,2) node[not] {} -- (-1,3) node[xi] {};\draw[kernels2] (0,2)  -- (1,3) node[xi] {}; }

\DeclareSymbol{cI1Xi4ab}{2}{\draw[kernels2] (0,0) node[not] {} -- (-1,1) ; \draw[kernels2] (0,0) node[not] {} -- (1,1) node[xic] {};\draw (-1,1) node[xic] {} -- (0,2) ; \draw[kernels2] (0,2) node[not] {} -- (-1,3) node[xi] {};\draw[kernels2] (0,2)  -- (1,3) node[xi] {}; }

\DeclareSymbol{I1Xi4bc}{2}{\draw (0,0) node[xi] {} -- (-1,1) node[not] {}; \draw[kernels2] (-1,1) -- (0,2) ; 
\draw[kernels2] (-1,1) -- (-2,2) node[xi] {} ; \draw[kernels2] (0,2) node[not] {} -- (-1,3) node[xi] {};\draw[kernels2] (0,2)  -- (1,3) node[xi] {};
}

\DeclareSymbol{cI1Xi4bc}{2}{\draw (0,0) node[xic] {} -- (-1,1) node[not] {}; \draw[kernels2] (-1,1) -- (0,2) ; 
\draw[kernels2] (-1,1) -- (-2,2) node[xic] {} ; \draw[kernels2] (0,2) node[not] {} -- (-1,3) node[xi] {};\draw[kernels2] (0,2)  -- (1,3) node[xi] {};
}

\DeclareSymbol{I1Xi4abcc1}{2}{\draw[kernels2] (0,0) node[not] {} -- (-1,1) node[not] {}
-- (-2,2) node[not]{} -- (-3,3) node[xic]  {};
\draw[kernels2] (0,0) -- (1,1) node[xic] {};
\draw[kernels2] (-1,1) -- (0,2) node[xi] {};
\draw[kernels2] (-2,2) -- (-1,3) node[xi] {};
}

\DeclareSymbol{I1Xi4abcc1b}{2}{\draw[kernels2] (0,0) node[not] {} -- (-1,1) node[not] {}
-- (-2,2) node[not]{} -- (-3,3) node[xi]  {};
\draw[kernels2] (0,0) -- (1,1) node[xic] {};
\draw[kernels2] (-1,1) -- (0,2) node[xic] {};
\draw[kernels2] (-2,2) -- (-1,3) node[xi] {};
}

\DeclareSymbol{I1Xi4abcc2}{2}{\draw[kernels2] (0,0) node[not] {} -- (-1,1) node[not] {}
-- (-2,2) node[not]{} -- (-3,3) node[xic]  {};
\draw[kernels2] (0,0) -- (1,1) node[xi] {};
\draw[kernels2] (-1,1) -- (0,2) node[xi] {};
\draw[kernels2] (-2,2) -- (-1,3) node[xic] {};
}

\DeclareSymbol{I1Xi4ac}{2}{\draw[kernels2] (0,0) node[not] {} -- (-1,1) ; \draw[kernels2] (0,0) node[not] {} -- (1,1) node[xi] {}; 
\draw[kernels2] (-1,1) node[not] {} -- (0,2) ; 
\draw[kernels2] (-1,1) -- (-2,2) node[xi] {} ;
\draw (0,2) node[xi] {} -- (-1,3) node[xi] {};}

\DeclareSymbol{cI1Xi4ac}{2}{\draw[kernels2] (0,0) node[not] {} -- (-1,1) ; \draw[kernels2] (0,0) node[not] {} -- (1,1) node[xic] {}; 
\draw[kernels2] (-1,1) node[not] {} -- (0,2) ; 
\draw[kernels2] (-1,1) -- (-2,2) node[xic] {} ;
\draw (0,2) node[xi] {} -- (-1,3) node[xi] {};}

\DeclareSymbol{I1Xi4acc1}{2}{\draw[kernels2] (0,0) node[not] {} -- (-1,1) ; \draw[kernels2] (0,0) node[not] {} -- (1,1) node[xic] {}; 
\draw[kernels2] (-1,1) node[not] {} -- (0,2) ; 
\draw[kernels2] (-1,1) -- (-2,2) node[xi] {} ;
\draw (0,2) node[xic] {} -- (-1,3) node[xi] {};}

\DeclareSymbol{I1Xi4acc2}{2}{\draw[kernels2] (0,0) node[not] {} -- (-1,1) ; \draw[kernels2] (0,0) node[not] {} -- (1,1) node[xic] {}; 
\draw[kernels2] (-1,1) node[not] {} -- (0,2) ; 
\draw[kernels2] (-1,1) -- (-2,2) node[xi] {} ;
\draw (0,2) node[xi] {} -- (-1,3) node[xic] {};}

\DeclareSymbol{2I1Xi4}{2}{\draw[kernels2] (0,0) node[not] {} -- (-1,1) node[not] {};
\draw[kernels2] (0,0) -- (1,1) node[not] {};
\draw[kernels2] (-1,1) -- (-1.5,2.5) node[xi] {};
\draw[kernels2] (-1,1) -- (-0.5,2.5) node[xi] {};
\draw[kernels2] (1,1) -- (0.5,2.5) node[xi] {};
\draw[kernels2] (1,1) -- (1.5,2.5) node[xi] {};
}

\DeclareSymbol{2I1Xi4dis}{2}{\draw[kernels2] (0,0) node[not] {} -- (-1,1) node[not] {};
\draw[kernels2] (0,0) -- (1,1) node[not] {};
\draw[kernels2] (-1,1) -- (-1.5,2.5) node[xies] {};
\draw[kernels2] (-1,1) -- (-0.5,2.5) node[xies] {};
\draw[kernels2] (1,1) -- (0.5,2.5) node[xies] {};
\draw[kernels2] (1,1) -- (1.5,2.5) node[xies] {};
}

\DeclareSymbol{2I1Xi4c1}{2}{\draw[kernels2] (0,0) node[not] {} -- (-1,1) node[not] {};
\draw[kernels2] (0,0) -- (1,1) node[not] {};
\draw[kernels2] (-1,1) -- (-1.5,2.5) node[xic] {};
\draw[kernels2] (-1,1) -- (-0.5,2.5) node[xi] {};
\draw[kernels2] (1,1) -- (0.5,2.5) node[xic] {};
\draw[kernels2] (1,1) -- (1.5,2.5) node[xi] {};
}

\DeclareSymbol{2I1Xi4c2}{2}{\draw[kernels2] (0,0) node[not] {} -- (-1,1) node[not] {};
\draw[kernels2] (0,0) -- (1,1) node[not] {};
\draw[kernels2] (-1,1) -- (-1.5,2.5) node[xic] {};
\draw[kernels2] (-1,1) -- (-0.5,2.5) node[xic] {};
\draw[kernels2] (1,1) -- (0.5,2.5) node[xi] {};
\draw[kernels2] (1,1) -- (1.5,2.5) node[xi] {};
}

\DeclareSymbol{2I1Xi4b}{2}{\draw[kernels2] (0,0) node[not] {} -- (-1,1) ;
\draw[kernels2] (0,0) -- (1,1);
\draw (-1,1) node[xi] {} -- (-1,2.5) node[xi] {};
\draw (1,1)  node[xi] {} -- (1,2.5) node[xi] {};
}

\DeclareSymbol{2I1Xi4bb}{2}{\draw[kernels2] (0,0) node[not] {} -- (-1,1) ;
\draw[kernels2] (0,0) -- (1,1);
\draw (-1,1) node[xi] {} -- (-1,2.5) node[xiesf] {};
\draw (1,1)  node[xi] {} -- (1,2.5) node[xic] {};
}

\DeclareSymbol{2I1Xi4c}{2}{\draw[kernels2] (0,0) node[not] {} -- (-1,1);
\draw[kernels2] (0,0) -- (1,1) node[not] {};
\draw (-1,1)  node[xi] {} -- (-1,2.5) node[xi] {};
\draw[kernels2] (1,1) -- (0.4,2.5) node[xi] {};
\draw[kernels2] (1,1) -- (1.6,2.5) node[xi] {};
}

\DeclareSymbol{2I1Xi4cc1}{2}{\draw[kernels2] (0,0) node[not] {} -- (-1,1);
\draw[kernels2] (0,0) -- (1,1) node[not] {};
\draw (-1,1)  node[xic] {} -- (-1,2.5) node[xi] {};
\draw[kernels2] (1,1) -- (0.4,2.5) node[xic] {};
\draw[kernels2] (1,1) -- (1.6,2.5) node[xi] {};
}

\DeclareSymbol{2I1Xi4cc2}{2}{\draw[kernels2] (0,0) node[not] {} -- (-1,1);
\draw[kernels2] (0,0) -- (1,1) node[not] {};
\draw (-1,1)  node[xic] {} -- (-1,2.5) node[xic] {};
\draw[kernels2] (1,1) -- (0.4,2.5) node[xi] {};
\draw[kernels2] (1,1) -- (1.6,2.5) node[xi] {};
}

\DeclareSymbol{Xi4ba}{0}{\draw(-0.5,1.5) node[xi] {} -- (0,0); \draw (-1.5,1) node[xi] {} -- (0,0) node[not] {}; \draw[kernels2] (0,0) -- (1.5,1) node[xi] {};
\draw[kernels2] (0,0) -- (0.5,1.5) node[xi] {} ;}

\DeclareSymbol{Xi4badis}{0}{\draw(-0.5,1.5) node[xies] {} -- (0,0); \draw (-1.5,1) node[xies] {} -- (0,0) node[not] {}; \draw[kernels2] (0,0) -- (1.5,1) node[xies] {};
\draw[kernels2] (0,0) -- (0.5,1.5) node[xies] {} ;}

\DeclareSymbol{Xi4ba1}{0}{\draw(-0.5,1.5) node[xi] {} -- (0,0); \draw (-1.5,1) node[xi] {} -- (0,0) node[not] {}; \draw[kernels2] (0,0) -- (1.5,1) node[xic] {};
\draw[kernels2] (0,0) -- (0.5,1.5) node[xic] {} ;}

\DeclareSymbol{Xi4ba1b}{0}{\draw(-0.5,1.5) node[xic] {} -- (0,0); \draw (-1.5,1) node[xic] {} -- (0,0) node[not] {}; \draw[kernels2] (0,0) -- (1.5,1) node[xi] {};
\draw[kernels2] (0,0) -- (0.5,1.5) node[xi] {} ;}

\DeclareSymbol{Xi4ba1bdiff}{0}{\draw(-0.5,1.5) node[xic] {} -- (0,0); \draw (-1.5,1) node[xic] {} -- (0,0) node[not] {}; \draw (0,0) -- (1.5,1) node[xi] {};
\draw (0,0) -- (0.5,1.5) node[xi] {};
\draw(0,0) node[diff] {};}

\DeclareSymbol{Xi4ba1bb}{0}{\draw(-0.5,1.5) node[xic] {} -- (0,0); \draw (-1.5,1) node[xiesf] {} -- (0,0) node[not] {}; \draw[kernels2] (0,0) -- (1.5,1) node[xi] {};
\draw[kernels2] (0,0) -- (0.5,1.5) node[xi] {} ;}

\DeclareSymbol{Xi4ba2}{0}{\draw(-0.5,1.5) node[xi] {} -- (0,0); \draw (-1.5,1) node[xic] {} -- (0,0) node[not] {}; \draw[kernels2] (0,0) -- (1.5,1) node[xi] {};
\draw[kernels2] (0,0) -- (0.5,1.5) node[xic] {} ;}

\DeclareSymbol{Xi4ba2b}{0}{\draw(-0.5,1.5) node[xi] {} -- (0,0); \draw (-1.5,1) node[xic] {} -- (0,0) node[not] {}; \draw[kernels2] (0,0) -- (1.5,1) node[xi] {};
\draw[kernels2] (0,0) -- (0.5,1.5) node[xiesf] {} ;}

\DeclareSymbol{Xi4ca}{0}{\draw (0,1) -- (-1,2.2) node[xi] {};\draw (0,-0.25) node[xi] {} -- (0,1) ; \draw[kernels2] (0,1) node[not] {} -- (1,2.2) node[xi] {};
\draw[kernels2] (0,1) {} -- (0,2.7) node[xi] {};
}

\DeclareSymbol{Xi4cb}{0}{\draw (-1,1) -- (-2,2) node[xi] {};\draw[kernels2] (0,0)  -- (-1,1) node[xi] {} ; \draw[kernels2] (0,0) node[not] {} -- (1,1) node[xi] {} ; 
\draw (-1,1) node[xi] {} -- (0,2) node[xi] {};}

\DeclareSymbol{Xi4cbb}{0}{\draw (-1,1) -- (-2,2) node[xiesf] {};\draw[kernels2] (0,0)  -- (-1,1) node[xi] {} ; \draw[kernels2] (0,0) node[not] {} -- (1,1) node[xi] {} ; 
\draw (-1,1) node[xi] {} -- (0,2) node[xic] {};}

\DeclareSymbol{Xi4cbc1}{0}{\draw (-1,1) -- (-2,2) node[xic] {};\draw[kernels2] (0,0)  -- (-1,1) node[xic] {} ; \draw[kernels2] (0,0) node[not] {} -- (1,1) node[xi] {} ; 
\draw (-1,1) node[xic] {} -- (0,2) node[xi] {};}

\DeclareSymbol{Xi4cbc2}{0}{\draw (-1,1) -- (-2,2) node[xi] {};\draw[kernels2] (0,0)  -- (-1,1) node[xi] {} ; \draw[kernels2] (0,0) node[not] {} -- (1,1) node[xic] {} ; 
\draw (-1,1) node[xic] {} -- (0,2) node[xi] {};}

\DeclareSymbol{Xi4cab}{0}{\draw (-1,1) -- (-2,2) node[xi] {};\draw[kernels2] (0,0)  -- (-1,1); \draw[kernels2] (0,0) node[not] {} -- (1,1) node[xi] {} ; 
\draw[kernels2] (-1,1)  {} -- (0,2) node[xi] {};
\draw[kernels2] (-1,1) node[not] {} -- (-1,2.5) node[xi] {};
}

\DeclareSymbol{Xi4cabdis}{0}{\draw (-1,1) -- (-2,2) node[xies] {};\draw[kernels2] (0,0)  -- (-1,1); \draw[kernels2] (0,0) node[not] {} -- (1,1) node[xies] {} ; 
\draw[kernels2] (-1,1)  {} -- (0,2) node[xies] {};
\draw[kernels2] (-1,1) node[not] {} -- (-1,2.5) node[xies] {};
}

\DeclareSymbol{Xi4cabc1}{0}{\draw (-1,1) -- (-2,2) node[xi] {};\draw[kernels2] (0,0)  -- (-1,1); \draw[kernels2] (0,0) node[not] {} -- (1,1) node[xic] {} ; 
\draw[kernels2] (-1,1)  {} -- (0,2) node[xic] {};
\draw[kernels2] (-1,1) node[not] {} -- (-1,2.5) node[xi] {};
}

\DeclareSymbol{Xi4cabc2}{0}{\draw (-1,1) -- (-2,2) node[xic] {};\draw[kernels2] (0,0)  -- (-1,1); \draw[kernels2] (0,0) node[not] {} -- (1,1) node[xic] {} ; 
\draw[kernels2] (-1,1)  {} -- (0,2) node[xi] {};
\draw[kernels2] (-1,1) node[not] {} -- (-1,2.5) node[xi] {};
}

\DeclareSymbol{Xi4ea}{1.5}{\draw (-1,2.5) node[xi] {} -- (-1,1) node[xi] {} -- (0,0); 
 \draw[kernels2] (0,0)  -- (1,1) node[xi] {};
\draw[kernels2] (0,0) node[not] {} -- (0,1.5) node[xi] {}; }

\DeclareSymbol{Xi4eac1}{1.5}{\draw (-1,2.5) node[xic] {} -- (-1,1) node[xi] {} -- (0,0); 
 \draw[kernels2] (0,0)  -- (1,1) node[xic] {};
\draw[kernels2] (0,0) node[not] {} -- (0,1.5) node[xi] {}; }

\DeclareSymbol{Xi4eac1b}{1.5}{\draw (-1,2.5) node[xic] {} -- (-1,1) node[xi] {} -- (0,0); 
 \draw[kernels2] (0,0)  -- (1,1) node[xiesf] {};
\draw[kernels2] (0,0) node[not] {} -- (0,1.5) node[xi] {}; }

\DeclareSymbol{Xi4eac2}{1.5}{\draw (-1,2.5) node[xic] {} -- (-1,1) node[xic] {} -- (0,0); 
 \draw[kernels2] (0,0)  -- (1,1) node[xi] {};
\draw[kernels2] (0,0) node[not] {} -- (0,1.5) node[xi] {}; }

\DeclareSymbol{Xi4eact1}{1.5}{\draw (-1,2.5) node[xic] {} -- (-1,1) node[xi] {} -- (0,0); 
 \draw (0,0)  -- (1,1) node[xic] {};
\draw[rho] (0,0) node[not] {} -- (0,1.5) node[xi] {}; }

\DeclareSymbol{Xi4eact2}{1.5}{\draw[rho] (-1,2.5) node[xic] {} -- (-1,1) node[xi] {} -- (0,0); 
 \draw (0,0)  -- (1,1) node[xic] {};
\draw (0,0) node[not] {} -- (0,1.5) node[xi] {}; }

\DeclareSymbol{Xi4eabis}{1.5}{\draw (-1,2.5) node[xi] {} -- (-1,1) ; \draw[kernels2] (-1,1) node[xi] {} -- (0,0); 
 \draw (0,0)  -- (1,1) node[xi] {};
\draw[kernels2] (0,0) node[not] {} -- (0,1.5) node[xi] {}; }

\DeclareSymbol{Xi4eabisc1}{1.5}{\draw (-1,2.5) node[xic] {} -- (-1,1) ; \draw[kernels2] (-1,1) node[xi] {} -- (0,0); 
 \draw (0,0)  -- (1,1) node[xi] {};
\draw[kernels2] (0,0) node[not] {} -- (0,1.5) node[xic] {}; }

\DeclareSymbol{Xi4eabisc1b}{1.5}{\draw (-1,2.5) node[xic] {} -- (-1,1) ; \draw[kernels2] (-1,1) node[xi] {} -- (0,0); 
 \draw (0,0)  -- (1,1) node[xi] {};
\draw[kernels2] (0,0) node[not] {} -- (0,1.5) node[xiesf] {}; }

\DeclareSymbol{Xi4eabisc1bis}{1.5}{\draw (-1,2.5) node[xi] {} -- (-1,1) ; \draw[kernels2] (-1,1) node[xi] {} -- (0,0); 
 \draw (0,0)  -- (1,1) node[xi] {};
\draw[kernels2] (0,0) node[not] {} -- (0,1.5) node[xi] {};
\draw (-2,1) node[] {\tiny{$i$}};
\draw (-2,2.5) node[] {\tiny{$\ell$}};
\draw (2,1) node[] {\tiny{$k$}};
\draw (0,2.5) node[] {\tiny{$j$}};
 }

\DeclareSymbol{Xi4eabisc1tris}{1.5}{\draw (-1,2.5) node[xi] {} -- (-1,1) ; \draw[kernels2] (-1,1) node[xi] {} -- (0,0); 
 \draw (0,0)  -- (1,1) node[xi] {};
\draw[kernels2] (0,0) node[not] {} -- (0,1.5) node[xi] {};
\draw (-2,1) node[] {\tiny{i}};
\draw (-2,2.5) node[] {\tiny{j}};
\draw (2,1) node[] {\tiny{j}};
\draw (0,2.5) node[] {\tiny{i}};
 }

\DeclareSymbol{Xi4eabisc1quater}{1.5}{\draw (-1,2.5) node[xic] {} -- (-1,1) ; \draw[kernels2] (-1,1) node[xi] {} -- (0,0); 
 \draw (0,0)  -- (1,1) node[xic] {};
\draw[kernels2] (0,0) node[not] {} -- (0,1.5) node[xi] {};
 }

\DeclareSymbol{Xi4eabisc2}{1.5}{\draw (-1,2.5) node[xic] {} -- (-1,1) ; \draw[kernels2] (-1,1) node[xi] {} -- (0,0); 
 \draw (0,0)  -- (1,1) node[xic] {};
\draw[kernels2] (0,0) node[not] {} -- (0,1.5) node[xi] {}; }

\DeclareSymbol{Xi4eabisc2l}{1.5}{\draw (-1,2.5) node[xiesf] {} -- (-1,1) ; \draw[kernels2] (-1,1) node[xi] {} -- (0,0); 
 \draw (0,0)  -- (1,1) node[xic] {};
\draw[kernels2] (0,0) node[not] {} -- (0,1.5) node[xi] {}; }

\DeclareSymbol{Xi4eabisc2r}{1.5}{\draw (-1,2.5) node[xic] {} -- (-1,1) ; \draw[kernels2] (-1,1) node[xi] {} -- (0,0); 
 \draw (0,0)  -- (1,1) node[xiesf] {};
\draw[kernels2] (0,0) node[not] {} -- (0,1.5) node[xi] {}; }

\DeclareSymbol{Xi4eabisc3}{1.5}{\draw (-1,2.5) node[xic] {} -- (-1,1) ; \draw[kernels2] (-1,1) node[xic] {} -- (0,0); 
 \draw (0,0)  -- (1,1) node[xi] {};
\draw[kernels2] (0,0) node[not] {} -- (0,1.5) node[xi] {}; }

\DeclareSymbol{Xi4eb}{0}{
\draw[kernels2] (0,2) node[xi] {} -- (-1,1) ; \draw[kernels2] (-2,2)  node[xi] {} -- (-1,1) ; \draw (-1,1)  node[not] {} -- (0,0); 
 \draw (0,0) node[xi] {}  -- (1,1) node[xi] {};
}

\DeclareSymbol{Xi4eab}{1.5}{\draw[kernels2] (-1,2.5) node[xi] {} -- (-1,1) ; \draw[kernels2] (-2,2)  node[xi] {} -- (-1,1) ; \draw (-1,1)  node[not] {} -- (0,0); 
 \draw[kernels2] (0,0)  -- (1,1) node[xi] {};
\draw[kernels2] (0,0) node[not] {} -- (0,1.5) node[xi] {}; 
}

\DeclareSymbol{Xi4eabdis}{1.5}{\draw[kernels2] (-1,2.5) node[xies] {} -- (-1,1) ; \draw[kernels2] (-2,2)  node[xies] {} -- (-1,1) ; \draw (-1,1)  node[not] {} -- (0,0); 
 \draw[kernels2] (0,0)  -- (1,1) node[xies] {};
\draw[kernels2] (0,0) node[not] {} -- (0,1.5) node[xies] {}; 
}

\DeclareSymbol{Xi4eabc1}{1.5}{\draw[kernels2] (-1,2.5) node[xic] {} -- (-1,1) ; \draw[kernels2] (-2,2)  node[xi] {} -- (-1,1) ; \draw (-1,1)  node[not] {} -- (0,0); 
 \draw[kernels2] (0,0)  -- (1,1) node[xic] {};
\draw[kernels2] (0,0) node[not] {} -- (0,1.5) node[xi] {}; 
}

\DeclareSymbol{Xi4eabc2}{1.5}{\draw[kernels2] (-1,2.5) node[xi] {} -- (-1,1) ; \draw[kernels2] (-2,2)  node[xi] {} -- (-1,1) ; \draw (-1,1)  node[not] {} -- (0,0); 
 \draw[kernels2] (0,0)  -- (1,1) node[xic] {};
\draw[kernels2] (0,0) node[not] {} -- (0,1.5) node[xic] {}; 
}

\DeclareSymbol{Xi4eabbis}{1.5}{\draw[kernels2] (-1,2.5) node[xi] {} -- (-1,1) ; \draw[kernels2] (-2,2)  node[xi] {} -- (-1,1) ; \draw[kernels2] (-1,1)  node[not] {} -- (0,0); 
 \draw (0,0)  -- (1,1) node[xi] {};
\draw[kernels2] (0,0) node[not] {} -- (0,1.5) node[xi] {}; 
}

\DeclareSymbol{Xi4eabbisc1}{1.5}{\draw[kernels2] (-1,2.5) node[xic] {} -- (-1,1) ; \draw[kernels2] (-2,2)  node[xi] {} -- (-1,1) ; \draw[kernels2] (-1,1)  node[not] {} -- (0,0); 
 \draw (0,0)  -- (1,1) node[xic] {};
\draw[kernels2] (0,0) node[not] {} -- (0,1.5) node[xi] {}; 
}

\DeclareSymbol{Xi4eabbisc1perm}{1.5}{\draw[kernels2] (-1,2.5) node[xi] {} -- (-1,1) ; \draw[kernels2] (-2,2)  node[xic] {} -- (-1,1) ; \draw[kernels2] (-1,1)  node[not] {} -- (0,0); 
 \draw (0,0)  -- (1,1) node[xic] {};
\draw[kernels2] (0,0) node[not] {} -- (0,1.5) node[xi] {}; 
}

\DeclareSymbol{Xi4eabbisc2}{1.5}{\draw[kernels2] (-1,2.5) node[xi] {} -- (-1,1) ; \draw[kernels2] (-2,2)  node[xi] {} -- (-1,1) ; \draw[kernels2] (-1,1)  node[not] {} -- (0,0); 
 \draw (0,0)  -- (1,1) node[xic] {};
\draw[kernels2] (0,0) node[not] {} -- (0,1.5) node[xic] {}; 
}

\DeclareSymbol{Xi2cbis}{0}{\draw[kernels2] (0,1) -- (0.8,2.2) node[xi] {};\draw[kernels2] (0,-0.25) node[not] {} -- (0,1); \draw[kernels2] (0,1) node[not] {} -- (-0.8,2.2) node[xi] {};}

\DeclareSymbol{Xi2cbis1}{0}{\draw (0,1) -- (-0.8,2.2) node[xi] {};\draw[kernels2] (0,-0.25) node[not] {} -- (0,1) node[xi] {}; }

\DeclareSymbol{Xi2Xbis}{-2}{\draw[kernels2] (0,-0.25)  -- (-1,1) ; \draw (-1,1) node[xix] {};
\draw[kernels2] (0,-0.25) node[not] {} -- (1,1) node[xi] {};}

\DeclareSymbol{XXi2bis}{-2}{\draw[kernels2] (0,-0.25) -- (-1,1) node[xi] {};
\draw[kernels2] (0,-0.25) node[X] {} -- (1,1) node[xi] {};}

\DeclareSymbol{I1XiIXi}{0}{\draw[kernels2] (0,-0.25) -- (1,1) node[xi] {};
\draw (0,-0.25) node[not] {} -- (-1,1) node[xi] {};}

\DeclareSymbol{I1XiIXib}{0}{\draw  (0,-0.25) node[xi] {} -- (0,1) node[not] {};
\draw[kernels2] (0,1) -- (0,2.25) ; \draw (0,2.25) node[xi]{}; }

\DeclareSymbol{I1XiIXic}{0}{
\draw[kernels2] (0,0) -- (1,1) node[xi] {} ; 
\draw[kernels2] (0,0) node[not] {}  -- (-1,1) node[not] {} -- (0,2) node[xi] {};
}

\DeclareSymbol{thin}{1.4}{\draw[pagebackground] (-0.3,0) -- (0.3,0); \draw  (0,0) -- (0,2);}
\DeclareSymbol{thin2}{1.4}{\draw[pagebackground] (-0.3,0) -- (0.3,0); \draw[tinydots]  (0,0) -- (0,2);}

\DeclareSymbol{thick}{1.4}{\draw[pagebackground] (-0.3,0) -- (0.3,0); \draw[kernels2]  (0,0) -- (0,2);}

\DeclareSymbol{thick2}{1.4}{\draw[pagebackground] (-0.3,0) -- (0.3,0); \draw[kernels2,tinydots]  (0,0) -- (0,2);}

\DeclareSymbol{Xi4ind}{2}{\draw (0,0) node[xi,label={[label distance=-0.2em]right: \scriptsize  $ i $}]  { } -- (-1,1) node[xi,label={[label distance=-0.2em]left: \scriptsize  $ j $}] {} -- (0,2) node[xi,label={[label distance=-0.2em]right: \scriptsize  $ k $}] {} -- (-1,3) node[xi,label={[label distance=-0.2em]left: \scriptsize  $ \ell $}] {};}

\DeclareSymbol{Xi4c1}{2}{\draw (0,0) node[xic] {} -- (-1,1) node[xi] {} -- (0,2) node[xic] {} -- (-1,3) node[xi] {};} 
\DeclareSymbol{IXi2ex}{0}{\draw (0,-0.25) node[xie] {} -- (-1,1) node[xi] {} ; \draw (0,-0.25)-- (1,1) node[xi] {};}
\DeclareSymbol{IXi2ex1}{0}{\draw (0,-0.25) node[xie] {} -- (-1,1) node[xi] {} -- (0,2) node[xi] {};}

\DeclareSymbol{Xi4b1}{0}{\draw(0,1.5) node[xic] {} -- (0,0); \draw (-1,1) node[xic] {} -- (0,0) node[xi] {} -- (1,1) node[xi] {};}

\DeclareSymbol{Xi4ec1}{0}{\draw (0,2) node[xi] {} -- (-1,1) node[xic] {} -- (0,0) node[xic] {} -- (1,1) node[xi] {};}
\DeclareSymbol{Xi4ec2}{0}{\draw (0,2) node[xic] {} -- (-1,1) node[xi] {} -- (0,0) node[xic] {} -- (1,1) node[xi] {};}
\DeclareSymbol{Xi4ec3}{0}{\draw (0,2) node[xic] {} -- (-1,1) node[xic] {} -- (0,0) node[xi] {} -- (1,1) node[xi] {};}

\DeclareSymbol{I1Xi4ac1}{2}{\draw[kernels2] (0,0) node[not] {} -- (-1,1) ; \draw[kernels2] (0,0) node[not] {} -- (1,1) node[xic] {} ;
\draw (-1,1) node[xi] {} -- (0,2) node[xic] {} -- (-1,3) node[xi] {};}

\DeclareSymbol{I1Xi4ac2}{2}{\draw[kernels2] (0,0) node[not] {} -- (-1,1) ; \draw[kernels2] (0,0) node[not] {} -- (1,1) node[xic] {} ;
\draw (-1,1) node[xi] {} -- (0,2) node[xi] {} -- (-1,3) node[xic] {};}

\DeclareSymbol{I1Xi4bp}{2}{\draw (0,0) node[not] {} -- (-1,1) node[xi] {} -- (0,2) ; \draw[kernels2] (0,2) node[not] {} -- (-1,3) node[xi] {};\draw[kernels2] (0,2)  -- (1,3) node[xi] {};
}

\DeclareSymbol{I1Xi4bc1}{2}{\draw (0,0) node[xic] {} -- (-1,1) node[xi] {} -- (0,2) ; \draw[kernels2] (0,2) node[not] {} -- (-1,3) node[xi] {};\draw[kernels2] (0,2)  -- (1,3) node[xic] {};
}

\DeclareSymbol{I1Xi4bc2}{2}{\draw (0,0) node[xic] {} -- (-1,1) node[xi] {} -- (0,2) ; \draw[kernels2] (0,2) node[not] {} -- (-1,3) node[xic] {};\draw[kernels2] (0,2)  -- (1,3) node[xi] {};
}

\DeclareSymbol{I1Xi4cp}{2}{\draw (0,0) node[not] {} -- (-1,1) node[not] {}; \draw[kernels2] (-1,1) -- (0,2) ; 
\draw[kernels2] (-1,1) -- (-2,2) node[xi] {} ;
\draw (0,2) node[xi] {} -- (-1,3) node[xi] {};}

\DeclareSymbol{I1Xi4cc1}{2}{\draw (0,0) node[xic] {} -- (-1,1) node[not] {}; \draw[kernels2] (-1,1) -- (0,2) ; 
\draw[kernels2] (-1,1) -- (-2,2) node[xi] {} ;
\draw (0,2) node[xic] {} -- (-1,3) node[xi] {};}

\DeclareSymbol{I1Xi4cc2}{2}{\draw (0,0) node[xic] {} -- (-1,1) node[not] {}; \draw[kernels2] (-1,1) -- (0,2) ; 
\draw[kernels2] (-1,1) -- (-2,2) node[xi] {} ;
\draw (0,2) node[xi] {} -- (-1,3) node[xic] {};}

\DeclareSymbol{I1Xi4abc1}{2}{\draw[kernels2] (0,0) node[not] {} -- (-1,1) ; \draw[kernels2] (0,0) node[not] {} -- (1,1) node[xic] {};\draw (-1,1) node[xi] {} -- (0,2) ; \draw[kernels2] (0,2) node[not] {} -- (-1,3) node[xic] {};\draw[kernels2] (0,2)  -- (1,3) node[xi] {}; }

\DeclareSymbol{I1Xi4abc2}{2}{\draw[kernels2] (0,0) node[not] {} -- (-1,1) ; \draw[kernels2] (0,0) node[not] {} -- (1,1) node[xic] {};\draw (-1,1) node[xi] {} -- (0,2) ; \draw[kernels2] (0,2) node[not] {} -- (-1,3) node[xi] {};\draw[kernels2] (0,2)  -- (1,3) node[xic] {}; }

\DeclareSymbol{R1}{0}{\draw (-1,1) node[xi] {} -- (0,0) node[not] {};
\draw[kernels2] (0,1.5) node[xic] {} -- (0,0) -- (1,1) node[xic] {};}
\DeclareSymbol{R2}{0}{\draw (-1,1) node[xic] {} -- (0,0) node[not] {};
\draw[kernels2] (0,1.5)  {} -- (0,0) -- (1,1)  {};
\draw (0,1.5) node[xi] {};
\draw (1,1) node[xic] {};
}
\DeclareSymbol{R3}{1}{\draw[kernels2] (-1,1.5)  {} -- (0,0) node[not] {} -- (1,1.5);
\draw (-1,1.5) node[xi] {};
\draw[kernels2] (0,3) {} -- (1,1.5) -- (2,3)  {};
\draw  (0,3) node[xic] {} ;
\draw (2,3) node[xic] {};}
\DeclareSymbol{R4}{1}{\draw[kernels2] (-1,1.5) node[xic] {} -- (0,0) node[not] {} -- (1,1.5);
\draw[kernels2] (0,3) {} -- (1,1.5) -- (2,3) node[xic] {};
\draw (0,3) node[xi] {};}

\DeclareSymbol{I1Xi4bcp}{2}{\draw (0,0) node[not] {} -- (-1,1) node[not] {}; \draw[kernels2] (-1,1) -- (0,2) ; 
\draw[kernels2] (-1,1) -- (-2,2) node[xi] {} ; \draw[kernels2] (0,2) node[not] {} -- (-1,3) node[xi] {};\draw[kernels2] (0,2)  -- (1,3) node[xi] {};
}

\DeclareSymbol{I1Xi4bcc1}{2}{\draw (0,0) node[xic] {} -- (-1,1) node[not] {}; \draw[kernels2] (-1,1) -- (0,2) ; 
\draw[kernels2] (-1,1) -- (-2,2) node[xi] {} ; \draw[kernels2] (0,2) node[not] {} -- (-1,3) node[xi] {};\draw[kernels2] (0,2)  -- (1,3) node[xic] {};
}

\DeclareSymbol{I1Xi4bcc2}{2}{\draw (0,0) node[xic] {} -- (-1,1) node[not] {}; \draw[kernels2] (-1,1) -- (0,2) ; 
\draw[kernels2] (-1,1) -- (-2,2) node[xi] {} ; \draw[kernels2] (0,2) node[not] {} -- (-1,3) node[xic] {};\draw[kernels2] (0,2)  -- (1,3) node[xi] {};
} 

\DeclareSymbol{2I1Xi4bc1}{2}{\draw[kernels2] (0,0) node[not] {} -- (-1,1) ;
\draw[kernels2] (0,0) -- (1,1);
\draw (-1,1) node[xic] {} -- (-1,2.5) node[xi] {};
\draw (1,1)  node[xic] {} -- (1,2.5) node[xi] {};
}

\DeclareSymbol{2I1Xi4bc2}{2}{\draw[kernels2] (0,0) node[not] {} -- (-1,1) ;
\draw[kernels2] (0,0) -- (1,1);
\draw (-1,1) node[xi] {} -- (-1,2.5) node[xic] {};
\draw (1,1)  node[xic] {} -- (1,2.5) node[xi] {};
}

\DeclareSymbol{diff2I1Xi4bc2}{2}{\draw (0,0) node[diff] {} -- (-1,1) ;
\draw (0,0) -- (1,1);
\draw (-1,1) node[xi] {} -- (-1,2.5) node[xic] {};
\draw (1,1)  node[xic] {} -- (1,2.5) node[xi] {};
}

\DeclareSymbol{2I1Xi4bc3}{2}{\draw[kernels2] (0,0) node[not] {} -- (-1,1) ;
\draw[kernels2] (0,0) -- (1,1);
\draw (-1,1) node[xic] {} -- (-1,2.5) node[xic] {};
\draw (1,1)  node[xi] {} -- (1,2.5) node[xi] {};
}

\DeclareSymbol{Xi41}{0}{\draw (0,1) -- (0.8,2.2) node[xic] {};\draw (0,-0.25) node[xi] {} -- (0,1) node[xi] {} -- (-0.8,2.2) node[xic] {};} 

\DeclareSymbol{Xi42}{0}{\draw (0,1) -- (0.8,2.2) node[xi] {};\draw (0,-0.25) node[xic] {} -- (0,1) node[xi] {} -- (-0.8,2.2) node[xic] {};}

\DeclareSymbol{Xi4ca1}{0}{\draw (0,1) -- (-1,2.2) node[xic] {};\draw (0,-0.25) node[xi] {} -- (0,1) ; \draw[kernels2] (0,1) node[not] {} -- (1,2.2) node[xic] {};
\draw[kernels2] (0,1) {} -- (0,2.7) node[xi] {};
}

\DeclareSymbol{Xi4ca2}{0}{\draw (0,1) -- (-1,2.2) node[xi] {};\draw (0,-0.25) node[xi] {} -- (0,1) ; \draw[kernels2] (0,1) node[not] {} -- (1,2.2) node[xic] {};
\draw[kernels2] (0,1) {} -- (0,2.7) node[xic] {};
}

\DeclareSymbol{Xi4cap}{0}{\draw (0,1) -- (-1,2.2) node[xi] {};\draw (0,-0.25) node[not] {} -- (0,1) ; \draw[kernels2] (0,1) node[not] {} -- (1,2.2) node[xi] {};
\draw[kernels2] (0,1) {} -- (0,2.7) node[xi] {};
}

\DeclareSymbol{Xi3a}{0}{
 \draw (-1,1)  node[xi] {} -- (0,0); 
 \draw (0,0) node[xi] {}  -- (1,1) node[xi] {};
}

\DeclareSymbol{Xi4ebc1}{0}{
\draw[kernels2] (0,2) node[xi] {} -- (-1,1) ; \draw[kernels2] (-2,2)  node[xic] {} -- (-1,1) ; \draw (-1,1)  node[not] {} -- (0,0); 
 \draw (0,0) node[xic] {}  -- (1,1) node[xi] {};
}

\DeclareSymbol{Xi4ebc2}{0}{
\draw[kernels2] (0,2) node[xi] {} -- (-1,1) ; \draw[kernels2] (-2,2)  node[xi] {} -- (-1,1) ; \draw (-1,1)  node[not] {} -- (0,0); 
 \draw (0,0) node[xic] {}  -- (1,1) node[xic] {};
}

\DeclareSymbol{Xi2cbispex}{0}{\draw[kernels2] (0,1) -- (0.8,2.2) node[xi] {};\draw (0,-0.25) node[xie] {} -- (0,1); \draw[kernels2] (0,1) node[not] {} -- (-0.8,2.2) node[xi] {};}

\DeclareSymbol{Xi2cbis1p}{0}{\draw (0,1) -- (-0.8,2.2) node[xi] {};\draw (0,-0.25) node[not] {} -- (0,1) node[xi] {}; }

\DeclareSymbol{Xi2Xp}{-2}{\draw (0,-0.25) node[not] {} -- (-1,1) node[xix] {};} 

\DeclareSymbol{I1XiIXib}{0}{\draw  (0,-0.25) node[xi] {} -- (0,1) node[not] {};
\draw[kernels2] (0,1) -- (0,2.25) ; \draw (0,2.25) node[xi]{}; }

\DeclareSymbol{IXi2b}{0}{\draw  (0,-0.25) node[xi] {} -- (0,1) node[not] {};
\draw (0,1) -- (0,2.25) ; \draw (0,2.25) node[xi]{}; }

\DeclareSymbol{IXi2bex}{0}{\draw  (0,-0.25) node[xi] {} -- (0,1) node[xie] {};
\draw (0,1) -- (0,2.25) ; \draw (0,2.25) node[xi]{}; }

 \def\1{\mathbf{\symbol{1}}}

\DeclareSymbol{diff}{0}{
\draw (0,0.5) node[diff] {};
}

\DeclareSymbol{diff1}{0}{
\draw (0,0.5) node[diff1] {};
}

\DeclareSymbol{diff2}{0}{
\draw (0,0.5) node[diff2] {};
}

\DeclareSymbol{geo}{0}{
\draw (0,0) node[diff] {};
\draw (0.3,0) node[diff] {};
}

\DeclareSymbol{generic}{0}{
\draw (0,0.6) node[xi] {};
}

\DeclareSymbol{g}{0}{
\draw (0,0.6) node[g] {};
}

\DeclareSymbol{Ito}{0}{
\draw (0,0.6) node[xies] {};
}

\DeclareSymbol{Itob}{0}{
\draw (0,0.6) node[xiesf] {};
}

\DeclareSymbol{greycirc}{0}{
\draw (0,0.3) node[xi] {};
}

\DeclareSymbol{not}{0}{
\draw (0,0.6) node[not] {};
\draw[tinydots] (0,0.6) circle (0.8);
}

\DeclareSymbol{genericb}{0}{
\draw (0,0.6) node[xic] {};
}

\DeclareSymbol{bluecirc}{0}{
\draw (0,0.3) node[xic] {};
}

\DeclareSymbol{genericxix}{0}{
\draw (0,0.6) node[xix] {};
}

\DeclareSymbol{genericX}{0}{
\draw (0,0.6) node[X] {};
}

\DeclareSymbol{diffIto}{1}{
\draw  (0,2.5) -- (0,0) ;
\draw (0,-0.1) node[diff] {};
\draw (0,2.5) node[xies] {};
}
\DeclareSymbol{Itodiff}{2}{
\draw(0,2.9) -- (0,-0.2);
\draw (0,2.9) node[diff] {};
\draw (0,-0.1) node[xies] {};
}

\DeclareSymbol{diffgeneric}{1}{
\draw  (0,2.5) -- (0,0) ;
\draw (0,-0.1) node[diff] {};
\draw (0,2.5) node[xi] {};
}

\DeclareSymbol{genericdiff}{2}{
\draw(0,2.9) -- (0,-0.2);
\draw (0,2.9) node[diff] {};
\draw (0,-0.1) node[xi] {};
}

\DeclareSymbol{diffdot}{2}{
\draw  (0,3) -- (0,-0.1) ;
\draw (0,3) node[not] {};
\draw (0,-0.1) node[diff] {};
}

\DeclareSymbol{diffdotmini}{0}{
\draw  (0,0) -- (0,1.2) ;
\draw (0,1.2) node[not] {};
\draw (0,0) node[diffmini] {};
}

\DeclareSymbol{dotdiff}{2}{
\draw[kernelsmod]  (0,3) -- (0,-0.1) ;
\draw (0,3) node[diff] {};
\draw (0,-0.1) node[not] {};
}

\DeclareSymbol{dotdiff1}{2}{
\draw[kernelsmod]  (0,3) -- (0,-0.1) ;
\draw (0,3) node[diff1] {};
\draw (0,-0.1) node[not] {};
}

\DeclareSymbol{dotdiff1mini}{0}{
\draw[kernelsmod]  (0,1.2) -- (0,0) ;
\draw (0,1.2) node[diffmini] {};
\draw (0,0) node[not] {};
}

\DeclareSymbol{dotdiff2}{2}{
\draw (0,3) -- (0,-0.1) ;
\draw (0,3) node[diff] {};
\draw (0,-0.1) node[not] {};
}

\DeclareSymbol{dotdiff2mini}{0}{
\draw (0,1.2) -- (0,0) ;
\draw (0,1.2) node[diffmini] {};
\draw (0,0) node[not] {};
}

\DeclareSymbol{dotdiffstraight}{0}{
\draw  (0,3) -- (0,-0.1) ;
\draw (0,3) node[diff] {};
\draw (0,-0.1) node[not] {};
}

\DeclareSymbol{arbre1}{0}{
\draw  (0,0) -- (1.5,1.5) ;
\draw (1.5,1.5) node[not] {};
\draw (0,0) node[not] {};
}

\DeclareSymbol{arbre2}{0}{
\draw  (0,0) -- (1.5,1.5) ;
\draw[kernelsmod] (0,0) -- (-1.5,1.5);
\draw (1.5,1.5) node[not] {};
\draw (0,0) node[not] {};
\draw (-1.5,1.5) node[xi] {};
}

\DeclareSymbol{arbre3}{0}{
\draw  (0,0) -- (1.5,1.5) ;
\draw[kernelsmod] (1.5,1.5) -- (0,3);
\draw (0,0) node[not] {};
\draw (1.5,1.5) node[not] {};
\draw (0,3) node[xi] {};
}

\DeclareSymbol{treeeval}{0}{
\draw (0,0) -- (1,1);
\draw (0,0) node[xi] {};
\draw (1.25,1.25) node[xi] {};
\draw (-0.6,0.6) node[]{\tiny{$i$}};
\draw (0.65,1.85) node[]{\tiny{$j$}};
}

\DeclareSymbol{testeval}{0}{
\draw (0,0) -- (1,1);
\draw (0,0) -- (-1,1);
\draw (0,0) node[xi] {};
\draw (1.25,1.25) node[xi] {};
\draw (-1.25,1.25) node[xi] {};
\draw (-0.6,-0.6) node[]{\tiny{$i$}};
\draw (0.65,1.85) node[]{\tiny{$j$}};
\draw (-1.95,1.85) node[]{\tiny{$k$}};
}

\DeclareSymbol{treeeval2}{0}{
\draw[kernelsmod] (-0.25,-1) -- (1,0.5) ;
\draw[kernelsmod] (1,0.5) -- (-0.25,2);
\draw (1,0.5) node[diff2] {};
\draw (-0.25,-1) node[not] {};
\draw (-0.25,2) node[xi] {};
\draw (-0.6,1.2) node[]{\tiny{1}};
}

\DeclareSymbol{arbreact}{1}{
\draw (0,0) node[not] {};
\draw[kernelsmod] (0,0) -- (1,1);
\draw[kernelsmod] (0,0) -- (-1,1);
\draw (-1,1) node[xic] {};
\draw  (0,2) -- (1,1) ;
\draw (0,2) node[xic] {};
\draw (1,1) node[xi] {};
}

\DeclareSymbol{arbreact1}{0}{
\draw (0,-1.5) -- (0,0);
\draw[kernelsmod] (0,0) -- (1,1);
\draw[kernelsmod] (0,0) -- (-1,1);
\draw  (0,2) -- (1,1) ;
\draw (0,-1.5) node[diff] {};
\draw (0,0) node[not] {};
\draw (-1,1) node[xic] {};
\draw (0,2) node[xic] {};
\draw (1,1) node[xi] {};
}

\DeclareSymbol{arbreact2}{0}{
\draw (0,-0.75) -- (-1,0.5); 
\draw (0,-0.75) -- (1,0.5);
\draw (0,1.5) -- (1,0.5);
\draw (0,1.5) node[xic] {};
\draw (1,0.5) node[xi] {};
\draw (-1,0.5) node[xic] {};
\draw (0,-0.75) node[diff] {};
}

\DeclareSymbol{arbreact3}{0}{
\draw[kernelsmod] (0,-0.75) -- (-1,0.5); 
\draw[kernelsmod] (0,-0.75) -- (1,0.5);
\draw (0,1.75) -- (1,0.5);
\draw (2,1.75) -- (1,0.5);
\draw (0,1.75) node[xic] {};
\draw (1,0.5) node[diff] {};
\draw (-1,0.5) node[xic] {};
\draw (2,1.75) node[xi] {};
\draw (0,-0.75) node[not] {};
}

\DeclareSymbol{pre_im_I1Xitwo}{0}{
\draw[kernels2] (0,-0.3) node[not] {} -- (-0.6,0.7) ;
\draw[kernels2] (0,-0.3) -- (0.6,0.7);
\draw (0,0.9) node[g] {};
}

\DeclareSymbol{pre_im_cI1Xi4ab}{2}{
\draw[kernels2] (0,-1) node[not] {} -- (-0.6,0) ;
\draw[kernels2] (0,-1) -- (0.6,0);
\draw (0,0.2) node[g] {};
\draw (0,0.6) -- (0,1.5);
\draw[kernels2] (0,1.5) node[not] {} -- (-0.6,2.5) ;
\draw[kernels2] (0,1.5) -- (0.6,2.5);
\draw (0,2.7) node[g] {};
}

\DeclareSymbol{pre_im_I1Xi4acc2}{0}{
\draw[kernels2] (-1,-0.5) node[not] {} -- (-1.6,0.5) ;
\draw[kernels2] (-1,-0.5) -- (-0.4,0.5);
\draw[kernels2] (-1,-0.5) -- (0.2,-1.5) node[not] {} ;
\draw (-1,1.1) -- (-1,2);
\draw[kernels2] (0.2,-1.5) -- (0.2,2);
\draw (-1,0.7) node[g] {};
\draw (-0.3,2.2) node[g] {};
}

\DeclareSymbol{pre_im_I1Xi4abcc2}{2}{
\draw[kernels2] (0,-1) node[not] {} -- (-1,0) node[not] {};
\draw[kernels2] (-1,1.2) node[not] {} -- (-1,0);
\draw[kernels2] (-1,1.2) -- (-1.5,2.5);
\draw[kernels2] (-1,1.2) -- (-0.5,2.5);
\draw[kernels2] (-1,0) -- (0.7,2.5);
\draw[kernels2] (0,-1) -- (1.5,2.5);
\draw (-1,2.7) node[g] {};
\draw (1,2.7) node[g] {};
}

\DeclareSymbol{pre_im_2I1Xi4c1}{2}{
\draw[kernels2] (0,-0.5) node[not] {} -- (-1,0.5) node[not] {};
\draw[kernels2] (0,-0.5) -- (1,0.5) node[not] {};
\draw[kernels2] (-1,0.5) node[not] {}-- (-1.7,2);
\draw[kernels2]  (-1,2) -- (1,0.5);
\draw[kernels2] (-1,0.5) -- (1,2);
\draw[kernels2] (1,0.5) -- (1.7,2);
\draw (-1.2,2.2) node[g] {};
\draw (1.2,2.2) node[g] {};
}

\DeclareSymbol{pre_im_Xi4eabisc2}{2}{
\draw[kernels2] (1.2,-0.5) node[not] {} -- (-0.7,0.8) ;
\draw[kernels2] (1.2,-0.5) -- (0.4,0.8);
\draw (0,1.4)  -- (0,2.2);
\draw (1.2,2.2) -- (1.2,-0.6);
\draw (0,1) node[g] {};
\draw (0.6,2.4) node[g] {};
}

\DeclareSymbol{pre_im_Xi4eabisc22}{2}{
\draw (1.2,-0.5) node[not] {} -- (-0.7,0.8) ;
\draw[kernels2] (1.2,-0.5) -- (0.4,0.8);
\draw (0,1.4)  -- (0,2.2);
\draw[kernels2] (1.2,2.2) -- (1.2,-0.6);
\draw (0,1) node[g] {};
\draw (0.6,2.4) node[g] {};
}

\DeclareSymbol{pre_im_Xi4eabisc222}{2}{
\draw[kernels2] (0.4,-0.5) node[not] {} -- (-0.6,1) ;
\draw[kernels2] (1.2,0) -- (0.3,1);
\draw (0,1.1)  -- (0,2.5);
\draw[kernels2] (1.2,2.5) -- (1.2,0) node[not] {} -- (0.4,-0.6);
\draw (0,1.2) node[g] {};
\draw (0.6,2.5) node[g] {};
}

\DeclareSymbol{pre_im_Xi4eabc2}{2}{
\draw (0,-0.5) node[not] {} -- (-1,0.5) node[not] {};
\draw[kernels2] (-1,0.5) -- (-1.5,2);
\draw[kernels2] (0,-0.5)  -- (0.7,2);
\draw[kernels2] (-1,0.5) -- (-0.5,2);
\draw[kernels2] (0,-0.5) -- (1.5,2);
\draw (-1,2.2) node[g] {};
\draw (1,2.2) node[g] {};
}

\DeclareSymbol{pre_im_Xi4eabbisc2}{2}{
\draw[kernels2] (0,-0.5) node[not] {} -- (-1,0.5) node[not] {};
\draw[kernels2] (-1,0.5) -- (-1.5,2);
\draw[kernels2] (0,-0.5)  -- (0.7,2);
\draw[kernels2] (-1,0.5) -- (-0.5,2);
\draw (0,-0.5) -- (1.5,2);
\draw (-1,2.2) node[g] {};
\draw (1,2.2) node[g] {};
}

\DeclareSymbol{pre_im_I1Xi4abcc1}{2}{
\draw[kernels2] (0,-1) node[not] {} -- (-1,0) node[not] {};
\draw[kernels2] (0,1.1) node[not] {} -- (-1,0);
\draw[kernels2] (-1,0) -- (-1.5,2.5);
\draw[kernels2] (0,1.1) node[not] {} -- (-0.5,2.5);
\draw[kernels2] (0,1.1) -- (0.5,2.5);
\draw[kernels2] (0,-1) -- (1.5,2.5);
\draw (-1,2.7) node[g] {};
\draw (1,2.7) node[g] {};
}

\DeclareSymbol{pre_im_Xi4eabc1}{2}{
\draw (0,-0.5) node[not] {} -- (-1,0.5) node[not] {};
\draw[kernels2] (-1,0.5) -- (-1.5,2);
\draw[kernels2] (0,-0.5)  -- (-0.5,2);
\draw[kernels2] (-1,0.5) -- (0.8,2);
\draw[kernels2] (0,-0.5) -- (1.5,2);
\draw (-1,2.2) node[g] {};
\draw (1,2.2) node[g] {};
}

\DeclareSymbol{pre_im_Xi4ba1b}{2}{
\draw[kernels2] (0,0) node[not] {}  -- (1.8,1.5);
\draw[kernels2] (0,0) -- (0.8,1.5);
\draw (0,-0.1) -- (-1.8,1.5);
\draw (0,-0.1) -- (-0.8,1.5);
\draw (-1,1.7) node[g] {};
\draw (1,1.7) node[g] {};
}

\DeclareSymbol{pre_im_Xi4ba2}{2}{
\draw (0,-0.1) node[not] {}  -- (1.8,1.5);
\draw[kernels2] (0,0) -- (0.8,1.5);
\draw (0,-0.1) -- (-1.8,1.5);
\draw[kernels2] (0,0) -- (-0.8,1.5);
\draw (-1,1.7) node[g] {};
\draw (1,1.7) node[g] {};
}

\DeclareSymbol{pre_im_Xi4cabc2}{2}{
\draw[kernels2] (0,-0.5) node[not] {} -- (-1,0.5) node[not] {};
\draw[kernels2] (-1,0.5) -- (-1.5,2);
\draw (-1,0.5)  -- (0.7,2);
\draw[kernels2] (-1,0.5) -- (-0.5,2);
\draw[kernels2] (0,-0.5) -- (1.5,2);
\draw (-1,2.2) node[g] {};
\draw (1,2.2) node[g] {};
}

\DeclareSymbol{pre_im_Xi4cabc1}{2}{
\draw[kernels2] (0,-0.5) node[not] {} -- (-1,0.5) node[not] {};
\draw (-1,0.5) -- (-1.5,2);
\draw[kernels2] (-1,0.5)  -- (0.7,2);
\draw[kernels2] (-1,0.5) -- (-0.5,2);
\draw[kernels2] (0,-0.5) -- (1.5,2);
\draw (-1,2.2) node[g] {};
\draw (1,2.2) node[g] {};
}

\DeclareSymbol{pre_im_Xi4eabbisc1}{2}{
\draw[kernels2] (0,-0.5) node[not] {} -- (-1,0.5) node[not] {};
\draw[kernels2] (-1,0.5) -- (-1.5,2);
\draw[kernels2] (0,-0.5)  -- (-0.5,2);
\draw[kernels2] (-1,0.5) -- (0.8,2);
\draw (0,-0.5) -- (1.5,2);
\draw (-1,2.2) node[g] {};
\draw (1,2.2) node[g] {};
}

\DeclareSymbol{pre_im_1}{0}{
\draw[kernels2] (0,-0.5) node[not] {} -- (-0.6,0.5) ;
\draw[kernels2] (0,-0.5) -- (0.6,0.5);
\draw (0,1.1)  -- (-0.55,2);
\draw (0,1.1)  -- (0.55,2);
\draw (0,0.7) node[g] {};
\draw (0,2.2) node[g] {};
}

\DeclareSymbol{disconnect}{0}{
\draw[kernels2] (0,-0.5) node[not] {} -- (-0.6,0.5) ;
\draw[kernels2] (0,-0.5) -- (0.6,0.5);
\draw (-0.55,1.1)  -- (-0.55,2.3);
\draw (0.55,2.3) -- (0.55,1.5) -- (1.2,1.5) -- (1.2,3.5) -- (0.55,3.5) -- (0.55,2.7);
\draw (0,0.7) node[g] {};
\draw (0,2.5) node[g] {};
}

\DeclareSymbol{pre_im_2}{2}{\draw[kernels2] (0,0) node[not] {} -- (-1,1) node[not] {};
\draw[kernels2] (0,0) -- (1,1) node[not] {};
\draw[kernels2] (-1,1) -- (-1.5,2.5);
\draw[kernels2] (-1,1) -- (-0.5,2.5);
\draw[kernels2] (1,1) -- (0.5,2.5);
\draw[kernels2] (1,1) -- (1.5,2.5);
\draw (-1,2.7) node[g] {};
\draw (1,2.7) node[g] {};
}

\DeclareSymbol{CX_rec}{0}{
\draw [black] (-0.3,1) to (-0.3,-0.3);
\draw [black] (0.3,1) to (0.3,-0.3);
\draw [black] (-0.3,1) to (-0.3,2.3);
\draw [black] (0.3,1) to (0.3,2.3);
\draw (0,1) node[rec] {};
}

\DeclareSymbol{CX_cerc}{0}{
\draw [black] (0,1) to (0,-0.3);
\draw (0,1) node[cerc] {};
}

\DeclareMathAlphabet{\mathpzc}{OT1}{pzc}{m}{it}

\def\eqref#1{(\ref{#1})}

\makeatletter 
\newcommand*{\bigcdot}{}
\DeclareRobustCommand*{\bigcdot}{%
  \mathbin{\mathpalette\bigcdot@{}}%
}
\newcommand*{\bigcdot@scalefactor}{.5}
\newcommand*{\bigcdot@widthfactor}{1.15}
\newcommand*{\bigcdot@}[2]{%
  \sbox0{$#1\vcenter{}$}
  \sbox2{$#1\cdot\m@th$}%
  \hbox to \bigcdot@widthfactor\wd2{%
    \hfil
    \raise\ht0\hbox{%
      \scalebox{\bigcdot@scalefactor}{%
        \lower\ht0\hbox{$#1\bullet\m@th$}%
      }%
    }%
    \hfil
  }%
}
\makeatother

\def\act{\bigcdot}

\tcbset
{colframe=boxcolor,colback=symbols!7!pagebackground,coltext=pageforeground,
fonttitle=\bfseries,nobeforeafter,center title,size=fbox,boxsep=1.5pt,
top=0mm,bottom=0mm,boxsep=0mm,tcbox raise base}

\def\two{{\<generic>\kern0.05em\<genericb>}}
\def\twoI{{\<Ito>\kern0.05em\<Itob>}}

\def\mail#1{\burlalt{#1}{mailto:#1}}

\usepackage{thmtools}
\usepackage{cleveref}
\begin{document}
\renewcommand\thmcontinues[1]{Continued}

\title{Convergence of space-discretised gKPZ via Regularity Structures}

\author{Yvain Bruned$^1$, Usama Nadeem$^2$}
\institute{ 
 IECL (UMR 7502), Université de Lorraine
 \and University of Edinburgh \\
Email:\ \begin{minipage}[t]{\linewidth}
\mail{yvain.bruned@univ-lorraine.fr}
\\ \mail{M.U.Nadeem@sms.ed.ac.uk}
\end{minipage}}

\maketitle

\begin{abstract}
In this work, we show a convergence result for the discrete formulation of the generalised KPZ equation $\partial_t u = (\Delta u) + g(u)(\nabla u)^2 + k(\nabla u) + h(u) + f(u)\xi_t(x)$, where $\xi$ is real-valued, $\Delta$ is the discrete Laplacian, and $\nabla$ is a discrete gradient, without fixing the spatial dimension. Our convergence result is established within the discrete regularity structures introduced by Hairer and Erhard \cite{EH17}. We extend with new ideas the convergence result found in \cite{MH21} that deals with a discrete form of the Parabolic Anderson model driven by a (rescaled) symmetric simple exclusion process. This is the first time that a discrete generalised KPZ equation is treated and it is a major step toward a general convergence result that will cover a large family of discrete models.
\\[.4em]
\end{abstract}
\setcounter{tocdepth}{2}
\tableofcontents

\section{Introduction}

Martin Hairer's theory of Regularity Structures \cite{reg} has been applied to construct a solution theory for a large subclass of singular stochastic PDEs of the form:

\beq\label{eq:sSPDE}
\partial_t u - \mathcal{L}u = F(u,\nabla u,\xi)
\eeq

where $\mathcal{L}$ is some differential operator and $F$ is some non-linearity affine in the noise $ \xi $. Central to this solution theory is the idea of renormalisation which usually involves subtracting certain counterterms from the equation so as to deal with products of distributions that are undefined in the classical sense. 

Since the original paper \cite{reg}, the programme for generating the solution theory for a given SPDE has been automated. In \cite{BHZ}, the authors explained how to extract \textit{rules} from a given equation, and then proceduralised the construction of a renormalisation group that can affect renormalisation on the equation. In this programme, the choice of the aforementioned counter-terms comes from the classic BPHZ formalism \cite{BP57,KH69,WZ69}. In \cite{CH16}, the authors provide a functionally black-box like result that automatically produces the required stochastic estimates for the renormalised stochastic objects coming from \cite{BHZ}. The last step completed in \cite{BCCH} was to write the fixed point and the action of the renormalisation onto the right hand side of \eqref{eq:sSPDE}. Hence, the programme in the continuum has been automated by this series of papers. 

Many equations of the form \eqref{eq:sSPDE}, arise from scaling limits of microscopic models - consider \cite{BG97,MW17,K16} for derivations of Stochastic Burgers, KPZ, $\Phi^4$ and the parabolic Anderson model - and as such there is interest in discretisations of these equations. A classic example is from \cite{BG97}, where employing an approach based on the Cole-Hopf transform, the authors proved that a (rescaled) particle system converges to their notion of solution of the KPZ equation. In \cite{GJ10,GJ14}, the authors introduced a martingale type approach to the problem - the so-called energy solution. Further, they showed that limits of a sizeable class of suitably scaled interacting particle systems are also energy solutions. The uniqueness of the energy solution has been proved in \cite{GP18} which implies various applications given in \cite{GP16,DGP17}. These approaches that we have listed thus far have certain drawbacks -  for the former the difficulty is that the Hopf-Cole transform is difficult to implement at a discrete level, and for the latter, knowledge and control of the invariant measure for the system is a requisite.

The use of regularity structures circumvents both of these problems and hence has been shown to have great promise for discrete problems such as these. In \cite{HM18}, the authors developed a framework to adapt the theory of regularity structure for certain spatial discretisations. Further generality was achieved in \cite{EH17}, where the authors do not fix any particular discretisation procedure. As applications to the discrete regularity structures constructed in these papers, in \cite{HM18} a space-discretised KPZ was proven to converge to the solution of the generic KPZ, and then in \cite{CM18} the same was proven for a space-time discretisation of the KPZ equation. Convergence of more general discretisations driven by interacting particle systems to stochastic PDEs in paper such as  \cite{Mat18} and \cite{MH21}, was addressed. In the former, the author used multiple stochastic integrals with respect to martingales with particular properties that allow for certain moment bounds and expansions that are analogous to the Nelson's Estimate and Wiener chaos expansions that are used in \cite{reg} in the case of (homogeneous) Gaussian Noise. In the latter, the authors begin with a discrete parabolic Anderson model driven by a symmetric simple exclusion process and show that the solutions to a suitably renormalised version of the equation converge in law to the solution of the same equation driven by a generalised Ornstein-Uhlenbeck process. This is achieved via estimates on joint cumulants of arbitrary large order for the exclusion process. 
 
Another reason that drives interest in discrete regularity structures is that it may possibly provide a way to link invariant measures to their dynamics. See \cite{BGHZ22} for a conjecture connecting the Brownian loop measure to stochastic geometric heat equations, that falls short of being a theorem due to the lack of a convergence result in the way of \cite{CH16} for discrete regularity structure. In \cite{C22, CCHS20, CCHS22} the authors study the stochastic quantisation for the Yang-Mills in the 2 and 3 dimensional Euclidean space and they have the same type of open problem for the invariant measure.

Apart from Regularity Structures, other theories have been successful in treating singular SPDEs. One such is via the use of paracontrolled calculus as in \cite{GIP13}, which much in the way of Hairer, describes the spectral features of a function in terms of paracontrolled distributions. This methodology has also been used to look at discrete problems. In \cite{ZZ14}, for example, the author use the theory to approximate the Navier-Stokes equation, while in \cite{ZZ15} the authors construct a piecewise linear approximation for the $\Phi^4_3$ model. In \cite{GP17}, the authors achieve a similar convergence result as we present here, for Sasamoto-Spohn type discretisations of the stochastic Burgers equation, in that they show that the limit of such discretisations solves (a version of) the continuous Burgers equations. In \cite{CGP16}, the authors also look at the discrete PAM, but the tool for proving convergence is paracontrolled calculus and finally in \cite{MP19} the authors develop a discrete version of paracontrolled distributions.

In this paper, we intend to extend the work done in \cite{MH21} by presenting a renormalisation procedure that is able to handle more complicated equations. We fix the following discretisation of the generalised KPZ equation:
\begin{equs}\label{eq:gKPZ}\tag{gKPZ}
\partial_t u = \Delta u + g(u)(\nabla u)^2+k(u) \nabla u+h(u)+f(u)\xi_t(x),
\end{equs}
where $f,\,g,\,k,\,h$ are assumed to be smooth, to showcase our methodology. Here $t\ge 0$, $x\in\mathbb{Z}^d$, $\nabla$ is a discrete gradient and $\Delta$ is the discrete Laplacian defined by:
\begin{equs}
\Delta u(t,x) = \sum_{y:\,x\sim y} [u(y,t) - u(x,t)],
\end{equs} 

where $x\sim y$ means that $x$ and $y$ are nearest neighbours with respect to the Euclidean norm on $\mathbb{Z}^d$ and $\xi$ is a $\mathbb{R}$-valued random field which we will not fix although we do assume that it is centred.

As in \cite{MH21} we would like to prove that the solution of \eqref{eq:gKPZ} converges, in some sense to be specified later, to the solution $\bar u$ of the gKPZ equation driven by some space-time random field $Y$ defined on $\mb{R}_+\times\mb{R}^d$:
\begin{equs}\label{eq:gKPZcon}
\partial_t \bar u = \Delta \bar u + g(\bar u)(\nabla \bar u)^2+k(\bar u) \nabla \bar u+h(\bar u)+f(\bar u)Y_t(x).
\end{equs}
To this end, we rescale the space by reformulating our equation on $\mathbb{Z}^d_N = (\mathbb{Z}/2^N\mathbb{Z})^d$:
\begin{equs}
\label{eq:gKPZ1}\partial_t\hat{u}^N & = (\Delta^N\hat{u}^N) + g(\hat{u}^N)(\nabla_{\! N}\hat{u}^N)^2 \\ & + k(\hat{u}^N)(\nabla_{\! N}\hat{u}^N)+ h(\hat u^n) + 2^{Nd/2}{\xi}_t^Nf(\hat{u}^N),
\end{equs}
where $\Delta_N = 2^{2N}\Delta$, $\nabla_{\! N} = 2^{N}\nabla$ are the rescaled discrete Laplacian and discrete derivative defined on $2^{N}\mathbb{Z}^d$, with $x\in 2^{-N}\mathbb{Z}^d$,\,$t\ge 0$, $\xi^N_t(x) = \xi_{2^{2N}t}(2^Nx)$. Recall that we did not fix any particular discrete derivative, but we do require from here on that $\nabla$ must be chosen so that we have $\|\nabla_{\! N}\varphi - \varphi'\|_{\infty}\rightarrow 0$.
In the sequel, we will use the shorthand notation $ \mathbb{S} $ for  
$\mathbb{Z}^d_N$.

Of course, with reference to the discussion in the opening regarding the ill-posedness of the products of distributions in general, it turns out that both \eqref{eq:gKPZ} and \eqref{eq:gKPZcon} fail to have any canonical meaning and instead of \eqref{eq:gKPZ1} one should be looking at:
\begin{equs}\label{eq:gKPZ2}
\partial_t\hat{u}^N & = (\Delta^N\hat{u}^N) + g(\hat{u}^N)(\nabla_{\! N}\hat{u}^N)^2 \\ & + k(\hat{u}^N)(\nabla_{\! N}\hat{u}^N)+ \bar{h}(\hat u^n) + 2^{Nd/2}{\xi}_t^Nf(\hat{u}^N),
\end{equs}
with $\bar{h}$ taking the form:
\begin{equs}
\bar{h}(u) & = h(u) + \sum_{\btau \in \mc{T}_-} \frac{\Upsilon[\btau]}{S(\btau)} C_N(\btau).
\end{equs}
where $ \mc{T}_- $ is a finite set of decorated trees, $ \Upsilon[\btau] $ are elementary differentials, $ S(\btau) $ are symmetric factors and the $ C_N(\btau) $ are renormalisation constants coming from the general theory which was developed in \cite{BCCH} and has since been used for stochastic geometric equations in \cite{BGHZ22}. For our purposes, the renormalisation constants need to be adjusted and we discuss this in section~\ref{sec:RnmMdl}.

We impose the following assumption on the cumulants of the noise (refer to Appendix~\ref{sec:cum} for a primer on cumulants): 
\begin{assumption}\label{ass:cyc}
We assume that $\xi^N_t(x)$ defined on $Z^d_N$, is such that for $t\ge 0$ and some index set $|\mathcal{A}|\ge 2$, one has the following bound: 
\beq
\mathbb{E}_c\left\{\xi^N_{t_a}(x_a)\,:\,a\in\mathcal{A}\right\}\lesssim \sum_{\sigma\in\sigma(\mathcal{A})}\prod_{i=1}^{|\mathcal{A}|}\left(\|z_{\sigma(i+1)}-z_{\sigma(i)}\|_\s\vee 2^{-N}\right)^{-\frac{3}{2}},
\eeq
where $\|\cdot\|_\s$ is as in \eqref{eq:normRd}, uniformly over $N > 0$ and over all collections of time-space points $z_a=(t_a,x_a)$ indexed by $a\in\mathcal{A}$.
\end{assumption}

Then, our main result can be seen as an extension of the main result obtained in \cite[Thm 1.1]{MH21}:

\begin{theorem}[Meta-theorem] Denote by $u^N$ the solution of \eqref{eq:gKPZ2} defined on the rescaled torus $\mathbb{T}^d_N = 2^{-N}\mathbb{Z}^d_N$ and by $u$ the solution of the renormalised version of \eqref{eq:gKPZcon}. If there exists a sequence of initial conditions $(u^{N}_0)_{N\in\mb{N}}$ such that there is $\eta\in(0,1)$ and $u_0\in\mathcal{C}^\eta$ for which it is true that
$$\lim_{n\rightarrow\infty}\|u_0;u^N_0\|_\eta = 0,$$ and one has $$\lim_{N\rightarrow\infty}\sup_{(t,x)\in[0,T]\times 2^{-N}\mb{S}}|\xi^{\delta,N}(t,x)-Y^\delta(t,x)|=0,$$
where $\xi^{\delta,N}$ and $Y^{\delta}$ are constructed from $\xi$ and $Y$ by convolution as explained in Sec~\ref{sec:convergence}. Then there exists a sequence of diverging constants such that the sequence of $u^N$ converges in law to $u$ as a member of $\mc{C}^{\bar\eta,T}_N$ for all $\bar\eta\in(0,\frac{1}{2}\wedge\eta)$.

\end{theorem}

\begin{remarque}
The motivation behind this assumption and the convergence result comes from the work that Hairer and Erhard have done in \cite{MH21}, wherein they studied a discrete analogue of the parabolic Anderson model: $$\partial_t u(x,t) = \bigl(\Delta u\bigr)(x,t)-\xi_t(x)\,u(x,t)$$ defined on $x\in\mathbb{Z}^d$. Unlike the usual (continuous) SPDE setup where the random noise comes from the space-time white noise, the authors in \cite{MH21} instead solve a discretisation of the equation driven by a stationary symmetric simple exclusion process using discrete regularity structures, which they further show to converge to a parabolic Anderson model that is driven by a generalised Ornstein-Uhlenbeck process. The result they achieve is in $d=3$, which is not surprising because for this spatial dimensionality, one expects the space-time white noise to be far too irregular to be amenable to the usual programme, while the Ornstein-Uhlenbeck process is more regular and hence susceptible to it.
\end{remarque}
\begin{remarque}
We adopt a flipping of the perspective. Instead of fixing a noise that drives the discretisation, we can begin with positing the bounds on the cumulants of the noise, for which the solution theory still works. A major work in \cite{MH21} was to check that Assumption~\ref{ass:cyc} holds for a stationary symmetric simple exclusion process. We assume these bounds were given to us. 
\end{remarque}
\begin{remarque}
As the noise no longer needs to be Gaussian, the second order moments are no longer sufficient in finding bounds on higher order moments and as such higher order cumulants are needed for the stochastic estimates central to the application of regularity structures in solving singular SPDEs. This means that more Feynman diagrams have to be controlled.
\end{remarque}

It is well established in the literature of (discrete) regularity structures that certain stochastic estimates are sufficient for ensuring convergence post renormalisation; one can refer to the prototypical Theorem 10.7 in \cite{reg} in the continuous setting. As we have remarked in the opening, these stochastic estimates have been automated for the continuum but in discrete settings corresponding results are lacking. In lieu of those results, the authors in \cite{MH21} appealed to the rather general bounds on labelled rooted binary trees presented in \cite[Appendix A]{HQ15} to establish certain estimates. The work presented here is motivated by the same but it turns out that there are certain estimates needed for the \eqref{eq:gKPZ} that are not tractable under the existing framework. Indeed, one specific example is a singular integral with multiple points:
\begin{equs}
I = \int f(x_{0},x_1,...,x_n) \prod_{i=1
}^n K_i(x_i-x'_i) d x_1 ...d x_n 
\end{equs}
where the $ x'_i $ are fixed and $f$ is singular when the points $ x_0,...,x_n $ are close. Our main new idea is to use a local transformation by introducing Taylor expansions via telescopic sums:
\begin{equs}
I & = \int f(x_{0},x_1,...,x_n) \left( K_1(x_1-x'_1) - \sum_{k \leq n_f} \frac{(x_1-x_0)^k}{k!}D^k K_1(x_0 - x'_1)\right) \\ & \prod_{i \neq 1
} K_i(x_i-x'_i) d x_1 ...d x_n + \sum_{k \leq n_f} D^{k} K_1(x_0 - x'_1) \int f(x_{0},x_1,...,x_n)   \frac{(x_1-x_0)^k}{k!}  \\ & \prod_{i \neq 1
} K_i(x_i-x'_i)d x_1 \hdots d x_n. 
\end{equs}
The first term contains a Taylor expansion that renormalised $ f $ up to its degree of divergence given by $ n_f $ whereas the second term is simpler as it contains fewer $K_i$'s. One now moves all the $K_i$'s to focus in on the following integral:
\begin{equs}
\int f(x_{0},x_1,...,x_n) \prod_{i=1}^n (x_i-x_0)^{k_i} dx_1 \hdots dx_n. 
\end{equs}
Then one integrates over all the $x_i,\text{ i}\neq 0$, and finally uses translation invariance to get rid of the $x_0$. In the convergence theorem given in \cite{HQ15,MH21}, one can only treat the case with only one $ K_i $ which is enough for PAM but not sufficient for gKPZ as they are terms with two $ K_i$'s. 
 Our contribution is to extend the \cite{HQ15} result (compare our Ass.~\ref{ass:grph} with \cite[Ass. A.1]{HQ15}) and couple it with a renormalisation procedure (see Section~\ref{sec:RenormProc}) that allows us to establish the erstwhile untractable stochastic estimates.
Our approach is valid when one has to face subdivergences with no nested structures and with no overlap.
\newline
General estimates in the discrete setting will need a full extension of \cite{CH16} or to push forward the approach advocated in this work, which we expect to be a very challenging task. Another route is to rely on a complete recursive proof of these estimates as it is performed with multi-indices in \cite{OSSW,LOT,LOTT} or to use the local renormalisation approach developed in \cite{BR18}. It has been used successfully in \cite{BB21,BB21b} for writing the renormalised equation in a non-translation invariant setting.

\begin{remark}
    Several months after this work was completed, many papers appeared on the convergence of discrete/continuous singular SPDEs. Let us summarise how these works compare with the techniques of the present paper. In the continuous setting, for a long time the main convergence theorem relied on \cite{CH16}. Recently in \cite{LOTT}, the authors proposed a non-diagrammatic proof based on a spectral inequality satisfied by the noise. Their proof has been written for a quasilinear SPDE using multi-indices instead of decorated trees.
    In \cite{BN23}, the authors of the present paper has explained the main algebraic steps of \cite{LOTT} into the context of decorated trees and for a large class of singular SPDEs. The first author of the present work together with Ismaël Bailleul wrote a short proof of the convergence of the renormalised model for continuous gKPZ with a space-time white noise in \cite{BB23} using the main results of \cite{BN23} and building upon the techniques of this paper. Almost at the same time, a recursive proof of the convergence of the renormalised model has been produced in \cite{HS23}. In the discrete setting, the recent works \cite{Mat18,CS23} use techniques from \cite{reg,HP15} which cannot cover the case of gKPZ.
\end{remark}

{\bf Overview of the paper:}  In Section \ref{sec:DisRS}, we recall for the reader some major aspects of the theory of Discrete Regularity Structures; the presentation here follows very much that of \cite{MH21}. We also briefly discuss the implementation of the renormalisation procedure, employing the recursive procedure found in \cite{BR18}.

In Section \ref{sec:lblgrph}, we present an extension of the bounds to be found in \cite[Appendix A]{HQ15}. We show that one has certain bounds (see Theorem \ref{th:HQbnd}) on integrals over the so-called "Generalised Convolutions" (see \eqref{eq:genconv}) under certain assumptions (see Assumption~\ref{ass:grph}) and these bounds we prove via very general bounds (see Lemma~\ref{lem:HQ}) on rooted binary trees. The interest in these integrals stems from the fact that they can be specialised to integrals one needs to bound to get convergence in the Regularity Structure paradigm.

In Section~\ref{sec:elblgraph}, we present the notion of an "elementary graph" (see Definition~\ref{def:ELG}, which is elementary in the sense that they can be morphed together to construct all the relevant graphs in our analysis. We also explain how to associate to the multiple integrals in $\Pi^N_0\btau$ these elementary graphs. We will then present assumptions (see Assumption~\ref{ass3}) under which Theorem~\ref{th:HQbnd} holds for the elementary graphs as well. 

In Section~\ref{sec:gKPZ}, we use the bounds previously constructed to achieve the stochastic estimates to prove the convergence of the discrete \eqref{eq:gKPZ}. 

\subsection*{Acknowledgements}

{\small
Y.B. thanks the Max Planck Institute for Mathematics in the Sciences (MiS) in Leipzig for having supported his research via a long stay in Leipzig from January to June 2022. U. N. thanks the Max Planck Institute for Mathematics in the Sciences (MiS) for a short stay in Leipzig where parts of this work were discussed.
}  

\section{Discrete Regularity Structures} \label{sec:DisRS}

A major tool in our analysis of \eqref{eq:gKPZ} is the theory of discrete Regularity Structures, as in \cite{EH17}. In this section we recall the very basics to motivate the work we do in this paper. The presentation here is heavily inspired by \cite[Section 2]{MH21}.

We start by recalling some notations. A scaling on $\mathbb{R}^{d+1}$ is the vector $\s=\{\s_0,\s_1,\hdots,\s_d\}\in\mathbb{N}^{d+1}_{\ge 1}$ and it allows the definition of a norm on $\mathbb{R}^{d+1}$ given by:

\begin{equs}\label{eq:normRd}||z||_\s:=\sup_{i\in\{0,\hdots,d\}}|z_i|^{1/\s_i}\end{equs}

We also set $|\s|=\sum_{i=0}^d|\s_i|$. In this paper it will always be the parabolic scaling, which is given by $\s=\{2,1,\hdots,1\}$. We will also find use of the following scaling function (and functional): for a $\lambda>0$, and $\varphi:\mathbb{R}^{d+1}\rightarrow\mathbb{R}$, $\mathcal{S}_\s^\lambda(z_0,z_1,\hdots,z_d) = (\lambda^{-\s_0}z_0,\lambda^{-\s_1}z_1,\hdots,\lambda^{-\s_d}z_d)$ and $(\mathcal{S}^{\lambda}_{\s,z}\varphi)(y)=\lambda^{-|\s|}\varphi(\mathcal{S^\lambda_\s}(y-z))$. Furthermore we denote by $(e_0,e_1,\hdots,e_d)$ the canonical basis of $\mb{N}^{d+1}$.

For our purposes the role that H\"older spaces play in the theory of regularity structures, will be played by the following discrete counter parts:
\bed
For $\eta\in (0,1)$, we define discrete H\"older spaces $\mc{C}_N^\eta(\mb{T}_N^d,\R)$ as the space of all elements $f\in\mb{R}^{\mb{T}_N^d}$, with norm 
\begin{equation}
\|f\|_{\mc{C}_N^\eta}\coloneqq \sup_{x\in\mb{T}_N^d}|f(x)| + \sup_{x\neq y\in\mb{T}_N^d}
\frac{|f(x)-u(y)|}{|x-y|^\eta}.
\end{equation}
\eed
Let $\eta\in (0,1)$. To compare an element $f\in\mc{C}^\eta(\mb{T}^d,\mb{R})$ in the usual H\"older space with an element $f^N\in \mc{C}_N^\eta(\mb{T}_N^d,\mb{R})$ we introduce the distance
\begin{equation}\label{eq:discreteHolder}
\begin{aligned}
\|f;f^N\|_{\eta}\coloneqq
&\sup_{x\in\mb{T}_N^d}|f(x)-f^N(x)| + \sup_{x\neq y\in\mb{T}_N^d}
\frac{|(f(x)-f(y))-(f^N(x)-f^N(y))|}{|x-y|^\eta}\\
 &+\sup_{\substack{x,y\in\mb{R}^d:\, |x-y|<2^{-N}}}
\frac{|f(x)-f(y)|}{|x-y|^\eta}.
\end{aligned}
\end{equation}
To compare functions $f\in\mc{C}_\s^\eta ([0,T]\times \mb{T}^d,\mb{R})$ and $f^N\colon [0,T]\times \mb{T}_N^d \to \mb{R}$, we define a ``distance'' by
\begin{equs}
\|f;f^N\|_{\mc{C}_N^{\eta,T}}&\coloneqq \hspace{-5mm}\sup_{(t,x)\in [0,T]\times \mb{T}_N^d}\hspace{-2mm}|f(t,x)-f^N(t,x)|+\hspace{-3mm}\sup_{\substack{(t,x), (s,y)\in [0,T]\times\mb{T}_N^d\\ \|(t,x)-(s,y)\|_\s < 2^{-N}}}\frac{|f(t,x)-f(s,y)|}{\|(t,x)-(s,y)\|_\s^\eta}\\
& + \sup_{\substack{(t,x), (s,y)\in [0,T]\times\mb{T}_N^d\\ \|(t,x)-(s,y)\|_\s \geq 2^{-N}}} \frac{|(f(t,x)-f(s,y))-(f^N(t,x)-f^N(s,y))|}{\|(t,x)-(s,y)\|_\s^\eta}.
\end{equs}

\subsection{Discrete Models}

\begin{definition}[Regularity Structures]\label{def:regstr}
The pair $(\mc{T},\mathcal{G})$ is called a regularity structure if one has the following:
\begin{itemize}
\item $\mathcal{T}$ is a graded vector space $\bigoplus_{\alpha\in A} T_\alpha$, where $A$ is a locally finite and bounded from below index set, $A\subset \mathbb{R}$, $T_\alpha$ is a Banach space. To each element $\colb{\tau_\alpha}\in T_\alpha$, we associate the notion of a homogeneity: $\|\btau\| = \alpha$.
\item $\mathcal{G}$ is a group of continuous linear operators defined on $T$ in such a manner that one has for every $\Gamma\in\mathcal{G}$ and $\colb{\tau_\alpha}\in T_\alpha$:
$$\Gamma\colb{\tau_\alpha}-\colb{\tau_a}\in\bigoplus_{\beta<\alpha}T_\beta$$
\end{itemize}
\end{definition}

We will always use $\colb{\text{blue}}$ for elements of any given regularity.

To construct the Regularity Structure $\mc{T}$ we will need, we begin with a bigger structure $\mc{F}$ first. This is defined recursively by beginning with $\{\colb{1},\colb{X_0},\colb{X_1},\cdots,\colb{X_d},\bXi\}\subset\mc{F}$.  Here $(\colb{X_0},\colb{X_1},\hdots,\colb{X_d})$ corresponds to $(t,x_1,\hdots,x_d)$ with the parabolic scaling and for $k\in\mathbb{N}^{d+1}$, and we use the term $\colb{X^k}$ to represent $\colb{X_0^{2k_0}X_1^{k_1}\cdots X^{k_d}_d}$. Upon this set we require the existence of a product, i.e. $\colb{\tau_1},\hdots,\colb{\tau_n}\in\mc{F}$, then $\colb{\tau_1\cdots \tau_n}\in\mc{F}$, which is assumed to be associative and commutative. $\bXi$ is included here to represent the noise. Finally we require that if $\btau\in\mc{F}\setminus\{\colb{1},\colb{X^k}:k\in\mathbb{N}^{d+1}\}$ then $\{\bcI{\btau},\colb{\CI_{e_1}(\tau)},\hdots,\colb{\CI_{e_d}(\tau)}\}\subset\mc{F}$, where in general for $k\in\mb{N}^{d+1}$, $\colb{\mc{I}_k(\cdot)}$ represents convolution with the kernel differentiated $k_i$ times in the $i$-th component and in particular $\colb{\mc{I}} \coloneqq \colb{\mc{I}_{\{0,0,\hdots,0\}}}$.

To each $\btau\in\mc{F}$, we associate a homogeneity $|\btau|_\s$ which belongs to the reals. This is also defined recursively by setting $|\bXi|_\s = -\frac{3}{2}-\kappa$, $|\colb{X_i}|=1$ unless $i=0$ in which case $|\colb{X_0}|=2$, $|\colb{1}|_\s=0$ and then requiring $|\colb{\tau_1\cdots\tau_n}|_\s = |\colb{\tau_1}|_\s+\hdots+|\colb{\tau_n}|_\s$ and $\colb{\mc{I}_k(\tau)}=|\btau|_\s + 2 - |k|_\s$.

The recursive nature of this definition makes the space far too large. We instead restrict ourselves to those symbols that we expect to be able to extract from \eqref{eq:gKPZ}:
\begin{equs} \label{rules_gKPZ}
\mc{R} & =\{\colb{X^k\mc{I}(\cdot)^\ell},\colb{X^k\mc{I}(\cdot)^{\ell}\mc{I}_{e_i}(\cdot)},\colb{X^k\mc{I}(\cdot)^{\ell}\mc{I}_{e_i}(\cdot)\mc{I}_{e_j}(\cdot)},\colb{X^k\mc{I}(\cdot)^{\ell}\Xi}, \\ &  k \in \mb{N}^{d+1}, \, \ell \in \mb{N}, \, i,j \in \lbrace 1,...,d  \rbrace \}
\end{equs}
 Then we define:
\begin{equs}
\mc{H} = \{\btau\in\mc{F}:\btau = R(\colb{\tau_1},\hdots,\colb{\tau_n}), R\in\mc{R} \text{ and }\colb{\tau_1},\hdots,\colb{\tau_n}\in\mc{H}\text{ or }\btau = \Xi\}
\end{equs}
and finally $\mc{T} = \text{Vec}\,\mc{H}$. $\mc{H}_\alpha$ is naturally the set of all symbols in $\mc{H}$ that have a homogeneity of $\alpha$, and then $\mc{T}_\alpha$ is then the span of $\mc{H}_\alpha$, whereby $A$ is the just the set of all of these $\alpha$.

It is known that for a small enough choice of $\kappa$, $\mc{T}_\alpha$ is finite dimensional for all $\alpha$. Finally we impose $\colb{\mc{I}(X^k)}=0$ for $k\in\mathbb{N}^{d+1}$ because the function such a convolution would realise would be smooth and we already have the polynomial regularity structure to describe smooth functions.

With a Regularity Structure fixed, one needs to define a model on it. One fixes first $\mc{X}_N = \mc{D}(\mathbb{R},\mathbb{R}^{\mathbb{T}^d_N})$, by which we mean the set of all c\'{a}dl\'{a}g functions from $\mathbb{R}$ to $\mathbb{R}^{\mathbb{T}^d_N}$. We also define an inclusion map $\iota_N:\mc{X}_N\rightarrow\mc{S}'(\mathbb{R}^{d+1})$: 

\beq
(\iota_N\,f)(\varphi)=2^{-dN}\sum_{x\in\mathbb{T}^d_N}\int f(t,x)\varphi(t,x)\,dt.
\eeq

On $\chi_N$ we define a family of seminorms $\|\cdot\|_{\alpha;\mc{k}_N;z;N}$ in the following way:

$$\|f\|_{\alpha;\K_N;z;N} = \sup_{\lambda\in(0,2^{-N}]}\sup_{\varphi\in\Phi_{N,z}^\lambda}\lambda^{-\alpha}\left|\int f(s,x)(\mc{S}^\lambda_{2,t}\varphi)(s)\,ds\right|$$

Here $\alpha$ ranges over $\mathbb{R}$, $\K_N$ ranges over all compact subsets in $\mathbb{R}^{d+1}$ with diameter bounded by $2^{-N+1}$, and $z=(t,x)\in\mathbb{R}\times\mathbb{T}^d_N$ is such that $z\in\K_N$. With $r > |\min A|$ and $\lambda\in(0,2^{-N}]$ fixed, $\Phi_{N,z}^\lambda$ the set of all functions $\varphi:\mathbb{R}\rightarrow\mathbb{R}$ with $\|\varphi\|_{\mc{C}^r}\le 1$ and support contained in the unit ball so that $\supp\,\mc{S}^\lambda_{2,t}\varphi\subset\{s\in\mathbb{R}:(s,x)\in\K_N\}$. Further by $\|\cdot\|_\ell$ we will denote the norm of the $\ell$-th component in $\mc{T}$. With these ingredients we define:

\begin{definition}[Discrete Model]
A discrete model consists of a collection of maps $z\rightarrow \Pi_x^N\in\mc{L}(\mc{T},\mc{X}_N)$ and $\Gamma^N:\mathbb{R}^{d+1}\times\mathbb{R}^{d+1}\rightarrow\mc{G}$ such that:
\begin{itemize}
\item $\Gamma^N_{zz}\equiv\text{id}$ the identity operator, and the $\Gamma^n_{xy}\Gamma^N_{yz}=\Gamma^N_{xz}$, $\forall\; x,\,y,\,z\in\mathbb{R}$.
\item $\Pi^N_z = \Pi^N_y\Gamma^N_{yz}$, $\forall y,\,z\in\mathbb{R}^{d+1}$
\end{itemize}

Such that the following estimates hold, for any compact set $\K\subset\mathbb{R}^d$ and every $\btau\in\mc{T}_\ell$:
\beq
|(\iota_N\Pi_z^N\btau)(\mc{S}^\lambda_{\s,z}\varphi)|\lesssim\|\btau\|_\ell\lambda^\ell,\qquad\qquad\|\Pi^N_z\btau\|_{\ell;\K_N,z;N}\lesssim\|\btau\|_{\ell},
\eeq

and 
\beq
\|\Gamma_{zz'}^N\btau\|_m\lesssim\|\btau\|_\ell\|z-z'\|^{\ell - m}_\s,\qquad\qquad\$z\mapsto\Gamma_{zz'}^N\btau\$_{\ell;\K;N}\lesssim\|\btau\|_\ell
\eeq

for some fixed $\gamma > 0$, uniformly over $\lambda\in(2^{-N},1]$, for all $\varphi\in\Phi^\lambda_{N,z}$ , all $\ell\in A$ such that $\ell < \gamma$ and $m < \ell$, all $\btau\in\mc{T}_\ell$, and locally uniform over $z,z'\in\K$ such that $\|z-z'\|\s\in(2^{-N},1]$ and over all compact subsets $\K$ of diameter bounded by $2^{-N+1}$.
\end{definition}

\begin{remark}
The norm $\$\cdot\$_{\ell,\mathfrak{K},N}$ has not yet been defined. Further below we replace it with another norm on functions of the form $f:\mb{R}^{d+1}\rightarrow\mc{T}_{<\gamma}$ that we do explicitly state, when we discuss the $``t=0"$-hyperplane. 
\end{remark}

Central to the application of Regularity Structures to the study of semilinear SPDEs is the Green's function of the differential operator of the equation. For the \eqref{eq:gKPZ}, this is $\partial_t - \Delta_N$. It is known, from \cite[Sec. A.1]{HQ15} say, that one can decompose the rescaled Green's function of this operator as follows:

\beq\label{eq:Greens} 2^{dN}G^N = K^N + R^N,\eeq

where the function $K^N$ can be decomposed as $K^N = \sum_{n=1}^N K_n$ where each $K_n$ has a support contained in the set $\{z\in\mathbb{R}^{d+1}:\|z\|\lesssim 2^{-n}\}$, and annihilate polynomials of scaled degree less than or equal to two. For each $n\in\{1,\cdots,N-1\}$, and each multi-index $|k|_\s\le 2$,
$$|D^k K_n(z)|\lesssim 2^{n(|\s|-2+|k|_\s)}$$
uniformly in all parameters. If $n=N$, then the same estimate holds for $k=0$. What this inequality means that this function can be extended to a function in $\mathbb{R}^{d+1}$, for which the estimate is true. $R^N$ is compactly supported, is anticipative, and $\|R^N\|_{\mc{C}^2}$ is bounded uniformly in $N$, with the same interpretation for the bound as before. Let $\zeta\in\mathbb{R}$ and $n\le N$, we define:
\begin{equation}
(T_{n,\zeta+2}^N F)(z)= \sum_{|k|_\s<\zeta+2}\frac{X^k}{k!}\CQ_k((T_{n,\zeta+\beta}^{N} F)(z)),
\end{equation}
where for $z=(t,x)\in \mathbb{R}^{d+1}$, $n\in\{1,\ldots, N\}$, and $|k|_\s <\zeta+2$,
\begin{equation}
\CQ_k((T_{n,\zeta+2}^{N} F)(z))= 
\begin{cases}
2^{-dN}\sum_{y\in\mb{T}_N^d}\int D_1^k K_n( z,(s,y))F(y,s)\, ds, &\text{if }n<N,\\
\delta_{k=0} 2^{-dN}\sum_{y\in\mb{T}_N^d}\int K_N( z,(s,y))F(y,s)\,ds, &\text{otherwise},
\end{cases}
\end{equation}
for all $F\in\mc{X}_N$ for which the above expression makes sense.

\begin{definition}[Admissible Models]\label{def:admmodel}
Given $K^N$ as above, we define the set of admissible models $\mc{M}$ as consisting of all models such that for every multi-index $k$:
\begin{equation}
\label{eq:Xk}
(\Pi_{(t,x)}^N \colb{X^k})(s,y)= ((s,y)-(t,x))^k
\end{equation}
such that:
\begin{equation}
\label{eq:I}
\begin{aligned}
(\Pi_{(t,x)}^N \bcI{\tau})(s,y)= &2^{-dN}\sum_{\bar x\in\mb{T}_N^d}\int K^N((s,y),(\bar s,\bar x))(\Pi_{(t,x)}^{N}\btau)(\bar s,\bar x)\, d\bar s\\
&-\sum_{|k|_\s<|\tau|+2} \frac{((s,y)-(t,x))^k}{k!}Q_k\big((T_{|\tau|+2}^{N}\Pi_{(t,x)}^N \btau )(t,x)\big).
\end{aligned}
\end{equation}
where the following shorthands have been employed:
\begin{equation}
T_{\zeta+2}^N F= \sum_{n=1}^{N}T_{n,\zeta+2}^N F,\quad\mbox{and}\quad 
\CQ_k(T_{\zeta+2}^N F)(\cdot))=\sum_{n=1}^{N}\CQ_k(T_{n,\zeta+2}^N F)(\cdot)).
\end{equation}
\end{definition}

The standard way of constructing an admissible model \cite{reg} is, to begin with:
\begin{equation}
(\Pi_{(t,x)}^N\bXi)(s,y)=\xi_s^N(y), 
\end{equation}
for $(t,x)\in\mathbb{R}^{d+1}$ and $(s,y)\in\mathbb{R}\times\mathbb{T}^d_N$ and $\Pi_{(t,x)}^N\colb{X^k}$ as in \eqref{eq:Xk} and the relationship:
\begin{equation}
\label{eq:product}
\Pi_{(t,x)}^N\btau\colb{\bar{\tau}}= \Pi_{(t,x)}^N\btau\,\Pi_{(t,x)}^N\colb{\bar{\tau}}.
\end{equation}

To construct a solution theory, we extract from the equation an abstract fixed problem, and for that we need a notion of a $\mc{T}$-valued functions on our physical space. Recall that $\Pi_z\btau$ is to be understood as a local description of a possibly globally defined distribution around $z$ and motivated by this we would like our $\mc{T}$-valued function $\colb{f}$ to be defined in such a manner that $\Pi_z\colb{f}(z)$ is able to mimic a global distribution locally around $z$ for all $z\in\mathbb{R}^{d+1}$. To this end, we define first the $``t=0"$-hyperplane via $P=\{(t,x)\in\mathbb{R}^{d+1}:t = 0\}$, and $||z||p = 1 \wedge \inf_{y\in P}\|z - y\|_\s$ and $||y,z||_P = ||y||_P\wedge ||z||_P$, and finally:
$$\K_P=\{(y,z)\in (\K\setminus P)^2:\, y\neq z,\, \|y-z\|_\s\leq \|y,z\|_{P}\}$$

for some set $\K\subset\mathbb{R}^{d+1}$

Further for fixed $\gamma > 0$, $f:\mathbb{R}^{d+1}\setminus P\rightarrow \mc{T}_{<\gamma}$. Then one defines for any compact set $\mc{K}\subset\mathbb{R}^{d+1}$:
$$\$ f\$_{\gamma,\eta;\K;N}=\sup_{\substack{z\in\K\setminus P\cap\mb{R}\times \mb{T}_N^d\\ \|z\|_P < 2^{-N}}}\sup_{\beta <\gamma}\frac{\|f(z)\|_\beta}{\|z\|_{P}^{(\eta-\beta)\wedge 0}}
+ \sup_{\substack{y,z\in\K_P\cap\mb{R}\times \mb{T}_N^d,\\ \|y-z\|_\s<2^{-N}}}
\sup_{\beta <\gamma}\frac{\|f(z)-\Gamma^{N}_{zy} f(y)\|_\beta}{\|y-z\|_\s^{\gamma-\beta}\|y,z\|_{P}^{\eta-\gamma}}$$

The need to be careful with $P$, is due to the singularity of the initial condition on it. For particulars one can refer to \cite[Section 2]{reg}.

\begin{definition}[Discrete Model] For a regularity structure $\mc{T}$, a discrete model $(\Pi^N,\,\Gamma^N)$ and $\eta\in\mathbb{R}$, the space $\mathcal{D}^{\gamma,\eta}_N$ consists of functions $f:\mathbb{R}^{d+1}\setminus P \rightarrow \mathcal{T}_{<\gamma}$ for which the following norm is finite:
\begin{equation}
\label{eq:weightedDgamma}
\sup_{\substack{z\in\K\setminus P\\ \|z\|_P\geq 2^{-N}}}\sup_{\beta <\gamma}\frac{\| f(z)\|_{\beta}}{\| z\|_{P}^{(\eta-\beta)\wedge 0}} + \sup_{\substack{(y,z)\in\K_P,\\ 2^{-N}\leq \|y-z\|_\s\leq 1}}\sup_{\beta <\gamma}
\frac{\|f(z)-\Gamma^{N}_{zy}f(y)\|_{\beta}}{\|y-z\|_\s^{\gamma-\beta}\|y,z\|_{P}^{\eta-\gamma}}
+\$f\$_{\gamma,\eta;\K;N}.
\end{equation}
\end{definition}

For $\colb{f}$ to be an adequate description of a global distribution, we should be able to recover (or "reconstruct") it from $\colb{f}$. This is done via the reconstruction operator, $\mc{R}^N:\mc{D}^{\gamma,\eta}_N\rightarrow\mc{X}_N$ defined in the discrete case, simply as:

$$(\mc{R}^N\colb{f})(z)=(\Pi^N_z\colb{f}(z)),$$

for $z \in \mathbb{R}\times\mathbb{T}^d_N$.

With this diversion settled, we return to our fixed point equation. For \eqref{eq:gKPZ}, the fixed point equation that is to be solved is given by:

\beq\label{eq:absfixpoint}\colb{U^N} = \colb{\mc{G}^N}(\hat F_\gamma\colb{U^N}\mathbbm{1}_{\{t\ge 0\}})+(G^N u^N_0)\eeq

where $\mc{G}^N$ is to be understood as the abstract version of convolution with the (rescaled) heat kernel $G^N$ such that it satisfies the property:

\beq\mc{R}^N\colb{\mc{G}^N} = G^N\eeq

and $F_\gamma(\btau)=(\colb{G}(\btau)(\mc{D}\btau)^2 + \colb{H}\mc{D}\btau+\colb{K}(\btau)+\colb{F}(\btau)\Xi)$ with $\colb{G},\,\colb{H},\,\colb{K},\,\colb{F}$ being the lifts of the smooth functions $g,\,h,\,k,\,f$ as defined in \cite[Section 5.2]{EH17}, the abstract derivative operator $\mc{D}$ as seen in \cite[Section 5.3]{EH17} and the products of the objects in $\mc{D}^{\gamma,\eta}_N$ made sense of as in \cite[Section 5.1]{EH17}. Finally, following \cite[Section 4]{EH17} and the particular discretisation in \cite{MH21}, $G^N u^N_0$ is defined in the following manner: 

\begin{equation}
(G^Nu_0^N)(t,x)= \sum_{|k|_\s<\gamma}\frac{X^k}{k!}\mc{Q}_k((G^N u_0^N)(t,x),
\end{equation}
where $\mc{Q}_k((G^N u_0^N)(t,x))$ is given by
\begin{equation}
\begin{cases}
2^{-dN}\sum_{n<N}\sum_{y\in\mathbb{T}_N^d}\int D_1^k[ K_n+ R^N]((t,x),(s,y))\, u_0^N(y)\, ds, &\mbox{if $|k|_\s >0$},\\
2^{-dN}\sum_{y\in\mathbb{T}_N^d}\int [K^N+ R^N]((t,x),(s,y))\, u_0^N(y)\, ds, &\mbox{otherwise.}
\end{cases}
\end{equation}

That \eqref{eq:absfixpoint} is wellposed comes from (a slight modification of) the work in \cite{BCCH}, where the authors introduce a notion of ``coherence with non-linearity" in Definition~3.20, prove in Lemma~3.21 that being coherent with non-linearity for $\colb{U^N}$ is equivalent to $\colb{U^N}$ solving \eqref{eq:absfixpoint}, and finally in \cite[Section 2.8.1]{BCCH} check the coherence condition for \eqref{eq:gKPZ}. Moreover, it turns out that $\mc{R}^N\colb{U^N}$ solves (or perhaps more accurately coincides) with the solution of \eqref{eq:gKPZ}. At this juncture to show the convergence of the equations, one might think that we should just consider the convergence of the model, but it so transpires that the sequence fails to converge in the space of admissible models. The remedy to this situation is an application of a set of continuous transformations $\hat{M}^{N}$ to the model $(\hat{\Pi}^N,\hat\Gamma^N) = \hat M^N(\Pi^N,\Gamma^N)$ so that it has a {\it renormalisation} effect. A general construction of the underlying renormalisation group has been achieved in \cite{BHZ}, and we will outline some of the features in the following section and show that our renormalisation procedure is completely consistent with the contraction extraction method they employ.

Before we explicate this renormalisation group, we remark that due to the recursive nature of an admissible model, a very natural graphical scheme is used in the literature as a shorthand to avoid cumbersome notation.

For $\bXi$ we use the notation \begin{tikzpicture}
\node at (0,0) [bar] (a) {};
\end{tikzpicture}. For the abstract integration of the previous symbol - $\bcI{\Xi}$ - we add a downward facing edge to the previous diagram,
\begin{tikzpicture}[scale=0.8]
\node at (0,0) [bar] (a) {};
\draw[] (a) to (0,-0.35);
\end{tikzpicture},
and for {\color{blue} $\CI_{e_i}(\Xi)$}, for $i\in\{1,\hdots,d\}$, we signal the presence of a derivative by making the edge thicker
\begin{tikzpicture}[scale=0.8]
\node at (0,0) [bar] (a) {};
\draw[ultra thick] (a) to (0,-0.35);
\end{tikzpicture} without distinguishing between which direction the derivative is taken in. For this reason we will use the notation $\colb{\CI'(\Xi)}$ for $\colb{\CI_{e_i}(\Xi)}$ with $i$ unspecified. Multiplication of symbols is represented by joining them at the root, so for example one would have \begin{tikzpicture}
\node at (0,0) [bar] (a) {};
\node at (0.5, 0) [bar] (c) {};
\draw[] (a) to (0.25,-0.25);
\draw[] (c) to (0.25,-0.25);
\end{tikzpicture}
to represent $\bcI{\Xi}\bcI{\Xi}$, except for multiplication of $\btau$ by $\colb{X}$ in which case one uses the symbol \begin{tikzpicture}
\node at (0,0) [barx] (a) {};
\end{tikzpicture}. If $\bXi$ is in $\btau$, one uses the symbol
\begin{tikzpicture}
\node at (0,0) [barx2] (a) {};
\end{tikzpicture} instead. In Table~\ref{table1} we compile all the symbols with homogeneity less than or equal to naught; whenever there's a thick line in the symbol, it is meant to represent all of the $d$ trees with the derivatives taken along different spatial dimensions.

\begin{table}[!htbp]\label{table1}
\centering
\begin{tabular}{||l l ||} 
 \hline
 Homogeneity & Symbol\\ [0.5ex] 
 \hline\hline
 $-3/2-\kappa$ & \begin{tikzpicture}[baseline=0cm,scale=0.8]
 \node at (0,0.15) [bar] (a) {};
 \end{tikzpicture}\\
 \hline
 $-1-2\kappa$ & \begin{tikzpicture}[baseline=0cm,scale=0.8]
 \node at (-0.25,0.25) [bar] (a) {};
 \node at (0,0) [bar] (b) {};
 \draw[] (a) to (b);
 \end{tikzpicture}\,,\,\,\,\begin{tikzpicture}[baseline=0cm,scale=0.8]
 \node at (-0.25,0.25) [bar] (a) {};
 \node at (0.25,0.25) [bar] (b) {};
 \node at (0,0) [bnode] (c) {};
 \draw[ultra thick] (a) to (0,0);
 \draw[ultra thick] (0,0) to (b);
 \end{tikzpicture}\\
 \hline
 $-1/2-3\kappa$ & \begin{tikzpicture}[baseline=0cm,scale=0.8]
 \node at (-0.25,0.25) [bar] (a) {};
 \node at (0.25,0.25) [bar] (b) {};
 \node at (0,0) [bar] (c) {};
 \node at (0,-0.05) {};
 \draw[] (a) to (c);
 \draw[] (c) to (b);
 \end{tikzpicture}\,,\,\,\,\begin{tikzpicture}[baseline=0cm,scale=0.8]
 \node at (0.25,0.45) [bar] (a) {};
 \node at (0,0.20) [bar] (b) {};
 \node at (0.25,-0.05) [bar] (c) {};
 \draw[] (a) to (b);
 \draw[] (b) to (c);
 \end{tikzpicture}\,,\,\,\,\begin{tikzpicture}[baseline=0cm,scale=0.8]
 \node at (0.25,0.20) [bar] (a) {};
 \node at (-0.25,0.20) [bar] (b) {};
 \node at (0,0.45) [bar] (c) {};
 \node at (0,0) [bnode] (d) {};
 \draw[ultra thick] (a) to (0,0);
 \draw[ultra thick] (0,0) to (b);
 \draw[] (a) to (c);
 \end{tikzpicture}\,,\,\,\,\begin{tikzpicture}[baseline=0cm,scale=0.8]
 \node at (0.25,0.45) [bar] (a) {};
 \node at (-0.25,0.45) [bar] (b) {};
 \node at (0.25,-0.05) [bar] (c) {};
 \node at (0,0.20) [bnode] (d) {};
 \draw[ultra thick] (a) to (0,0.20);
 \draw[ultra thick] (0,0.20) to (b);
 \draw[] (0,0.19) to (c);
 \end{tikzpicture}\,,\,\,\,\begin{tikzpicture}[baseline=0cm,scale=0.8]
 \node at (0,0.25) [bar] (a) {};
 \node at (-0.25,0.15) [bar] (b) {};
 \node at (0.25,0.15) [bar] (c) {};
 \node at (0,-0.05) [bnode] (d) {};
 \draw[ultra thick] (a) to (0,-0.05);
 \draw[ultra thick] (0,-0.05) to (c);
 \draw[] (0,-0.1) to (b);
 \end{tikzpicture}\,,\,\,\,\begin{tikzpicture}[baseline=0cm,scale=0.8]
 \node at (0,0.5) {};
 \node at (0.25,0.45) [bar] (a) {};
 \node at (-0.25,0.45) [bar] (b) {};
 \node at (0.5,0.20) [bar] (c) {};
 \node at (0.25,-0.05) [bnode] (d) {};
 \draw[ultra thick] (a) to (0,0.20);
 \draw[ultra thick] (0,0.20) to (b);
 \draw[ultra thick] (0,0.20) to (0.25,-0.05);
 \draw[ultra thick] (c) to (0.25,-0.05);
 \end{tikzpicture}\,\\
 \hline
 $-1/2 - \kappa$ &
 \begin{tikzpicture}[baseline=0cm,scale=0.8]
 \node at (0,0.15) [barx] (a) {};
 \end{tikzpicture}\,,\,\,\,\begin{tikzpicture}[baseline=0cm,scale=0.8]
 \node at (-0.25,0.20) [bar] (a) {};
\draw[ultra thick] (a) to (0,-0.05);
\end{tikzpicture} \\
 \hline
 $-4\kappa$ & \begin{tikzpicture}[baseline=0cm,scale=0.8]
 \node at (0,0.8) {};
 \node at (0,0.7) [bar] (d) {};
 \node at (0.25,0.45) [bar] (a) {};
 \node at (0,0.20) [bar] (b) {};
 \node at (0.25,-0.05) [bar] (c) {};
 \draw[] (a) to (b);
 \draw[] (b) to (c);
 \draw[] (a) to (d);
 \end{tikzpicture}\,,\,\,\,\begin{tikzpicture}[baseline=0cm,scale=0.8]
 \node at (0,0.7) [bar] (d) {};
 \node at (0.25,0.45) [bar] (a) {};
 \node at (0,0.20) [bar] (b) {};
 \node at (0.25,-0.05) [bnode] (c) {};
 \node at (0.5,0.20) [bar] (e) {};
 \draw[] (a) to (b);
 \draw[] (a) to (d);
 \draw[ultra thick] (b) to (0.25,-0.05);
 \draw[ultra thick] (0.25,-0.05) to (e);
 \end{tikzpicture}\,,\,\,\,\begin{tikzpicture}[baseline=0cm,scale=0.8]
 \node at (0,0.7) [bar] (d) {};
 \node at (0.25,0.45) [bnode] (a) {};
 \node at (0,0.20) [bar] (b) {};
 \node at (0.25,-0.05) [bar] (c) {};
 \node at (0.5,0.7) [bar] (e) {};
 \draw[ultra thick] (d) to (0.25,0.45);
 \draw[ultra thick] (0.25,0.45) to (e);
 \draw[] (b) to (c);
 \draw[] (a) to (b);
 \end{tikzpicture}\,,\,\,\,\begin{tikzpicture}[baseline=0cm,scale=0.8]
 \node at (0,0.7) [bar] (d) {};
 \node at (0.25,0.45) [bar] (a) {};
 \node at (-0.25,0.45) [bar] (b) {};
 \node at (0,0.20) [bnode] (c) {};
 \node at (0.25,-0.05) [bar] (e) {};
 \draw[ultra thick] (a) to (0,0.20);
 \draw[ultra thick] (0,0.20) to (b);
 \draw[] (d) to (a);
 \draw[] (0,0.20) to (e);
 \end{tikzpicture}\,,\,\,\,\begin{tikzpicture}[baseline=0cm,scale=0.8]
 \node at (-0.25,0.7) [bar] (d) {};
 \node at (0.25,0.7) [bar] (a) {};
 \node at (0,0.45) [bnode] (b) {};
 \node at (-0.25,0.20) [bar] (c) {};
 \node at (0.25,0.20) [bar] (e) {};
 \node at (0,-0.05) [bnode] (f) {};
 \draw[ultra thick] (d) to (0,0.45);
 \draw[ultra thick] (0,0.45) to (a);
 \draw[] (b) to (c);
 \draw[ultra thick] (c) to (0,-0.05);
 \draw[ultra thick] (e) to (0,-0.05);
 \end{tikzpicture}\,,\,\,\,\begin{tikzpicture}[baseline=0cm,scale=0.8]
 \node at (0.5,0.7) [bar] (d) {};
 \node at (0.25,0.45) [bnode] (a) {};
 \node at (0,0.7) [bar] (b) {};
 \node at (0,0.20) [bnode] (c) {};
 \node at (-0.25,0.45) [bar] (e) {};
 \node at (0.25,-0.05) [bar] (f) {};
 \draw[ultra thick] (d) to (0.25,0.45);
 \draw[ultra thick] (0.25,0.45) to (b);
 \draw[ultra thick] (0.25,0.45) to (0,0.20);
 \draw[ultra thick] (0,0.20) to (e);
 \draw[] (f) to (0,0.20);
 \end{tikzpicture}\,,\,\,\,\begin{tikzpicture}[baseline=0cm,scale=0.8]
 \node at (-0.25,0.7) [bar] (d) {};
 \node at (0,0.45) [bar] (a) {};
 \node at (-0.25,0.20) [bnode] (b) {};
 \node at (-0.5,0.45) [bar] (c) {};
 \node at (0,-0.05) [bnode] (e) {};
 \node at (0.25,0.20) [bar] (f) {};
 \draw[] (d) to (a);
 \draw[ultra thick] (a) to (-0.25,0.20);
 \draw[ultra thick] (-0.25,0.20) to (c);
 \draw[ultra thick] (-0.25,0.20) to (0,-0.05);
 \draw[ultra thick] (f) to (0,-0.05);
 \end{tikzpicture}\,,\,\,\,\begin{tikzpicture}[baseline=0cm,scale=0.8]
 \node at (-0.75,0.70) [bar] (d) {};
 \node at (-0.5,0.45) [bnode] (a) {};
 \node at (-0.25,0.70) [bar] (b) {};
 \node at (-0.25,0.20) [bnode] (c) {};
 \node at (0,0.45) [bar] (e) {};
 \node at (0,-0.05) [bnode] (f) {};
 \node at (0.25,0.20) [bar] (g) {};
 \draw[ultra thick] (d) to (-0.5,0.45);
 \draw[ultra thick] (-0.5,0.45) to (b);
 \draw[ultra thick] (-0.25,0.20) to (-0.5,0.45);
 \draw[ultra thick] (-0.25,0.20) to (e);
 \draw[ultra thick] (-0.25,0.20) to (0,-0.05);
 \draw[ultra thick] (0,-0.05) to (g);
 \end{tikzpicture}\,,\,\,\,\begin{tikzpicture}[baseline=0cm,scale=0.8]
 \node at (0,-0.05) [bnode] (a) {};
 \node at (-0.25, 0.2) [bnode] (b) {};
 \node at (0.25,0.2) [bnode] (c) {};
 \node at (-0.5,0.45) [bar] (d) {};
 \node at (0.5,0.45) [bar] (e) {};
 \node at (-0.12,0.45) [bar] (f) {};
 \node at (0.12,0.45) [bar] (g) {};
 \draw[ultra thick] (d) to (-0.25,0.2);
 \draw[ultra thick] (f) to (-0.25,0.2);
 \draw[ultra thick] (e) to (0.25,0.2);
 \draw[ultra thick] (g) to (0.25,0.2);
 \draw[ultra thick] (-0.25,0.2) to (0,-0.05);
 \draw[ultra thick] (0.25,0.2) to (0,-0.05);
 \end{tikzpicture}\\
  \,    & \begin{tikzpicture}[baseline=0cm,scale=0.8]
 \node at (0,-0.05) [bar] (a) {};
 \node at (0,0.30) [bar] (b) {};
 \node at (-0.25,0.20) [bar] (c) {};
 \node at (0.25, 0.20) [bar] (d) {};
 \node at (0,0.5) {};
 \draw[] (a) to (b);
 \draw[] (a) to (c);
 \draw[] (a) to (d);
 \end{tikzpicture}\,,\,\,\, \begin{tikzpicture}[baseline=0cm,scale=0.8]
 \node at (0,-0.05) [bar] (a) {};
 \node at (-0.12,0.30) [bar] (b) {};
 \node at (0.12,0.30) [bar] (c) {};
 \node at (-0.3, 0.15) [bar] (d) {};
 \node at (0.3,0.15) [bar] (e) {};
 \draw[] (a) to (b);
 \draw[] (a) to (d);
 \draw[ultra thick] (a) to (c);
 \draw[ultra thick] (a) to (e); 
 \end{tikzpicture}\,,\,\,\, \begin{tikzpicture}[baseline=0cm,scale=0.8]
\node at (0,-0.05) [bnode] (a) {};
\node at (-0.25, 0.2) [bar] (d) {};
\node at (0.25,0.2) [bar] (e) {};
\node at (-0.25,0.45) [bar] (f) {};
\node at (0.25,0.45) [bar] (g) {};
\draw[ultra thick] (0,-0.05) to (d);
\draw[ultra thick] (0,-0.05) to (e);
\draw[] (d) to (f);
\draw[] (e) to (g);
\end{tikzpicture}\,,\,\,\,\begin{tikzpicture}[baseline=0cm,scale=0.8]
\node at (0,-0.05) [bnode] (a) {};
\node at (-0.25, 0.2) [bar] (d) {};
\node at (0.25,0.2) [bnode] (e) {};
\node at (-0.25,0.45) [bar] (f) {};
\node at (0.4,0.45) [bar] (g) {};
\node at (0.10,0.45) [bar] (h) {};
\draw[ultra thick] (0,-0.05) to (d);
\draw[ultra thick] (0,-0.05) to (0.25,0.2);
\draw[] (d) to (f);
\draw[ultra thick] (0.25,0.2) to (g);
\draw[ultra thick] (0.25,0.2) to (h);
\end{tikzpicture}\,,\,\,\,\begin{tikzpicture}[baseline=0cm,scale=0.8]
\node at (0,-0.05) [bar] (a) {};
\node at (0, 0.25) [bar] (d) {};
\node at (-0.25,0.45) [bar] (e) {};
\node at (0.25,0.45) [bar] (f) {};
\draw[] (a) to (d);
\draw[] (d) to (e);
\draw[] (d) to (f);
\end{tikzpicture}\,,\,\,\,\begin{tikzpicture}[baseline=0cm,scale=0.8]
\node at (0,-0.05) [bnode] (a) {};
\node at (-0.25, 0.2) [bar] (d) {};
\node at (0.25,0.2) [bar] (e) {};
\node at (-0.40,0.45) [bar] (f) {};
\node at (-0.10,0.45) [bar] (g) {};
\draw[ultra thick] (0,-0.05) to (e);
\draw[ultra thick] (0,-0.05) to (d);
\draw[] (d) to (f);
\draw[] (d) to (g);
\end{tikzpicture}\,,\,\,\,\begin{tikzpicture}[baseline=0cm,scale=0.8]
\node at (0,-0.05) [bar] (a) {};
\node at (0,0.2) [bnode] (d) {};
\node at (0,0.55) [bar] (e) {};
\node at (-0.25,0.45) [bar] (f) {};
\node at (0.25,0.45) [bar] (g) {};
\draw[] (a) to (d);
\draw[ultra thick] (0,0.2) to (e);
\draw[ultra thick] (0,0.2) to (g);
\draw[] (0,0.2) to (f);
\end{tikzpicture}\,,\,\,\,\begin{tikzpicture}[baseline=0cm,scale=0.8]
\node at (0,-0.05) [bar] (a) {};
\node at (0.25,0.2) [bar] (d) {};
\node at (-0.25,0.2) [bar] (e) {};
\node at (0,0.45) [bar] (f) {};
\draw[] (a) to (d);
\draw[] (a) to (e);
\draw[] (e) to (f);
\end{tikzpicture}\,,\,\,\,\begin{tikzpicture}[baseline=0cm,scale=0.8]
\node at (0,-0.05) [bnode] (a) {};
\node at (0,0.3) [bar] (d) {};
\node at (-0.25,0.25) [bar] (e) {};
\node at (0.25,0.25) [bar] (f) {};
\node at (-0.25, 0.5) [bar] (g) {};
\draw[ultra thick] (0,-0.05) to (d);
\draw[] (0,-0.05) to (f);
\draw[ultra thick] (0,-0.05) to (e);
\draw[] (e) to (g);
\end{tikzpicture}\,,\,\,\,\begin{tikzpicture}[baseline=0cm,scale=0.8]
\node at (0,-0.05) [bnode] (a) {};
\node at (0,0.3) [bar] (d) {};
\node at (-0.25,0.25) [bar] (e) {};
\node at (0.25,0.25) [bar] (f) {};
\node at (-0.25, 0.5) [bar] (g) {};
\draw[ultra thick] (0,-0.05) to (d);
\draw[ultra thick] (0,-0.05) to (f);
\draw[] (0,-0.05) to (e);
\draw[] (e) to (g);
\end{tikzpicture}
\\ 
 & \begin{tikzpicture}[baseline=0cm,scale=0.8]
\node at (0,-0.05) [bar] (a) {};
\node at (0.25,0.2) [bar] (b) {};
\node at (-0.25,0.2) [bnode] (c) {};
\node at (-0.4,0.45) [bar] (d) {};
\node at (-0.1, 0.45) [bar] (e) {};
\node at (0,0.65) {};
\draw[] (a) to (b);
\draw[] (a) to (c);
\draw[ultra thick] (-0.25,0.2) to (d);
\draw[ultra thick] (-0.25,0.2) to (e);
\end{tikzpicture}\,,\,\,\,\begin{tikzpicture}[baseline=0cm,scale=0.8]
\node at (0,-0.05) [bnode] (a) {};
\node at (0,0.3) [bar] (b) {};
\node at (0.25,0.25) [bar] (c) {};
\node at (-0.25,0.25) [bnode] (d) {};
\node at (-0.55, 0.5) [bar] (g) {};
\node at (-0.25,0.55) [bar] (e) {};
\draw[ultra thick] (0,-0.05) to (b);
\draw[ultra thick] (0,-0.05) to (c);
\draw[] (0,-0.05) to (-0.25,0.25);
\draw[ultra thick] (-0.25,0.25) to (g);
\draw[ultra thick] (-0.25,0.25) to (e);
\end{tikzpicture}\,,\,\,\,\begin{tikzpicture}[baseline=0cm,scale=0.8]
\node at (0,-0.05) [bnode] (a) {};
\node at (0,0.3) [bar] (b) {};
\node at (0.25,0.25) [bar] (c) {};
\node at (-0.25,0.25) [bnode] (d) {};
\node at (-0.55, 0.5) [bar] (g) {};
\node at (-0.25,0.55) [bar] (e) {};
\draw[ultra thick] (0,-0.05) to (b);
\draw[] (0,-0.05) to (c);
\draw[ultra thick] (0,-0.05) to (-0.25,0.25);
\draw[ultra thick] (-0.25,0.25) to (g);
\draw[ultra thick] (-0.25,0.25) to (e);
\end{tikzpicture}\,,\,\,\,\begin{tikzpicture}[baseline=0cm,scale=0.8]
\node at (0,-0.05) [bnode] (a) {};
\node at (0,0.5) [bar] (b) {};
\node at (0.25,0.25) [bar] (c) {};
\node at (-0.25,0.25) [bnode] (d) {};
\node at (-0.5, 0.5) [bar] (g) {};
\node at (-0.25,0.6) [bar] (e) {};
\draw[ultra thick] (-0.25,0.25) to (b);
\draw[ultra thick] (0,-0.05) to (c);
\draw[ultra thick] (0,-0.05) to (-0.25,0.25);
\draw[] (-0.25,0.25) to (g);
\draw[ultra thick] (-0.25,0.25) to (e);
\end{tikzpicture}
 \\ 
\hline
$-2\kappa$ & \begin{tikzpicture}[baseline=0cm,scale=0.8]
\node at (0,0.05) [bar] (a) {};
\node at (-0.25,0.3) [barx] (b) {};
\draw[] (a) to (b);
\end{tikzpicture}\,,\,\,\,\begin{tikzpicture}[baseline=0cm,scale=0.8]
\node at (0,-0.05) [bnode] (a) {};
\node at (-0.25,0.2) [barx] (b) {};
\node at (0.25,0.2) [bar] (c) {};
\draw[ultra thick] (0,-0.05) to (b);
\draw[ultra thick] (0,-0.05) to (c);
\end{tikzpicture}\,,\,\,\,\begin{tikzpicture}[baseline=0cm,scale=0.8]
\node at (0,0.05) [barx] (a) {};
\node at (-0.25,0.3) [bar] (b) {};
\draw[] (a) to (b);
\end{tikzpicture}\,,\,\,\,\begin{tikzpicture}[baseline=0cm,scale=0.8]
\node at (0,0.05) [barx2] (a) {};
\node at (0.25,0.3) [bar] (b) {};
\node at (-0.25,0.3) [bar] (c) {};
\draw[ultra thick] (a) to (b);
\draw[ultra thick] (a) to (c);
\end{tikzpicture}\,,\,\,\,\begin{tikzpicture}[baseline=0cm,scale=0.8]
\node at (0,-0.05) [bnode] (a) {};
\node at (0,0.2) [bar] (b) {};
\node at (-0.25,0.5) [bar] (c) {}; 
\draw[ultra thick] (0,-0.05) to (b);
\draw[] (b) to (c);
\end{tikzpicture}\,,\,\,\,\begin{tikzpicture}[baseline=0cm,scale=0.8]
\node at (0,-0.05) [bnode] (a) {};
\node at (0,0.2) [bnode] (b) {};
\node at (-0.25,0.5) [bar] (c) {};
\node at (0.25,0.5) [bar] (d) {};
\draw[ultra thick] (0,-0.05) to (0,0.2);
\draw[ultra thick] (0,0.2) to (c);
\draw[ultra thick] (0,0.2) to (d);
\end{tikzpicture}\,,\,\,\,\begin{tikzpicture}[baseline=0cm,scale=0.8]
\node at (0,-0.05) [bnode] (a) {};
\node at (-0.25,0.2) [bar] (b) {};
\node at (0.25,0.2) [bar] (c) {};
\draw[] (0,-0.05) to (b);
\draw[ultra thick] (0,-0.05) to (c);
\end{tikzpicture}\,,\,\,\,\begin{tikzpicture}[baseline=0cm,scale=0.8]
\node at (0,-0.05) [bar] (a) {};
\node at (0,0.2) [bnode] (b) {};
\node at (0,0.45) [bar] (c) {};
\draw[] (a) to (b);
\draw[ultra thick] (0,0.2) to (c);
\end{tikzpicture}\,,\,\,\,\begin{tikzpicture}[baseline=0cm,scale=0.8]
\node at (0,0.5) {};
\node at (0,-0.05) [bnode] (a) {};
\node at (0.25,0.2) [bar] (d) {};
\node at (-0.25,0.2) [bnode] (e) {};
\node at (0,0.45) [bar] (f) {};
\draw[ultra thick] (0,-0.05) to (d);
\draw[ultra thick] (0,-0.05) to (-0.25,0.2);
\draw[ultra thick] (-0.25,0.2) to (f);
\end{tikzpicture}
\\
\hline
$0$ & \begin{tikzpicture}[baseline=0cm,scale=0.8]
\node at (0,0) {$\colb{\mathbbm{1}}$};
\end{tikzpicture} \\
[0.5ex]
 \hline
\end{tabular}
\caption{List of symbols with negative homogeneity}
\label{table1}
\end{table}

\subsection{Renormalised Model}\label{sec:RnmMdl}

In this subsection, we get into the definition of the renormalisation map announced in the previous section. The systematic study of the implementation of such a renormalisation is due to \cite{BHZ}, where \eqref{eq:product} is redefined because in general we expect $(\Pi^{N}_{x})(\btau)(\Pi^{N}_{x})({\color{blue} \bar \tau})$ to diverge. The work \cite{BHZ} uses heavily extended decorations which are not needed for the semi-general convergence theorem, we want to implement. Instead, we follow the approach with the local renormalisation map introduced in \cite{BR18}.
The works \cite{BB21,BB21b} based on this approach give a more general proof of the renormalised equation introduced in \cite{BCCH}.

To showcase the formulation of that result, we require the so-called decorated trees, which in turn are defined via rooted trees. By a rooted tree $T$, we mean a finite tree (connected graph without simple cycles) with a distinguished vertex $\rho_T$, called the root, and a function $\mathfrak{l}:L_T\sqcup E_T\rightarrow\CL$, where $\CL$ is a fixed non-empty set of types, $E_T\subset N_T\times N_T$ is the set of edges of $T$, with $N_T$ denoting the nodes of $T$, and finally $L_T\subset N_T\setminus\{\rho_T\}$ is the set of leaves, that is nodes that are adjacent to a single edge. Pick two symbols $\mcI$ and $\Xi$ and let $ \mathcal{D} := \lbrace \mcI,\Xi \rbrace \times \mb{N}^{d+1}$ define the set of edge decorations. Decorated trees over $\mathcal{D}$ are of the form  $T_{\Labe}^{\Labn} =  (T,\Labn,\Labe) $ where $T$ is a non-planar rooted tree. We also define a partial order $\geq$ on rooted trees by saying that $x\geq y$ for $x,\,y\in N_T$ whenever $y$ lies on the shortest path from $\rho_T$ to $x$. The maps $\Labn : N_T \rightarrow \mathbb{N}^{d+1}$ and $\Labe : E_T \rightarrow \mathcal{D}$ are node, respectively edge, decorations. We denote the set of decorated trees by $ \mfT $. The tree product is defined by 
\begin{equation}  \label{treeproduct}
 	(T,\Labn,\Labe) \cdot  (T',\Labn',\Labe') 
 	= (T \cdot T',\Labn + \Labn', \Labe + \Labe')\;, 
\end{equation} 
where $T \cdot T'$ is the rooted tree obtained by identifying the roots of $ T$ and $T'$. The sums $ \Labn + \Labn'$ and $\Labe + \Labe'$ mean that decorations are added at the root and extended to the disjoint union by setting them to vanish on the other tree. We now make the connection with the symbolic notation introduced in the previous part.

\begin{enumerate}
   \item[--] An edge decorated by  $ (\mcI,a) \in \mathcal{D} $  is denoted by $ \colb{\mathcal{I}_{a}} $. The symbol $  \mathcal{I}_{a} $ is also viewed as the operation that grafts a tree onto a new root via a new edge with edge decoration $ a $. The new root at hand remains decorated with $0$. 
   
   \item[--]  An edge decorated by $ (\Xi,0) \in \mathcal{D} $ is denoted by $  \colb{\Xi} $.

   \item[--] A factor $\colb{ X^k}$   encodes a single node  $ \bullet^{k} $ decorated by $ k \in \mathbb{N}^{d+1}$. We write $ \colb{X_i}$, $ i \in \lbrace 0,1,\ldots,d\rbrace $, to denote $ \colb{X^{e_i}}$. The element $ \colb{X^0} $ is identified with the empty tree $\mathbbm{1}$.
 \end{enumerate}
 
 \begin{remark}
Notice that the noise symbol $\bXi$ was depicted as a node in the previous section, but is seen as a decorated edge above. The result from \cite{BHZ} is of more general applicability, in particular to cases where one might have any finite number of noises attached to the same inner node, for which the notion of a noise edge is natural. For \eqref{eq:gKPZ}, each inner node is at most attached to one instance of noise, so for notational ease, we prefer illustrating the presence of the noise with a node. Of course for this work, the equivalence is obvious.
\end{remark}
 
 The degree $ |\cdot|_{\s} $ of a decorated tree is the one applied to its symbolic notation. The following is easily seen to be true:
\begin{equs}
|{\color{blue} \tau\bar \tau}|_{\s} = |{\color{blue} \tau}|_{\s} + |{\color{blue} \bar \tau}|_{\s}
\end{equs}

 We restrict ourselves to the trees that are generated by the graphical algorithm outlined in the last section from symbols residing only in $\mc{H}$. The collection of these trees is denoted by $\mfT_0$.

 We say a tree $T\in\mfT$ is planted if either $T = \bXi$ or \colb{if} there exists $k\in\mb{N}^{d+1}$ and $\bar T\in\mfT$ such that $T =  \colb{\mathcal{I}_k(\bar \tau)}$, where $\colb{ \bar\tau}$ is the symbol associated to $\bar T$.
 
The set $\mfT_{-}$ is then the set of unplanted trees of negative degree and zero polynomial decoration at the root, which is to say:
$$\mfT_-\coloneqq\{\btau = T^{\mfn}_\mfe\in\mfT_0\,:\,|\btau|_\s < 0,\,\mfn(\rho_T)= 0,\,\btau\text{ is not planted}\}.$$

We set $\mc{T}_{-}\coloneqq\text{Vec}\,\mfT_-$ and then define the local extraction/contraction map $\Deltam_r$ that will be central to the construction of the renormalisation map. It was first introduced in \cite{BR18}.

\begin{definition}\label{def:Deltam}
We define for $T^\mfn_\mfe\in\mfT_0$
\beq\begin{split}
\Deltam_r T^{\mfn}_{\mfe}=\sum_{A\subset_r T}\sum_{\mfe_A,\,\mfn_A}\frac{1}{\mfe_A!}\binom{\mfn}{\mfn_A}&\pi_-(A,\mfn_A+\pi\mfe_A,\mfe\upharpoonright E_A) \\
&\otimes (T/A,[\mfn - \mfn_A]_A,\mfe+\mfe_A),
\end{split}
\eeq
where:
\begin{itemize}
\item $\pi_-:\mathcal{T}\rightarrow\mathcal{T}_-$ which projects every $A\in \mathcal{T}  \setminus \mathcal{T}_-$ to $0$,
\item For $C\subset D$, and $f:D\rightarrow\mb{N}^{d+1}$, the restriction of $f$ to $C$, is denoted by $f\upharpoonright C$. Further one defines by $f! = \prod_{x\in D} f(x)!$ where $n! = \prod_{i=0}^d n_i!$ for $n\in\mb{N}^{d+1}$. Given another $g:C\rightarrow\mb{N}^{d+1}$, we denote:
$$\binom{f}{g}=\prod_{x\in C}\binom{f(x)}{g(x)},$$
and unsurprisingly $\binom{n}{m}=\prod_{i=0}^{d}\binom{n_i}{m_i}$ for $k,\,n\in\mb{N}^{d+1}$.
\item The outer sum is over all subtrees $A$ of $T$ which have the same root as $ T $. The inner sum runs over all $\mfn_A:N_A\rightarrow\mb{N}^{d+1}$ with $\mfn_A\le \mfn$, and $\mfe_A:\partial(A,T)\rightarrow\mb{N}^{d+1}$, where the \textit{boundary of} $A$ in $T$ is defined by:
$$\partial(A,T)\coloneqq\{(x,y)\in E_T\setminus E_A : y \geq x, x\in A\}$$
\item For $f:E_T\rightarrow\mb{N}^{d+1}$, we set for every $x\in N_T$, $(\pi f)(x) = \sum_{(x,y)\in E_T}f(x,y)$
\item We write $T/A$ for the tree obtained by contracting  $A$ to the root $ \rho_T $ of $ T $. For $f:N_T\rightarrow\mb{N}^{d+1}$ we define $[f]_A: N_{T/A}\rightarrow\mb{N}^{d+1}$ by $[f]_A(\rho_T)=\sum_{y\in N_{A}} f(y)$ and $[f]_A(x) = f(x)$ for $x\in N_{T/A}\setminus\{ \rho_{T}\}$.
\end{itemize}
\end{definition}

The definition above gives us a simple representation for the renormalisation map $ R_g $ that we will use for the model:
\begin{equs}
R_g = \left(  g\otimes \id \right) \Deltam_r
\end{equs}
with $ g  $ a linear map from $ \CT_- $ into $ \mathbb{R} $. 

\begin{example}
Consider the simple tree given by: $ T = \begin{tikzpicture}[baseline = 0.05cm]
\node at (0,0) [bar] (a) {};
\node at (0.25,0.25) [bar] (b) {};
\draw[] (a) to (b);
\end{tikzpicture}$, for which one can check that $|T|_\s = -1-$. Then an easy computation gives:
\begin{equs}
\Deltam_r \begin{tikzpicture}[baseline = 0.05cm]
\node at (0,0) [bar] (a) {};
\node at (0.25,0.25) [bar] (b) {};
\draw[] (a) to (b);
\end{tikzpicture} = \mathbbm{1}\otimes\begin{tikzpicture}[baseline = 0.05cm]
\node at (0,0) [bar] (a) {};
\node at (0.25,0.25) [bar] (b) {};
\draw[] (a) to (b);
\end{tikzpicture} + \begin{tikzpicture}[baseline = 0.05cm]
\node at (0,0) [bar] (a) {};
\node at (0.25,0.25) [bar] (b) {};
\draw[] (a) to (b);
\end{tikzpicture}\otimes\mathbbm{1},
\quad
R_g\begin{tikzpicture}[baseline = 0.05cm]
\node at (0,0) [bar] (a) {};
\node at (0.25,0.25) [bar] (b) {};
\draw[] (a) to (b);
\end{tikzpicture} = g(\mathbbm{1})\begin{tikzpicture}[baseline = 0.05cm]
\node at (0,0) [bar] (a) {};
\node at (0.25,0.25) [bar] (b) {};
\draw[] (a) to (b);
\end{tikzpicture} - g\left(\begin{tikzpicture}[baseline = 0.05cm]
\node at (0,0) [bar] (a) {};
\node at (0.25,0.25) [bar] (b) {};
\draw[] (a) to (b);
\end{tikzpicture}\right)\mathbbm{1}
\end{equs}
where we perform the following identification for the tensor product $a\otimes{ \color{blue} \tau} = a\btau$ with $a\in\mathbb{R}$. As $g$ is a character one has directly that $g(\mathbbm{1}) = 1$ for any choice of $g$.
\end{example}

The BPHZ renormalisation given in \cite{BHZ} can be recovered by a specific choice of $ g $:
\begin{equation}\label{eq:BPZg}
g_{\text{\tiny{BPHZ}}}(\btau) = -  \mathbb{E} \left(  (\Pi^{N,g} \btau)(0) \right)
\end{equation}
where one has:
$$\Pi^{N,g} \btau   = \hat{\Pi}^{N,g} R_g \btau, \quad  \hat{\Pi}^{N,g}  \btau {\color{blue} \bar \tau }=  \hat{\Pi}^{N,g}  \btau  \hat{\Pi}^{N,g}   {\color{blue} \bar \tau}. \quad \Pi^{N,g} \colb{X^k}(s,y)  = (s,y)^k,$$
$$(\hat{\Pi}^{N,g} \bcI{\tau})(s,y)  = 2^{-dN}\sum_{\bar x\in\mb{T}_N^d}\int K^N((s,y),(\bar s,\bar x))(\Pi^{N,g}\btau)(\bar s,\bar x)\, d\bar s.$$
This recursive formula has been introduced in \cite{BB21b} in a non-translation invariant setting.
Then, the renormalised model is described recursively by a similar formula:
\be
\begin{split}
\label{eq:Xk}
(\hat{\Pi}_{(t,x)}^{N,g} \colb{X^k})(s,y) &= ((s,y)-(t,x))^k \quad
\Pi^{N,g}_{(t,x)} \btau   = \hat{\Pi}^{N,g}_{(t,x)} R_g \btau, \\  &\hat{\Pi}^{N,g}_{(t,x)}  \btau {\color{blue} \bar \tau} =  \hat{\Pi}^{N,g}_{(t,x)}  \btau  \hat{\Pi}^{N,g}_{(t,x)}   {\color{blue} \bar \tau},
\end{split}
\ee
\be
\begin{split}
(\hat{\Pi}_{(t,x)}^{N,g}\bcI{\tau})(s,y)= 2^{-dN}\sum_{\bar x\in\mb{T}_N^d}\int K^N((s,y),(\bar s,\bar x))(&\Pi_{(t,x)}^{N,g}\btau)(\bar s,\bar x)\, d\bar s \\
-\sum_{|k|_\s<|\btau|+2} \frac{((s,y)-(t,x))^k}{k!}Q_k\big((T_{|\btau|+2}^{N}\Pi_{(t,x)}^{N,g} \btau )(t,x)\big).
\end{split}
\ee

\begin{example}
For $\btau = \begin{tikzpicture}[baseline = 0.05cm]
\node at (0,0) [bar] (a) {};
\node at (0.25,0.25) [bar] (b) {};
\draw[] (a) to (b);
\end{tikzpicture}$, let us compute $g\left(\begin{tikzpicture}[baseline = 0.05cm]
\node at (0,0) [bar] (a) {};
\node at (0.25,0.25) [bar] (b) {};
\draw[] (a) to (b);
\end{tikzpicture}\right)$ under the BPHZ paradigm. Due to \eqref{eq:BPZg} and the definition of $\Pi^{N,g}$ following it, the quantity that we are interested in is:
\be
\begin{split}
\mb{E}\left[\left(\Pi^{N,g}\begin{tikzpicture}[baseline = 0.05cm]
\node at (0,0) [bar] (a) {};
\node at (0.25,0.25) [bar] (b) {};
\draw[] (a) to (b);
\end{tikzpicture}\right)(0)\right] = \mb{E}\left[\left(\Pi^{N,g}\begin{tikzpicture}[baseline = 0.05cm]
\node at (0.25,0.25) [bar] (b) {};
\draw[] (0,0) to (b);
\end{tikzpicture}\right)(0)\left(\Pi^{N,g}\begin{tikzpicture}[baseline = 0.05cm]
\node at (0,0.125) [bar] (b) {};
\end{tikzpicture}\right)(0)\right] &= \mb{E}\left[\int K^N(-z)\xi^{N}(z)\xi^{N}(0)\,dz\right] \\
&= \int K^N(-z)\mb{E}\left[\xi^N(z)\xi^N(0)\right]\,dz,
\end{split}
\ee
where we have used the shorthands $z = (t,x)\in\mb{R}^{d+1}$ and $\xi^N(z) = \xi^N_t(x) = \xi_{2^{2N}t}(2^N\,x)$ and the integral is the semi-discrete integral i.e. integrals over space are actually the Riemann sums $2^{-dN}\sum_{x\in\mb{T}^{d}_N}$. Notice that due to assumption of centrality on our noise and the properties of cumulants, we have $\mb{E}\left[\xi^N(z)\xi^N(0)\right] = \mb{E}_c\left[\xi^N(z),\xi^N(0)\right]$. One can check that the constant $c_N$ in \cite[Sec. 2.3]{MH21} associated to $\btau$ is the same as the quantity above and our renormalisation map achieves the same effect theirs does.
\end{example}

In the manner explained above one can compute the counter terms $C_N$ associated with any $\btau\in\mc{T}_-$. The following example explains the difference that stems from the fact that we allow for noises that are (generally) non-Gaussian.

\begin{example}
Consider the tree: $\btau = \begin{tikzpicture}[baseline = 0.1cm]
\node at (0,0) [bar] (a) {};
\node at (-0.25,0.25) [bar] (b) {};
\node at (0,0.5) [bar] (c) {};
\draw[] (a) to (b);
\draw[] (b) to (c);
\end{tikzpicture}$. One checks then that:

$$\Deltam_r\begin{tikzpicture}[baseline = 0.1cm]
\node at (0,0) [bar] (a) {};
\node at (-0.25,0.25) [bar] (b) {};
\node at (0,0.5) [bar] (c) {};
\draw[] (a) to (b);
\draw[] (b) to (c);
\end{tikzpicture}=\mathbbm{1}\otimes\begin{tikzpicture}[baseline = 0.1cm]
\node at (0,0) [bar] (a) {};
\node at (-0.25,0.25) [bar] (b) {};
\node at (0,0.5) [bar] (c) {};
\draw[] (a) to (b);
\draw[] (b) to (c);
\end{tikzpicture} + \begin{tikzpicture}[baseline=0.1cm]
\node at (0,0) [bar] (a) {};
\node at (-0.25,0.25) [bar] (b) {};
\draw[] (a) to (b);
\end{tikzpicture}\otimes\begin{tikzpicture}[baseline=0.1cm]
\node at (0,0) [bar] (a) {};
\node at (0.25,0.25) [bar] (b) {};
\draw[] (a) to (b);
\end{tikzpicture}+\begin{tikzpicture}[baseline = 0.1cm]
\node at (0,0) [bar] (a) {};
\node at (-0.25,0.25) [bar] (b) {};
\node at (0,0.5) [bar] (c) {};
\draw[] (a) to (b);
\draw[] (b) to (c);
\end{tikzpicture}\otimes\mathbbm{1}.$$

We have already seen how to compute $g(\mathbbm{1})$ and $g\left(\begin{tikzpicture}[baseline = 0.05cm,scale=0.8]
\node at (0,0) [bar] (a) {};
\node at (0.25,0.25) [bar] (b) {};
\draw[] (a) to (b);
\end{tikzpicture}\right)$, the new calculation at hand is $g\left(\begin{tikzpicture}[baseline = 0.1cm]
\node at (0,0) [bar] (a) {};
\node at (-0.25,0.25) [bar] (b) {};
\node at (0,0.5) [bar] (c) {};
\draw[] (a) to (b);
\draw[] (b) to (c);
\end{tikzpicture}\right)$. A very similar computation as in the previous example gives us:
\be
\mb{E}\left[\left(\Pi^{N,g}\begin{tikzpicture}[baseline = 0.1cm,scale=0.8]
\node at (0,0) [bar] (a) {};
\node at (-0.25,0.25) [bar] (b) {};
\node at (0,0.5) [bar] (c) {};
\draw[] (a) to (b);
\draw[] (b) to (c);
\end{tikzpicture}\right)(0)\right]=\int K^{N}(-z_1)K^N(z_1-z_2)\mb{E}\left[\xi^{N}(z_1)\xi^N(z_2)\xi^N(0)\right]\,dz_1 dz_2,
\ee

Were our noises centred Gaussian, the above expectation would have been zero due to the centredness and Wicks Formula. In the non-Gaussian case, the above integral is no longer guaranteed to vanish. As we have assumed control of the cumulants, we again use the assumption of centred noises to give $\mb{E}\left[\xi^{N}(z_1)\xi^N(z_2)\xi^N(0)\right] = \mb{E}_c\left[\xi^{N}(z_1),\xi^N(z_2),\xi^N(0)\right]$. Hence we find that this constant is the same as  $c^{(1)}_N$ in \cite[Sec. 2.3]{MH21}.
\end{example}

We have thus seen how the counter-terms that renormalise our equation are generated for each of the symbols in Table~\ref{table1}. In subtracting these counter terms, one cannot expect that our admissible model solves, in the manner described in the previous section, the equation \eqref{eq:gKPZ} we stated in the introduction. The "renormalised equation" that it does solve, which is due to  \cite[Th. 2.22]{BCCH}, takes the form:
\begin{equs}
\partial_t u = \Delta u + g(u)(\nabla_{\!x} u)^2+k(u)(\nabla_{\!x}u)+\bar{h}(u)+f(u)\xi_t(x),
\end{equs}
where $\bar{h}$ is given by:
\begin{equs}
\bar{h}(u) & = h(u) + \sum_{\btau \in \mc{T}_-} \frac{\Upsilon[\btau]}{S(\btau)} C_N(\btau),
\end{equs}
where the $C_N$ are found exactly in the manner we have explained above. The definitions of the symmetry factors and the elementary differential operators are directly from \cite[Sec 2.7]{BCCH}. Let's mention that \cite{BCCH} uses heavily decorated trees with extended decorations while \cite{BB21,BB21b} proposes a shorter proof valid for local renormalisation maps and without the extended decorations.
 
\section{Labelled Graphs} \label{sec:lblgrph}

It is known from the literature on Discrete Regularity Structure \cite{EH17,HM18} that the sort of convergence we seek above relies on certain stochastic estimates on $\hat \Pi^N$ and $\Pi^N$. As the general theory that exists in the continuum for getting these stochastic estimates from \cite{CH16} remains out of reach for the discrete context, we follow the method in \cite{MH21} that builds on the theory from \cite{HQ15}. Our contribution here is to extend the theory so that it is applicable to \eqref{eq:gKPZ}.

\subsection{Directed MultiGraph and Contracted Graphs}
The main ingredient will be finite directed multigraphs $\ds{G}=(\ds{V},\ds{E})$, with edges $e = (e_-,\,e_+)\in\ds{E}$ that originate from $e_-\in\ds{V},$ culminate at $e_+\in\ds{V}\hspace{-0.5mm},$ that carry the label $(a_e,\,r_e,\,v_e)\in\mathbb{R}\times\mathbb{Z}\times\ds{V},$ and are associated with compactly supported kernels $K_e:\mathbb{R}\times\mathbb{T}^d_N\mapsto\mathbb{R}$ that satisfy the following bound:
 \beq\label{eq:KerTay}
 |D^{k}K_e(z)|\lesssim\|z\|_\s^{-a_e-|k|_\s},
 \eeq
 uniformly over $\|z\|_\s\le 1$ and for all multi-indices $k$. In this sense, $a_e$ is the order of the singularity of the kernel $K_e$ associated to the edge. The renormalisation procedure we purpose will involve Taylor expansions whose particulars are given by the quantities $r_e,\,v_e$. That is to say, we associate to each edge $e\in\ds{E}$, a kernel $\hat{K}_e$ via a Taylor expansion around another point and the length of this expansion is encoded by $r_e\in\mathbb{Z}_{\ge 0}$ while the point of expansion $x_{v_e}\in\mathbb{R}^d$ is encoded by $v_e$:
 \begin{equation}\label{eq:defKernel}
 \hat{K}_e(x_{e_-},x_{v_e},x_{e_+})=K_e(x_{e_+}-x_{e_-})-\sum_{|j|_\s<r_e}\frac{(x_{e_+}-x_{v_e})^j}{j!}D^jK_e(x_{v_e}-x_{e_-}).
 \end{equation}
Whenever $v_e = v_0$, we suppress the dependence to write $\hat{K}_e(x_{e_-},x_{e_+})$. In the above definition, one notices the first major departure from the \cite{HQ15} and all derivative literature. The only meaningful value for $r_e$ to take here is non-negative, whereas in \cite{HQ15} the handling of $r_e < 0$ is very non-trivial. Furthermore, in \cite{HQ15} the role of $v_e$ is played by the distinguished vertex $0$ (see below), whereas we allow it to be any $v_e\in\ds{V}$.

In the case that we have multiple edges with the same origin and destination vertices, we will assume only one has a non-zero $r_e$ label, and we will concatenate all the edges into one edge with the label $(\hat{a}_e,r_e)$ where $\hat{a}_e$ denotes the sum of all the $a_e$ for the edges that were concatenated. We require $\mathds{G}$ to have the distinguished vertex $0\in\mathds{V}$ that will be connected by distinguished edges $e_{\star,1},\,\hdots,\,e_{\star,M}$ to $M$ distinguished vertices $v_{\star,1},\,\hdots,\,v_{\star,M}$. All of these distinguished edges will carry the label $(0,0,0)$. We also define the sets $\ds{V}_0\coloneqq\ds{V}\setminus{0}$ and $\ds{V}_\star\coloneqq\{0,v_{\star,1},\,\hdots,\,v_{\star,M}\}$.
Finally, we define for any subset $\bar{\mathds{V}}\subseteq{\mathds{V}}$ the following sets of edges:
\begin{equs}
& \mathds{E}^\uparrow(\bar{\mathds{V}}) \coloneqq \{e\in \mathds{E}:e_-\in\bar{\mathds{V}}\}, \quad \mathds{E}^\downarrow(\bar{\mathds{V}}) \coloneqq \{e\in \mathds{E}:e_+\in\bar{\mathds{V}}\}, \\ & \mathds{E}_0(\bar{\mathds{V}}) \coloneqq \{e\in \mathds{E}:e_\pm\in\bar{\mathds{V}}\}, \quad 
\mathds{E}(\bar{\mathds{V}}) \coloneqq \{e\in \mathds{E}:e_+\in\bar{\mathds{V}}\textnormal{ or }e_-\in\bar{\mathds{V}}\},
\\ & \mathds{E}_+(\bar{\mathds{V}})\coloneqq\{e\in \mathds{E}(\bar{\mathds{V}}):r_e > 0\}, \quad \mathds{E}_+^\uparrow(\bar{\mathds{V}}): =\mathds{E}_+(\bar{\mathds{V}})\cap\mathds{E}^\uparrow(\bar{\mathds{V}}),
\\ &  \mathds{E}_+^\downarrow(\bar{\mathds{V}})\coloneqq\mathds{E}_+(\bar{\mathds{V}})\cap\mathds{E}^\downarrow(\bar{\mathds{V}}), \quad \mathds{E}_{\star} \coloneqq \{ e \in \mathds{E} : \, \mathds{V}_{\star} \cap e = e \}.
\end{equs}
We levy the following assumptions on our graphs:
\begin{assumption}\label{ass:grph} For a given graph $\mathds{G} = (\mathds{V}, \mathds{E})$ we require:
\begin{itemize}
    \item For every subset $\bar{\mathds{V}} \subset {\mathds{V}}$, such that $|\dsV|\ge 2$, one has:
\beq\label{eq:ass2e1}\sum_{e \in \mathds{E}_0 (\bar{\mathds{V}})} \hat{a}_e + \hspace{-2mm}\sum_{e \in \mathds{E}^{\uparrow}_+ (\bar {\mathds{V}})} \hspace{-2mm}\mathbbm{1}_{\{v_e\in\bar{\mathds{V}}\}}(\hat{a}_e + r_e - 1) - \hspace{-2mm}\sum_{e \in \mathds{E}^{\downarrow} (\bar{\mathds{V}})} \hspace{-2mm}\mathbbm{1}_{\{v_e\in\bar{\mathds{V}}\}}r_e < \bigl(|\bar{\mathds{V}}| - 1\bigr)|\mathfrak{s}|.\eeq
 \item For every non-empty subset $\bar{\mathds{V}} \subset \hat{\mathds{V}}\setminus\mathds{V}_{\star}$ one has:
\beq\label{eq:ass2e2}
\begin{split}\sum_{e \in \mathds{E}_0(\bar{\mathds{V}})} \hspace{-3mm}\hat{a}_e + \hspace{-3mm}\sum_{e \in \mathds{E}^{\downarrow}(\bar{\mathds{V}})} \hspace{-2mm}\bigl(\mathbbm{1}_{\{v_e\in\bar{\mathds{V}}\vee r_e=0\}}&(\hat{a}_e + r_e - 1)-(r_e - 1)\bigr) \\
&+\sum_{e \in \mathds{E}^{\uparrow}(\bar{\mathds{V}})}((\hat{a}_e + r_e) - \mathbbm{1}_{\{v_e\in\hat{\mathds{V}}\}}r_e) > |\bar{\mathds{V}}| |\mathfrak{s}|,\end{split}
\eeq
\end{itemize}
\end{assumption}

where $\vee$ denotes the logical OR operator. An initial observation to make here is that these assumptions are different from the ones in \cite[Assumption A.1]{MH21}. The "renormalisation" and "recentring" conditions of there's, have been adapted to our renormalisation paradigm, while the milder condition on integrability (i.e. A.1.1 and A.1.2. in \cite{MH21}) have been absorbed into \eqref{eq:ass2e1}. Indeed if one sets $\bar{\dsV}=\{(e_-,e_+)\}$ for any edge $e\in\dsE$, \eqref{eq:ass2e1} becomes $\hat{a}_e < |\s|$ which is exactly Assumption A.1.1 seeing as $r_e \wedge 0 =0$ for every edge in our setup. Similarly whenever $v_e\not\in\bar\dsV$, \eqref{eq:ass2e1} becomes $$\sum_{e\in\dsE_0(\bar{\dsV})}\hat{a}_e<(|\hat{\dsV}|-1)|\s|,$$ which is just Assumption A.1.2. Why we need $v_e\in\bar{\dsV}$, will be clearer to the reader once we have explained our renormalisation procedure.
Given our kernels $K_e$, we will also assume the existence of the following particular decomposition:
\begin{assumption}\label{ass:Kernel}
Given $K_e$ as above, we assume that there exist $\{K_e^{(n)}\}_{0\le n \le N}$ satisfying:
\begin{itemize}
\item $K_e(z) = \sum_{0\le n\le N}K^{(n)}_e$ for all $z\neq 0$.
\item For all $0\le n < N$ the functions $K^{(n)}_e$ are supported in the annulus $2^{-(n+2)}\le\|z\|_\s\le 2^{-n}$ and $\text{supp}(K^{N}_e)\subseteq\{z:\|z\|_\s\le 2^{-N}\}$
\item for all $p<\infty$ and some $C<\infty$
\begin{equation}\label{e:boundKn}
\begin{aligned}
\sup_{\substack{|k| \le p\\ 0\leq n\leq N-1 }} 2^{-(a_e + |k|_\s)n}
|D^k K_e^{(n)}(z)| &\le C \quad \text{and}\\
2^{-a_eN} |K_e^{(N)}(z)| & \le C \, ;
\end{aligned}
\end{equation}
\end{itemize}
\end{assumption}
where the $N$ in the above definition is as it was fixed in $\mb{Z}_N^d$. The smallest $C$ that appears in \eqref{e:boundKn} will be denoted by $\|K_e\|_{a_e;p}$. For a smooth test function $\varphi$ and $z=(t,x)$, we set $$\phi_{\lambda,\mu}(z) = \lambda^{-|\s\setminus\s_1|}\mu^{-|\s_1|} \phi(t/\mu^{|\s_1|},x/\lambda).$$ This is another departure from \cite{HQ15}, but this time towards more generality as seen in \cite{MH21}. One has to take care to distinguish between the two regimes: $2^{-N}\le\lambda=\mu\text{ or }\mu<2^{-N}=\lambda$. In the former, $\mu = \lambda$ regime, the arguments are almost the same as in \cite{HQ15}, so we explicate on the more demanding case. 

The purpose of this section is to be able to bound \textit{Generalised Convolutions} of the following form:
\begin{equation}\label{eq:genconv}
\CI^\ds{G}(\phi_{\lambda,\mu},K)\coloneqq  \int_{(\mathbb{R}\times\mathbb{T}_N^d)^{\ds{V}_0}} \prod_{e\in \ds{E}}{\hat K}_e(z_{e_-}, z_{e_+})\prod_{i=1}^M \phi_{\lambda,\mu}(z_{v_{\star,i}}) \,dz,
\end{equation}
where $\mathbb{T}^d_N=2^{-N}\mathbb{Z}^d_N$. When $\lambda = \mu$, we simply write $\CI^{\ds{G}}(\varphi_\lambda,K)$. The main result (adapted to this setup) provides exactly this:
\begin{theorem}\label{th:HQbnd}  Let $\ds{G}=(\ds{V},\ds{E})$ be a finite directed multigraph with labels $(a_e,r_e,v_e) $ and kernels
$K_e$ with $e\in \ds{E}$, such that Assumption~\ref{ass:grph} are satisfied and that the kernels satisfying Assumption~\ref{ass:Kernel}.
Then, there exists  $C$ depending only on the 
structure of the graph $(\ds{V},\ds{E})$ and the labels $r_e$ such that 
\begin{equation}\label{eq:genconvBound}
\CI^\ds{G}(\phi_{\lambda,\mu},K) 
\le C\lambda^{\tilde \alpha}\;,
\end{equation}
if either $2^{-N}\leq \lambda=\mu$ or $\mu <2^{-N}=\lambda$.
Here,
\begin{equation}\tilde \alpha = |\s||\ds{V}\setminus \ds{V}_\star| - \sum_{e\in \ds{E}} a_e. 
\end{equation} 
\end{theorem}

We do not prove directly Theorem~\ref{th:HQbnd} but we rewrite it via a multi-scale decomposition of the kernels
$ K_e $. We need to introduce some notations to present this expansion. The first building block is to define a cutoff function $\Psi^{(n)}$ for every $ n \in \mathbb{Z} $ using a smooth function supported on $[3/8,1]$, $\psi:\mathbb{R}\rightarrow [0,1]$ such that $\sum_{n\in\mathbb{Z}}\psi(2^nx)=1 \text{ for every } x \neq 0$. We set for every $n\in\mathbb{N}$: $\Psi^{(n)}=\psi(2^nx)$.

\begin{definition}\label{def:Kn}
For any $\textbf{n}_1:\ds{E}\setminus\ds{E}_\star\mapsto\{0,1,\hdots,N-1,N\}^3$ and $e\in\ds{E}\setminus\ds{E}_\star$, we define a function $\hat{K}_e^{(\textbf{n}_1(e))}(z_1,z_2)$ as follows: if $r_e \le 0$ then $\hat{K}_e^{(\textbf{n}_1(e))}=0$ unless $\textbf{n}_1=(k,0,0)$ in which case $\hat{K}_e^{(\textbf{n}_1(e))}=K^{(k)}_e(z_2-z_1)$; if $r_e > 0$, $\hat{K}_e^{(\textbf{n}_1(e))}=0$, unless $\mathbf{n}=(k,p,m)\in\{0,1,\hdots,N-1\}^3$, in which case we have:
\begin{equs}
\hat{K}_e^{(k,p,m)}(z_1,z_2)&=\Psi^{(k)}(z_2-z_1)\Psi^{(p)}(z_{v_e}-z_1)\Psi^{(m)}(z_2-z_{v_e}) \\ &\quad\Bigl(K_e(z_2-z_1)-\sum_{|j|_{\mathfrak{s}}<r_e}\frac{(z_2-z_{v_e})^{j}}{j!}D^jK_e(z_{v_e}-z_1)\Bigr),
\end{equs}
or $\textbf{n}_1(e)=(N,0,0)$ in which case:
$$\hat{K}^{(\textbf{n}_1(e))}_e(z_1,z_2)=\Psi^{(N)}(z_2-z_1)K_e(z_2-z_1).$$
\end{definition}
Note that the main difference in our treatment is in defining $\hat{K}_e^{(k,p,m)}$. Finally to avoid issues of differentiability at $t=0$, we also define for a function $\textbf{n}_2:\ds{E}_\star\mapsto\{0,1,\hdots\lceil|\log_2(\mu)|\rceil\}^3$ and $e\in\ds{E}_\star$, the kernel $\hat{K}_e^{(\textbf{n}_2(e))}$ as being zero unless $\textbf{n}_2(e)=(k,0,0)$ in which case $\hat{K}_e^{(\textbf{n}_2(e))}=\Psi^{(k)}(z_2-z_1)K^{(k)}_e(z_2-z_1)$.
The multiplicity in the definitions allows the following definition:
\begin{equation}\label{def:kayhat}
\hat{K}^{(\textbf{n})}(z) = \prod_{e \in \ds{E}} \hat{K}_e^{(\textbf{n}_e)}(z_{e_-},z_{e_+}),
\end{equation}
where we have defined:
$$\textbf{n}_e=\begin{cases}
\textbf{n}_1(e)\mbox{ if }e\in\ds{E}\setminus\ds{E}_\star \\
\textbf{n}_2(e)\mbox{ if }e\in\ds{E}_\star
\end{cases}$$
Whenever $\textbf{n}$ is given in such a form we write $\textbf{n}=\textbf{n}_1\cup\, \textbf{n}_2$. For $ 0 < \mu < \lambda = 2^{-N} \le 1$, we set:
\begin{equation}
\CN_{\lambda,\mu}\coloneqq\{ 
\textbf{n}=\textbf{n}_1\sqcup\textbf{n}_2 \colon\, 2^{-|\textbf{n}_2(e_{\star,i})|}\leq \mu,\,  i=1,\ldots,M\}\;,
\end{equation}
And then finally:
\begin{equation}\label{eq:bigsum}
\CI_{\lambda,\mu}^\ds{G}(K) \coloneqq \sum_{\textbf{n} \in \CN_{\lambda,\mu}} \int_{(\mathbb{R}\times\mathbb{T}_N^d)^{\ds{V}_0}} \hat{K}^{(\textbf{n})}(z)\,dz\;.
\end{equation}
The benefit of the above definition is that the following theorem, when established will give us Theorem~\ref{th:HQbnd}. Indeed assume that we know the following theorem to be true. Then the facts that test functions can be interpreted as kernels with $a_e = 0$ and that $\|K_{e}\|_{a_e;\,p}\lesssim\lambda^{-|\s\setminus \s_1|}\mu^{-|\s_1|}$ allow us to conclude.
\begin{theorem}
\label{lem:wantedBound2}
Under the same assumptions as in Theorem~\ref{th:HQbnd}, the 
inequality 
\begin{equation}\label{e:wantedBound2}
|\CI_{\lambda,\mu}^\ds{G}(K)| \leq C\lambda^\alpha (\mu^{|\s_1|}2^{N|\s_1|})^M\prod_{e\in\ds{E}}\|K_e\|_{a_e;p},
\end{equation}
holds. 
If $\lambda=\mu$, we also have:
\begin{equation}\label{e:wantedBound21}
|\CI_{\lambda,\mu}^\ds{G}(K)| \leq C\lambda^\alpha \prod_{e\in\ds{E}}\|K_e\|_{a_e;p}.
\end{equation}
In both cases $\alpha =|\s| |\ds{V}_0| - \sum_{e \in \ds{E}} a_e$.
\end{theorem}
 In the next section we prove Theorem~\ref{lem:wantedBound2} and hence Theorem~\ref{th:HQbnd}. 

\subsection{Partition of Integral Domain}

As in \cite{HQ15}, we want to associate to every point $z\in(\mathbb{R}\times\mathbb{T}^d_N)^{{\mathds{V}}}$ a rooted labelled binary tree, where by rooted tree we mean a connected graph without cycles, which has a distinguished node called the root, which is labelled in the sense that a mapping from the set of its edges into reals exists and finally binary in the usual graph theoretic way.

We denote by $\mathcal{B}({\mathds{V}})$ the set of all labelled rooted binary trees which have ${\mathds{V}}$ as their set of leaves. We further define a partial order by saying that for inner nodes $\nu$ and $\omega$ in a tree in $\mathcal{B}(\ds{V})$, $\nu \ge \omega$ means that $\omega$ lies on the shortest path from the root to $\nu$. We write $\nu\wedge\omega$ for the most recent ancestor of $\nu$ and $\omega$. The main departure here from \cite{HQ15} which has been introduced in \cite{MH21}, is that when $\nu$ is the most recent ancestor of two elements in $\ds{V}_\star$, in lieu of the usual node labelling $\ell$ we assign to it the pair of node labels: $(\ell^T_\nu,\ell_\nu)=(\lceil|\log_2(\mu)|\rceil,N)$. This change is necessitated by the fact that under the scale $2^{-N}$ space does not matter and the time variable does not need to be controlled beyond scale $\mu$. Now we may construct our trees so as to be able to impose:
$$\ell^\alpha(v) \ge \ell^\beta(w) \text{ for } v\ge w,$$
where the $\alpha\text{ and }\beta$ can either be empty (in which case we mean $\ell$) or it could be $T$ (by which we mean $\ell^{T}$). The set of labelled trees that is constructed in this manner is denoted by $\mathbb{T}(\dsV)$ with a general element denoted by $(T,\ell,\ell^T)$.

\begin{example}
Consider for example $\dsV = \{v_1,v_2,v_3\}$. The following diagram is of one possible example of a rooted tree on it (with $r$ denoting the root).
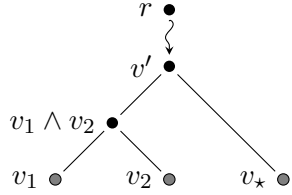
\begin{center}
    \begin{tikzpicture}[scale=0.75]
\node at (0,0) [circle,fill=black, inner sep=0pt, minimum size=1.5mm, label=left:$r$] (a) {};
\node at (0,-1) [circle,fill=black, inner sep=0pt, minimum size=1.5mm, label=left:$v'$] (b) {};
\node at (-1,-2) [circle,fill=black, inner sep=0pt, minimum size=1.5mm, label=left:$v_1\wedge v_2$] (c) {};
\node at (-2,-3) [circle,draw=black,fill=gray, inner sep=0pt, minimum size=1.5mm, label=left:$v_1$] (e) {};
\node at (0,-3) [circle,draw=black,fill=gray, inner sep=0pt, minimum size=1.5mm, label=left:$v_2$] (f) {};
\node at (2,-3) [circle,draw=black,fill=gray, inner sep=0pt, minimum size=1.5mm, label=left:$v_\star$] (g) {};
\draw[->, snake=snake, segment amplitude = 0.5mm, line after snake = 0.5mm, shorten >=2pt,shorten <=2pt] (a) -- (b);
\draw[shorten >=2pt,shorten <=2pt] (b) -- (c);
\draw[shorten >=2pt,shorten <=2pt] (c) -- (e);
\draw[shorten >=2pt,shorten <=2pt] (c) -- (f);
\draw[shorten >=2pt,shorten <=2pt] (b) -- (g);
\end{tikzpicture}
\end{center}
Here for example we have $\ell^\alpha_{v_1\wedge v_2}\ge\ell^\beta_{v'}$ exactly because $v'$ lies on the path from $r$ to $v_1\wedge v_2$.
\end{example}

 The labelled trees $(T, \ell,\ell^T)$ partition the domain of integration and the labelling $\ell$ is further such that for $v, \ w\in\mathds{V}$, we have:
\begin{equs}
\label{eq:labelbnd}\|z_\nu-z_\omega\|_{\mathfrak{s}}\thicksim 2^{\colb{-}\ell_{\nu\wedge 
\omega}}.
\end{equs}
For the above bound we have to replace $\ell_{v\wedge w}$ by $\ell_{v\wedge w}^T$ if $v,w\in\ds{V}_\star$.
\begin{definition}
Set $c := \log|\dsV| + 2$. We define $\mathcal{N}(T,\ell,\ell^T)$ as the collection of all functions
$\mathbf{n}=\mathbf{n}_1\cup\mathbf{n}_2$ such that:
\begin{itemize}
\item for every edge $e = (v,w)\in\ds{E}\setminus\ds{E}_{\star}$ if one has $\textbf{n}_e = (k,0,0)$ with
$|k - \ell_{v\wedge w}| \le c$, alternatively if one has $\textbf{n}_e = (k,p,m)$ with
$|k - \ell_{v\wedge w}| \le c$, $|p - \ell_{v\wedge v_e}| \le c$, and
$|m - \ell_{w\wedge v_e}| \le c$;
\item for every edge $e=(v,w)\in\ds{E}_{\star}$ with $\textbf{n}_e=(k,0,0)$  one has 
$|k - \ell_{v\wedge w}| \le c$.
\end{itemize}
\end{definition}

With this definition, one has the following result:

\begin{lemma}\label{lem:kfcn}
For every $\mathbf{n}:\hat{\mathds{E}}\rightarrow \mathbb{N}^3$ such that $\hat{K}^{(n)}$ as defined as before is non-vanishing, there exists an element $(T,\,\ell,\,\ell^T)\in\mathcal{B}(\hat{\mathds{V}})$ with $\mathbf{n}\in\mathcal{N}(T,\,\ell,\,\ell^T)$.
\end{lemma}

\begin{definition}
Denote by $\mathbb{T}_{\lambda,\mu}(\ds{V})$ those subsets in $\mathbb{T}(\ds{V})$ such that $2^{-\ell^T_{v\wedge w}}\leq\mu$ for all $v,w\in\ds{V}_\star$.
 \end{definition}

One can now use Lemma~\ref{lem:kfcn} (which is the extension of \cite[Lemma A.9]{HQ15} and is proven in virtually the same manner) to turn the sum over $\mathcal{N}_\lambda$ into a sum over $\mathbb{T}_{\lambda,\mu}(\dsV)$:

\begin{lemma}\label{lem:intbnd} 
For $\CI_{\lambda,\mu}^{G}(K)$ as in \ref{eq:bigsum}, we have the following bound
\begin{equation}
|\CI_{\lambda,\mu}^{G}(K)|\lesssim \sum_{(T,\ell,\ell^T)\in \mb{T}_{\lambda,\mu}(\ds{V})}\sum_{\mathbf{n}\in\CN(T,\ell,\ell^T)}
\Big|\int_{(\mathbb{R}\times\mathbb{T}_N^d)^{\ds{V}_0}}\hat{K}^{(\textbf{n})}(z)\, dz\Big|.
\end{equation}
\end{lemma}

For the proof of Lemma~\ref{lem:intbnd} one can refer to \cite[Lemma A.9]{MH21}.

As is usual, we represent by $T^\circ$ the interior nodes of $T$, while setting $T^\circ_\star$ as the set of nodes $\nu$ such that $\nu=v_\star\wedge w_\star$ for some $v_\star,\, w_\star\in\ds{V}_\star$, and $\CD(T,\ell,\ell^T)\subset(\mathbb{R}\times\mathbb{T}^d_N)^{\ds{V}}$ such that $\|z_v-z_w\|_\s\le|\ds{V}|2^{-\ell_{v\wedge w}}$, unless $v,\,w\in T^\circ_\star$ in which case we assume that $\|z_v - z_w\|_\s\le|\ds{V}|2^{-\ell^T_{v\wedge w}}$. In the same vein, we are able to define $\CN_{\lambda,\,\mu}(T^\circ)$, where instead of having a labelling fixed and the trees constituting the elements of the collection, we fix a tree and ask for all the labellings for which the previous construction holds true. With this notation we get the following volumetric bound:

\begin{lemma}\label{lem:tildeK}
With $\CD(T,\ell,\ell^{T})$ as above, the volume (in the lebesgue sense) of $\CD(T,\ell,\ell^T)$ is bounded, that is:
\begin{align}\label{eq:secondmainstep}\mu\left(\CD(T,\ell,\ell^{T})\right)\lesssim\prod_{v \in T^\circ\setminus T^{\circ}_{\star}}2^{-\ell_v |\s|}\prod_{v\in T^{\circ}_{\star}}2^{-\ell_v^T|\s_1|-\ell_v|\s\setminus\s_1|}
\end{align}
\end{lemma}
\begin{proof} The bound is derived in \cite[Lemma~A.11]{MH21}.
\end{proof}

We digress here to introduce the general construction that will be pivotal for the bounds we are looking for. Given a rooted binary tree $T$ with some distinguished leaves $0,v_{\star,1},\cdots,v_{\star,M}$ and a distinguished node $\nu_\star$, we set $T^\circ$ as the set of inner nodes of $T$ and $\CN_{\lambda,\mu}(T^\circ)$ has the same description as before. Given two functions $\eta,\;\eta^T:T^\circ\rightarrow\mathbb{R}$, we write:
\begin{equs}\CI_{\lambda,\mu}(\eta,\eta^T)=\sum_{(\ell,\ell^T)\in\CN_{\lambda,\mu}(T^\circ)}\prod_{\nu\in T^\circ}2^{-\ell_\nu\eta_\nu - \ell^T_\nu\eta^T_\nu}.
\end{equs}
Set $|\eta|=\sum_{\nu\in T^\circ} \eta_\nu$ and similarly for $\eta^T$. Then we have:
\begin{lemma}\label{lem:HQ}
If $\eta$ satisfies the following:
\begin{itemize}
\item For every $\nu\in {T}^\circ$, one has the $\sum_{v\ge \nu} \eta_v > 0$
\item For every $\nu\in {T}^\circ$ such that $\nu \le \nu_\star$, one has $\sum_{u \ngeqslant v} \eta_u < 0$, if this sum contains at least one term.
\end{itemize}
and $\eta^T$ is such that $\eta^T_\nu\neq 0\Leftrightarrow \ell^T > 0$ and that necessarily $\eta^T_v$ is positive in this case, then one has the bound $\mathcal{I}_{\lambda,\mu}(\eta,\eta^T)\lesssim\lambda^{|\eta|}\mu^{|\eta^T|}$, uniformly over $0<\mu<2^{-N}=\lambda$.
\end{lemma}
\begin{proof}
Refer to \cite[Lemma A.10]{MH21}
\end{proof}
Notice now that \eqref{eq:secondmainstep} put together with the fact that $\supp[\hat{K}^{(\mathbf{n})}]\subseteq\CD(T,\ell,\ell^T)$ and Lemma~\ref{lem:intbnd} gives us:
\begin{align}
|\CI_{\lambda,\mu}^\ds{G}(K)| & \lesssim
 \sum_{(T,\ell,\ell^T) \in \mb{T}_{\lambda,\mu}(\ds{V})}\sum_{\textbf{n} \in \CN(T,\ell,\ell^T)} \nonumber \\ & \Bigl(\prod_{v \in T^\circ\setminus T^{\circ}_{\star}}2^{-\ell_v |\s|}\prod_{v\in T^{\circ}_{\star}}2^{-\ell_v^T|\s_1|-\ell_v|\s\setminus\s_1|}\Bigr)\sup_{\mathbf{n}\in\mathcal{N}(T,\ell,\ell^T)}\sup_{z}|\hat K^{(\textbf{n})}(z)|.
\end{align}
This reduces our problem to bounding $\sup_{z}|\hat{K}^{(\mathbf{n})}(z)|$. To this end, we begin with the observation that:
\beq\label{eq:indbnd}\sup_{z}|\hat{K}^{(\textbf{n})}(z)|\le\prod_{e\in\ds{E}}\sup_x|\hat{K}_e^{(\textbf{n}_e)}(z_{e_-},z_{e_+})|,\eeq
which further reduces our problem to bounding $\hat{K}^{(\textbf{n}_e)}(z_{e_-},z_{e_+})$. It is here that we would like to use Lemma~\ref{lem:HQ}. To this end, we first make the following definitions. With $e_{\uparrow}:= e_+ \wedge e_-$, define the following configuration: $A^+\hspace{-0.5mm}\text{, }A\text{ and }A^-$ as follows:\\
$$A^+=\begin{tikzpicture}[scale=0.75]
\node at (0,0) [circle,fill=black, inner sep=0pt, minimum size=1.5mm, label=left:$r$] (a) {};
\node at (0,-1) [circle,fill=black, inner sep=0pt, minimum size=1.5mm, label=left:$e_{\uparrow}$] (b) {};
\node at (-1,-2) [circle,fill=black, inner sep=0pt, minimum size=1.5mm, label=left:$e_+\wedge v_e$] (c) {};
\node at (-2,-3) [circle,draw=black,fill=gray, inner sep=0pt, minimum size=1.5mm, label=left:$e_+$] (e) {};
\node at (0,-3) [circle,draw=black,fill=gray, inner sep=0pt, minimum size=1.5mm, label=left:$v_e$] (f) {};
\node at (2,-3) [circle,draw=black,fill=gray, inner sep=0pt, minimum size=1.5mm, label=left:$e_-$] (g) {};
\draw[->, snake=snake, segment amplitude = 0.5mm, line after snake = 0.5mm, shorten >=2pt,shorten <=2pt] (a) -- (b);
\draw[shorten >=2pt,shorten <=2pt] (b) -- (c);
\draw[shorten >=2pt,shorten <=2pt] (c) -- (e);
\draw[shorten >=2pt,shorten <=2pt] (c) -- (f);
\draw[shorten >=2pt,shorten <=2pt] (b) -- (g);
\end{tikzpicture},\qquad\qquad A=\begin{tikzpicture}[scale=0.75]
\node at (0,0) [circle,fill=black, inner sep=0pt, minimum size=1.5mm, label=left:$r$] (a) {};
\node at (0,-1) [circle,fill=black, inner sep=0pt, minimum size=1.5mm, label=left:$e_+\wedge v_e$] (b) {};
\node at (-1,-2) [circle,fill=black, inner sep=0pt, minimum size=1.5mm, label=left:$e_{\uparrow}$] (c) {};
\node at (-2,-3) [circle,draw=black,fill=gray, inner sep=0pt, minimum size=1.5mm, label=left:$e_+$] (e) {};
\node at (0,-3) [circle,draw=black,fill=gray, inner sep=0pt, minimum size=1.5mm, label=left:$e_-$] (f) {};
\node at (2,-3) [circle,draw=black,fill=gray, inner sep=0pt, minimum size=1.5mm, label=left:$v_e$] (g) {};
\draw[->, snake=snake, segment amplitude = 0.5mm, line after snake = 0.5mm, shorten >=2pt,shorten <=2pt] (a) -- (b);
\draw[shorten >=2pt,shorten <=2pt] (b) -- (c);
\draw[shorten >=2pt,shorten <=2pt] (c) -- (e);
\draw[shorten >=2pt,shorten <=2pt] (c) -- (f);
\draw[shorten >=2pt,shorten <=2pt] (b) -- (g);
\end{tikzpicture},$$ \newline \beq\label{eq:config} A^-=\begin{tikzpicture}[scale=0.75]
\node at (0,0) [circle,fill=black, inner sep=0pt, minimum size=1.5mm, label=left:$r$] (a) {};
\node at (0,-1) [circle,fill=black, inner sep=0pt, minimum size=1.5mm, label=left:$e_{\uparrow}$] (b) {};
\node at (-1,-2) [circle,fill=black, inner sep=0pt, minimum size=1.5mm, label=left:$e_-\wedge v_e$] (c) {};
\node at (-2,-3) [circle,draw=black,fill=gray, inner sep=0pt, minimum size=1.5mm, label=left:$e_-$] (e) {};
\node at (0,-3) [circle,draw=black,fill=gray, inner sep=0pt, minimum size=1.5mm, label=left:$v_e$] (f) {};
\node at (2,-3) [circle,draw=black,fill=gray, inner sep=0pt, minimum size=1.5mm, label=left:$e_+$] (g) {};
\draw[->, snake=snake, segment amplitude = 0.5mm, line after snake = 0.5mm, shorten >=2pt,shorten <=2pt] (a) -- (b);
\draw[shorten >=2pt,shorten <=2pt] (b) -- (c);
\draw[shorten >=2pt,shorten <=2pt] (c) -- (e);
\draw[shorten >=2pt,shorten <=2pt] (c) -- (f);
\draw[shorten >=2pt,shorten <=2pt] (b) -- (g);
\end{tikzpicture}\eeq

and by $A_0$ those edges such that $r_e = 0$.

On the interior of $T$ consider $\eta:T^{\circ}\rightarrow\mathbb{R}$ given by $\eta=\sum_{e\in \mathds{E}}\eta_e$ where:
\begin{equation}\label{eq:etaa}\begin{split}\eta_e(v)&=\mathbbm{1}_{e\in A}(-a_e\mathbbm{1}_{e_{\uparrow}}(v))+\mathbbm{1}_{e\in A^+}(r_e\mathbbm{1}_{{e_+}\wedge v_e}(v)-(a_e+r_e)\mathbbm{1}_{e_{\uparrow}}(v))\\
&-\mathbbm{1}_{e\in A^-}((a_e+r_e-1)\mathbbm{1}_{e_-\wedge v_e}(v)-(r_e -1)\mathbbm{1}_{e_\uparrow}(v))+\mathbbm{1}_{e \in A_0}(-a_e\mathbbm{1}_{e_\uparrow}(v)),\end{split}\end{equation}
and $\eta^T:T^\circ\mapsto\mathbb{R}$ given by:
\begin{equation}\label{eq:etaT}
\eta^T(\nu) = \left\{\begin{array}{lr}
|\s_1| & \mbox{if }\nu\in T^\circ_\star \\
0, & \mbox{otherwise.}
\end{array}\right.
\end{equation}
and finally $\tilde{\eta}(v) = |\s|+\eta(v)$.

\begin{lemma}\label{lem:Knebnd}With $\hat{K}^{(\mathbf{n}_e)}$ defined as above and $\eta_e$ and $\eta^T$ as in \eqref{eq:etaa}, we have the following bound:
$$\sup_z\left|\prod_{e = (e_-,e_+)}\hat{K}^{(\mathbf{n}_e)}_e(z_{e_-},z_{e_+})\right|\lesssim\prod_{v\in T^\circ}2^{-\ell_v\eta_e(v)}$$
for every edge $(e_-,e_+)\in\hat{\dsE}$.
\end{lemma}
\begin{proof}
Let $e\in \mathds{E}$, and $\mathbf{n}_e=(k,p,m)$. Recall that:
\beq\label{eq:kpmbnd}|k-\ell_{z_{e_-}\wedge \ z_{e_+}}|\le c;\quad |p-\ell_{z_{e_-}\wedge \ z_{v_e}}|\le c;\quad |m-\ell_{z_{e_+}\wedge \ z_{v_e}}|\le c  \eeq
when $r_e > 0$. With reference to \cite[Proposition A.1]{reg} one is able to write:
\begin{equation*}\begin{split}
    \hat K_e^{(\mathbf{n})}(z_{e_-},z_{e_+}) = &\Psi^{(k)}(z_{e_+}-z_{e_-}) \\
    &\Psi^{(p)}(z_{v_e}-z_{e_-})\Psi^{(m)}(z_{e_+}-z_{v_e}) \sum_{|r|_{\mathfrak{s}}=r_e}
    \int_{\mathbb{R}^{d+1}}D^rK_e(y)\mathcal{Q}_e^r(x,dy)
\end{split}\end{equation*}
where the kernel has the property:
$$\mathcal{Q}_e^r(z,\mathbb{R}^{d+1})\lesssim\|z_{e_+}-z_{v_e}\|_\mathfrak{s}^{r_e}$$
We can bound this in different ways, depending on how $m$ and $k$ compare.
We consider the following possibilities with some constant $C_0$:
\begin{itemize}
    \item If $m \ge k + C_0$, then:
    $$\sup_z|\hat K_e^{(\mathbf{n}_e)}(z)|\lesssim 2^{-r_em+(a_e+r_e)k}\sim 2^{-r_e\ell_{z_{e_+}\wedge z_{v_e}}+(a_e+r_e)\ell_{z_{e_+}\wedge z_{e_-}}}$$
    \item If $k\ge m + C_0$, then:
    $$\sup_z|\hat K_e^{(\mathbf{n}_e)}(z)|\lesssim 2^{a_e k}+\sum_{|j|_\mathfrak{s} < r_e}2^{-m|j|_\mathfrak{s}+ (a_e+|j|_\mathfrak{s})p}\lesssim 2^{a_e k}$$
    \item if $k\sim m$ then:
    \begin{equation*}
    \begin{split}
        \sup_z|\hat K_e^{(\mathbf{n}_e)}(z)|\lesssim 2^{a_e k}+\sum_{|j|_\mathfrak{s} < r_e}&2^{-m|j|_\mathfrak{s}+ (a_e+|j|_\mathfrak{s})p} \lesssim 2^{(a_e + r_e - 1)p - (r_e - 1)m}
    \end{split}    
    \end{equation*}
    
\end{itemize}

When $r_e = 0$, the following bound is straightforward:
$$\sup_z |\hat{K}_e^{(\mathbf{n}_e)}(z_{e_-},z_{e_+})|\lesssim 2^{a_ek}.$$
Notice finally that these above three cases correspond to the three configurations in \eqref{eq:config}. Then we have the following bound:
$$\sup_z|\hat{K}_e^{(\mathbf{n}_e)}(z_{e_-},z_{e_+})|\lesssim\prod_{v \in T^{\circ}}2^{-\ell_v\eta_e(v)}.$$
\end{proof}
\begin{lemma}\label{lem:etasat}
For $\eta$ and $\eta^T$ as given in \eqref{eq:etaa} and \eqref{eq:etaT}, the hypothesis of Lemma~\ref{lem:HQ} is satisfied.
\end{lemma}
\begin{proof}
The first comment we make is that the condition on $\eta^T$ is certainly met. Indeed, if $\eta^T_\nu > 0$, then necessarily it is $|\s_1|$ and the corresponding $\ell^T_\nu$ is $\lceil|\log_2(\mu)|\rceil$ which is strictly positive because $\mu < 2^{-N}$ (the other regime falls in the domain of the original Hairer-Quastel result). The other direction follows similarly. Now let $v \in T^\circ$ and we consider $L_v\subseteq \ds{V}$ the leaves attached to $v$. We have:
 \begin{align*}
&\sum_{u\ge v}\tilde \eta(u)= |\mathfrak{s}|(|L_v|-1) + \sum_{u \ge v}\biggl(\hspace{1mm}\sum_{e\in A}-a_e\mathbbm{1}_{\{e_{\uparrow}\}}(u)+\sum_{e\in A^+}(r_e\mathbbm{1}_{\{{e_+}\wedge v_e\}}(u)\\ 
&-(a_e+r_e)\mathbbm{1}_{\{e_{\uparrow}\}}(u))-\sum_{e\in A^-}((a_e+r_e-1)\mathbbm{1}_{\{e_-\wedge v_e\}}(u)\\
&\hspace{4cm}-(r_e -1)\mathbbm{1}_{\{e_\uparrow\}}(u))+\sum_{e \in A_0}(-a_e\mathbbm{1}_{\{e_\uparrow\}}(u))\biggr) \\
&= |\mathfrak{s}|(|L_v|-1) -\hspace{-4mm} \sum_{e \in \mathds{E}_0(L_v)}\hspace{-2mm}a_e - \hspace{-4mm}\sum_{e \in \mathds{E}^{\uparrow}(L_v)}\hspace{-3mm}\mathbbm{1}_{\{v_e\in L_v\}}(a_e+r_e - 1)+\hspace{-4mm}\sum_{e\in\mathds{E}^{\downarrow}(L_v)}\hspace{-3mm}\mathbbm{1}_{\{v_e\in L_v\}}(r_e),
     \end{align*}
where the second equality is justified by the observation that $A,\,A^+,\,A^-,\,A_0$ exhaust all possible configurations for the leaves so that for example $\sum_{u\ge v}(\sum_{e\in A} \bullet + \sum_{e\in A^+} \bullet + \sum_{e\in A^-} \bullet + \sum_{e\in A_0}\bullet) = \sum_{e\in\dsE_0(L_v)}\bullet$. The other summands are explained by considering alternate arguments for the indicators above. To see that this quantity is indeed strictly positive, one can look at the contributions to each of the summands above from $A,\,A^+,\,A^-,\,A_0$ separately:
$$ \begin{array}{||c|c|c|c||}
     \hline
      & \mathds{E}_0(L_v) & \mathds{E}^{\uparrow}(L_v) & \mathds{E}^{\downarrow}(L_v) \\
    \hline
    A^+ & a_e & 0 & -\mathbbm{1}_{\{v_e\in\bar{\mathds{V}}\}}(r_e) \\
    \hline
    A & a_e & 0 & 0 \\
    \hline
    A^- & a_e & \mathbbm{1}_{\{v_e\in\bar{\mathds{V}}\}}(a_e + r_e - 1) & 0
    \\
    \hline
    A_0 & a_e & 0 & 0 \\
    \hline
 \end{array}$$
To check the second condition, we fix some node $v \in T^\circ$ such that $v_\star \ge v$. Denote by \newline $U_v = \{u \in T^{\circ}: u \ngeqslant v\}$ and $\bar{\mathds{V}}$ the set of leaves attached to $U_v$. We must have $|\bar{\mathds{V}}|=|U_v|$ and then it follows:
\begin{equation*}
     \begin{split}
         \sum_{u\in U_v}\tilde \eta(u) &= |\mathfrak{s}||\bar{\mathds{V}}| + \sum_{u \in U_v}\biggl(\hspace{1mm}\sum_{e\in A}-a_e\mathbbm{1}_{\{e_{\uparrow}\}}(u)+\sum_{e\in A^+}(r_e\mathbbm{1}_{\{{e_+}\wedge v_e\}}(u)-(a_e+r_e)\mathbbm{1}_{\{e_{\uparrow}\}}(u))\\
&-\sum_{e\in A^-}((a_e+r_e-1)\mathbbm{1}_{\{e_-\wedge v_e\}}(u)-(r_e -1)\mathbbm{1}_{\{e_\uparrow\}}(u))+\sum_{e \in A_0}\left(-a_e\mathbbm{1}_{\{e_\uparrow\}}(u)\right)\biggr)\\
&\leq |\mathfrak{s}||\bar{\mathds{V}}| - \sum_{e \in {\mathds{E}}_0(\bar{\mathds{V}})} \hspace{-3mm}\hat{a}_e + \hspace{-3mm}\sum_{e \in {\mathds{E}}^{\downarrow}(\bar{\mathds{V}})} \hspace{-2mm}\bigl(\mathbbm{1}_{\{v_e\in\bar{\mathds{V}}\vee r_e=0\}}(\hat{a}_e + r_e - 1)-(r_e - 1)\bigr) \\
&\qquad+\sum_{e \in {\mathds{E}}^\uparrow(\bar{\mathds{V}})}((\hat a_e + \hat r_e) - \mathbbm{1}_{\{v_e\in{\mathds{V}}\}}r_e) < 0.
     \end{split}
 \end{equation*}
The argument is as before but we fail to have equality because not all configurations for the indicator's arguments are possible. It only remains to check that this is strictly negative, for which, as before we check the relevant contributions.
$$ \begin{array}{||c|c|c|c||}
     \hline
      & \mathds{E}_0(L_v) & \mathds{E}^{\uparrow}(L_v) & \mathds{E}^{\downarrow}(L_v) \\
    \hline
    A^+ & a_e & a_e + r_e -\mathbbm{1}_{\{v_e\in\bar{\mathds{V}}\}}(r_e) & a_e \\
    \hline
    A & a_e & a_e & a_e \\
    \hline
    A^- & a_e & a_e & \mathbbm{1}_{\{v_e\in\bar{\mathds{V}}\}}(a_e + r_e - 1) - (r_e - 1)
    \\
    \hline
    A_0 & a_e & a_e & a_e \\
    \hline
 \end{array}$$
\end{proof}

We can now conclude this section by putting all of this together and proving Theorem~\ref{lem:wantedBound2}.
\begin{proof} \hspace{-1.5mm}$\left[\textsc{Of Theorem}~\ref{lem:wantedBound2}\right]$ From Lemma~\ref{lem:Knebnd} and \eqref{eq:indbnd}, we have the following:
$$\sup_z |\hat{K}^{(\mathbf{n})}(z)|\lesssim\prod_{v \in B^{\circ}}2^{-\ell_v\eta  (v)}$$

Inserting the above bound into \eqref{eq:secondmainstep}, and recalling that $2^{-\ell_v}=2^{-N}$ for all $v\in T^\circ_\star$, gives us the following bound:
\begin{equation}\label{eq:sup}
    |\CI_{\lambda,\mu}^\ds{G}(K)|\lesssim\sum_{(T,\ell,\ell^T)\in\mathbb{T}_{\lambda,\mu}(\ds{G})}\sum_{\textbf{n}\in\CN(T,\ell,\ell^T)}\prod_{v\in T^\circ}2^{-\ell_v\tilde{\eta}(v)-\ell^T_v\eta^T(v)}\;2^{N|\s_1|M}
\end{equation}
with $\tilde \eta(v) := |\mathfrak{s}|+\eta(v)$.
To conclude now we need to appeal to Theorem~\ref{lem:HQ}, the hypothesis of which we have already checked in Lemma~\ref{lem:etasat}.
\end{proof}
\section{Elementary Labelled Graphs} \label{sec:elblgraph}

In the previous section, we presented our primary tool for obtaining the bounds we need on our stochastic integrals. We want to be able to prove that the graphs that show up in the analysis of the \eqref{eq:gKPZ} satisfy Assumption~\ref{ass:grph}. It turns out that while checking condition \eqref{eq:ass2e1} is very non-trivial, one is able to link to the labelled graphs of the previous section a new construction called elementary labelled graphs, that always satisfy (a version of) \eqref{eq:ass2e2}. In this section, we will give a recursive construction of these elementary graphs.

\subsection{Construction of Elementary Labelled Graphs}

\begin{definition}[Elementary Labelled Graph]\label{def:ELG}
An elementary labelled graph is the graph ${G}=(\ds{V},\ds{E})$ connected with two distinguished vertices $\ds{V}_{\star}=\{v_0,v_\star\}$ with an edge label: $(a_e,r_e,v_e)\in\mathbb{R}\times\mathbb{N}\times\ds{V}$ such that:
\begin{itemize}
\item it is almost a tree in the sense that $\bar{T}=(\bar{\ds{V}},\bar{\ds{E}})=\left(\ds{V}\setminus\{v_0\},\ds{E}\setminus \ds{E}^{\uparrow}\left(\{v_0\}\right)\right)$ is a rooted tree with root $v_\star$. Furthermore, we can ascribe to it a labelled tree, $T^\mathfrak{e}_\mathfrak{n}$, as follows:
\begin{itemize}
\item $E_T=\bar{\ds{E}}$,\;$N_{\bar{T}}=N_T$, and $L_{\bar{T}}=L_{T},$
\item for every edge $e\in E_T$, one has:
\begin{equs}\label{eq:ELGedgedef}
a_e=|\s|+|\mathfrak{e}(e)|_\s-|\mathfrak{l}|_\s,
\end{equs}
\item for every edge $e\in\dsE^{\uparrow}(L_T)$, $a_e=|\s|+\kappa$, $|\mathfrak{e}(e)|_\s=|\mathfrak{l}(e)|_\s=0$ and $|\mathfrak{l}(e_-)|_\s=-|s|/2-\kappa=\alpha$, with $\kappa > 0,$
\item for every node $v\in N_T$, one has: 
\begin{equs}\label{eq:ELGleafedge}(a_{(v_0,v)}, r_{(v_0,v)}, v_{(v_0,v)})=(-|\mathfrak{n}(v)|_\s, 0, v_0).
\end{equs}
\end{itemize}
\item For every edge of $\bar{T}$, one has: $r_e=\lceil|T_e|_\s\rceil\vee 0$ where $T_e$ is the tree above edge $e$. In set-theoretic notation we may write $T_e=(\ds{V}_e,\mathds{E}_0(\ds{V}_e))$, where $\ds{V}_e=\{v\in\ds{V}\setminus\{v_0\}:e_+\wedge v = e_+\}.$ 
\end{itemize}
\end{definition}
\begin{definition}[Homogeneity]
To each type $\ell\in\mathfrak{l}$ we associate a homogeneity $|\ell|_\s\in\mathbb{R}$. The homogeneity of a labelled tree $T^\mathfrak{n}_\mathfrak{e}$ then is given by:
$$|{T}^\mathfrak{n}_\mathfrak{e}|_\s=\hspace{-2mm}\sum_{u\in L_T \sqcup E_T}\hspace{-2mm}|\mathfrak{\ell}(u)|_\s+\sum_{x\in\mathring{N}_T}|\mathfrak{e}(x)|_\s-\sum_{e\in E_T}|\mathfrak{e}(e)|_\s$$
For an elementary labelled graph, $G$, the homogeneity is defined as:
\begin{equs}\label{eq:defdegELG}
|G|_\s=\left(|\dsV|-2\right)|\s|-\sum_{e\in\dsE}a_e
\end{equs}
\end{definition}

The following proposition is an easy consequence of the definitions:

\begin{proposition}
Let $G=(\dsV,\dsE)$ be an elementary labelled graph and ${T}^\mathfrak{n}_\mathfrak{e}$ the labelled graph associated to it. Then $|{T}^\mathfrak{n}_\mathfrak{e}|_\s=|G|_\s$
\end{proposition}
\begin{proof}
From the definition \eqref{eq:defdegELG} and \eqref{eq:ELGedgedef} and \eqref{eq:ELGleafedge}, one has:
\begin{equs}\label{eq:ELGproof1}
|G|_\s &= \left(|\dsV| - 2\right)|\s| - \sum_{e\in\dsE}a_e = \left(|\dsV| - 2\right)|\s| - \sum_{e\in\dsE(\hat T)}a_e -\hspace{-3mm}\sum_{e\in\dsE^{\uparrow}(\{v_0\})}\hspace{-2mm}a_e \\
&=\left(|\dsV| - 2\right)|\s| - \hspace{-8mm}\sum_{e\in\dsE(\hat T)\setminus\dsE^{\uparrow}(L_{\hat T})}\hspace{-8mm}\left(|\s|+|\mathfrak{e}(e)|_\s-|\mathfrak{l}|_\s\right) -\hspace{-3mm}\sum_{(v_0,v)\in\dsE}|\mathfrak{n}(v)|_\s - |L_{\hat T}|\kappa
\end{equs}
To conclude we notice that:
\begin{equs}\label{eq:ELGproof2}
(|\dsV|-2)|\s| - \hspace{-8mm}\sum_{u\in\dsE(\hat T)\setminus\dsE^{\uparrow}(L_T)}\hspace{-8mm}|\s|-|L_{\hat T}|\kappa &=(|\hat\dsV|-1)|\s|-\hspace{-8mm}\sum_{u\in\dsE(\hat T)\setminus\dsE^{\uparrow}(L_{\hat T})}\hspace{-8mm}|\s|-|L_{\hat T}|\kappa \\
&=|L_{\hat T}|\alpha = \sum_{v\in L_{\hat T}}|\mathfrak{l}(u)|_{\s},
\end{equs}
and that:
\begin{equs}\label{eq:ELGproof3}
\sum_{(v_0,v)\in\dsE}|\mathfrak{n}(v)|_\s = \sum_{v\in\mathring{N}_{\hat T}}|\mathfrak{n}(v)|_\s.
\end{equs}
It is now just a matter of putting \eqref{eq:ELGproof2} and \eqref{eq:ELGproof3} into \eqref{eq:ELGproof1}.
\end{proof}

Recall that the regularity structure in Section~\ref{sec:DisRS} was built up recursively from ${\color{blue} \Xi}$ and ${\color{blue} X}$. As we intend to codify $\hat\Pi_0^NT^\mfn_\mfe$ with elementary graphs, we built up our space of elementary labelled graph similarly from graphs that encode $\Pi^N_0\,{\color{blue}\Xi}$ and $\Pi^N_0\,{\color{blue}X}$:

\begin{definition}\label{def:elg1} For $\colb{\Xi},\,\colb{X}\in\mathcal{T}$, we define the elementary labelled graphs as:
$$\dsG({\color{blue}\Xi})=
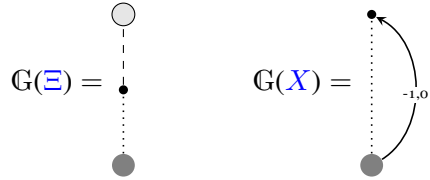
\begin{tikzpicture}[baseline=0cm]
\node at (0,1) [var] (a) {};
\node at (0,0) [dot] (b) {};
\node at (0,-1) [root] (c) {};
\draw[dashed] (b) to (a);
\draw[dotted, semithick] (c) to (b);
\end{tikzpicture}
\qquad\qquad \dsG({\color{blue}X})=\begin{tikzpicture}[baseline=0.5cm]
\node at (0,1.5) [dot] (a) {};
\node at (0,-0.5) [root] (b) {};
\draw[dotted, semithick] (b) to (a);
\draw[kepsus] (b) to[bend right = 60] node[labl,pos=0.45] {\tiny -1,0} (a);
\end{tikzpicture}$$
\end{definition}

With reference to the relation $\Pi_z^N\colb{\tau}\colb{\bar\tau}=\Pi_z^N\colb{\tau}\Pi_z^N\colb{\bar\tau}$, we define the product of elementary graphs in the following way:

\begin{definition}\label{def:elg2}
Given two Elementary Labelled Graphs $\dsG_1=\left(\dsE_1,\dsV_2\right)$ and $\dsG_2=\left(\dsE_2,\dsV_2\right)$ define a new Elementary Labelled Graph $\bar{\dsG}:=\dsG_1\otimes\dsG_2 = \left(\bar\dsE,\bar\dsV\right)$, by $\bar\dsV:=\dsV_1\cup\dsV_2$ and $\bar\dsE:=\dsE_1\cup\dsE_2$. With $v^{i}_0$ and $v^{i}_\star$ for $i\in\{1,2\}$ denoting the distinguished nodes coming from $\dsG_i$ and $v_0$, $v_\star$ the distinguished nodes from $\bar{\dsG}$, we identify $v_0 \sim v^{1}_0 \sim ~ v^{2}_0$ and $v_\star \sim v^1_\star \sim v^2_\star$.
\end{definition}
\begin{example}
Consider the product $\colb{\Xi X}$. Putting the previous two definitions together gives us: 
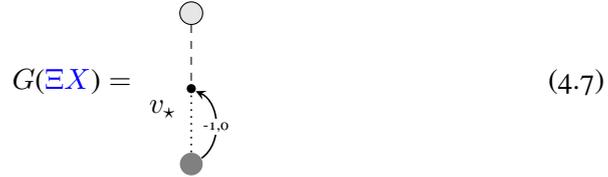
\beq
G(\colb{\Xi X})=\begin{tikzpicture}[baseline=0cm]
\node at (0,1) [var] (a) {};
\node at (0,0) [dot,label=200:$v_\star$] (b) {};
\node at (0,-1) [root] (c) {};
\draw[dashed] (a) to (b);
\draw[dotted, semithick] (c) to (b);
\draw[kepsus] (c) to[bend right = 60] node[labl,pos=0.45] {\tiny -1,0} (b);
\end{tikzpicture}
\eeq
\end{example}

Further with reference to \eqref{eq:I} we define the integration of graphs in the following manner:

\begin{definition}\label{def:elg3}
The integration of an elementary graph $\dsG=(\dsV,\dsE)$ is the elementary graph $\bar{\dsG}=\CI_n(G)$ given by: $(\ds{V}\cup\{v'_\star\},\ds{E}')$ where $\ds{E}'$ is defined by
\begin{itemize}
\item if $a_{e_\star}\neq 0$ then $\ds{E}'=\ds{E}\cup\{(v'_\star,v_0),(v'_\star,v_\star)\}$ with the label $(v'_\star,v_0)$ given by $(0,0)$.
\item else $\ds{E}'=(\ds{E}\setminus\{(v_0,v_\star)\})\cup\{(v_0,v'_\star),(v_\star,v'_\star)\}$.
where the edge label for $(v_0,v'_\star)$ and $(v_\star,v'_\star)$ are given by $(0,0)$ and $(|\s|-|\CI(\cdot)|_\s+|n|_\s,0\vee\lceil|G'|\rceil,v_0).$
\end{itemize}
\end{definition}
Consider the following examples of this construction:

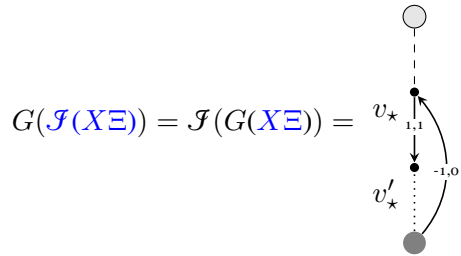
\be
G\left(\colb{\CI(X\Xi)}\right)=\CI\left(G(\colb{X\Xi})\right)=\begin{tikzpicture}[baseline=0.5cm]
\node at (0,2) [var] (a) {};
\node at (0,1) [dot,label=200:$v_\star$] (b) {};
\node at (0,0) [dot,label=200:$v'_\star$] (c) {};
\node at (0,-1) [root] (d) {};
\draw[dashed] (a) to (b);
\draw[kepsus] (b) to node[labl,pos=0.45] {\tiny 1,1} (c);
\draw[kepsus] (d) to[out=50,in=310] node[labl,pos=0.45] {\tiny\textit -1,0} (b);
\draw[dotted, semithick] (d) to (c);
\end{tikzpicture}
\ee
\be
G\left(\colb{\CI(\Xi)\Xi}\right)=G\left(\colb{\CI(\Xi)}\right)\oast G\left(\colb{\Xi}\right)=\begin{tikzpicture}[baseline=0.5cm]
\node at (1,2) [var] (a) {};
\node at (1,1) [var] (e) {};
\node at (0,1) [dot,label=200:$v_\star$] (b) {};
\node at (0,0) [dot,label=200:$v'_\star$] (c) {};
\node at (0,-1) [root] (d) {};
\draw[dashed] (a) to (b);
\draw[kepsus] (b) to node[labl,pos=0.45] {\tiny 1,1} (c);
\draw[dashed] (e) to (c);
\draw[dotted, semithick] (d) to (c);
\end{tikzpicture}
\ee

We differentiate between the "integration" of the graph and the "integration" of the abstract symbol by colouring the latter in blue. 


\begin{definition}[Subtree]
Given a elementary labelled graph $\dsG = (\dsV,\dsE)$, a subtree $\bar T = (\bar\dsV,\bar\dsE)$ is defined via the inclusions $\bar\dsV\subset\dsV\text{ and }\bar\dsE\subset\dsE$ and one defines the homogeneity of $\bar T$ as:
\beq\label{def:homsub}|\bar{T}|_\s=\left(|\bar\dsV| - \mathbbm{1}_{\{v_\star\in\bar\dsV\}}\right)|\s|-\sum_{e\in\dsE_0(\bar\dsV)}a_e - \sum_{e\in\dsE^\uparrow(\bar\dsV)}a_e
\eeq
\end{definition}
As a sanity check, one may chose $\bar{\dsV} = \dsV$, and notice that in this case, the second sum above would be over an empty index and what remains is just $|G|_{\s}$ \eqref{eq:defdegELG}.

Recursively one could generate any number of graphs, but in our analysis of gKPZ, we will only be concerned with those that are constructed with the specific set of rules given in $ \eqref{rules_gKPZ} $.
So for example, the rules do not allow for any powers of $\colb{\Xi}$.  We will denote by $\CG_{\mathcal{R}}$ the set of elementary graphs built from $\CR $ and $G(\colb\Xi)$ where all the symbols, products and integrations are replaced by the counterparts we have given in Definitions \ref{def:elg1}, \ref{def:elg2}, \ref{def:elg3}.
This recursive construction matches the recursive construction of the model $ \Pi^N $ as we have a one to one correspondence between a subclass of elementary graphs and the iterated integrals generated by $ \Pi^N \btau $. 

$\CR$ is a locally subcritical set of rules, which we recall means that for any subtree $T$ of any elementary labelled graph $\dsG\in \CG_{\CR}$ such that it is not $\colb{\Xi}$, we always have that $|T|_\s > |\colb{\Xi}|_\s = \alpha$.

\subsection{Bounds on Elementary Labelled Graphs}

To use the bounds of the last section on these elementary graphs, we will want to use Theorem \ref{th:HQbnd}. To this end we will suggest bounds on Elementary Labelled Graphs, that will give us the bounds \eqref{eq:ass2e1} and \eqref{eq:ass2e2} and hence satisfy the hypothesis of the theorem in question.

\begin{assumption}\label{ass3} \begin{itemize}
\item[1.] For every subset $\bar{\ds{V}}\subset\ds{V}_0$ one has:
\beq\label{elgass1}
\sum_{e\in\mathds{E}_0(\bar{\ds{V}})} a_e \le \left(|\bar{\ds{V}}|-\frac{1}{2}\right)|\s|+\kappa.
\eeq
\item[2.] For every subset $\bar{\ds{V}}\subset\ds{V}$ one has for $|\bar{\ds{V}}|\ge 3$:
\beq\label{elgass2}
\begin{split}
\sum_{e\in\mathds{E}_0(\bar{\ds{V}})} a_e + \sum_{e\in\mathds{E}^\uparrow(\bar{\ds{V}})}\mathbbm{1}_{\{v_e\in\bar{\ds{V}}\wedge r_e>0\}}(a_e&+r_e-1)-\sum_{e\in\mathds{E}^\downarrow(\bar{\ds{V}})}\mathbbm{1}_{\{v\in\bar{\ds{V}}\}}r_e \\
&<\left(|\bar{\ds{V}}|-1\right)|\s|,
\end{split}
\eeq
where we treat the vertex $0$ as an inner node.
\item[3.] For every non-empty subset $\bar{\ds{V}}\subsetneq \ds{V}_0$, one has:
\beq\begin{split}\label{elgass3}
&\sum_{e\in\mathds{E}_0(\bar{\ds{V}})}a_e + \sum_{e\in\mathds{E}^\downarrow(\bar{\ds{V})}}\bigl(\mathbbm{1}_{\{v_e\in\bar{\ds{V}}\vee r_e=0\}}(a_e+r_e-1)-(r_e-1)\bigr) \\ &+\sum_{e\in\mathds{E}^\uparrow(\bar{\ds{V}})}\bigl((a_e+r_e)-\mathbbm{1}_{\{v_e\in\bar{\ds{V}}\}}r_e\bigr)+\mathbbm{1}_{\{v_\star\in\bar{\ds{V}}\}}|G|_\s
>\biggl(|\bar{\ds{V}}|-\mathbbm{1}_{\{v_\star\in\bar{\ds{V}}\}}\biggr)|\s|.\end{split}
\eeq

We note that if $\bar{\ds{V}}=\ds{V}_0$, we obtain equality in the previous bound which is the definition of $|G|_\s$.
\end{itemize}
\end{assumption}

The following result tells us that \eqref{elgass3} is always satisfied for the graphs generated from $\mathcal{G}_{\mathcal{R}}$.

\begin{proposition}\label{prop:assump4-3}
Let $\dsG=(\dsV,\dsE)\in\CG_{\CR}$ and $\bar{{T}}=(\bar{\dsV},\bar{\dsE})$ a sub-tree of $\dsG$ such that $\bar{\dsV}\subsetneq\dsV_0$ then $\bar{\dsV}$ satisfies the assumption \eqref{elgass3}.
\end{proposition}
\begin{proof}
Let $\rho$ be the root of $\bar T$ and recall that $T_\rho$ denotes the tree above $\rho$. If $\rho\neq v_\star$, we replace $T_\rho$ by $T_u$ where $(u,\rho)\in\dsE$. To each edge $e=(e_-,e_+)\in\dsE^{\downarrow}(\bar{\dsV})$, we can associate a tree $T_e=T_{e_-}$, and notice that we have the following equalities:
$$T_\rho=\bar{T}\cup\hspace{-3mm}\bigcup_{v\in\dsE^{\downarrow}(\bar{\dsV})}T_e,\qquad |T_\rho|_\s=|\bar{T}|_\s + \hspace{-3mm}\sum_{e\in\dsE^{\downarrow}(\bar \dsV)}\hspace{-2mm}|T_e|_\s$$
Then in view of \eqref{def:homsub} and the second equality above we get:
\beq\label{eq:T_rho}|T_\rho|_\s=\left(|\bar{\dsV}|-\mathbbm{1}_{\{v_\star\in\bar{\dsV}\}}\right)|\s|-\hspace{-3mm}\sum_{e\in\dsE_0(\bar{\dsV})}\hspace{-2mm}a_e-\hspace{-3mm}\sum_{e\in\dsE^{\uparrow}(\bar{\dsV})}a_e + \sum_{e\in\dsE^{\downarrow}(\bar{\dsV})}|T_e|_\s\eeq
Assume first that $\rho\neq v_\star$ and recall that by definition one has $\sum_{e\in\dsE^{\uparrow}(\bar{\dsV})}r_e > |T_\rho|_\s$. This implies that:
\begin{equs}
\sum_{e\in\dsE^{\uparrow}(\bar{\dsV})}r_e >\left(|\bar{\dsV}|-\mathbbm{1}_{\{v_\star\in\bar{\dsV}\}}\right)|\s|-\hspace{-3mm}\sum_{e\in\dsE_0(\bar{\dsV})}\hspace{-2mm}a_e-\hspace{-3mm}\sum_{e\in\dsE^{\uparrow}(\bar{\dsV})}a_e - \sum_{e\in\dsE^{\downarrow}(\bar{\dsV})}|T_e|_\s \\
\Rightarrow \sum_{e\in\dsE_0(\bar{\dsV})}\hspace{-2mm}a_e +\hspace{-3mm}\sum_{e\in\dsE^{\uparrow}(\bar{\dsV})}a_e + r_e > \left(|\bar{\dsV}|-\mathbbm{1}_{\{v_\star\in\bar{\dsV}\}}\right)|\s| + \sum_{e\in\dsE^{\downarrow}(\bar{\dsV})}|T_e|_\s 
\end{equs}
To bound $|T_e|_\s$ recall that by subcriticality of $G$, one has at least that $|T_e|_\s >-\frac{\s}{2}-\kappa>-a_e$. If we know that $r_e > 0$, we can improve the bound to $|T_e|_\s > r_e - 1$. Putting this bound back in gives us:
\begin{equs}
\sum_{e\in\dsE_0(\bar{\dsV})}a_e + \hspace{-3mm}\sum_{e\in\dsE^{\uparrow}(\bar{\dsV})}(a_e+r_e) + \hspace{-3mm}\sum_{e\in\dsE^{\downarrow}(\bar{\dsV})}&\left(\mathbbm{1}_{\{r_e=0\}}(a_e+r_e-1)-(r_e-1)\right) \\ &\hphantom{a}\hspace{10mm} > \left(|\bar{\dsV}|-\mathbbm{1}_{\{v_\star\in\bar{\dsV}\}}\right)|\s|
\end{equs}
Assume now that $|T_\rho|_\s=|G|_\s$. With the bound on $|T_e|_\s$ the same, one gets the following:
\begin{equs}
\sum_{e\in\dsE_0(\bar{\dsV})}a_e + \hspace{-3mm}\sum_{e\in\dsE^{\uparrow}(\bar{\dsV})}a_e + \hspace{-3mm}\sum_{e\in\dsE^{\downarrow}(\bar{\dsV})}&\left(\mathbbm{1}_{\{r_e=0\}}(a_e+r_e-1)-(r_e-1)\right) \\ &\hphantom{a}\hspace{5mm}+|G|_\s > \left(|\bar{\dsV}|-\mathbbm{1}_{\{v_\star\in\bar{\dsV}\}}\right)|\s|
\end{equs}
Finally, we can consolidate the two previous bounds into the required one:
\begin{equation*}\begin{split}
\sum_{e\in\dsE_0(\bar{\dsV})}a_e + \hspace{-3mm}\sum_{e\in\dsE^{\uparrow}(\bar{\dsV})}(a_e+r_e-&\mathbbm{1}_{\{v_e\in\bar\dsV\}}r_e) + \hspace{-3mm}\sum_{e\in\dsE^{\downarrow}(\bar{\dsV})}\left(\mathbbm{1}_{\{r_e=0\}}(a_e+r_e-1)-(r_e-1)\right) \\ &+\mathbbm{1}_{\{v_\star\in\bar{\dsV}\}}|G|_\s > \left(|\bar{\dsV}|-\mathbbm{1}_{\{v_\star\in\bar{\dsV}\}}\right)|\s|,
\end{split}
\end{equation*}
where the indicator $\mathbbm{1}_{\{v_e\in\bar\dsV\}}r_e$ comes from the fact that when $\rho_\star = v_\star$, one trivially has $v_e\in\bar\dsV$.
\end{proof}

Let $G=(\dsV,\dsE)$ and $\bar{\dsV}\subsetneq\dsV_0$. We consider $\bar{G}=(\bar{\dsV},\bar{\dsE})$ where $\bar{\dsE}=\dsE_0(\bar{\dsV})$. Then the graph $\bar{G}$ admits the following decomposition: $\bar{G}=\sqcup_{j\in K} \bar{G}_j$ where $G_j=(\dsV_j,\dsE_j)$, with $\dsE_j\coloneqq\dsE_0(\dsV_j)$, are disjoint subtrees of $\bar G$ and $|K|\geq 1$.
\begin{example}
    In graph $G$ given below, we colour in red $\bar{G}$ and then exhibit the decomposition $\bar{G}_1\sqcup \bar{G}_2$ in colours green and yellow.
    
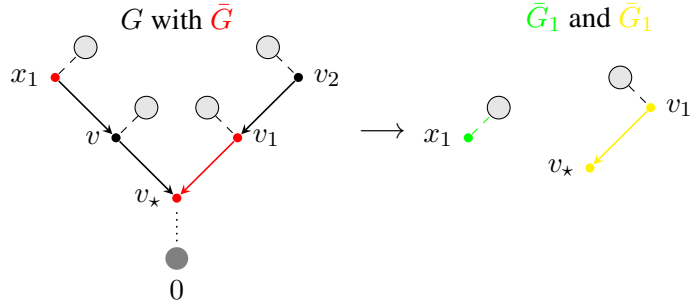
\begin{equs}
        \begin{tikzpicture}[baseline=0cm,scale=0.8]
\node at (0,2) {$G$ with {\color{red}$\bar G$}};
\node at (0,-2) [root, label=below:$0$] (a) {};
\node at (0,-1) [dot,color=red,label=left:$v_\star$] (b) {};
\node at (-1,0) [dot,label=left:$v$] (c) {};
\node at (1,0) [dot,color=red,label=right:$v_1$] (d) {};
\node at (-2,1) [dot,color=red,label=left:$x_1$] (e) {};
\node at (-1.5,1.5) [var] (i) {};
\node at (-0.5,0.5) [var] (f) {};
\node at (0.5,0.5) [var] (g) {};
\node at (2,1) [dot,label=right:$v_2$] (h) {};
\node at (1.5,1.5) [var] (j) {};
\draw[testfcn] (a) to (b);
\draw[kepsus] (c) to (b);
\draw[kepsus,color=red] (d) to (b);
\draw[kepsus] (e) to (c);
\draw[dashed] (f) to (c);
\draw[dashed] (e) to (i);
\draw[kepsus] (h) to (d);
\draw[dashed] (g) to (d);
\draw[dashed] (j) to (h);
\end{tikzpicture}\longrightarrow \begin{tikzpicture}[baseline=0cm,scale=0.8]
\node at (0,2) {{\color{green} $\bar{G}_1$} and {\color{yellow} $\bar{G}_1$}};
\node at (0,-0.5) [dot,color=yellow,label=left:$v_\star$] (b) {};
\node at (1,0.5) [dot,color=yellow,label=right:$v_1$] (d) {};
\node at (-2,0) [dot,color=green,label=left:$x_1$] (e) {};
\node at (0.5,1) [var] (g) {};
\node at (-1.5,0.5) [var] (i) {};
\draw[kepsus,color=yellow] (d) to (b);
\draw[dashed,color=green] (e) to (i);
\draw[dashed] (g) to (d);
\end{tikzpicture}
    \end{equs}
    
\end{example}

Using this characterisation we are able to prove:
\begin{proposition}
Let $G=(\dsV,\dsE)\in\CG_{\CR}$ and let $\bar{\dsV}\subset\dsV\backslash \dsV_\star$. Then $\bar{\dsV}$ satisfies the assumption \eqref{elgass3}.
\end{proposition}
\begin{proof}
We decompose $\bar{\dsV}=\bigsqcup_{j\in K}\dsV_j$ where $\dsV_j$ are disjoint sets and $T_j = (\dsV_j,\dsE_0(V_j))$ is a subtree of $G$. Then we apply the previous proposition on each of $\dsV_j$ and by summing the bounds, we obtain the required result.
\end{proof}
\begin{proposition}\label{prop:ass4-2}
Let $G=(\dsV,\dsE)$ a labelled graph and $\bar{\dsV}\subset\dsV$ such that $v_0\in\bar{\dsV}$ and such that $\tilde{\dsV}=\dsV\backslash\bar{\dsV}$ satisfies \eqref{elgass3} then $\bar{\dsV}$ satisfies the assumption \eqref{elgass2}.
\end{proposition}
\begin{proof}
We suppose that $\bar{\dsV}$ does not satisfy \eqref{elgass2} which yields:
\begin{equs}\label{eq:ELGproof4}\sum_{e\in\dsE_0(\bar{V})}a_e + \hspace{-3mm}\sum_{e\in\dsE^{\uparrow}(\bar{\dsV})}&\mathbbm{1}_{\{v_e\in\bar{\dsV}\wedge r_e > 0\}}(a_e + r_e - 1) \\ 
- &\sum_{e\in\dsE^{\downarrow}(\bar{\dsV})}\mathbbm{1}_{\{v_e\in\bar{\dsV}\}}r_e\ge\left(|\bar{\dsV}|-1\right)|\s|\end{equs}

On the other hand, $\tilde{\dsV}=\dsV\backslash\bar{\dsV}$ satisfies \eqref{elgass3}:
\begin{equs}\label{eq:ELGproof5}
\sum_{e\in\dsE_0(\tilde{V})}a_e &+ \hspace{-3mm}\sum_{e\in\dsE^{\uparrow}(\tilde{\dsV})}\left(a_e + r_e - \mathbbm{1}_{\{v_e\in\tilde{\dsV}\}}r_e\right) \\
&+\hspace{-3mm}\sum_{e\in\dsE^{\downarrow}(\tilde{\dsV})}(\mathbbm{1}_{\{v_e\in\tilde{\dsV}\vee r_e=0\}}(a_e+r_e-1)-(r_e-1)) \\
&+ \mathbbm{1}_{\{v_\star\in\tilde{\dsV}\}}|G|_\s\ge\left(|\tilde{\dsV}|-\mathbbm{1}_{\{v_\star\in\tilde\dsV\}}\right)|\s|
\end{equs}
Notice that the definitions of $\bar{\dsV}$ and $\tilde{\dsV}$, and in particular $\tilde{\dsV}^{\textnormal{c}}=\bar\dsV$, imply that $\dsE_0(\dsV) = \dsE_0(\bar\dsV)\sqcup\dsE_0(\tilde{\dsV})$, $\dsE^{\uparrow}(\bar{\dsV})=\dsE^{\downarrow}(\tilde{\dsV})$, $\dsE^{\downarrow}(\bar{\dsV})=\dsE^{\uparrow}(\tilde{\dsV})$, and:
\begin{equs}\mathbbm{1}_{\{v_e\in\bar{\dsV}\}} + \mathbbm{1}_{\{v_e\in\tilde{\dsV}\}} = 1,\qquad \mathbbm{1}_{\{v_e\in\tilde{\dsV}\vee r_e=0\}}+\mathbbm{1}_{\{v_e\in\bar{\dsV}\wedge r_e>0\}}=1,
\end{equs}
which means that adding \eqref{eq:ELGproof4} and \eqref{eq:ELGproof5}, gives:
\begin{equation}\label{eq:proof1}
 \sum_{e\in\dsE(\dsV)}a_e + \mathbbm{1}_{\{v_\star\in\tilde{\dsV}\}}|G|_\s > \left(|\dsV|-1-\mathbbm{1}_{\{v_\star\in\tilde{\dsV}\}}\right)|\s| 
 \end{equation}
\noindent The subcriticality of $G$ also gives us that:
 $$-\frac{|\s|}{2}-\kappa\le|G|_\s=\left(|\dsV|-2\right)|\s|-\sum_{e\in\dsE}a_e$$
which can be rearranged to give:
 $$\sum_{e\in\dsE}a_e\le\left(|\dsV|-\frac{3}{2}\right)|\s|+\kappa$$
 But this of course contradicts \eqref{eq:proof1} for all sufficiently small $\kappa > 0$. It follows then that $\bar{\dsV}$ satisfies the condition \eqref{elgass2}.
\end{proof}

We are now able to prove that all graphs in $\CG_{\CR}$ satisfy \eqref{elgass2}.

\begin{proposition}
Every $G=(\dsV,\dsE)\in\CG_{\CR}$ satisfies the condition \eqref{elgass2} for $\bar{\dsV}\subset\dsV$ and $v_0\in\bar{\dsV}$
\end{proposition}
\begin{proof}
Let $\tilde{\dsV}=\dsV\backslash\bar{\dsV}$. From Proposition~\ref{prop:assump4-3}, we know that \eqref{elgass3} is satisfied for $\tilde{\dsV}$, which in turn implies that $\bar{\dsV}$ satisfies \eqref{elgass2} due to  the Proposition~\ref{prop:ass4-2}.
\end{proof}

\begin{proposition}
Let $G=(\dsV,\dsE)\in\CG_{\CR}$ and $\bar{T}=(\bar{\dsV},\bar{\dsE})=\bigsqcup_{j\in K}T_j$ such that $\bar{\dsV}\subset\dsV_0$. Then
\begin{itemize}
\item If $|K|\ge 3$ or $|\bar{T}|_\s>0$ then the condition \eqref{elgass2} is satisfied.
\item If $|K|=2$ and there exists $j'\in K$ such that $|T_{j'}|_\s>-\frac{|s|}{2}-\kappa$ then the condition \eqref{elgass2} is satisfied.
\item If $|K|=1$ and $|\bar{T}|_\s < 0$ then the condition \eqref{elgass1} is satisfied but \eqref{elgass2} is not.
\end{itemize}
\end{proposition}
\begin{proof}
For every $j \in K$, we denote by $\dsV_j$, the nodes of the tree $T_j$.
From \eqref{def:homsub}, one has the bound:
\begin{equs}\label{eq:proof414-1}
    \left(|\dsV_j| - 1\right)|\s| -\hspace{-3mm}\sum_{e\in\dsE_0(\dsV_j)}\hspace{-2mm}a_e \ge |T_j|_\s
\end{equs}
If $|\bar{T}|_\s > 0$, then one necessarily has that $\sum_{j\in K}|T_j|_\s > 0$, which implies that:
\begin{equs}
    \sum_{e\in\dsE_0(\bar\dsV)}a_e = \sum_{j\in K}\sum_{e\in\dsE_0(\dsV_j)}a_e &\le \sum_{j\in K}\left(|\dsV_j| - 1\right)|\s|, \\
    &\le (|\dsV|-1)|\s|
\end{equs}
proving the result. If $|K|=2$ and there is a $T_j$ with strictly positive homogeneity, the result follows similarly. Consider finally the case $|K|\geq 3$, where subcriticality gives us that $|T_j|_\s\ge -\frac{|\s|}{2}-\kappa$ (we fail to have a strict inequality because not every $T_j$ is necessarily a proper subtree of $G$), which weakens the lower bound in \eqref{eq:proof414-1}, yielding:
 \begin{equation*}
  \sum_{e\in\dsE_0(\dsV_j)}a_e \le \left(|\dsV_{j}| - \frac{1}{2}\right)|\s| + \kappa
  \end{equation*} 
And again by summing over $j\in K$, we get
\begin{equs}
\sum_{e\in\dsE_0(\bar{\dsV})}a_e & \le \sum_{j \in K} \left(|\dsV_{j}| - \frac{1}{2} + \frac{\kappa}{|\s|}\right)|\s| \\
&=\left(|\bar{\dsV}|-\sum_{j\in K}\left(\frac{1}{2} - \frac{\kappa}{|\s|}\right)\right)
\end{equs}
If $|K| > 3$, $\sum_{j\in K}\left(\frac{1}{2}-\frac{\kappa}{|\s|}\right)>1$, which can be used to bound the expression.
For the last assertion when $|K|=1$ and $|\bar{T}|_\s < 0$, the fact that $|\bar{T}|_\s > -\frac{|\s|}{2} - \kappa$ proves the condition \eqref{elgass1} but $|\bar{T}|_\s < 0$ is in contradiction with \eqref{elgass2}.
\end{proof}
\begin{proposition}\label{prop12}
Let $G=(\dsV,\dsE)\in\CG_{\CR}$ and $\bar{T}=(\bar{\dsV},\bar{\dsE})$ a subtree of $G$. Suppose that there exists a $\mathscr{l}\in\bar{\dsV}_{\ell}$ such that $\mathscr{l}\notin\dsV_\ell$ then if $|\bar{\dsV}|>1$, the condition \eqref{eq:ass2e1} is satisfied. 
\end{proposition}
\begin{proof}
We construct a new tree $\tilde T$, where we replace $\mathscr{l}$ by a leaf in $\bar{T}$, which counts for $-|\s|/2$. This means that $|\bar T|_\s = |\tilde{T}|_\s + |\s|/2 $ but coupled with subcriticality this means that $|\bar{T}|_\s>-\kappa$. For the set of homogeneities is a discrete set and the next homogeneity above $-\kappa$ is $0$, we must have that $|T| > 0$. We conclude that \eqref{eq:ass2e1} is satisfied.
\end{proof}
\section{Solving the Generalised KPZ Equation} \label{sec:gKPZ}

\subsection{Bounds on the Discrete Model}
To simplify the task of finding the bounds, we employ the graphical shorthand that is already well-established in literature. In particular we will use "\,\tikz \node [root] {};\,"  to denote a special node that represents the origin and "\,\tikz \node [var] {};\," for an instance of the noise. Furthermore "\,\tikz \node [dot] {};\," will denote dummy variables that are to be integrated out. Arrows of the form "\,\tikz \draw[kepsus] (0,0) to node[labl,pos=0.45] {\tiny $a_e,r_e,v_e$} (1.5,0);\," will denote Kernels with the label $(a_e,r_e,v_e)$ although when $v_e = 0$ in the interest of brevity, we will simply write $(a_e,r_e)$. We assume for our analysis that we are given the kernels $K^N(\cdot,\cdot)$ for which the Assumption~\ref{ass:Kernel} hold true. Given such kernels, we define the following functions that are kernels in the same sense:
\begin{equation}\label{eq:KerType}
\begin{split}
& K^N(t-s,y-x), \quad  K^N(t-s,y-x)-K^N(-s-x), \\ &  K^N(t-s,y-x)-K^N(-s-x)-\sum_{i=1}^dy_i\partial_i K^N(-s,-x), \quad \sum_{i=1}^d\partial_i K^N
\end{split}
\end{equation}
 These kernels come with the labels $(|\s|-2,0),\,(|\s|-2,1),\,(|\s|-2,2),\,((|\s|-2)+1,0)$. In \cite{MH21} the authors work in $d=3$ so we see kernels with weights $(3,0),\,(3,1),\,(3,2),\,(4,0)$. Furthermore, we use the arrow: `\,\tikz[baseline=-0.1cm] \draw[testfcn] (0,0) to (1,0);\,\!', to represent the test function $\varphi^\lambda$. Finally, a red polygon with $p$ dots inside will represent a joint cumulant $p$-th order of $\xi^N$. 
  example: $\begin{tikzpicture}
\node[] (e) at (0,-0.5) {};
\node[cumu4] (e-) at (e) {};
\node[dot] at (e.north east) {};
\node[dot] at (e.north west) {};
\node[dot] at (e.south east) {};
\node[dot] at (e.south west) {};
\end{tikzpicture}$; represents a $4$th order cumulant.

\begin{example}\label{ex:cum} Consider the following graph:
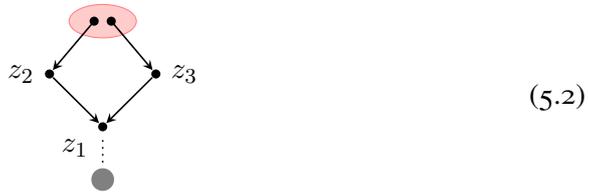
\beq\begin{tikzpicture}[scale=0.7,baseline=0.3cm]
\node at (0,-1)  [root] (root) {};
\node at (0,0) [dot,label=210:$z_1$] (b) {};
\node at (-1,1)  [dot,label=left:$z_2$] (left) {};
\node at (1,1)  [dot,label=right:$z_3$] (right) {};
\node[cumu2n] (a) at (0,2){};
\draw[cumu2] (a) ellipse (18pt and 9pt);
\draw[testfcn] (root) to (b);
\draw[kepsus] (right) to (b);
\draw[kepsus] (left) to (b);
\draw[kepsus] (a.west) node[dot] {} to (left);
\draw[kepsus] (a.east) node[dot] {} to (right); 
\end{tikzpicture}\eeq

The integration variables here comprise $z_1 = (t_1,x_1),\,z_2=(t_2,x_2),\,z_3=(t_3,x_3)$ and the integral it represents is:

$$\int \varphi_0^\lambda(z')\left[K(z_1 - z_2)K(z_1 - z_3)\right]\mathbb{E}_c\left[\xi^N_{t_2}(x_2),\xi^N_{t_3}(x_3)\right]dz$$
\end{example}

where by the integral we mean the semi-discrete integral we introduced in Section~\ref{sec:DisRS} over all the variables.
 
To see why cumulant appears in the integral above consider first: $\Pi^N_0\begin{tikzpicture}
\node at (-0.25,0.25) [bar] (a) {};
\node at (0.25,0.25) [bar] (b) {};
\draw[] (a) to (0,0);
\draw[] (0,0) to (b);
\end{tikzpicture}$. With reference to \eqref{eq:I} and \eqref{eq:product}, we expect to see the integral: $$\int \varphi_0^\lambda(z')\left[K(z_1 - z_2)K(z_1 - z_3)\right]\xi^N_t(x_1)\xi^N_t(x_2)\,dz$$
The problem with analysing the integral above is that the singularity of the noises means that the product of the noises is not defined. The classical approach in quantum physics (for white noise at least) is to use Wicks products to renormalise the product. One can refer to \cite{Wick50}, where C. G. Wick first introduced the Wick product, and then \cite{HI67} where the authors use a version of the Wick product in the Stochastic Analysis setup. In \cite{MY89} finally, the authors extended it to stochastic distributions.
We know then by Definition~\ref{def:wick} that: $\xi^N_t(x_1)\xi^N_t(x_2) =\;\; :\!\xi^N_t(x_1)\xi^N_t(x_2)\!: + \;\mathbb{E}_c\left[\xi^N_t(x_1),\xi^N_t(x_2)\right]$. Substituting this into the integral above we get the sum:
\beq\label{eq:examplesplit}\begin{split}&\int \varphi_0^\lambda(z')\left[K(z_1 - z_2)K(z_1 - z_3)\right]:\!\xi^N_{t_2}(x_2),\xi^N_{t_3}(x_3)\!:dz \\ &+ \int \varphi_0^\lambda(z')K(z_1 - z_{2})K(z_{1} - z_{3})\mathbb{E}_c\left[\xi^N_{t_2}(x_2),\xi^N_{t_3}(x_3)\right]dz\end{split}\eeq

The integrals in \eqref{eq:examplesplit} are then represented via the graph notation as follows:
\beq\begin{tikzpicture}[scale=0.7,baseline=0.3cm]
\node at (0,-1)  [root] (root) {};
\node at (0,0) [dot,label=210:$x$] (b) {};
\node at (-1,1)  [dot,label=left:$z$] (left) {};
\node at (1,1)  [dot,label=right:$y$] (right) {};
\node at (1,2) [var] (r) {};
\node at (-1,2) [var] (l) {};
\draw[testfcn] (root) to (b);
\draw[kepsus] (right) to (b);
\draw[kepsus] (left) to (b);
\draw[dashed] (left) to (l);
\draw[dashed] (right) to (r);
\end{tikzpicture} + \begin{tikzpicture}[scale=0.7,baseline=0.3cm]
\node at (0,-1)  [root] (root) {};
\node at (0,0) [dot,label=210:$x$] (b) {};
\node at (-1,1)  [dot,label=left:$z$] (left) {};
\node at (1,1)  [dot,label=right:$y$] (right) {};
\node[cumu2n] (a) at (0,2) {};
\draw[cumu2] (a) ellipse (18pt and 9pt);
\draw[testfcn] (root) to (b);
\draw[kepsus] (right) to (b);
\draw[kepsus] (left) to (b);
\draw[kepsus] (a.west) node[dot] {} to (left);
\draw[kepsus] (a.east) node[dot] {} to (right); 
\end{tikzpicture}\eeq

Looking at our definition of a model one can infer that the integrals we will be interested in, will always take the following form:

\beq\label{eq:genKer}
\int \varphi_0^\lambda(z)\prod_j\mc{K}_j^N(z^{\{j,2\}},\,z^{\{j,1\}}):\!\!\prod_k\xi^{N}_{t^{(k)}}(x^{(k)})\!\!:dz\eeq

where $\mc{K}^N_j$ stands for any the four possibilities in \eqref{eq:KerType}.

The quantities of interest then will be the $p$-th moment of \eqref{eq:genKer}:

$$\int \mathbb{E}\left[\prod_i^p\varphi_{0}^\lambda(z_i)\left[\prod_j\mc{K}_j^N(z_i^{\{j,2\}},\,z_i^{\{j,1\}})\right]:\!\!\prod_k\xi^{N}_{t_i^{(k)}}(x_i^{(k)})\!\!:\right]dz_i$$

We want to use the bounds on the cumulants to bound these moments. For this, we use Lemma~\ref{lem:cyc}, to transition to:
\beq\label{eq:cumgraph}
\sum_\pi\int\prod_{i=1}^p\varphi_0^\lambda(z_i)\left[\prod_{j}\mc{K}_j^N(z_i^{\{j,2\}},\,z_i^{\{j,1\}})\right]\prod_{B\in\pi}\mathbb{E}_c\left(\left[\xi^{N}_{t}(x):(t,x) \in B\right]\right)dz_i
\eeq

The graphical shorthands we have defined in the beginning of this section are used to codify \eqref{eq:cumgraph} (for some fixed $p$) are illustrated as graphs generated via the following algorithm:

\begin{itemize}
\item take $p$ copies of $\dsG$,
\item fix a partition $\pi$ of $\dsV_{\ell}^{\otimes p}$(set of all the leaves) such that each $B\in\pi$ contains at
least two elements that from different copies of $\dsG$,
\item for each $B\in\pi$ draw a red polygon with $|B|$ dots inside and connect these $|B|$ dots to the elements of $B$,
\item sum over all expression obtained in this way.
\end{itemize}

We illustrate this with an example:

\begin{example} Consider for example the term $\colb{\Xi\CI(\Xi)}$. In its treatment (the manipulation being the same as in \cite{HP15} but the contractions refer to cumulants arising as in Example~\ref{ex:cum}) one sees the following graphs:
\begin{equation}
\label{eq:xiixibasic}
\begin{aligned}
(\hat\Pi_0^{N}\colb{\Xi\CI(\Xi)})(\varphi_0^{\lambda})
&= \begin{tikzpicture}[scale=0.5,baseline=0.3cm]
\node at (-1,4) {\textcircled{1}};
\node at (-2,-1)  [root] (root) {};
\node at (-2,1)  [dot,label=left:$z^{(1)}$] (left) {};
\node at (-2,3)  [dot,label=left:$z^{(2)}$] (left1) {};
\node at (0,1) [var] (variable1) {};
\node at (0,3) [var] (variable2) {};

\draw[testfcn] (root) to  (left);

\draw[kepsus] (left1) to  node[labl, pos=0.45] {\tiny $3,1$} (left);
\draw[kepsus] (variable2) to (left1); 
\draw[kepsus] (variable1) to (left); 
\end{tikzpicture}\;
- \;
\begin{tikzpicture}[scale=0.5,baseline=0.3cm]
\node at (-2,4) {\textcircled{2}};
\node at (-2,-1)  [root] (root) {};
\node at (-2,1)  [dot,label=210:$z^{(1)}$] (left) {};
\node at (-2,3)  [dot,label=below:$z^{(2)}$] (top) {};
\node[cumu2n] (a) at (-4,2){};
\draw[cumu2] (a) ellipse (12pt and 24pt);

\draw[testfcn] (root) to  (left);

\draw[kepsus] (top) to [bend left=60] node[labl,pos=0.45] {\tiny $3,0$} (root);
\draw[kepsus] (a.south) node[dot] {} to (left);
\draw[kepsus] (a.north)node[dot] {} to (top); 
\end{tikzpicture}\;.
\end{aligned}
\end{equation}
\end{example}
 
 


 
Our aim is to be able to bound graphs such as these and for graphs like \textcircled{2} in \eqref{eq:xiixibasic} this is a straightforward matter. One notices that it is non-random and as such it is a straightforward application of Theorem~\ref{th:HQbnd}. As per Assumption~\ref{ass:cyc}, we can replace each cumulant kernel in \textcircled{2}, by a cycle of weight $\frac{3}{2}+\kappa$ and then check Assumption~\ref{ass:grph}. So for example one may check that for the subtree $\{z^{(2)},\,z^{(1)}\}$, condition \eqref{eq:ass2e1} reads: $|\s| - 2 < (2 - 1)|\s| = |\s|$ which is of course true for all $\s$. Similarly for $\{z^{(2)},\,z^{(1)},\,0\}$ one check that it becomes: $|\s|- 2 + |\s| - 2  = 2(|\s| - 2) < (3 - 1)|\s| = 2|\s|$, which again holds true always. Condition \eqref{eq:ass2e2} is checked similarly. From thence it is only a matter of invoking Theorem~\ref{th:HQbnd}.

The arguments needed for \textcircled{1} in \eqref{eq:xiixibasic} become more complicated because now we have a random object instead of a deterministic object. Now one applies the general program we expounded on before. The family of quantities (indexed by $p$) we are looking to bound is given by:
\beq\label{eq:moment}\int\mathbb{E}\left[\prod_{i=1}^p\varphi_0^\lambda(z_i)[K^N(z^{(1)}_i-z^{(2)}_i)-K^{N}(-z^{(2)})]\xi^{N}_{t^{(1)}}(x^{(1)})\xi^{N}_{t^{(2)}}(x^{(2)})\right]\,dz_i\eeq

To be able to use Assumption~\ref{ass:cyc} we need to transition from expectations to cumulants in \eqref{eq:moment}. Via Lemma~\ref{lem:cyc}, one gets:
\beq\label{eq:cumulant}\sum_{\pi}\int\prod_{i=1}^p\varphi_0^\lambda(z_i)[K^N(z_i^{(1)}-z_i^{(2)})-K^N(-z_i^{(2)})\prod_{B\in\pi}\mathbb{E}_c(\left\{\xi_t^N(x):(t,x)\in B\right\})\,dz\eeq

with $\pi$ running over all partitions in $\mathcal{P}_{1,2}(\{1,2\}\times\{1,2,\cdots,p\})$. Let us now illustrate the algorithm for constructing illustrations of \eqref{eq:cumulant}. Let us fix $p=2$, which means we begin by taking two copies of \textcircled{1} in \eqref{eq:xiixibasic}. With the four possible leaves, the only partition that makes sense is the one where we pair off one leaf from either copy with the other. This results in:

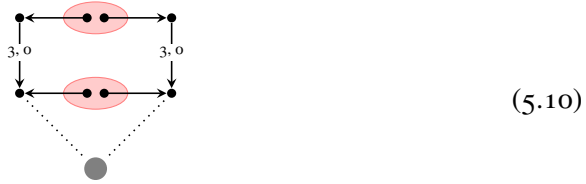
\beq\label{eq:xiixibasic2}\begin{tikzpicture}[scale=0.5,baseline=0.3cm]
\node at (0,-1)  [root] (root) {};
\node at (-2,1)  [dot] (left) {};
\node at (-2,3)  [dot] (left1) {};
\node at (2,1)   [dot] (right) {};
\node at (2,3)   [dot] (right1) {}; 
\node[cumu2n] (a) at (0,3){};
\draw[cumu2] (a) ellipse (24pt and 12pt);
\node[cumu2n] (b) at (0,1){};
\draw[cumu2] (b) ellipse (24pt and 12pt);

\draw[testfcn] (left) to  (root);
\draw[testfcn] (right) to (root);

\draw[kepsus] (left1) to node[labl,pos=0.45] {\tiny 3, 0} (left);
\draw[kepsus] (right1) to node[labl,pos=0.45] {\tiny 3, 0} (right);
\draw[kepsus] (a.west) node[dot] {} to (left1);
\draw[kepsus] (a.east) node[dot] {} to (right1);
\draw[kepsus] (b.west) node[dot] {} to (left);
\draw[kepsus] (b.east) node[dot] {} to (right);
\end{tikzpicture}\;\eeq

The next step would be to bound each of these diagrams. It was achieved via \cite[Proposition 4.7]{MH21} in the original paper through a general criterion that when met checks automatically Assumption~\ref{ass:grph} for all the graphs generated for the $p$ moments. In the section that follows we present our version of that proposition.

\subsection{General Criterion}

Fix a graph $\ds{G}$, like the first one in \eqref{eq:xiixibasic}. As in \cite{MH21} we would like to leverage the fact that the "cumulant" kernels appearing in the cumulant terms are replaced with cycles with some prescribed weight. To formulate the conditions we will use ${\dsV}_\xi$ which we define as the subset of ${\dsV}$ comprising interior nodes such that they are attached to an instance of the noise.

\begin{proposition}\label{prop:pcondition}
Starting with a graph $\dsG=(\dsV,\dsE)$ for which we have  replaced each cumulant term by a simple weighted cycle with weight $3/2+\eta$, if the following bounds are met for every subset $\bar\dsV\subseteq\dsV_0$:
 \beq\label{eq:cond1a} \sum_{e\in\dsE_0(\bar\dsV)} a_e + \frac{3}{2}|\bar\dsV\cap\dsV_\xi| \le \left(|\bar\dsV| - \frac{1}{4}\right)|\s|+\kappa\eeq
\beq\label{eq:cond1b} \sum_{e\in\dsE_0(\bar\dsV)} a_e \le \left(|\bar\dsV| - \frac{1}{2}\right)|\s|+\kappa\eeq

then all the $p$-th moments of $\dsG$ satisfy \eqref{elgass1}.

Further if for every subset $\bar\dsV\subseteq\dsV$, such that $|\bar\dsV|\ge 3$, one has the following bounds:
 \beq\label{eq:cond2a}\begin{split}\sum_{e\in\dsE_0(\bar\dsV)} &a_e + \hspace{-2mm}\sum_{e\in\dsE^\uparrow(\bar\dsV)}\mathbbm{1}_{\{v_e\in\bar\dsV\vee r_e > 0\}}(a_e + r_e - 1) \\ &- \sum_{e\in\dsE^\downarrow(\bar\dsV)}\mathbbm{1}_{\{v_e\in\bar\dsV\}}r_e + \frac{3}{2}|\bar\dsV\cap\dsV_\xi| < \left(|\bar\dsV|-\frac{1}{2}\right)|\s|\end{split}\eeq
\beq\label{eq:cond2b}\begin{split}\sum_{e\in\dsE_0(\bar\dsV)} a_e + \sum_{e\in\dsE^\uparrow(\bar\dsV)}&\mathbbm{1}_{\{v_e\in\bar\dsV\vee r_e > 0\}}(a_e + r_e - 1) \\ &- \sum_{e\in\dsE^\downarrow(\bar\dsV)}\mathbbm{1}_{\{v_e\in\bar\dsV\}}r_e < \left(|\bar\dsV|-1\right)|\s|\end{split}\eeq
then the $p$-th moments of the graph satisfy \eqref{elgass2}.

Finally, if for every non-empty subset $\bar\dsV\subset\dsV_0$, one has:
\beq\begin{split}\label{eq:cond3}&\sum_{e\in\dsE_0(\bar\dsV)}a_e + \sum_{e\in\dsE^\downarrow(\bar\dsV)}\left(\mathbbm{1}_{\{v_\star\in\bar\dsV\vee r_e > 0\}}(a_e + r_e - 1) - (r_e - 1)\right)
\\
&+\sum_{e\in\dsE^\uparrow(\bar\dsV)}\left((a_e+r_e) - \mathbbm{1}_{v_e\in\bar\dsV}r_e\right)
+ \mathbbm{1}_{\{v_\star\in\bar\dsV\}}|\dsG|_\s
\\
&+\frac{3}{2}|\bar\dsV\cap\dsV_\xi|> \left(|\bar\dsV| - \mathbbm{1}_{\{v_\star\in\bar\dsV\}}\right)|\s|,
\end{split}\eeq
then the $p$-th moment of the graph satisfies \eqref{elgass3}.
\end{proposition}
\begin{proof}
In the case that $p=1$, \eqref{eq:cond1a} directly gives us \eqref{elgass1}, so we may assume $p\ge 2$. We need to argue that for any subset $\bar{\dsV}\subset\dsV_0$, which is the union of all the nodes in the $p$ copies, satisfies \eqref{elgass1}. To this end we introduce the notation $\bar{\dsV}_{j}$ for the nodes in $\bar{\dsV}$ that come from the $j$-th copy; the decomposition $\bar{\dsV}=\cup_{j=1}^p\bar{\dsV}_j$ is what we have in mind.

Now as in \cite{MH21}, we proceed by decomposing the sum $\sum_{e\in{\dsE}_0({\bar{\dsV}})}\hat{a}_e$. The idea is that the nodes in $\dsV_j$ that also lie in $\dsV_\xi$, will lead to a cumulant cycle after contraction which under our strategy can be seen as an edge with weight $3/2+\eta$. This means that it will add a weight of $\tfrac{3}{2}|\dsV_j\cap\dsV_\xi|$ to the existing sum. Notice that here we have ignored the $\eta > 0$ and will do so in the proof of the other inequalities too. This is justified by noticing that apart from \eqref{elgass1}, the inequalities we are interested in are strict, and in the case of \eqref{elgass1} we have an arbitrarily small constant, $\kappa$, on the other side which can be adjusted in conjunction with $\eta$. The nodes which are not attached to leaves, are unchanged and hence we are able to make the following calculation:
\be\begin{split}\sum_{e\in{\dsE}_0(\bar{\dsV})}\hat{a}_e
&\le\sum_{\substack{j=1 \\ \bar{\dsV}\!_j\cap\dsV_\xi\neq\emptyset}}^p\hspace{-3mm}\left(\sum_{e\in{\dsE}_0(\dsV_j)}a_e+\frac{3}{2}|\dsV_{j}\cap\dsV_\xi|\right)+\sum_{\substack{j=1 \\ \bar{\dsV}\!_j\cap\dsV_\xi=\emptyset}}^p\hspace{-3mm}\sum_{e\in{\dsE}_0(\dsV_j)}a_e \\
&<|\s|\hspace{-2mm}\sum_{\substack{j=1 \\ \bar{\dsV}\!_j\cap\dsV_\xi\neq\emptyset}}^p\hspace{-3mm}\left(|\bar{\ds{V}}|-\frac{1}{4}\right)+\kappa + |\s|\hspace{-2mm}\sum_{\substack{j=1 \\ \bar{\dsV}\!_j\cap\dsV_\xi=\emptyset}}^p\hspace{-3mm}\left(|\bar{\ds{V}}|-\frac{1}{4}\right)+\kappa \\
&< |\s|\left(|\bar{\ds{V}}|-\frac{1}{2}\right)+\kappa'
\end{split}
\ee
where the last inequality follows from recalling $p\ge 2$ and setting $\kappa' := 2\kappa$, proving \eqref{elgass1} for $p$-th moments.

Fix a graph $\ds{G}$, and assume we have a $\bar{\dsV}$ such that it has cardinality at least $3$ and it satisfies the conditions in \eqref{eq:cond2a} and \eqref{eq:cond2b}. We now argue for the second inequality. The case for $p=1$ is fulfilled directly by \eqref{eq:cond2b}, so we may assume $p\ge 2$. Consider the decomposition $\bar{\dsV}\setminus\{0\}:=\cup_{j=1}^p\bar{\dsV}_j$ where $\bar{\dsV}_j$ denote vertices belonging to the $j$-th copy in $\ds{G}$. The argument for \eqref{elgass2} proceeds in the same manner as in \eqref{elgass3}, except we bound the summands with qualified by $\dsV_j\cap\dsV_\xi\neq\emptyset$, by $\left(|\dsV_j\cup\{0\}|-1\right)|\s|$.

Finally for the last part, fix some $\bar{\dsV}\subsetneq\dsV_0$ and notice that for each $v\in\dsV_{\xi}$, there are two cumulant kernels emerging from it, and that $v$ can belong to at most one edge $e$ with $e\in\dsE^{\downarrow}(\bar\dsV)$. This means that when we decompose $\bar{\dsV}=\cup_{j=1}^p\bar{\dsV}_j$, we can bound the required quantity from below by using \eqref{eq:cond3}:
\be\begin{split}
&\sum_{e\in\mathds{E}_0(\bar{\ds{V}})}a_e + \sum_{e\in\mathds{E}^\downarrow(\bar{\ds{V})}}\bigl(\mathbbm{1}_{\{v_e\in\bar{\ds{V}}\vee r_e=0\}}(a_e+r_e-1)-(r_e-1)\bigr) \\ &+\sum_{e\in\mathds{E}^\uparrow(\bar{\ds{V}})}\bigl((a_e+r_e)-\mathbbm{1}_{\{v_e\in\bar{\ds{V}}\}}r_e\bigr) + \mathbbm{1}_{\{v_\star \in \bar\dsV\}}|\dsG|_\s \\
&\ge\sum_{j=1}^p\biggl(\sum_{e\in\mathds{E}_0(\bar{\ds{V}}_j)}a_e + \sum_{e\in\mathds{E}^\downarrow(\bar{\ds{V}}_j)}\bigl(\mathbbm{1}_{\{v_e\in\bar{\ds{V}}_j\vee r_e=0\}}(a_e+r_e-1)-(r_e-1)\bigr)  \\
&+\sum_{e\in\mathds{E}^\uparrow(\bar{\ds{V}})}\bigl((a_e+r_e)-\mathbbm{1}_{\{v_e\in\bar{\ds{V}}\}}r_e\bigr)+ \frac{3}{2}|\bar{\dsV}_j\cap\dsV_\xi|+\mathbbm{1}_{\{v_\star \in \bar\dsV\}}|\dsG|_\s\biggr) \\
&>\left(|\bar{\ds{V}}| - \mathbbm{1}_{v_\star\in\bar\dsV}\right)|\s|
\end{split}
\ee
\end{proof}

\subsection{Renormalisation Procedure}\label{sec:RenormProc}

In this section, we will develop our renormalisation procedure that is general enough to deal with the gKPZ. Let $\ds{G}=(\dsV,\dsE)$ be the kind of diagrammatic representation we have seen already. Assume it has some ``negative subtree'' $\bar{T}=(\bar{\dsV},\bar{\dsE})$ that requires renormalisation in the sense that it violates \eqref{eq:cond2a} or \eqref{eq:cond2b} - \eqref{eq:cond3} is always satisfied while \eqref{eq:cond1a} and \eqref{eq:cond1b} are mild enough that we do not expect them to be contravened. We will effect this renormalisation by changing the label of some $\gamma\in\mathds{E}^{\downarrow}(\bar\dsV)$ by substituting $v_\gamma= v_0$ with a node of $\bar{T}$ such that the new Taylor expansion point has a renormalisation effect on $T$.

Let $\gamma=(v_1,v_2)\in\mathds{E}^{\downarrow}(\bar\dsV)$ and $v\in\bar T$ such that there exists a $v'$ such that $(v,v')\in\mathds{E}^{\uparrow}(\bar\dsV)$. Diagrammatically we mean:$\begin{tikzpicture}[baseline=-1mm]
\node at (-2,0) [dot, label=above:$v$] (a) {}; 
\node at (0,0) [dot, label=above:$v_2$] (b) {};
\node at (2,0) [dot, label=above:$v_1$] (c) {};
\draw[semithick, - >] (c) to node[labl,pos=0.45] {\tiny $a_\gamma,r_\gamma,v_\gamma$} (b);
\draw[snake=zigzag, segment amplitude=0.5pt,segment length = 1mm, line after snake = 1mm, - >] (b) to (a);
\end{tikzpicture}$ where the symbol $\begin{tikzpicture}[baseline=-1mm]
\node at (0,0) (a) {}; 
\node at (1,0) (b) {};
\draw[snake=zigzag, segment amplitude=0.5pt,segment length = 1mm, line after snake = 1mm, - >] (b) to (a);
\end{tikzpicture}$ means that there exists a path between $v_2$ and $v$. In most practical examples it will be an edge. The label of $\gamma$ is replaced by $(a_\gamma, r_\gamma', v)$. This transformation and the choice of the level $r_\gamma'$ will depend on the subtree $\bar T$. For instance, we take $r_\gamma'=\max(\lceil-|\bar{T}|_\s\rceil,r_\gamma)$. Starting with $v_\gamma=v_0$, we want to rewrite the Taylor expansion in the point $v$. We proceed as follows:
\begin{align*}\hat{K}_\gamma(&x_{v_2}-x_{v_1})=K_\gamma(x_{v_2}-x_{v_1})-\hspace{-2mm}\sum_{|j|_\s<r_e}\frac{(x_{v_2})^j}{j!}K_\gamma^{(j)}(-x_{v_1})\\
&=K_\gamma(x_{v_2}-x_{v_1})-\hspace{-2mm}\sum_{|j|_\s<r_e'}\frac{(x_{v_2}-x_v)^j}{j!}K_\gamma^{(j)}(x_v-x_{v_1})\\
&+\sum_{|j|_\s<r_e'}\frac{(x_{v_2}-x_v)^j}{j!}K_\gamma^{(j)}(x_v-x_{v_1})-\hspace{-4mm}\sum_{|j+k|_\s<r_e}\frac{(x_{v_2}-x_v)^j(x_v)^k}{j!k!}K_\gamma^{(j+k)}(-x_{v_1})\\
&=K_\gamma(x_{v_2}-x_{v_1})-\hspace{-2mm}\sum_{|j|_\s<r_e'}\frac{(x_{v_2}-x_v)^j}{j!}K_\gamma^{(j)}(x_v-x_{v_1})\\
&+\sum_{|j|_\s<r_e'}\frac{(x_{v_2}-x_v)^j}{j!}\Bigl(K_\gamma^{(j)}(x_v-x_{v_1})-\hspace{-4mm}\sum_{|k|_\s<r_e-|j|_\s}\frac{(x_v)^k}{k!}K_\gamma^{(j+k)}(-x_{v_1})\Bigr)
\end{align*} 
Graphically, given a graph $\ds{G}$, we are effecting the following decomposition:
\beq\label{eqdecomp}
\begin{tikzpicture}[baseline=0.5mm]
\node at (0,0) [dot, label=right:$v$] (a) {}; 
\node at (0,1) [dot, label=right:$v_2$] (b) {};
\node at (0,2) [dot, label=right:$v_1$] (c) {};
\draw[semithick, - >] (c) to node[labl,pos=0.45] {\tiny $\gamma$} (b);
\draw[snake=zigzag, segment amplitude=0.5pt,segment length = 1mm, line after snake = 1mm, - >] (b) to (a);
\end{tikzpicture} = \begin{tikzpicture}[baseline=0.5mm]
\node at (0,0) [dot, label=right:$v$] (a) {}; 
\node at (0,1) [dot, label=right:$v_2$] (b) {};
\node at (0,2) [dot, label=right:$v_1$] (c) {};
\draw[semithick, - >] (c) to node[labl,pos=0.45] {\tiny $\gamma_e$} (b);
\draw[snake=zigzag, segment amplitude=0.5pt,segment length = 1mm, line after snake = 1mm, - >] (b) to (a);
\end{tikzpicture}+\sum_{|j|_\s<r_e'}\begin{tikzpicture}[baseline=0.1cm]
\node at (-1,1) [dot, label=left:$v_1$] (a) {};
\node at (1,1) [dot, label=right:$v_2$] (b) {};
\node at (0,0) [dot, label=right:$v$] (c) {};
\draw[semithick, - >] (a) to node[labl,pos=0.45] {\tiny $\gamma_j$} (c);
\draw[snake=zigzag, segment amplitude=0.5pt,segment length = 1mm, line after snake = 1mm, - >] (b) to (c);
\draw[semithick, - >] (c) to[bend left=60] node[labl,pos=0.45] {\tiny $e_j$} (b);
\end{tikzpicture}
\eeq

where $\gamma$, $\gamma_e$, $\gamma_j$, and $e_j$ stand for the labels $(a_\gamma,r_\gamma,v_0)$, $(a_\gamma,r'_\gamma,v)$, $(a_\gamma + |j|_\s, \max(r_\gamma-|j|_\s,0),v_0)$, $(-|j|_\s,0,v_0)$.

One is able to prove that in the case of the \eqref{eq:gKPZ}, this renormalisation procedure is relatively mild, in that we do not expect it to create new divergences. We prove the following results in this direction:

\begin{proposition}\label{prop:invcond3}
If either of the conditions - \eqref{eq:ass2e2}, \eqref{eq:cond3} - is satisfied in \eqref{eqdecomp} for the terms with the labels $\gamma$ and 
$\gamma_e$, then that condition is also satisfied on the other terms on $\bar{\dsV}\subset\dsV$ such that $\bar{\dsV}\cap\{v,v_1,v_2\}\neq\{v\}$.
\end{proposition}
\begin{proof}

Given that the conditions are satisfied on terms with $\gamma$ and $\gamma_e$, it suffices to prove that the contribution of $\gamma_j$ and $e_j$ to the left hand side of \eqref{eq:ass2e2} and \eqref{eq:cond3} is greater than the minimum of the contributions of $\gamma$ and $\gamma_e$. To this end we fix for $\bar{\dsV}\subset\dsV$ the set $\tilde{\ds{V}}=\bar{\dsV}\cap\{v,v_1,v_2\}$, and consider the the contribution of each of the four edges for every possible $\tilde{\dsV}$. So for example, if one considers $\tilde\dsV = \{v\}$ and the edge $\gamma_j$, then one calculates in the following way. For $\gamma_j$ carries the label $(a_\gamma+|j|_\s,\max(r_\gamma - |j|_\s,0),v_0)$, and we have that $v_0\not\in\{v\}$ one has to differentiate between the cases $r_\gamma - |j|_\s = 0$ and $r_\gamma - |j|_\s > 0$. In the former case, it gives us the contribution $a_\gamma + |j|_s$ and in the later yields a contribution of $-(r_e+|j|_s - 1)$. In the following table we collate all of these minor computations:
\\
\vspace{3mm}
\begin{center}
\begin{tabular}{|c|c|c|c|c|}
\hline 
\vphantom{\huge{A}} $\tilde{\dsV}$ & $\gamma_j$ & $e_j$ & $\gamma_e$ & $\gamma$ \\
\hline
$\substack{{\color{white} 0} \\ {\color{white} 0} \\ \{v\} \\ {\color{white} 0} \\ {\color{white} 0}}$ & $\substack{- \mathbbm{1}_{\{r_\gamma-|j|_\s>0\}}(r_\gamma-|j|_\s-1) \\ + \mathbbm{1}_{\{r_\gamma-|j|_\s= 0\}}(a_\gamma+|j|_\s)}$ & $-|j|_\s $ & $ 0 $ & $ 0 $\\
\hline
$\{v_1\} $ & $ a_\gamma+|j|_\s+\text{max}(r_\gamma-|j|_\s,0) $ & $ 0 $ & $ a_\gamma+r'_\gamma $ & $ a_\gamma+r_\gamma $\\
\hline
$\{v_2\} $ & $ 0 $ & $ -|j|_\s $ & $ -(r'_\gamma-1) $ & $ -(r_\gamma-1) $\\
\hline
$\{v,v_1\}$ & $ a_\gamma+|j|_\s $ & $ -|j|_\s $ & $ a_\gamma $ & $ a_\gamma + r_\gamma $ \\
\hline
$\substack{{\color{white} 0} \\ {\color{white} 0} \\ \{v,v_2\} \\ {\color{white} 0} \\ {\color{white} 0}}$ & $\substack{- \mathbbm{1}_{\{r_\gamma-|j|_\s>0\}}(r_\gamma-|j|_\s-1) \\ + \mathbbm{1}_{\{r_\gamma-|j|_\s=0\}}(a_\gamma+|j|_\s)} $ & $ -|j|_\s $ & $ -(r'_\gamma-1) $ & $ -(r_\gamma-1) $\\
\hline
$\{v_1,v_2\} $ & $ a_\gamma+|j|_\s+\text{max}(r_\gamma-|j|_\s,0) $ & $ -|j|_\s $ & $ a_\gamma $ & $ a_\gamma $\\
\hline
$\{v,v_1,v_2\}$ & $ a_\gamma+|j|_\s $ & $ -|j|_\s $ & $ a_\gamma$ & $ a_\gamma $\\
\hline
\end{tabular}
\end{center}
\vspace{3mm}
Hence the sum of the contributions of $\gamma_j$ and $e_j$ is indeed greater than the minimum between $\gamma_e$ and $\gamma$ except for $\tilde\dsV=\{v\}$, when $r_e-|j|_\s>0$. But this implies that $-(r_e - 1) \leq 0$, which in turn implies the required result.
\end{proof}

\begin{proposition}\label{prop:1}
If the condition \eqref{elgass2} is satisfied in \eqref{eqdecomp} for the terms with the labels $\gamma$ and $\gamma_e$ on some subset $\bar\dsV$, then this condition is satisfied for the other terms on the same subsets.
\end{proposition}
\begin{proof}
Let $\hat\dsV\subseteq \bar\dsV$, if $v_0 \in \hat\dsV$ then from the previous proposition the conditions \eqref{elgass3} are satisfied on $\bar\dsV\backslash \hat\dsV$, which due to Proposition~\ref{prop:ass4-2} gives us the required conclusion. It remains to consider the case $v_0\notin \hat\dsV$. As in the previous proposition, we calculate the contribution of the new edges $e_j$ and $\gamma_j$ to \eqref{elgass2} for all possible $\tilde\dsV=\hat\dsV\cap\{v,v_1,v_2\}$:
$$
\begin{array}{|c|c|c|c|c|}
\hline 
\tilde\dsV& \gamma_j & e_j & \gamma_e & \gamma \\
\hline
\{v\} & 0 & 0 & 0 & 0 \\
\hline
\{v_1\} & 0 & 0 & 0 & 0 \\
\hline
\{v_2\} & 0 & 0 & 0 & 0 \\
\hline
\{v,v_1\} & a_\gamma+|j|_\s & 0 & a_\gamma + r'_\gamma & a_\gamma \\
\hline
\{v,v_2\}& 0 & -|j|_\s & -r'_\gamma & 0 \\
\hline
\{v_1,v_2\} & 0 & 0 & a_\gamma & a_\gamma \\
\hline
\{v,v_1,v_2\}& a_\gamma+|j|_\s & -|j|_\s & a_\gamma & a_\gamma \\
\hline
\end{array}
$$
We see that the sum of the contributions from the edges in the terms in $\gamma_j$ and $e_j$ is smaller than the maximum between $\gamma_e$ and $\gamma$.
\end{proof}

After some number of contractions, one expects to see a graph with no instances of noises on it. If this graph still violates \eqref{eq:cond2a} or \eqref{eq:cond2b}, we would like to use it to motivate our definition of the renormalisation constant for our model. One has to be careful however that there are no Taylor expansions on these Kernels. Should this happen we remedy the situation by removing the Taylor expansions, by which we mean a transformation of the form $(a_e,r_e,v_0)\rightarrow (a_e,0)$. Let $e=(v_1,v_2)$ and edge in $\bar T$ with a label $(a_e,r_e,v_0)$ and $r_e>0$. We can perform the following decomposition to effect the previously stated transformation:
\begin{equation}\label{eq:tayexp}\begin{tikzpicture}[baseline=0cm,scale=0.8]
\node at (0,-1) [dot,label=right:$v_1$] (a) {};
\node at (0,1) [dot,label=right:$v_2$] (b) {};
\draw[kepsus] (b) to (a);
\end{tikzpicture}
= \begin{tikzpicture}[baseline=0cm,scale=0.8]
\node at (0,-1) [dot,label=right:$v_1$] (a) {};
\node at (0,1) [dot,label=right:$v_2$] (b) {};
\draw[kepsus] (b) to node[labl,pos=0.45] {\tiny $a_e,0$} (a);
\end{tikzpicture} - \sum_{|k|_\s < r_e}\begin{tikzpicture}[baseline=0cm,scale=0.8]
\node at (0,0) [dot,label=left:$v_1$] (a) {};
\node at (1,-1) [root,label=below:$0$] (b) {};
\node at (1,1) [dot, label=above:$v_2$] (c) {};
\draw[kepsus] (a) to node[labl,pos=0.45] {\tiny $-|j|_\s,0$} (b);
\draw[kepsus] (c) to node[labl,pos=0.45] {\tiny $a_e + |j|_\s$} (b);
\end{tikzpicture}
\end{equation}

For the generalised KPZ equation it is an easy power counting problem to prove that:

\begin{proposition}
The previous terms depending on $k$ satisfy the conditions \eqref{elgass1} and \eqref{elgass2} on $\dsV(\bar{T})$
\end{proposition}

In the following list we compile the sort of negative subtrees we expect to see in our analysis:

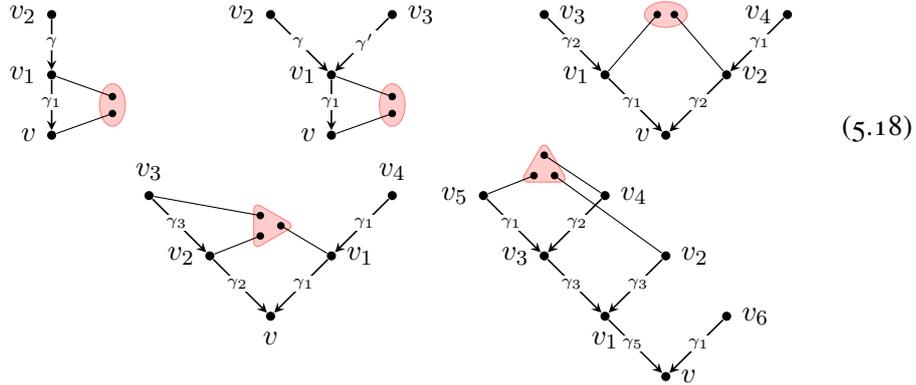
\begin{equation}\label{eqneg}
\begin{tikzpicture}[baseline=0cm,scale=0.8]
\node at (0,0) [dot,label=left:$v$] (a) {};
\node at (0,1) [dot,label=left:$v_1$] (b) {};
\node[cumu2n] (f) at (1,0.5) {};
\draw[cumu2] (f) ellipse (6pt and 10pt);
\node at (0,2) [dot,label=left:$v_2$] (d) {};
\draw[kepsus] (b) to node[labl,pos=0.45] {\tiny $\gamma_1$} (a);
\draw[kepsus] (d) to node[labl,pos=0.45] {\tiny $\gamma$} (b);
\draw[] (f.north) node[dot1] {} to (b);
\draw[] (f.south) node[dot1] {} to (a);
\end{tikzpicture}
\begin{tikzpicture}[baseline=0cm,scale=0.8]
\node at (0,0) [dot,label=left:$v$] (a) {};
\node[cumu2n] (f) at (1,0.5){};
\draw[cumu2] (f) ellipse (6pt and 10pt);
\node at (0,1) [dot,label=left:$v_1$] (c) {};
\node at (1,2) [dot,label=right:$v_3$] (d) {};
\node at (-1,2) [dot,label=left:$v_2$] (e) {};
\draw[kepsus] (c) to node[labl,pos=0.45] {\tiny $\gamma_1$} (a);
\draw[] (f.north) node[dot1] {} to (c);
\draw[] (f.south) node[dot1] {} to (a);
\draw[kepsus] (e) to node[labl,pos=0.45] {\tiny $\gamma$} (c);
\draw[kepsus] (d) to node[labl,pos=0.45] {\tiny $\gamma'$} (c);
\node at (3.5,2) [dot,label=right:$v_3$] (u) {};
\node at (7.5,2) [dot,label=left:$v_4$] (v) {};
\node at (5.5,0) [dot,label=left:$v$] (g) {};
\node at (4.5,1) [dot,label=left:$v_1$] (h) {};
\node at (6.5,1) [dot,label=right:$v_2$] (i) {};
\node[cumu2n] (j) at (5.5,2) {};
\draw[cumu2] (j) ellipse (10pt and 6pt);
\draw[] (j.west) node[dot1] {} to (h);
\draw[] (j.east) node[dot1] {} to (i);
\draw[kepsus] (v) to node[labl,pos=0.45] {\tiny $\gamma_1$} (i);
\draw[kepsus] (u) to node[labl,pos=0.45] {\tiny $\gamma_2$} (h);
\draw[kepsus] (h) to node[labl,pos=0.45] {\tiny $\gamma_1$} (g);
\draw[kepsus] (i) to node[labl,pos=0.45] {\tiny $\gamma_2$} (g);
\node[] (k) at (-1,-1.5) {};
\node[cumu3,rotate=270] (k-) at (-1.05,-1.5) {};
\node at (-3,-1) [dot,label=above:$v_3$] (l) {};
\node at (-2,-2) [dot,label=left:$v_2$] (m) {};
\node at (-1,-3) [dot,label=below:$v$] (n) {};
\node at (0,-2) [dot,label=right:$v_1$] (o) {};
\node at (1,-1) [dot,label=above:$v_4$] (w) {};
\draw[] (k.north west) node[dot1] {} to (l);
\draw[] (k.south west) node[dot1] {} to (m);
\draw[] (k.east) node[dot1] {} to (o); 
\draw[kepsus] (l) to node[labl,pos=0.45] {\tiny $\gamma_3$} (m);
\draw[kepsus] (m) to node[labl,pos=0.45] {\tiny $\gamma_2$} (n);
\draw[kepsus] (o) to node[labl,pos=0.45] {\tiny $\gamma_1$} (n);
\draw[kepsus] (w) to node[labl,pos=0.45] {\tiny $\gamma_1$} (o);
\node at (2.5,-1) [dot,label=left:$v_5$] (p) {};
\node at (3.5,-2) [dot,label=left:$v_3$] (q) {};
\node at (4.5,-1) [dot,label=right:$v_4$] (r) {};
\node at (4.5,-3) [dot,label=below:$v_1$] (s) {};
\node at (5.5,-2) [dot,label=right:$v_2$] (t) {};
\node at (5.5,-4) [dot,label=right:$v$] (x) {};
\node at (6.5,-3) [dot,label=right:$v_6$] (y) {};
\draw[kepsus] (p) to node[labl,pos=0.45] {\tiny $\gamma_1$} (q);
\draw[kepsus] (r) to node[labl,pos=0.45] {\tiny $\gamma_2$} (q);
\draw[kepsus] (q) to node[labl,pos=0.45] {\tiny $\gamma_3$} (s);
\draw[kepsus] (t) to node[labl,pos=0.45] {\tiny $\gamma_3$} (s);
\draw[kepsus] (y) to node[labl,pos=0.45] {\tiny $\gamma_1$} (x);
\draw[kepsus] (s) to node[labl,pos=0.45] {\tiny $\gamma_5$} (x);
\node at (3.5,-0.5) (u) {};
\node[cumu3] (u-) at (3.5,-0.55) {};
\draw[] (u.south west) node[dot1] {} to (p);
\draw[] (u.north) node[dot1] {} to (r);
\draw[] (u.south east) node[dot1] {} to (t); 
\end{tikzpicture}
\end{equation}

The first two subdivergences tend to be similar in that they can be treated in a general fashion which we encapsulate in the propositions that follow. For the other divergences, we will present some ad-hoc methods in Section~\ref{sec:exmclc}.

\begin{proposition}
Let $\ds{G}$ be one of the first two graphs in \eqref{eqneg}. If $\ds{G}$ satisfies the condition \eqref{eq:ass2e2} then the new graph $\ds{G}_*$ obtained from the transformation of the label of $\gamma=(v_2,v_1)$ to $(a_\gamma,r'_\gamma,v)$ satisfies the same condition.
\end{proposition}
\begin{proof}
We need only check that in changing to $v_\gamma$ from $v_0$ in $G$ the contribution of the edge $e$ is preserved. Let $\bar{\dsV}\subset\dsV$; we consider each case in turn:
\begin{itemize}
\item $\gamma\in\mathds{E}_0(\bar{\dsV})$ then we still have the same contribution of $a_\gamma$.
\item $\gamma\in\mathds{E}^{\downarrow}(\bar{\dsV})$, we must have $v_1\in \bar\dsV$. There are a few possibilities:
\begin{itemize}
\item If $v\in\bar{\dsV}$ then the new contribution is $a_\gamma$ which is the same when $r_\gamma = 0$ and which is better than $-(r_\gamma-1)$ when $r_\gamma>0$.
\item If $v\notin\bar{\dsV}$, then for the first graph the only possibility is $\dsV = \{v_1\}$. In this case if $r'_\gamma=0$, the contribution is the same, so the only relevant thing case to be considered is when $\bar{T}$ is a negative subtree. One is able to compute:
\begin{equs}
a_{\gamma_1} - |\s| = -|\bar T|_\s > (r'_e - 1),
\end{equs}
which means that \eqref{elgass3} is satisfied. For the second tree, if $v_3\not\in\bar\dsV$, the contribution $-(r_{\gamma^{'}}-1)$ is equal to zero. Otherwise, we may use the previous bound for $\{v_1\}$ and the fact that $r_{\gamma_1}\ge r_{\gamma^{'}}$ to conclude for $\{v_3\}$.

\end{itemize}
\item $\gamma\in\mathds{E}^{\uparrow}(\bar\dsV)$, we have $v_2\in \bar\dsV$.
\begin{itemize}
\item If $v\notin \bar\dsV$, the contribution is improved, that is $a_\gamma+r'_\gamma\geq a_\gamma+r_\gamma$.
\item Otherwise if $v\in \bar\dsV$, we lose a factor $r_\gamma$ in the contribution. Then for $r_{\gamma_1}>0$, one see that:
\begin{equs}
a_{\gamma_1} + \frac{3}{2} + r_{\gamma'} - 1 \leq 2|\s|,
\end{equs}
and a similar bound holds for $r_{\gamma_2}$. If we include $v_1$ to $\bar\dsV$, the bound will become sharper, but the argument to check \eqref{elgass3} is similar.
\end{itemize}
\end{itemize}
\end{proof}

\begin{proposition}
Let $\ds{G}$ be one of the two first graphs in \eqref{eqneg}. If $\ds{G}$ satisfies the conditions \eqref{eq:ass2e1} for some subsets $\bar{\dsV}\neq\{v,v_1\}$ then the new graph obtained from the transformation of the label of $\gamma=(v_2,v_1)$ to $(a_\gamma,r'_\gamma,v)$, satisfies this condition on the same subset and on $\{v,v_1\}$.
\end{proposition}
\begin{proof}
Let $\bar{\dsV} \subset G$, such that the condition \eqref{eq:ass2e1} is satisfied or such that $\bar\dsV = \{v,v_1\}$. We have to check that by changing $v_\gamma$ in $G$ the contribution of the edge $\gamma$ is preserved and indeed has been improved in the case of $\bar\dsV=\ds{G}'$. As before we proceed case by case:
\begin{itemize}
\item $\gamma\in\mathds{E}_0(\bar\dsV)$ then we still have the same contribution with $a_\gamma$
\item $\gamma\in\mathds{E}^{\uparrow}(\bar\dsV)$, we have $v_2\in \bar\dsV$.
\begin{itemize}
\item If $v\notin \bar\dsV$, the contribution does not change (from $0$).
\item If $v\in \bar\dsV$ (and $r_\gamma>0$ because otherwise the contribution is zero anyway) then the contribution improves to $a_\gamma + r'_\gamma - 1$ which is greater than $a_\gamma + r_\gamma - 1$ in the cases we are concerned with. By adding $v_1$ and using the fact that:
$$a_{\gamma_1} + 2\left(\frac{3}{2}\right) - (|\bar\dsV| - 1)|\s|+r'_\gamma-1>0$$
we obtain the desired bound.
\end{itemize}
\item $e\in\mathds{E}^{\downarrow}(\bar\dsV)$, we have $v_1\in V$
\begin{itemize}
\item If $v\in \bar\dsV$ then the new contribution is $-r'_\gamma$ which in generally is less than or equal to $-r_\gamma$. 
\item If $v\not\in \bar\dsV$, then $\bar\dsV$ is not equal to $\ds{G}'$ and the condition \eqref{eq:ass2e1} is satisfied.
\end{itemize}
\end{itemize}
\end{proof}

\begin{proposition}
For the graphs given in \eqref{eqneg}, if the condition \eqref{eq:ass2e2} and \eqref{eq:cond3} is satisfied on some subset $\bar\dsV$ in \eqref{eqdecomp} for the terms with the labels $\gamma$ and $\gamma_e$, then this condition is satisfied on the other terms for the same subsets. 
\end{proposition}
\begin{proof}
The result comes essentially from the proposition \eqref{prop:invcond3} except when we consider a subset $V$ such that $V\cap\{v,v_1,v_2\}=\{v\}$ and $j < r_\gamma$. In that case, the node $v$ has the contribution $-(r_\gamma-1)$, which due to the fact that the examples only have $r_\gamma \le 1$, gives the result.
\end{proof}

We iterate the renormalisation procedure on each edge in $\mathds{E}^{\downarrow}(\bar{\dsV})$ to obtain a graph with no leaves which can be divergent. In the generalised KPZ terms, the diverging graph has the following form:

\beq\label{eq:decomp3}
\begin{tikzpicture}[baseline=0cm,scale=0.8]
\node at (0,-1) [dot,label=left:$v$] (a) {};
\node at (0,1) [dot,label=left:$v_1$] (b) {};
\node[cumu2n] (z) at (1,0) {};
\draw[cumu2] (z) ellipse (6pt and 10pt);
\draw[] (z.north) node[dot1] {} to (b);
\draw[] (z.south) node[dot1] {} to (a); 
\draw[kepsus] (b) to node[labl,pos=0.45] {\tiny $\gamma$} (a);
\end{tikzpicture} = 
\begin{tikzpicture}[baseline=0cm,scale=0.8]
\node at (0,-1) [dot,label=left:$v$] (a) {};
\node at (0,1) [dot,label=left:$v_1$] (b) {};
\draw[kepsus] (b) to node[labl,pos=0.45] {\tiny $\gamma'$} (a);
\node[cumu2n] (z) at (1,0) {};
\draw[cumu2] (z) ellipse (6pt and 10pt);
\draw[] (z.north) node[dot1] {} to (b);
\draw[] (z.south) node[dot1] {} to (a);
\end{tikzpicture} + \sum_{k < r_e} \begin{tikzpicture}[baseline=0cm,scale=0.8]
\node at (1.5,-1) [root,label=below:$0$] (a) {};
\node at (1.5,1) [dot,label=right:$v_1$] (b) {};
\node at (0,0.5) [dot,label=left:$v$] (c) {};
\draw[kepsus] (b) to node[labl,pos=0.45] {\tiny $a_\gamma + k,0$} (a);
\draw[kepsus] (c) to node[labl,pos=0.45] {\tiny $-k,0$} (a);
\node[cumu2n] (z) at (0.5,1.5) {};
\draw[cumu2] (z) ellipse (10pt and 6pt);
\draw[] (z.east) node[dot1] {} to (b);
\draw[] (z.west) node[dot1] {} to (c);
\end{tikzpicture},
\eeq

where we have the following labels for $\gamma,\,\gamma'$: $(a_\gamma,r_\gamma,0),\,(a_\gamma,0,0,)$.
\begin{proposition}
The condition \eqref{elgass3} is satisfied for all the terms of the previous decomposition.
\end{proposition}
\begin{proof}
Let $\bar\dsV$ be a subset of $G$. We have for $\tilde\dsV=\bar\dsV\cap\{v,v_1\}$:
$$\begin{array}{|c|c|c|c|c|}
\hline
\tilde\dsV & \gamma & \gamma' & (v,0) & (v_1,0) \\
 \hline
 \{v,v_1\} & a_\gamma & a_\gamma & -k & a_\gamma + k \\
 \hline
 \{v\} & -(r_\gamma - 1) & a_\gamma & -k & 0 \\
 \hline
 \{v_1\} & a_\gamma + r_\gamma & a_\gamma & 0 & a_\gamma + k \\
 \hline
\end{array}$$
For the first two rows of the previous table, the contribution of $\gamma'$  and the sum of $(0,v)$ and $(0,v_1)$ are greater than that of $\gamma$. It is for $\{v_1\}$ and $r_\gamma > 0$, that we need to be more careful. In the case $v\in \bar\dsV$, the contribution of the edge from the cumulant is bounded by $|\s|$, which means that \eqref{elgass3} holds true. By removing this node and keeping only the contribution other than the half-edge from the cumulant, we notice that $a_\gamma$ is sufficient for the required bound instead of $a_\gamma + r_\gamma$.
\end{proof}

\begin{proposition}
The condition \eqref{elgass2} is satisfied on each term $G_k$ depending on $k < r_\gamma$ of the previous decomposition \eqref{eq:decomp3} and it is also satisfied for the first term $G$ on subset $V\neq\{v,v_1\}$
\end{proposition}
\begin{proof}
Let $\bar\dsV\subset G$. If $0\notin\bar\dsV$ then we can split the graph $G_k$ into two KPZ trees by splitting the cumulant into two leaves and the condition \eqref{elgass1} is satisfied on each tree which gives \eqref{elgass2} on $V$. If $0\in \dsV$ the previous proposition gives the condition \eqref{elgass3} on $G_k\setminus \bar\dsV$ which proves \eqref{elgass2} on $\bar\dsV$.
\end{proof}

If one compares the general renormalisation map in Section~\ref{sec:RnmMdl} (Definition~\ref{def:Deltam})  with the renormalisation procedure we have suggested, one finds that they are equivalent for these sorts of divergences. In fact, in our renormalisation procedure, we are moving errant edges to the root and excising the troublesome subtree through translation invariance. The extraction-contraction procedure of the renormalisation procedure amounts to the same. Indeed, for gKPZ, there are no nested neither overlapping divergences. One has to face only one subdivergence which is treated by the procedure described above and corresponds to the application of $ R_g $ to a node in the tree. We apply the renormalisation procedure to a graph that comes from some  $ \Pi^N \btau $ and it produces in the end a graph coming from $ \Pi^{N,\tilde{g}} \btau $  with $ \tilde{g} $ equal to $ g $ except on the empty tree, where it will be equal to zero. This forces the extraction of one subtree.
In the end, we get:
\begin{equs}
\Pi^{N,g} \btau = \Pi^{N} \btau + \Pi^{N,\tilde{g}} \btau
\end{equs}
This suggests that we are computing the correct counter-terms with our renormalisation procedure.

\begin{remark} For several disjoint subdivergences, the same procedure can be applied. Nested subdivergences could be handled in a recursive way by starting the procedure on the inner subdivergence and then proceeding to the outers. Overlapping divergences will require much more work and use similar techniques as in \cite{CH16} concerning safe/unsafe forests.
\end{remark}

\subsection{Example Computations}\label{sec:exmclc}

In this section, we would like to work out some calculations to show how the divergences will look like in practice, and how they are then resolved - in particular we will show how the subdivergences that were listed in the previous section are dealt with. Not to clutter notation we fix the following labels upfront $\gamma_1 = (|\s|-2,1,v_0)$, $\gamma_2=(|\s|-1,0,v_0)$, and $\gamma_3 = (|\s|-2,0,v_0)$. Consider then the tree coming from $\tau=\colb{\CI'(\CI(\Xi)\Xi)^2}$.

$$
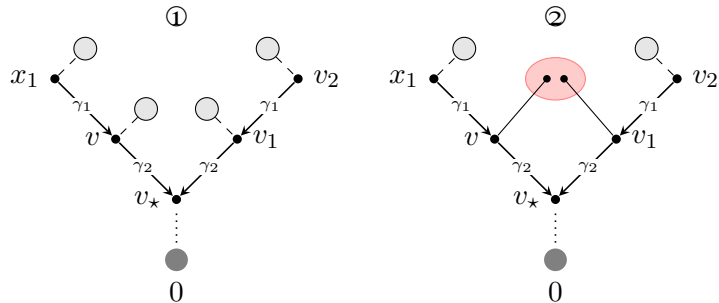
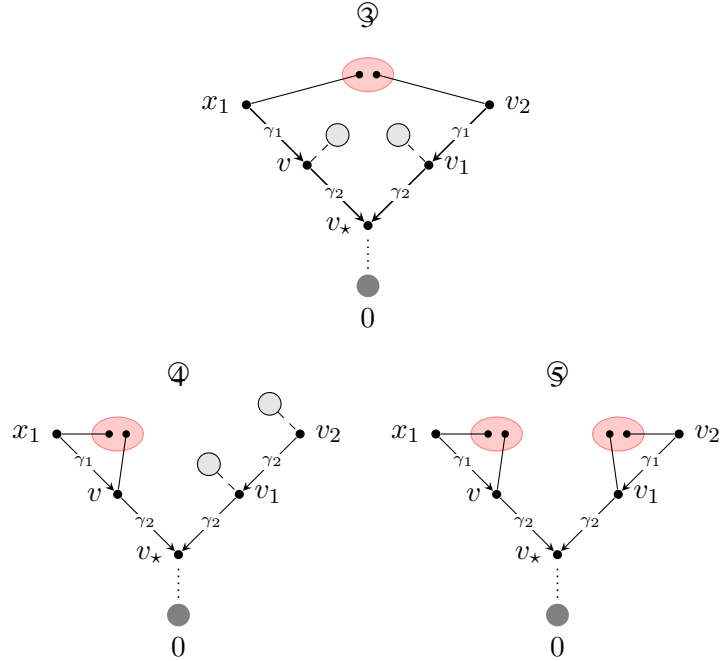
\begin{tikzpicture}[baseline=0cm,scale=0.8]
\node at (0,2) {\textcircled{1}};
\node at (0,-2) [root, label=below:$0$] (a) {};
\node at (0,-1) [dot,label=left:$v_\star$] (b) {};
\node at (-1,0) [dot,label=left:$v$] (c) {};
\node at (1,0) [dot,label=right:$v_1$] (d) {};
\node at (-2,1) [dot,label=left:$x_1$] (e) {};
\node at (-1.5,1.5) [var] (i) {};
\node at (-0.5,0.5) [var] (f) {};
\node at (0.5,0.5) [var] (g) {};
\node at (2,1) [dot,label=right:$v_2$] (h) {};
\node at (1.5,1.5) [var] (j) {};
\draw[testfcn] (a) to (b);
\draw[kepsus] (c) to node[labl,pos=0.45] {\tiny $\gamma_2$} (b);
\draw[kepsus] (d) to node[labl,pos=0.45] {\tiny $\gamma_2$} (b);
\draw[kepsus] (e) to node[labl,pos=0.45] {\tiny $\gamma_1$} (c);
\draw[dashed] (f) to (c);
\draw[dashed] (e) to (i);
\draw[kepsus] (h) to node[labl,pos=0.45] {\tiny $\gamma_1$} (d);
\draw[dashed] (g) to (d);
\draw[dashed] (j) to (h);
\end{tikzpicture}\quad\begin{tikzpicture}[baseline=0cm,scale=0.8]
\node at (0,2) {\textcircled{2}};
\node at (0,-2) [root,label=below:$0$] (a) {};
\node at (0,-1) [dot,label=left:$v_\star$] (b) {};
\node at (-1,0) [dot,label=left:$v$] (c) {};
\node at (1,0) [dot,label=right:$v_1$] (d) {};
\node at (-2,1) [dot,label=left:$x_1$] (e) {};
\node at (2,1) [dot,label=right:$v_2$] (h) {};
\node at (-1.5,1.5) [var] (i) {};
\node at (1.5,1.5) [var] (j) {};
\node[cumu2n] (f) at (0,1){};
\draw[cumu2] (f) ellipse (14pt and 10pt);
\draw[testfcn] (a) to (b);
\draw[kepsus] (e) to node[labl,pos=0.45] {\tiny $\gamma_1$} (c);
\draw[kepsus] (h) to node[labl,pos=0.45] {\tiny $\gamma_1$} (d);
\draw[] (f.west) node[dot1] {} to (c);
\draw[] (f.east) node[dot1] {} to (d); 
\draw[kepsus] (c) to node[labl,pos=0.45] {\tiny $\gamma_2$} (b);
\draw[kepsus] (d) to node[labl,pos=0.45] {\tiny $\gamma_2$} (b);
\draw[dashed] (e) to (i);
\draw[dashed] (h) to (j);
\end{tikzpicture}$$
$$\begin{tikzpicture}[baseline=0cm,scale=0.8]
\node at (0,2.5) {\textcircled{3}};
\node at (0,-2) [root, label=below:$0$] (a) {};
\node at (0,-1) [dot,label=left:$v_\star$] (b) {};
\node at (-1,0) [dot,label=left:$v$] (c) {};
\node at (1,0) [dot,label=right:$v_1$] (d) {};
\node at (-2,1) [dot,label=left:$x_1$] (e) {};
\node at (-0.5,0.5) [var] (y) {};
\node at (0.5,0.5) [var] (g) {};
\node at (2,1) [dot,label=right:$v_2$] (h) {};
\node[cumu2n] (z) at (0,1.5){};
\draw[cumu2] (z) ellipse (12pt and 8pt);
\draw[testfcn] (a) to (b);
\draw[kepsus] (c) to node[labl,pos=0.45] {\tiny $\gamma_2$} (b);
\draw[kepsus] (d) to node[labl,pos=0.45] {\tiny $\gamma_2$} (b);
\draw[kepsus] (e) to node[labl,pos=0.45] {\tiny $\gamma_1$} (c);
\draw[kepsus] (h) to node[labl,pos=0.45] {\tiny $\gamma_1$} (d);
\draw[dashed] (c) to (y);
\draw[dashed] (d) to (g);
\draw[] (z.west) node[dot1] {} to (e);
\draw[] (z.east) node[dot1] {} to (h); 
\end{tikzpicture}$$
$$\begin{tikzpicture}[baseline=0cm,scale=0.8]
\node at (0,2) {\textcircled{4}};
\node at (0,-2) [root, label=below:$0$] (a) {};
\node at (0,-1) [dot,label=left:$v_\star$] (b) {};
\node at (-1,0) [dot,label=left:$v$] (c) {};
\node at (1,0) [dot,label=right:$v_1$] (d) {};
\node at (-2,1) [dot,label=left:$x_1$] (e) {};
\node at (2,1) [dot,label=right:$v_2$] (f) {};
\node at (0.5,0.5) [var] (g) {};
\node at (1.5,1.5) [var] (k) {};
\node[cumu2n] (z) at (-1,1){};
\draw[cumu2] (z) ellipse (12pt and 8pt);
\draw[] (z.west) node[dot1] {} to (e);
\draw[] (z.east) node[dot1] {} to (c);
\draw[->] (e) to node[labl,pos=0.45] {\tiny $\gamma_1$} (c);
\draw[->] (c) to node[labl,pos=0.45] {\tiny $\gamma_2$} (b);
\draw[->] (d) to node[labl,pos=0.45] {\tiny $\gamma_2$} (b);
\draw[->] (f) to node[labl,pos=0.45] {\tiny $\gamma_2$} (d);
\draw[testfcn] (a) to (b);
\draw[dashed] (g) to (d);
\draw[dashed] (k) to (f);
\end{tikzpicture}\quad\begin{tikzpicture}[baseline=0cm,scale=0.8]
\node at (0,2) {\textcircled{5}};
\node at (0,-2) [root, label=below:$0$] (a) {};
\node at (0,-1) [dot,label=left:$v_\star$] (b) {};
\node at (-1,0) [dot,label=left:$v$] (c) {};
\node at (1,0) [dot,label=right:$v_1$] (d) {};
\node at (-2,1) [dot,label=left:$x_1$] (e) {};
\node at (2,1) [dot,label=right:$v_2$] (f) {};
\node[cumu2n] (z) at (-1,1){};
\draw[cumu2] (z) ellipse (12pt and 8pt);
\node[cumu2n] (y) at (1,1){};
\draw[cumu2] (y) ellipse (12pt and 8 pt);
\draw[] (z.west) node[dot1] {} to (e);
\draw[] (z.east) node[dot1] {} to (c);
\draw[] (y.west) node[dot1] {} to (d);
\draw[] (y.east) node[dot1] {} to (f);
\draw[->] (e) to node[labl,pos=0.45] {\tiny $\gamma_1$} (c);
\draw[->] (c) to node[labl,pos=0.45] {\tiny $\gamma_2$} (b);
\draw[->] (d) to node[labl,pos=0.45] {\tiny $\gamma_2$} (b);
\draw[->] (f) to node[labl,pos=0.45] {\tiny $\gamma_1$} (d);
\draw[testfcn] (a) to (b);
\end{tikzpicture}$$

In our case however we will also see higher order contractions. Some possible examples are:

$$
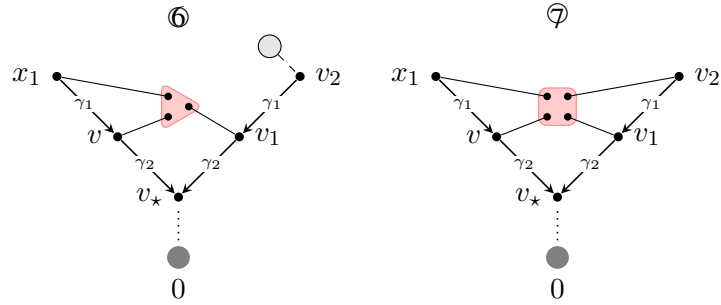
\begin{tikzpicture}[baseline=0cm,scale=0.8]
\node at (0,2) {\textcircled{6}};
\node at (0,-2) [root, label=below:$0$] (a) {};
\node at (0,-1) [dot,label=left:$v_\star$] (b) {};
\node at (-1,0) [dot,label=left:$v$] (c) {};
\node at (-2,1) [dot,label=left:$x_1$] (y) {};
\node at (1,0) [dot,label=right:$v_1$] (d) {};
\node at (2,1) [dot,label=right:$v_2$] (h) {};
\node at (1.5,1.5) [var] (k) {};
\node[] (e) at (0,0.5) {};
\node[cumu3,rotate=270] (e-) at (-0.05,0.5) {};
\draw[testfcn] (a) to (b);
\draw[kepsus] (c) to node[labl,pos=0.45] {\tiny $\gamma_2$} (b);
\draw[kepsus] (d) to node[labl,pos=0.45] {\tiny $\gamma_2$} (b);
\draw[kepsus] (h) to node[labl,pos=0.45] {\tiny $\gamma_1$} (d);
\draw[kepsus] (y) to node[labl,pos=0.45] {\tiny $\gamma_1$} (c);
\draw[] (e.north west) node[dot1] {} to (y);
\draw[] (e.south west) node[dot1] {} to (c);
\draw[] (e.east) node[dot1] {} to (d);
\draw[dashed] (k) to (h);
\end{tikzpicture}\quad\begin{tikzpicture}[baseline=0cm,scale=0.8]
\node at (0,-2) [root, label=below:$0$] (a) {};
\node at (0,2) {\textcircled{7}};
\node at (0,-1) [dot,label=left:$v_\star$] (b) {};
\node at (-1,0) [dot,label=left:$v$] (c) {};
\node at (-2,1) [dot,label=left:$x_1$] (y) {};
\node at (1,0) [dot,label=right:$v_1$] (d) {};
\node at (2,1) [dot,label=right:$v_2$] (h) {};
\node[] (e) at (0,0.5) {};
\node[cumu4] (e-) at (e) {};
\draw[testfcn] (a) to (b);
\draw[kepsus] (c) to node[labl,pos=0.45] {\tiny $\gamma_2$} (b);
\draw[kepsus] (d) to node[labl,pos=0.45] {\tiny $\gamma_2$} (b);
\draw[kepsus] (h) to node[labl,pos=0.45] {\tiny $\gamma_1$} (d);
\draw[kepsus] (y) to node[labl,pos=0.45] {\tiny $\gamma_1$} (c);
\draw[] (e.north west) node[dot1] {} to (y);
\draw[] (e.south west) node[dot1] {} to (c);
\draw[] (e.south east) node[dot1] {} to (d);
\draw[] (e.north east) node[dot1] {} to (h);
\end{tikzpicture}$$

The quantities $(|\dsV|-1)|\s|$ and $(|\dsV|-\frac{1}{2})|\s|$ will show up a lot in our analysis, so we define $A(n) \overset{\text{def}}{=}\left(n - \frac{1}{2}\right)|\s|,\,\,B(n)\overset{\text{def}}{=}(n-1)|\s|$.

For \textcircled{1}, we notice that each edge carries at most the weight $|\s|-1$, while on the other side we have the recursive identities: $A(n + 1)=A(n) + |\s|$, and $B(n+1) = B(n) + |\s|$. It is easy to see then, via an inductive argument for example, that \eqref{elgass1} and \eqref{elgass2} are satisfied for \textcircled{1}.

For \textcircled{2} in the list above, consider the cumulant term in the subtree $\{v,v_1\}$. As per our assumption, it will carry the weight of $3$ (the $\eta$ term is not really important here). Putting it into \eqref{elgass1}, one gets $2 - \kappa \le |\s|$. This gives us a first lower bound on $|\s|$ for the theory to work. Consider now the subtree $\{v,v_\star,v_1\}$, for which \eqref{elgass2} gives us $3 + |\s| - 1 + |\s| - 1 = 2|\s| + 1 \not < 2|\s|$. To remedy this we apply our renormalisation procedure:

\begin{equation}\label{eq94}\begin{tikzpicture}[baseline=0cm,scale=0.8]
\node at (0,-2) [root, label=below:$0$] (a) {};
\node at (0,-1) [dot,label=240:$v_\star$] (b) {};
\node at (-1,0) [dot,label=left:$v$] (c) {};
\node at (1,0) [dot,label=right:$v_1$] (d) {};
\node[cumu2n] (z) at (0,1){};
\draw[cumu2] (z) ellipse (14pt and 10pt);
\node at (-2,1) [dot,label=left:$x_1$] (e) {};
\node at (2,1) [dot,label=right:$v_2$] (h) {};
\node at (-1.5,1.5) [var] (i) {};
\node at (1.5,1.5) [var] (j) {};
\draw[dashed] (i) to (e);
\draw[dashed] (j) to (h);
\draw[testfcn] (a) to (b);
\draw[] (z.west) node[dot1] {} to (c);
\draw[] (z.east) node[dot1] {} to (d); 
\draw[kepsus] (c) to node[labl,pos=0.45] {\tiny $\gamma_2$} (b);
\draw[kepsus] (d) to node[labl,pos=0.45] {\tiny $\gamma_2$} (b);
\draw[kepsus] (e) to node[labl,pos=0.45] {\tiny $\gamma_1$} (c);
\draw[kepsus] (h) to node[labl,pos=0.45] {\tiny $\gamma_1$} (d);
\end{tikzpicture} =
\begin{tikzpicture}[baseline=0cm,scale=0.8]
\node at (0,-2) [root, label=below:$0$] (a) {};
\node at (0,-1) [dot,label=240:$v_\star$] (b) {};
\node at (-1,0) [dot,label=left:$v$] (c) {};
\node at (1,0) [dot,label=right:$v_1$] (d) {};
\node at (-2,1) [dot,label=left:$x_1$] (e) {};
\node at (-1.5,1.5) [var] (k) {};
\node at (1.5,1.5) [var] (l) {};
\node[cumu2n] (f) at (0,1){};
\draw[cumu2] (f) ellipse (14pt and 10pt);
\node at (2,1) [dot,label=right:$v_2$] (h) {};
\draw[testfcn] (a) to (b);
\draw[dashed] (e) to (k);
\draw[dashed] (h) to (l);
\draw[kepsus] (c) to node[labl,pos=0.45] {\tiny $\gamma_2$} (b);
\draw[kepsus] (d) to node[labl,pos=0.45] {\tiny $\gamma_2$} (b);
\draw[kepsus] (e) to node[labl,pos=0.45] {\tiny $\gamma_\star$} (c);
\draw[] (f.west) node[dot1] {} to (c);
\draw[kepsus] (h) to node[labl,pos=0.45] {\tiny $\gamma_1$} (d);
\draw[] (f.east) node[dot1] {} to (d);
\end{tikzpicture} + \sum_{j<2}\begin{tikzpicture}[baseline=0cm,scale=0.8]
\node at (0,-2) [root,label=below:$0$] (a) {};
\node at (0,-1) [dot,label=240:$v_\star$] (b) {};
\node[cumu2n] (f) at (0,0.5) {};
\draw[cumu2] (f) ellipse (14pt and 10pt);
\node at (-1,0) [dot,label=above:$v$] (c) {};
\node at (1,0) [dot,label=right:$v_1$] (d) {};
\node at (-1.8,-0.4) [var,label=$x_1$] (e) {};
\node at (1,1) [var,label=$v_2$] (h) {};
\draw[testfcn] (a) to (b);
\draw[kepsus] (e) to node[labl,pos=0.45] {\tiny $\gamma_j$} (b);
\draw[kepsus] (c) to[out=10,in=90] node[labl,pos=0.45] {\tiny $-j$,0} (b);
\draw[kepsus] (c) to node[labl,pos=0.45] {\tiny $\gamma_2$} (b);
\draw[] (f.west) node[dot1] {} to (c);
\draw[] (f.east) node[dot1] {} to (d);
\draw[kepsus] (d) to node[labl,pos=0.45] {\tiny $\gamma_2$} (b);
\draw[kepsus] (h) to node[labl,pos=0.45] {\tiny $\gamma_1$} (d);
\end{tikzpicture}
\end{equation}

Here $\gamma_\star=(|\s|-2,2,v_\star)$ and $\gamma_j=(|\s|-2+j,1-j,0)$. Our assumption on the cumulants gives us that $\begin{tikzpicture}[baseline=0cm,scale=0.8]
\node at (-1,0) (a) {};
\node at (1,0) (b) {};
\node[cumu2n] (z) at (0,0.5){};
\draw[cumu2] (z) ellipse (14pt and 10pt);
\draw[] (z.west) node[dot1] {} to (a);
\draw[] (z.east) node[dot1] {} to (b);
\end{tikzpicture}$ can be replaced by a kernel of weight $3$, and then it can be seen that all the above graphs satisfy \eqref{elgass1} (where we have assumed the lower bound $2 - \kappa \le |\s|$) and \eqref{elgass3}. 
The benefit of the decomposition can be seen now in the fact that the tree with $\gamma_\star$ satisfies \eqref{elgass2}\, - $3 + |\s| - 1 + |\s| - 1 + 0 - 2 = 2|\s| -  1 < 2|\s|$. For the terms with $\gamma_j$ the problem persists because notice that $3 + (|\s| - 1) + (|\s| - 1) - j = 2|\s| + 1 - j \nless 2|\s|$ for $j\in\{0,1\}$. Fortunately a couple of clever decompositions can fix this as well. For the case $j=1$, we may do:

\begin{equation*}
\begin{tikzpicture}[baseline=0cm,scale=0.8]
\node at (0,-2) [root,label=below:$0$] (a) {};
\node at (0,-1) [dot,label=240:$v_\star$] (b) {};
\node at (-1,0) [dot,label=above:$v$] (c) {};
\node at (1,0) [dot,label=right:$v_1$] (d) {};
\node at (-1.8,-0.4) [var,label=$x_1$] (e) {};
\node at (1,1) [var,label=$v_2$] (h) {};
\node[cumu2n] (z) at (0,0.5){};
\draw[cumu2] (z) ellipse (14pt and 10pt);
\draw[testfcn] (a) to (b);
\draw[kepsus] (e) to node[labl,pos=0.45] {\tiny $\gamma_2$} (b);
\draw[kepsus] (c) to[out=10,in=90] node[labl,pos=0.45] {\tiny -1,0} (b);
\draw[kepsus] (c) to node[labl,pos=0.45] {\tiny $\gamma_2$} (b);
\draw[kepsus] (d) to (b);
\draw[] (z.west) node[dot1] {} to (c);
\draw[] (z.east) node[dot1] {} to (d);
\draw[kepsus] (h) to node[labl,pos=0.45] {\tiny $\gamma_1$} (d);
\end{tikzpicture} = \begin{tikzpicture}[baseline=0cm,scale=0.8]
\node at (0,-2) [root,label=below:$0$] (a) {};
\node at (0,-1) [dot,label=240:$v_\star$] (b) {};
\node at (-1,0) [dot,label=above:$v$] (c) {};
\node at (1,0) [dot,label=right:$v_1$] (d) {};
\node at (-1.8,-0.4) [var,label=$x_1$] (e) {};
\node at (1,1) [var,label=$v_2$] (h) {};
\node[cumu2n] (z) at (0,0.5){};
\draw[cumu2] (z) ellipse (14pt and 10pt);
\draw[testfcn] (a) to (b);
\draw[kepsus] (e) to node[labl,pos=0.45] {\tiny $\gamma_2$} (b);
\draw[kepsus] (c) to[out=10,in=90] node[labl,pos=0.45] {\tiny -1,0} (b);
\draw[kepsus] (c) to node[labl,pos=0.45] {\tiny $\gamma_2$} (b);
\draw[kepsus] (d) to (b);
\draw[kepsus] (h) to node[labl,pos=0.45] {\tiny $\gamma_\star$} (d);
\draw[] (z.west) node[dot1] {} to (c);
\draw[] (z.east) node[dot1] {} to (d);
\end{tikzpicture} + C^1_0 \begin{tikzpicture}[baseline=0cm,scale=0.8]
\node at (0,-2) [root,label=below:$0$] (a) {};
\node at (0,-1) [dot,label=240:$v_\star$] (b) {};
\node at (-1,0.5) [var,label=$x_1$] (e) {};
\node at (1,0.5) [var,label=$v_2$] (h) {};
\draw[testfcn] (a) to (b);
\draw[kepsus] (e) to node[labl,pos=0.45] {\tiny $\gamma_2$} (b);
\draw[kepsus] (h) to node[labl,pos=0.45] {\tiny $\gamma_1$} (b);
\end{tikzpicture}
\end{equation*}

where $\gamma_\star = (1,1,v_\star)$ and $C^1_0=
\begin{tikzpicture}[baseline=0cm]
\node at (0,0) [dot,label=240:$0$] (b) {};
\node at (-1,1) [dot,label=$v$] (e) {};
\node at (1,1) [dot,label=$v_1$] (h) {};
\node[cumu2n] (z) at (0,1.5){};
\draw[cumu2] (z) ellipse (12pt and 8pt);
\draw[kepsus] (h) to node[labl, pos=0.45] {\tiny $\gamma_2$} (b);
\draw[kepsus] (e) to node[labl, pos=0.45] {\tiny $\gamma_2$} (b);
\draw[kepsus] (e) to[out=350,in=90] node[labl, pos=0.45] {\tiny -1,0} (b);
\draw[] (z.west) node[dot1] {} to (e);
\draw[] (z.east) node[dot1] {} to (h);
\end{tikzpicture}$. 

One can now check that for the first graph on the right hand side of the above equation one has:  $3 + (|\s| - 1) + (|\s| - 1) + 0 - 1 - 1 = 2|\s| - 1 < 2|\s|$.

When $j=0$, we may consider:
\begin{equation*}
\begin{tikzpicture}[baseline=-1cm,scale=0.7]
\node at (0,-2) [root,label=below:$0$] (a) {};
\node at (0,-1) [dot,label=240:$v_\star$] (b) {};
\node at (-1,0) [dot,label=above:$v$] (c) {};
\node at (1,0) [dot,label=right:$v_1$] (d) {};
\node at (-1.8,-0.4) [var,label=$x_1$] (e) {};
\node at (1,1.2) [var,label=$v_2$] (h) {};
\node[cumu2n] (z) at (0,0.5){};
\draw[cumu2] (z) ellipse (14pt and 10pt);
\draw[testfcn] (a) to (b);
\draw[kepsus] (e) to node[labl,pos=0.45] {\tiny $\gamma_1$} (b);
\draw[kepsus] (c) to node[labl,pos=0.45] {\tiny $\gamma_2$} (b);
\draw[] (z.west) node[dot1] {} to (c);
\draw[] (z.east) node[dot1] {} to (d);
\draw[kepsus] (d) to node[labl,pos=0.45] {\tiny $\gamma_2$} (b);
\draw[kepsus] (h) to node[labl,pos=0.45] {\tiny $\gamma_1$} (d);
\end{tikzpicture} = \begin{tikzpicture}[baseline=-1cm,scale=0.7]
\node at (0,-2) [root,label=below:$0$] (a) {};
\node at (0,-1) [dot,label=240:$v_\star$] (b) {};
\node at (-1,0) [dot,label=above:$v$] (c) {};
\node at (1,0) [dot,label=right:$v_1$] (d) {};
\node at (-1.8,-0.4) [var,label=left:$x_1$] (e) {};
\node at (1,1.2) [var,label=$v_2$] (h) {};
\node[cumu2n] (z) at (0,0.5){};
\draw[cumu2] (z) ellipse (14pt and 10pt);
\draw[testfcn] (a) to (b);
\draw[kepsus] (e) to node[labl,pos=0.45] {\tiny $\gamma_1$} (b);
\draw[kepsus] (c) to node[labl,pos=0.45] {\tiny $\gamma_2$} (b);
\draw[] (z.west) node[dot1] {} to (c);
\draw[] (z.east) node[dot1] {} to (d);
\draw[kepsus] (d) to node[labl,pos=0.45] {\tiny $\gamma_2$} (b);
\draw[kepsus] (h) to node[labl,pos=0.45] {\tiny $\gamma_\star$} (d);
\end{tikzpicture} +\hspace{2mm}C^2_0\hspace{-5mm}
\begin{tikzpicture}[baseline=-1cm,scale=0.7]
\node at (0,-2) [root,label=below:$0$] (a) {};
\node at (0,-1) [dot,label=240:$v_\star$] (b) {};
\node at (-1,0) [var,label=left:$x_1$] (c) {};
\node at (1,0) [var,label=right:$v_2$] (d) {};
\draw[testfcn] (a) to (b);
\draw[kepsus] (c) to node[labl,pos=0.45] {\tiny $\gamma_2$} (b);
\draw[kepsus] (d) to node[labl,pos=0.45] {\tiny $\gamma_1$} (b);
\end{tikzpicture} + \hspace{2mm}C_1\hspace{-5mm}\begin{tikzpicture}[baseline=-1cm,scale=0.7]
\node at (0,-2) [root,label=below:$0$] (a) {};
\node at (0,-1) [dot,label=240:$v_\star$] (b) {};
\node at (-1,0) [var,label=left:$x_1$] (c) {};
\node at (1,0) [var,label=right:$v_2$] (d) {};
\draw[testfcn] (a) to (b);
\draw[kepsus] (c) to node[labl,pos=0.45] {\tiny $\gamma_1$} (b);
\draw[kepsus] (d) to node[labl,pos=0.45] {\tiny $\gamma_1$} (b);
\end{tikzpicture}
\end{equation*}
where $\gamma_\star=(1,2,v_\star)$, and $C^2_0=\begin{tikzpicture}[baseline=0cm]
\node at (0,0) [dot,label=240:$0$] (b) {};
\node at (-1,1) [dot,label=$v$] (e) {};
\node at (1,1) [dot,label=$v_1$] (h) {};
\node[cumu2n] (z) at (0,1.5){};
\draw[cumu2] (z) ellipse (14pt and 10pt);
\draw[kepsus] (h) to node[labl, pos=0.45] {\tiny $\gamma_2$} (b);
\draw[kepsus] (e) to node[labl, pos=0.45] {\tiny $\gamma_2$} (b);
\draw[kepsus] (h) to[out=190,in=90] node[labl, pos=0.45] {\tiny -1,0} (b);
\draw[] (z.west) node[dot1] {} to (e);
\draw[] (z.east) node[dot1] {} to (h);
\end{tikzpicture}$ and $C_1=\begin{tikzpicture}[baseline=0cm]
\node at (0,0) [dot,label=240:$0$] (b) {};
\node at (-1,1) [dot,label=$v$] (e) {};
\node at (1,1) [dot,label=$v_1$] (h) {};
\node[cumu2n] (z) at (0,1.5){};
\draw[cumu2] (z) ellipse (14pt and 10pt);
\draw[kepsus] (h) to node[labl, pos=0.45] {\tiny $\gamma_2$} (b);
\draw[kepsus] (e) to node[labl, pos=0.45] {\tiny $\gamma_2$} (b);
\draw[] (z.west) node[dot1] {} to (e);
\draw[] (z.east) node[dot1] {} to (h);
\end{tikzpicture}$
Hence the renormalised term is given by:
\begin{equation*}
\begin{tikzpicture}[baseline=-1cm,scale=0.7]
\node at (0,-2) [root, label=below:$0$] (a) {};
\node at (0,-1) [dot,label=240:$v_\star$] (b) {};
\node at (-1,0) [dot,label=left:$v$] (c) {};
\node at (1,0) [dot,label=right:$v_1$] (d) {};
\node at (-1.8,1) [var,label=$x_1$] (e) {};
\node[cumu2n] (z) at (0,1){};
\draw[cumu2] (z) ellipse (14pt and 10pt);
\node at (1.8,1) [var,label=$v_2$] (h) {};
\draw[testfcn] (a) to (b);
\draw[kepsus] (c) to node[labl,pos=0.45] {\tiny $\gamma_2$} (b);
\draw[kepsus] (d) to node[labl,pos=0.45] {\tiny $\gamma_2$} (b);
\draw[kepsus] (e) to node[labl,pos=0.45] {\tiny $\gamma_1$} (c);
\draw[kepsus] (h) to node[labl,pos=0.45] {\tiny $\gamma_1$} (d);
\draw[] (z.west) node[dot1] {} to (c);
\draw[] (z.east) node[dot1] {} to (d);
\end{tikzpicture} - \hspace{2mm}C_1\hspace{-5mm}\begin{tikzpicture}[baseline=-1cm,scale=0.7]
\node at (0,-2) [root,label=below:$0$] (a) {};
\node at (0,-1) [dot,label=240:$v_\star$] (b) {};
\node at (-1,0) [var,label=left:$x_1$] (c) {};
\node at (1,0) [var,label=right:$v_2$] (d) {};
\draw[testfcn] (a) to (b);
\draw[kepsus] (c) to node[labl,pos=0.45] {\tiny $\gamma_1$} (b);
\draw[kepsus] (d) to node[labl,pos=0.45] {\tiny $\gamma_1$} (b);
\end{tikzpicture} -\hspace{2mm}C^2_0\hspace{-5mm}
\begin{tikzpicture}[baseline=-1cm,scale=0.7]
\node at (0,-2) [root,label=below:$0$] (a) {};
\node at (0,-1) [dot,label=240:$v_\star$] (b) {};
\node at (-1,0) [var,label=left:$x_1$] (c) {};
\node at (1,0) [var,label=right:$v_2$] (d) {};
\draw[testfcn] (a) to (b);
\draw[kepsus] (c) to node[labl,pos=0.45] {\tiny $\gamma_2$} (b);
\draw[kepsus] (d) to node[labl,pos=0.45] {\tiny $\gamma_1$} (b);
\end{tikzpicture}
\end{equation*}

For \textcircled{3} \eqref{elgass1} and \eqref{elgass3} are satisfied for the same reason as before and additionally, we notice that the subdivergence caused by \begin{tikzpicture}[scale=0.6]
\node at (-0.5,0) [dot] (a) {};
\node at (0.5,0) [dot] (b) {};
\node[cumu2n] (c) at (0,0.5){};
\draw[cumu2] (c) ellipse (12pt and 8pt);
\draw[] (c.west) node[dot1] {} to (a);
\draw[] (c.east) node[dot1] {} to (b);
\end{tikzpicture} in the second tree, does not cause a problem here.

Consider now \textcircled{4}. The subtree $\{x_1,v\}$ gives us a new construction to work with. As before the cumulant carries the weight $3$. Adding this into the label of the existing edge gives us a new edge with weight $\hat{a}_e = (|\s| - 2) + 3 = |\s| + 1$. One can check that \eqref{elgass1} is satisfied if again we have $2 - \kappa \le |\s|$. The checks for \eqref{elgass2} are similar as before except for $\{x_1,v,v_\star\}$ where we would have $3 + |\s| - 2 + |\s| - 1 = 2|\s| \not < 2|\s|$. When one moves on to $\{x_1, v, v_\star,v_1\}$, the condition is no longer violated because the addition of the interior node adds a weight of $|\s|$ to the RHS of the inequality, while the edge only adds a $|\s| - 1$ to the LHS. Our respite is in the fact that the problematic set is not connected directly to the noises and as such one is able to "excise" it. Consider first the following decomposition, which is due to \eqref{eq:tayexp}:

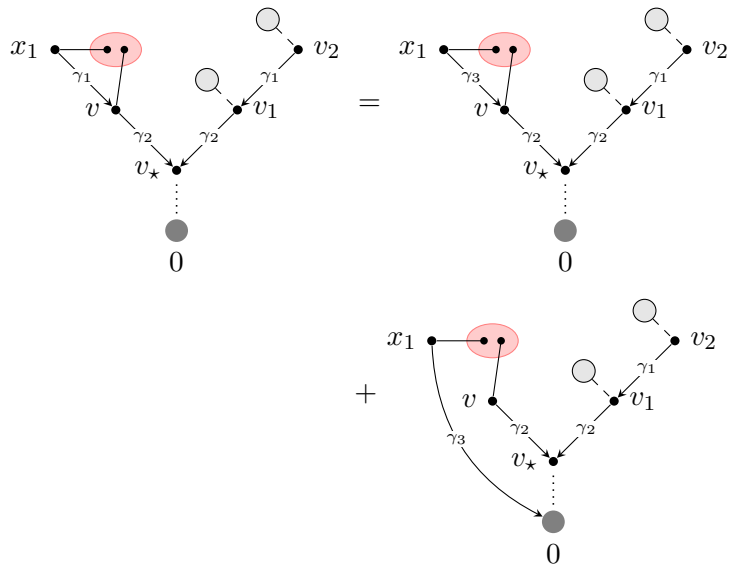
$$
\begin{tikzpicture}[baseline=0cm,scale=0.8]
\node at (0,-2) [root, label=below:$0$] (a) {};
\node at (0,-1) [dot,label=left:$v_\star$] (b) {};
\node at (-1,0) [dot,label=left:$v$] (c) {};
\node at (1,0) [dot,label=right:$v_1$] (d) {};
\node at (-2,1) [dot,label=left:$x_1$] (e) {};
\node at (2,1) [dot,label=right:$v_2$] (f) {};
\node at (0.5,0.5) [var] (g) {};
\node at (1.5,1.5) [var] (k) {};
\node[cumu2n] (z) at (-1,1){};
\draw[cumu2] (z) ellipse (12pt and 8pt);
\draw[] (z.west) node[dot1] {} to (e);
\draw[] (z.east) node[dot1] {} to (c);
\draw[->] (e) to node[labl,pos=0.45] {\tiny $\gamma_1$} (c);
\draw[->] (c) to node[labl,pos=0.45] {\tiny $\gamma_2$} (b);
\draw[->] (d) to node[labl,pos=0.45] {\tiny $\gamma_2$} (b);
\draw[->] (f) to node[labl,pos=0.45] {\tiny $\gamma_1$} (d);
\draw[testfcn] (a) to (b);
\draw[dashed] (g) to (d);
\draw[dashed] (k) to (f);
\end{tikzpicture}=\begin{tikzpicture}[baseline=0cm,scale=0.8]
\node at (0,-2) [root, label=below:$0$] (a) {};
\node at (0,-1) [dot,label=left:$v_\star$] (b) {};
\node at (-1,0) [dot,label=left:$v$] (c) {};
\node at (1,0) [dot,label=right:$v_1$] (d) {};
\node at (-2,1) [dot,label=left:$x_1$] (e) {};
\node at (2,1) [dot,label=right:$v_2$] (f) {};
\node at (0.5,0.5) [var] (g) {};
\node at (1.5,1.5) [var] (k) {};
\node[cumu2n] (z) at (-1,1){};
\draw[cumu2] (z) ellipse (12pt and 8pt);
\draw[] (z.west) node[dot1] {} to (e);
\draw[] (z.east) node[dot1] {} to (c);
\draw[->] (e) to node[labl,pos=0.45] {\tiny $\gamma_3$} (c);
\draw[->] (c) to node[labl,pos=0.45] {\tiny $\gamma_2$} (b);
\draw[->] (d) to node[labl,pos=0.45] {\tiny $\gamma_2$} (b);
\draw[->] (f) to node[labl,pos=0.45] {\tiny $\gamma_1$} (d);
\draw[testfcn] (a) to (b);
\draw[dashed] (g) to (d);
\draw[dashed] (k) to (f);
\end{tikzpicture}$$
$$ \hspace{4.5cm}+\begin{tikzpicture}[baseline=0cm,scale=0.8]
\node at (0,-2) [root, label=below:$0$] (a) {};
\node at (0,-1) [dot,label=left:$v_\star$] (b) {};
\node at (-1,0) [dot,label=left:$v$] (c) {};
\node at (1,0) [dot,label=right:$v_1$] (d) {};
\node at (-2,1) [dot,label=left:$x_1$] (e) {};
\node at (2,1) [dot,label=right:$v_2$] (f) {};
\node at (0.5,0.5) [var] (g) {};
\node at (1.5,1.5) [var] (k) {};
\node[cumu2n] (z) at (-1,1){};
\draw[cumu2] (z) ellipse (12pt and 8pt);
\draw[] (z.west) node[dot1] {} to (e);
\draw[] (z.east) node[dot1] {} to (c);
\draw[->] (e) to[bend right=30] node[labl,pos=0.45] {\tiny $\gamma_3$} (a);
\draw[->] (c) to node[labl,pos=0.45] {\tiny $\gamma_2$} (b);
\draw[->] (d) to node[labl,pos=0.45] {\tiny $\gamma_2$} (b);
\draw[->] (f) to node[labl,pos=0.45] {\tiny $\gamma_1$} (d);
\draw[testfcn] (a) to (b);
\draw[dashed] (g) to (d);
\draw[dashed] (k) to (f);
\end{tikzpicture}
$$

Notice that the disappearance of the edge $\{x_1,\,v\}$ from the second graph cures the divergence. The first graph on the left is rewritten via translation invariance.

$$
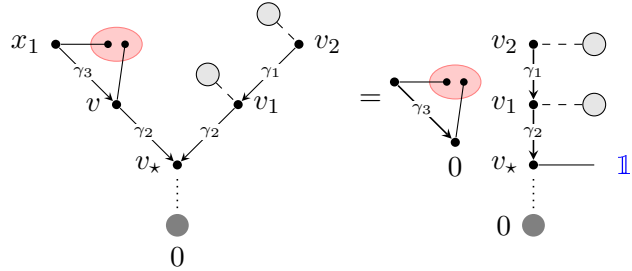
\begin{tikzpicture}[baseline=0cm,scale=0.8]
\node at (0,-2) [root, label=below:$0$] (a) {};
\node at (0,-1) [dot,label=left:$v_\star$] (b) {};
\node at (-1,0) [dot,label=left:$v$] (c) {};
\node at (1,0) [dot,label=right:$v_1$] (d) {};
\node at (-2,1) [dot,label=left:$x_1$] (e) {};
\node at (2,1) [dot,label=right:$v_2$] (f) {};
\node at (0.5,0.5) [var] (g) {};
\node at (1.5,1.5) [var] (k) {};
\node[cumu2n] (z) at (-1,1){};
\draw[cumu2] (z) ellipse (12pt and 8pt);
\draw[] (z.west) node[dot1] {} to (e);
\draw[] (z.east) node[dot1] {} to (c);
\draw[->] (e) to node[labl,pos=0.45] {\tiny $\gamma_3$} (c);
\draw[->] (c) to node[labl,pos=0.45] {\tiny $\gamma_2$} (b);
\draw[->] (d) to node[labl,pos=0.45] {\tiny $\gamma_2$} (b);
\draw[->] (f) to node[labl,pos=0.45] {\tiny $\gamma_1$} (d);
\draw[testfcn] (a) to (b);
\draw[dashed] (g) to (d);
\draw[dashed] (k) to (f);
\end{tikzpicture} = \begin{tikzpicture}[baseline=0.5cm,scale=0.8]
\node at (-1,0) [dot,label=below:$0$] (a) {};
\node at (-2,1) [dot] (b) {};
\node[cumu2n] (z) at (-1,1){};
\draw[cumu2] (z) ellipse (12pt and 8pt);
\draw[] (z.west) node[dot1] {} to (b);
\draw[] (z.east) node[dot1] {} to (a);
\draw[kepsus] (b) to node[labl,pos=0.45] {\tiny $\gamma_3$} (a);
\end{tikzpicture}\begin{tikzpicture}[baseline=0cm,scale=0.8]
\node at (0,-2) [root,label=left:$0$] (a) {};
\node at (0,-1) [dot,label=left:$v_\star$] (b) {};
\node at (0,0) [dot,label=left:$v_1$] (c) {};
\node at (0,1) [dot,label=left:$v_2$] (d) {};
\node at (1,0) [var] (e) {};
\node at (1,1) [var] (f) {};
\node at (1,-1) [label=right:$\colb{\mathbbm{1}}$] {};
\draw[] (1,-1) to (b);
\draw[testfcn] (a) to (b);
\draw[kepsus] (c) to node[labl,pos=0.45] {\tiny $\gamma_2$} (b);
\draw[kepsus] (d) to node[labl,pos=0.45] {\tiny $\gamma_1$} (c);
\draw[dashed] (e) to (c);
\draw[dashed] (f) to (d);
\end{tikzpicture}
$$

As before, notice that due to the fact that $K^N$ kills constants the second tree above on the right side is indeed zero, and so is the one on the left.

In \textcircled{6} one can check that the subset $\{x_1,\,v,\,v_\star,\,v_1\}$ fails $\eqref{eq:ass2e1}$. Indeed one can calculate that $|\s| - 2 + |\s| - 1 + |\s| - 1 + 3/2(3) + 0 - 0 = 3|\s| + 1/2 \nless 3|\s|$. To remedy this, consider first the decomposition:

$$
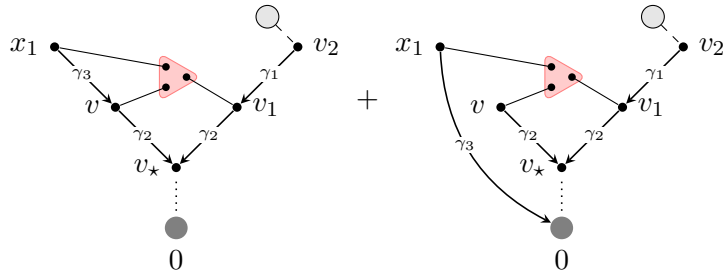
\begin{tikzpicture}[baseline=0cm,scale=0.8]
\node at (0,-2) [root, label=below:$0$] (a) {};
\node at (0,-1) [dot,label=left:$v_\star$] (b) {};
\node at (-1,0) [dot,label=left:$v$] (c) {};
\node at (-2,1) [dot,label=left:$x_1$] (y) {};
\node at (1,0) [dot,label=right:$v_1$] (d) {};
\node at (2,1) [dot,label=right:$v_2$] (h) {};
\node at (1.5,1.5) [var] (k) {};
\node[] (e) at (0,0.5) {};
\node[cumu3,rotate=270] (e-) at (-0.05,0.5) {};
\draw[testfcn] (a) to (b);
\draw[kepsus] (c) to node[labl,pos=0.45] {\tiny $\gamma_2$} (b);
\draw[kepsus] (d) to node[labl,pos=0.45] {\tiny $\gamma_2$} (b);
\draw[kepsus] (h) to node[labl,pos=0.45] {\tiny $\gamma_1$} (d);
\draw[kepsus] (y) to node[labl,pos=0.45] {\tiny $\gamma_3$} (c);
\draw[] (e.north west) node[dot1] {} to (y);
\draw[] (e.south west) node[dot1] {} to (c);
\draw[] (e.east) node[dot1] {} to (d);
\draw[dashed] (k) to (h);
\end{tikzpicture}+\begin{tikzpicture}[baseline=0cm,scale=0.8]
\node at (0,-2) [root, label=below:$0$] (a) {};
\node at (0,-1) [dot,label=left:$v_\star$] (b) {};
\node at (-1,0) [dot,label=left:$v$] (c) {};
\node at (-2,1) [dot,label=left:$x_1$] (y) {};
\node at (1,0) [dot,label=right:$v_1$] (d) {};
\node at (2,1) [dot,label=right:$v_2$] (h) {};
\node at (1.5,1.5) [var] (k) {};
\node[] (e) at (0,0.5) {};
\node[cumu3,rotate=270] (e-) at (-0.05,0.5) {};
\draw[testfcn] (a) to (b);
\draw[kepsus] (c) to node[labl,pos=0.45] {\tiny $\gamma_2$} (b);
\draw[kepsus] (d) to node[labl,pos=0.45] {\tiny $\gamma_2$} (b);
\draw[kepsus] (h) to node[labl,pos=0.45] {\tiny $\gamma_1$} (d);
\draw[kepsus] (y) to[bend right=30] node[labl,pos=0.45] {\tiny $\gamma_3$} (a);
\draw[] (e.north west) node[dot1] {} to (y);
\draw[] (e.south west) node[dot1] {} to (c);
\draw[] (e.east) node[dot1] {} to (d);
\draw[dashed] (k) to (h);
\end{tikzpicture}$$

Notice that for the graph on the right the divergence has been cured because the edge $(x_1,\,v)$ (and the weight that it carries) is removed from the calculation. To deal with the left tree we begin by redirecting the edge $(v_2,\,v_1)$ to $v_\star$, which by our renormalisation procedure gives rise to the graphs:
$$
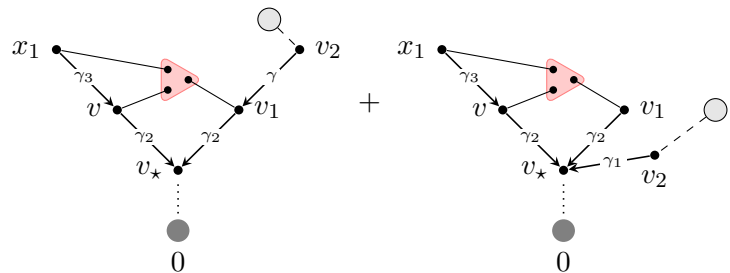
\begin{tikzpicture}[baseline=0cm,scale=0.8]
\node at (0,-2) [root, label=below:$0$] (a) {};
\node at (0,-1) [dot,label=left:$v_\star$] (b) {};
\node at (-1,0) [dot,label=left:$v$] (c) {};
\node at (-2,1) [dot,label=left:$x_1$] (y) {};
\node at (1,0) [dot,label=right:$v_1$] (d) {};
\node at (2,1) [dot,label=right:$v_2$] (h) {};
\node at (1.5,1.5) [var] (k) {};
\node[] (e) at (0,0.5) {};
\node[cumu3,rotate=270] (e-) at (-0.05,0.5) {};
\draw[testfcn] (a) to (b);
\draw[kepsus] (c) to node[labl,pos=0.45] {\tiny $\gamma_2$} (b);
\draw[kepsus] (d) to node[labl,pos=0.45] {\tiny $\gamma_2$} (b);
\draw[kepsus] (h) to node[labl,pos=0.45] {\tiny $\gamma$} (d);
\draw[kepsus] (y) to node[labl,pos=0.45] {\tiny $\gamma_3$} (c);
\draw[] (e.north west) node[dot1] {} to (y);
\draw[] (e.south west) node[dot1] {} to (c);
\draw[] (e.east) node[dot1] {} to (d);
\draw[dashed] (k) to (h);
\end{tikzpicture}+\begin{tikzpicture}[baseline=0cm,scale=0.8]
\node at (0,-2) [root, label=below:$0$] (a) {};
\node at (0,-1) [dot,label=left:$v_\star$] (b) {};
\node at (-1,0) [dot,label=left:$v$] (c) {};
\node at (-2,1) [dot,label=left:$x_1$] (y) {};
\node at (1,0) [dot,label=right:$v_1$] (d) {};
\node at (1.5,-0.75) [dot,label=below:$v_2$] (h) {};
\node at (2.5,0) [var] (k) {};
\node[] (e) at (0,0.5) {};
\node[cumu3,rotate=270] (e-) at (-0.05,0.5) {};
\draw[testfcn] (a) to (b);
\draw[kepsus] (c) to node[labl,pos=0.45] {\tiny $\gamma_2$} (b);
\draw[kepsus] (d) to node[labl,pos=0.45] {\tiny $\gamma_2$} (b);
\draw[kepsus] (h) to node[labl,pos=0.45] {\tiny $\gamma_1$} (b);
\draw[kepsus] (y) to node[labl,pos=0.45] {\tiny $\gamma_3$} (c);
\draw[] (e.north west) node[dot1] {} to (y);
\draw[] (e.south west) node[dot1] {} to (c);
\draw[] (e.east) node[dot1] {} to (d);
\draw[dashed] (k) to (h);
\end{tikzpicture}$$

where $\gamma = (|\s| - 2,\,1,\,v_\star)$. This means that the condition now becomes $|\s| - 2 + |\s| - 1 + |\s| - 1 + 3/2(3) - 1 = 3|\s| - 1/2 < 3|\s|$.
One defines $C_1\coloneqq
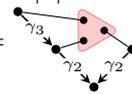
\begin{tikzpicture}[baseline=0cm,scale=0.5]
\node at (0,-1) [dot] (b) {};
\node at (-1,0) [dot] (c) {};
\node at (-2,1) [dot] (y) {};
\node at (1,0) [dot] (d) {};
\node[] (e) at (0,0.5) {};
\node[cumu3,rotate=270] (e-) at (-0.05,0.5) {};
\draw[kepsus] (c) to node[labl,pos=0.45] {\tiny $\gamma_2$} (b);
\draw[kepsus] (d) to node[labl,pos=0.45] {\tiny $\gamma_2$} (b);
\draw[kepsus] (y) to node[labl,pos=0.45] {\tiny $\gamma_3$} (c);
\draw[] (e.north west) node[dot1] {} to (y);
\draw[] (e.south west) node[dot1] {} to (c);
\draw[] (e.east) node[dot1] {} to (d);
\end{tikzpicture}$ and then renormalisation is a matter of removing a factor of
$C_1
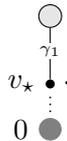
\begin{tikzpicture}[baseline=0cm,scale=0.6]
\node at (0,-1) [root,label=left:$0$] (a) {};
\node at (0,0) [dot,label=left:$v_\star$] (b) {};
\node at (0,1.5) [var] (c) {};
\draw[testfcn] (a) to (b);
\draw[] (c) to node[labl,pos=0.45] {\tiny $\gamma_1$} (b);
\end{tikzpicture}$.

Another example is furnished by the tree 
$\begin{tikzpicture}[baseline=0cm,scale=0.2]
\node at (0,0) [bar] (a) {};
\node at (1,1) [bar] (b) {};
\node at (-2,2) [bar] (d) {};
\node at (0,2) [bar] (e) {};
\draw (a) -- (b);
\draw (a) -- (-1,1);
\draw[thick] (-1,1) to (d);
\draw[thick] (-1,1) to (e); 
\end{tikzpicture}$. The renormalisation map for the homogeneous Gaussian case gives us the following:  

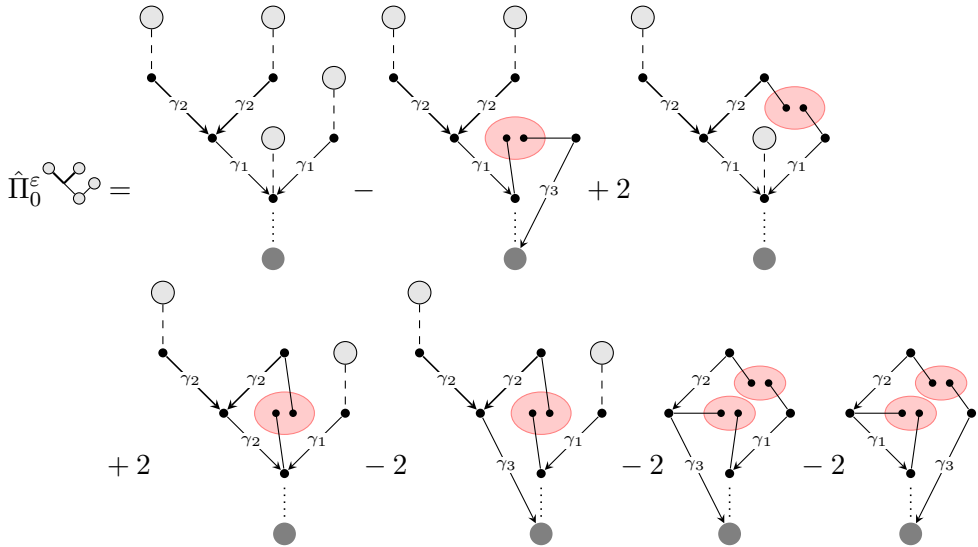
\begin{equation*}
\begin{split}
\hat{\Pi}_{0}^{\varepsilon}\begin{tikzpicture}[baseline=0cm,scale=0.2]
\node at (0,0) [bar] (a) {};
\node at (1,1) [bar] (b) {};
\node at (-2,2) [bar] (d) {};
\node at (0,2) [bar] (e) {};
\draw (a) -- (b);
\draw (a) -- (-1,1);
\draw[thick] (-1,1) -- (d);
\draw[thick] (-1,1) -- (e); 
\end{tikzpicture}&=
\begin{tikzpicture}[baseline=0cm,scale=0.8]
\node at (0,-1) [root] (a) {}; 
\node at (0,0) [dot] (b) {};
\node at (0,1) [var] (c) {};
\node at (1,1) [dot] (d) {};
\node at (1,2) [var] (e) {};
\node at (-1,1) [dot] (f) {};
\node at (0,2) [dot] (g) {};
\node at (-2,2) [dot] (h) {};
\node at (-2,3) [var] (i) {};
\node at (0,3) [var] (j) {};
\draw[testfcn] (a) to (b);
\draw[dashed] (b) -- (c);
\draw[->] (d) to node[labl,pos=0.45] {\tiny $\gamma_1$} (b);
\draw[dashed] (e) to (d);
\draw[kepsus] (g) to node[labl,pos=0.45] {\tiny $\gamma_2$} (f);
\draw[kepsus] (h) to node[labl,pos=0.45] {\tiny $\gamma_2$} (f);
\draw[dashed] (i) to (h);
\draw[->] (f) to node[labl,pos=0.45] {\tiny $\gamma_1$} (b);
\draw[dashed] (j) to (g);
\end{tikzpicture} - \begin{tikzpicture}[baseline=0cm,scale=0.8]
\node at (0,-1) [root] (a) {}; 
\node at (0,0) [dot] (b) {};
\node at (1,1) [dot] (d) {};
\node at (-1,1) [dot] (f) {};
\node at (0,2) [dot] (g) {};
\node at (-2,2) [dot] (h) {};
\node at (-2,3) [var] (i) {};
\node at (0,3) [var] (j) {};
\node[cumu2n] (z) at (0,1){};
\draw[cumu2] (z) ellipse (14pt and 10pt);
\draw[testfcn] (a) to (b);
\draw[->] (d) to node[labl,pos=0.45] {\tiny $\gamma_3$} (a);
\draw[kepsus] (g) to node[labl,pos=0.45] {\tiny $\gamma_2$} (f);
\draw[kepsus] (h) to node[labl,pos=0.45] {\tiny $\gamma_2$} (f);
\draw[dashed] (i) to (h);
\draw[->] (f) to node[labl,pos=0.45] {\tiny $\gamma_1$} (b);
\draw[dashed] (j) to (g);
\draw[] (z.west) node[dot1] {} to (b);
\draw[] (z.east) node[dot1] {} to (d);
\end{tikzpicture} + 2\begin{tikzpicture}[baseline=0cm,scale=0.8]
\node at (0,-1) [root] (a) {}; 
\node at (0,0) [dot] (b) {};
\node at (0,1) [var] (c) {};
\node at (1,1) [dot] (d) {};
\node at (-1,1) [dot] (f) {};
\node at (0,2) [dot] (g) {};
\node at (-2,2) [dot] (h) {};
\node at (-2,3) [var] (i) {};
\node[cumu2n] (z) at (0.5,1.5){};
\draw[cumu2] (z) ellipse (14pt and 10pt);
\draw[testfcn] (a) to (b);
\draw[dashed] (b) -- (c);
\draw[->] (d) to node[labl,pos=0.45] {\tiny $\gamma_1$} (b);
\draw[kepsus] (g) to node[labl,pos=0.45] {\tiny $\gamma_2$} (f);
\draw[kepsus] (h) to node[labl,pos=0.45] {\tiny $\gamma_2$} (f);
\draw[dashed] (i) to (h);
\draw[->] (f) to node[labl,pos=0.45] {\tiny $\gamma_1$} (b);
\draw[] (z.west) node[dot1] {} to (g);
\draw[] (z.east) node[dot1] {} to (d);
\end{tikzpicture}
\\
&+2\begin{tikzpicture}[baseline=0cm,scale=0.8]
\node at (0,-1) [root] (a) {}; 
\node at (0,0) [dot] (b) {};
\node at (1,1) [dot] (d) {};
\node at (1,2) [var] (e) {};
\node at (-1,1) [dot] (f) {};
\node at (0,2) [dot] (g) {};
\node at (-2,2) [dot] (h) {};
\node at (-2,3) [var] (i) {};
\node[cumu2n] (z) at (0,1){};
\draw[cumu2] (z) ellipse (14pt and 10pt);
\draw[testfcn] (a) to (b);
\draw[->] (d) to node[labl,pos=0.45] {\tiny $\gamma_1$} (b);
\draw[dashed] (e) to (d);
\draw[kepsus] (g) to node[labl,pos=0.45] {\tiny $\gamma_2$} (f);
\draw[kepsus] (h) to node[labl,pos=0.45] {\tiny $\gamma_2$} (f);
\draw[dashed] (i) to (h);
\draw[->] (f) to node[labl,pos=0.45] {\tiny $\gamma_2$} (b);
\draw[] (z.west) node[dot1] {} to (b);
\draw[] (z.east) node[dot1] {} to (g);
\end{tikzpicture} - 2\begin{tikzpicture}[baseline=0cm,scale=0.8]
\node at (0,-1) [root] (a) {}; 
\node at (0,0) [dot] (b) {};
\node at (1,1) [dot] (d) {};
\node at (1,2) [var] (e) {};
\node at (-1,1) [dot] (f) {};
\node at (0,2) [dot] (g) {};
\node at (-2,2) [dot] (h) {};
\node at (-2,3) [var] (i) {};
\node[cumu2n] (z) at (0,1){};
\draw[cumu2] (z) ellipse (14pt and 10pt);
\draw[testfcn] (a) to (b);
\draw[->] (d) to node[labl,pos=0.45] {\tiny $\gamma_1$} (b);
\draw[dashed] (e) to (d);
\draw[kepsus] (g) to node[labl,pos=0.45] {\tiny $\gamma_2$} (f);
\draw[kepsus] (h) to node[labl,pos=0.45] {\tiny $\gamma_2$} (f);
\draw[dashed] (i) to (h);
\draw[->] (f) to node[labl,pos=0.45] {\tiny $\gamma_3$} (a);
\draw[] (z.west) node[dot1] {} to (b);
\draw[] (z.east) node[dot1] {} to (g);
\end{tikzpicture} - 2\begin{tikzpicture}[baseline=0cm,scale=0.8]
\node at (0,-1) [root] (a) {}; 
\node at (0,0) [dot] (b) {};
\node at (1,1) [dot] (d) {};
\node at (-1,1) [dot] (f) {};
\node at (0,2) [dot] (g) {};
\node[cumu2n] (z) at (0.5,1.5){};
\draw[cumu2] (z) ellipse (12pt and 8pt);
\node[cumu2n] (y) at (0,1){};
\draw[cumu2] (y) ellipse (12pt and 8pt);
\draw[testfcn] (a) to (b);
\draw[->] (d) to node[labl,pos=0.45] {\tiny $\gamma_1$} (b);
\draw[->] (g) to node[labl,pos=0.45] {\tiny $\gamma_2$} (f);
\draw[->] (f) to node[labl,pos=0.45] {\tiny $\gamma_3$} (a);
\draw[] (z.west) node[dot1] {} to (g);
\draw[] (z.east) node[dot1] {} to (d);
\draw[] (y.east) node[dot1] {} to (b);
\draw[] (y.west) node[dot1] {} to (f);
\end{tikzpicture} - 2\begin{tikzpicture}[baseline=0cm,scale=0.8]
\node at (0,-1) [root] (a) {}; 
\node at (0,0) [dot] (b) {};
\node at (1,1) [dot] (d) {};
\node at (-1,1) [dot] (f) {};
\node at (0,2) [dot] (g) {};
\node[cumu2n] (z) at (0.5,1.5){};
\draw[cumu2] (z) ellipse (12pt and 8pt);
\node[cumu2n] (y) at (0,1){};
\draw[cumu2] (y) ellipse (12pt and 8pt);
\draw[testfcn] (a) to (b);
\draw[->] (d) to node[labl,pos=0.45] {\tiny $\gamma_3$} (a);
\draw[->] (g) to node[labl,pos=0.45] {\tiny $\gamma_2$} (f);
\draw[->] (f) to node[labl,pos=0.45] {\tiny $\gamma_1$} (b);
\draw[] (z.west) node[dot1] {} to (g);
\draw[] (z.east) node[dot1] {} to (d);
\draw[] (y.east) node[dot1] {} to (b);
\draw[] (y.west) node[dot1] {} to (f);
\end{tikzpicture}
\end{split}
\end{equation*}

Then, our usual assumption on the second cumulant, allows us to check Proposition~\ref{prop:pcondition}. In our case however we have non-vanishing higher order cumulants, so we will see the graphs like:

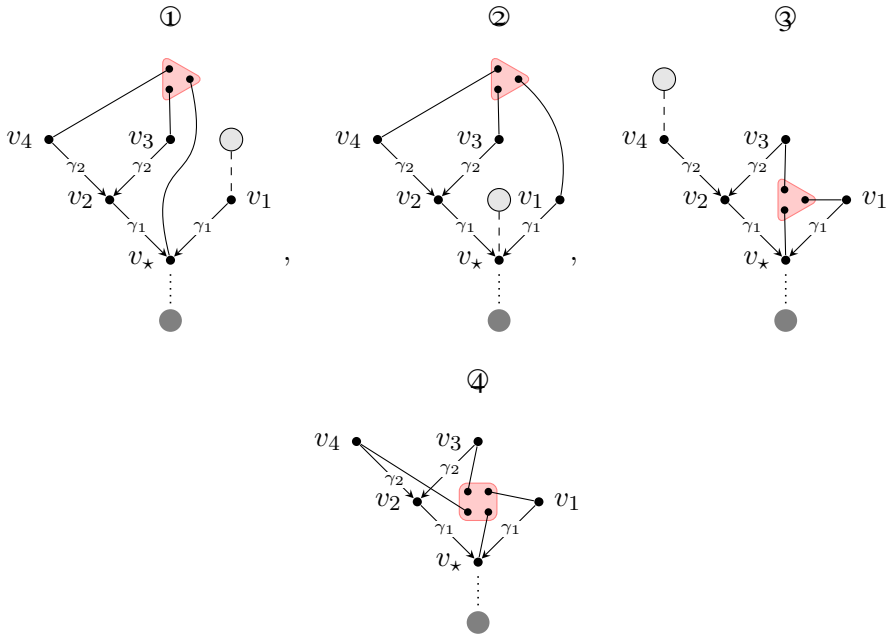
\be
\begin{tikzpicture}[baseline=0cm,scale=0.8]
\node at (0,4) {\textcircled{1}};
\node at (0,-1) [root] (a) {}; 
\node at (0,0) [dot,label=left:$v_\star$] (b) {};
\node at (1,1) [dot,label=right:$v_1$] (d) {};
\node at (-1,1) [dot,label=left:$v_2$] (f) {};
\node at (1,2) [var] (i) {};
\node at (0,2) [dot,label=left:$v_3$] (g) {};
\node at (-2,2) [dot,label=left:$v_4$] (h) {};
\node[] (e) at (0.15,3) {};
\node[cumu3,rotate=270] (e-) at (0.1,3) {};
\draw[testfcn] (a) to (b);
\draw[->] (d) to node[labl,pos=0.45] {\tiny $\gamma_1$} (b);
\draw[->] (g) to node[labl,pos=0.45] {\tiny $\gamma_2$} (f);
\draw[->] (h) to node[labl,pos=0.45] {\tiny $\gamma_2$} (f);
\draw[] (e.north west) node[dot1] {} to (h);
\draw[] (e.south west) node[dot1] {} to (g);
\draw (e.east) node[dot1] {} .. controls (0.8,1) and (-0.5,2) .. (b);
\draw[->] (f) to node[labl,pos=0.45] {\tiny $\gamma_1$} (b);
\draw[dashed] (i) to (d);
\end{tikzpicture},\quad
\begin{tikzpicture}[baseline=0cm,scale=0.8]
\node at (0,4) {\textcircled{2}};
\node at (0,-1) [root] (a) {}; 
\node at (0,0) [dot,label=left:$v_\star$] (b) {};
\node at (1,1) [dot,label=left:$v_1$] (d) {};
\node at (-1,1) [dot,label=left:$v_2$] (f) {};
\node at (0,1) [var] (i) {};
\node at (0,2) [dot,label=left:$v_3$] (g) {};
\node at (-2,2) [dot,label=left:$v_4$] (h) {};
\node[] (e) at (0.15,3) {};
\node[cumu3,rotate=270] (e-) at (0.1,3) {};
\draw[testfcn] (a) to (b);
\draw[->] (d) to node[labl,pos=0.45] {\tiny $\gamma_1$} (b);
\draw[->] (g) to node[labl,pos=0.45] {\tiny $\gamma_2$} (f);
\draw[->] (h) to node[labl,pos=0.45] {\tiny $\gamma_2$} (f);
\draw[] (e.north west) node[dot1] {} to (h);
\draw[] (e.south west) node[dot1] {} to (g);
\draw[] (e.east) node[dot1] {} to[bend left=30] (d);
\draw[->] (f) to node[labl,pos=0.45] {\tiny $\gamma_1$} (b);
\draw[dashed] (i) to (b);
\end{tikzpicture},\quad
\begin{tikzpicture}[baseline=0cm,scale=0.8]
\node at (0,4) {\textcircled{3}};
\node at (0,-1) [root] (a) {}; 
\node at (0,0) [dot,label=left:$v_\star$] (b) {};
\node at (1,1) [dot,label=right:$v_1$] (d) {};
\node at (-1,1) [dot,label=left:$v_2$] (f) {};
\node at (0,2) [dot,label=left:$v_3$] (g) {};
\node at (-2,2) [dot,label=left:$v_4$] (h) {};
\node at (-2,3) [var] (i) {};
\node[] (e) at (0.15,1) {};
\node[cumu3,rotate=270] (e-) at (0.1,1) {};
\draw[testfcn] (a) to (b);
\draw[->] (d) to node[labl,pos=0.45] {\tiny $\gamma_1$} (b);
\draw[->] (g) to node[labl,pos=0.45] {\tiny $\gamma_2$} (f);
\draw[->] (h) to node[labl,pos=0.45] {\tiny $\gamma_2$} (f);
\draw[] (e.north west) node[dot1] {} to (g);
\draw[] (e.south west) node[dot1] {} to (b);
\draw[] (e.east) node[dot1] {} to (d);
\draw[dashed] (i) to (h);
\draw[->] (f) to node[labl,pos=0.45] {\tiny $\gamma_1$} (b);
\end{tikzpicture}
\ee
\be
\begin{tikzpicture}[baseline=0cm,scale=0.8]
\node at (0,-1) [root] (a) {};
\node at (0,3) {\textcircled{4}};
\node at (0,0) [dot,label=left:$v_\star$] (b) {};
\node at (1,1) [dot,label=right:$v_1$] (d) {};
\node at (-1,1) [dot,label=left:$v_2$] (f) {};
\node at (0,2) [dot,label=left:$v_3$] (g) {};
\node at (-2,2) [dot,label=left:$v_4$] (h) {};
\node[] (e) at (0,1) {};
\node[cumu4] (e-) at (e) {};
\draw[testfcn] (a) to (b);
\draw[->] (d) to node[labl,pos=0.45] {\tiny $\gamma_1$} (b);
\draw[->] (g) to node[labl,pos=0.45] {\tiny $\gamma_2$} (f);
\draw[->] (h) to node[labl,pos=0.65] {\tiny $\gamma_2$} (f);
\draw[] (e.north west) node[dot1] {} to (g);
\draw[] (e.south west) node[dot1] {} to (h);
\draw[] (e.south east) node[dot1] {} to (b);
\draw[] (e.north east) node[dot1] {} to (d);
\draw[->] (f) to node[labl,pos=0.45] {\tiny $\gamma_1$} (b);
\end{tikzpicture}
\ee

In the above list, the tree \textcircled{3} can be seen to be unproblematic. In \textcircled{1} and \textcircled{2}, we see the subset $\bar{\dsV}=\{v_4,\,v_2,\,v_3\}$ we have needed to treat before but here that cumulant is of order 3, and so the extra weight of $3/2+$ associated to the cumulant terms attached to $v_4\text{ and }v_3$ will not be included in $\dsE^{\downarrow}\left(\bar{\dsV}\right)$. For this reason these trees do not cause any problems.

The next tree we can look at comes from \begin{tikzpicture}[baseline=0.25cm,scale=0.4]
\node at (0.5,0.5) [bar] (b) {};
\node at (-0.5,0.5) [bar] (c) {};
\node at (0,1) [bar] (d) {};
\node at (-0.5,1.5) [bar] (e) {};
\draw[thick] (0,0) to (b);
\draw[thick] (0,0) to (c);
\draw[] (c) to (d);
\draw[] (d) to (e);
\end{tikzpicture}. The tree in consideration is then:

$$
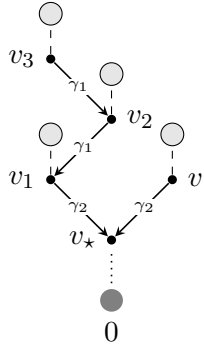
\begin{tikzpicture}[baseline=0cm,scale=0.8]
\node at (0,-2) [root,label=below:$0$] (a) {};
\node at (0,-1) [dot,label=left:$v_\star$] (b) {};
\node at (1,0) [dot,label=right:$v$] (c) {};
\node at (-1,0) [dot,label=left:$v_1$] (d) {};
\node at (0,1) [dot,label=right:$v_2$] (e) {};
\node at (-1,2) [dot,label=left:$v_3$] (f) {};
\node at (1,0.75) [var] (g) {};
\node at (-1,0.75) [var] (h) {};
\node at (0,1.75) [var] (i) {};
\node at (-1,2.75) [var] (j) {};
\draw[kepsus] (c) to node[labl,pos=0.45] {\tiny $\gamma_2$} (b);
\draw[kepsus] (d) to node[labl,pos=0.45] {\tiny $\gamma_2$} (b);
\draw[kepsus] (e) to node[labl,pos=0.45] {\tiny $\gamma_1$} (d);
\draw[kepsus] (f) to node[labl,pos=0.45] {\tiny $\gamma_1$} (e);
\draw[dashed] (c) to (g);
\draw[dashed] (d) to (h);
\draw[dashed] (e) to (i);
\draw[dashed] (f) to (j);
\draw[testfcn] (a) to (b);
\end{tikzpicture}
$$

It is an easy exercise in computations to check that this graph satisfies all the required condition.

Some of the second order contractions we expect to see are:

$$
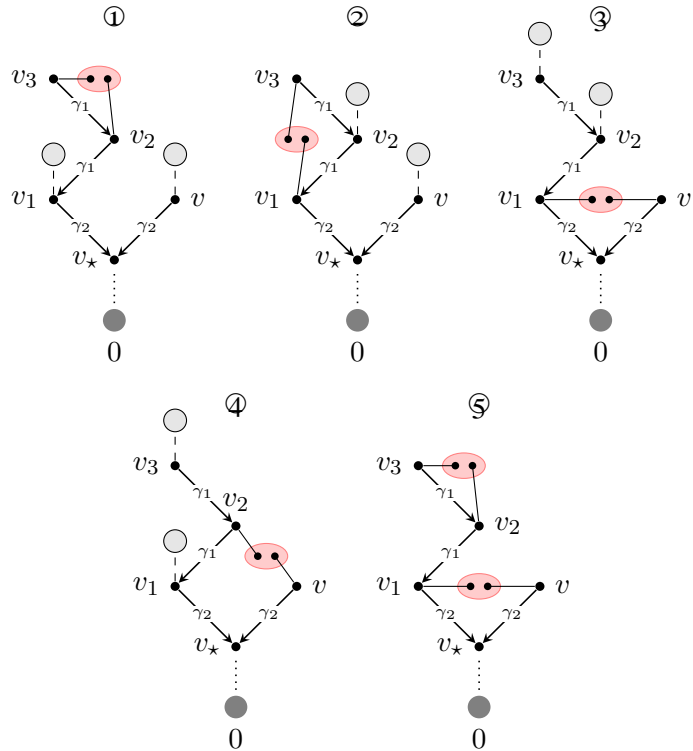
\begin{tikzpicture}[baseline=0cm,scale=0.8]
\node at (0,3) {\textcircled{1}};
\node at (0,-2) [root,label=below:$0$] (a) {};
\node at (0,-1) [dot,label=left:$v_\star$] (b) {};
\node at (1,0) [dot,label=right:$v$] (c) {};
\node at (-1,0) [dot,label=left:$v_1$] (d) {};
\node at (0,1) [dot,label=right:$v_2$] (e) {};
\node at (-1,2) [dot,label=left:$v_3$] (f) {};
\node at (1,0.75) [var] (g) {};
\node at (-1,0.75) [var] (h) {};
\draw[kepsus] (c) to node[labl,pos=0.45] {\tiny $\gamma_2$} (b);
\draw[kepsus] (d) to node[labl,pos=0.45] {\tiny $\gamma_2$} (b);
\draw[kepsus] (e) to node[labl,pos=0.45] {\tiny $\gamma_1$} (d);
\draw[kepsus] (f) to node[labl,pos=0.45] {\tiny $\gamma_1$} (e);
\draw[dashed] (c) to (g);
\draw[dashed] (d) to (h);
\draw[testfcn] (a) to (b);
\node[cumu2n] (z) at (-0.25,2){};
\draw[cumu2] (z) ellipse (10pt and 6pt);
\draw[] (z.west) node[dot1] {} to (f);
\draw[] (z.east) node[dot1] {} to (e);
\end{tikzpicture}\quad\begin{tikzpicture}[baseline=0cm,scale=0.8]
\node at (0,3) {\textcircled{2}};
\node at (0,-2) [root,label=below:$0$] (a) {};
\node at (0,-1) [dot,label=left:$v_\star$] (b) {};
\node at (1,0) [dot,label=right:$v$] (c) {};
\node at (-1,0) [dot,label=left:$v_1$] (d) {};
\node at (0,1) [dot,label=right:$v_2$] (e) {};
\node at (-1,2) [dot,label=left:$v_3$] (f) {};
\node at (1,0.75) [var] (g) {};
\node at (0,1.75) [var] (i) {};
\node[cumu2n] (z) at (-1,1){};
\draw[cumu2] (z) ellipse (10pt and 6pt);
\draw[kepsus] (c) to node[labl,pos=0.45] {\tiny $\gamma_2$} (b);
\draw[kepsus] (d) to node[labl,pos=0.45] {\tiny $\gamma_2$} (b);
\draw[kepsus] (e) to node[labl,pos=0.45] {\tiny $\gamma_1$} (d);
\draw[kepsus] (f) to node[labl,pos=0.45] {\tiny $\gamma_1$} (e);
\draw[dashed] (c) to (g);
\draw[dashed] (e) to (i);
\draw[] (z.west) node[dot1] {} to (f);
\draw[] (z.east) node[dot1] {} to (d);
\draw[testfcn] (a) to (b);
\end{tikzpicture}\quad\begin{tikzpicture}[baseline=0cm,scale=0.8]
\node at (0,3) {\textcircled{3}};
\node at (0,-2) [root,label=below:$0$] (a) {};
\node at (0,-1) [dot,label=left:$v_\star$] (b) {};
\node at (1,0) [dot,label=right:$v$] (c) {};
\node at (-1,0) [dot,label=left:$v_1$] (d) {};
\node at (0,1) [dot,label=right:$v_2$] (e) {};
\node at (-1,2) [dot,label=left:$v_3$] (f) {};
\node at (0,1.75) [var] (i) {};
\node at (-1,2.75) [var] (j) {};
\node[cumu2n] (z) at (0,0){};
\draw[cumu2] (z) ellipse (10pt and 6pt);
\draw[kepsus] (c) to node[labl,pos=0.45] {\tiny $\gamma_2$} (b);
\draw[kepsus] (d) to node[labl,pos=0.45] {\tiny $\gamma_2$} (b);
\draw[kepsus] (e) to node[labl,pos=0.45] {\tiny $\gamma_1$} (d);
\draw[kepsus] (f) to node[labl,pos=0.45] {\tiny $\gamma_1$} (e);
\draw[dashed] (e) to (i);
\draw[dashed] (f) to (j);
\draw[testfcn] (a) to (b);
\draw[] (z.west) node[dot1] {} to (d);
\draw[] (z.east) node[dot1] {} to (c);
\end{tikzpicture}
$$
$$
\begin{tikzpicture}[baseline=0cm,scale=0.8]
\node at (0,3) {\textcircled{4}};
\node at (0,-2) [root,label=below:$0$] (a) {};
\node at (0,-1) [dot,label=left:$v_\star$] (b) {};
\node at (1,0) [dot,label=right:$v$] (c) {};
\node at (-1,0) [dot,label=left:$v_1$] (d) {};
\node at (0,1) [dot,label=above:$v_2$] (e) {};
\node at (-1,2) [dot,label=left:$v_3$] (f) {};
\node at (-1,0.75) [var] (h) {};
\node at (-1,2.75) [var] (j) {};
\draw[kepsus] (c) to node[labl,pos=0.45] {\tiny $\gamma_2$} (b);
\draw[kepsus] (d) to node[labl,pos=0.45] {\tiny $\gamma_2$} (b);
\draw[kepsus] (e) to node[labl,pos=0.45] {\tiny $\gamma_1$} (d);
\draw[kepsus] (f) to node[labl,pos=0.45] {\tiny $\gamma_1$} (e);
\draw[dashed] (d) to (h);
\draw[dashed] (f) to (j);
\draw[testfcn] (a) to (b);
\node[cumu2n] (z) at (0.5,0.5){};
\draw[cumu2] (z) ellipse (10pt and 6pt);
\draw[] (z.west) node[dot1] {} to (e);
\draw[] (z.east) node[dot1] {} to (c);
\end{tikzpicture}\quad\begin{tikzpicture}[baseline=0cm,scale=0.8]
\node at (0,3) {\textcircled{5}};
\node at (0,-2) [root,label=below:$0$] (a) {};
\node at (0,-1) [dot,label=left:$v_\star$] (b) {};
\node at (1,0) [dot,label=right:$v$] (c) {};
\node at (-1,0) [dot,label=left:$v_1$] (d) {};
\node at (0,1) [dot,label=right:$v_2$] (e) {};
\node at (-1,2) [dot,label=left:$v_3$] (f) {};
\draw[kepsus] (c) to node[labl,pos=0.45] {\tiny $\gamma_2$} (b);
\draw[kepsus] (d) to node[labl,pos=0.45] {\tiny $\gamma_2$} (b);
\draw[kepsus] (e) to node[labl,pos=0.45] {\tiny $\gamma_1$} (d);
\draw[kepsus] (f) to node[labl,pos=0.45] {\tiny $\gamma_1$} (e);
\draw[testfcn] (a) to (b);
\node[cumu2n] (y) at (0,0){};
\draw[cumu2] (y) ellipse (10pt and 6pt);
\node[cumu2n] (z) at (-0.25,2){};
\draw[cumu2] (z) ellipse (10pt and 6pt);
\draw[] (z.west) node[dot1] {} to (f);
\draw[] (z.east) node[dot1] {} to (e);
\draw[] (y.west) node[dot1] {} to (d);
\draw[] (y.east) node[dot1] {} to (c);
\end{tikzpicture}
$$

Checking the conditions on \textcircled{1} and \textcircled{2} will follow a similar line of reasoning as the un-contracted tree. The only difference stems from the second order contractions, which occurs in the subsets $\{v_3,\,v_1\}$ and $\{v_3,\,v_2\}$. Consider first the effect of $(v_3,\,v_2)$ in \textcircled{1}. As before we expect this to add a weight of $3$ to the edge $(v_3,\,v_2)$ which is to say that $\hat{a}_{(v_3,v_2)}=|\s| - 2 + 3 = |\s| + 1$, but it is straightforward to see that conditions are satisfied for the same reasons as they were for the original tree. For \textcircled{2}, $\{v_3,\,v_2,\,v_1\}$ is a new structure, but one checks that $|\s| - 2 + |\s| - 2 + 3 = 2|\s| - 1 < A(3) = 2|\s|$. The rest of the tree shares arguments with the original tree.

In \textcircled{3}, we again see the sort of sub-divergence that we have had to contend with before in $\{v_1,\,v,\,v_\star\}$. To deal with it we will, much like before, move the edge $(v_2,\,v_1)$ to $v_\star$ instead. This decomposition will take the following form:

$$
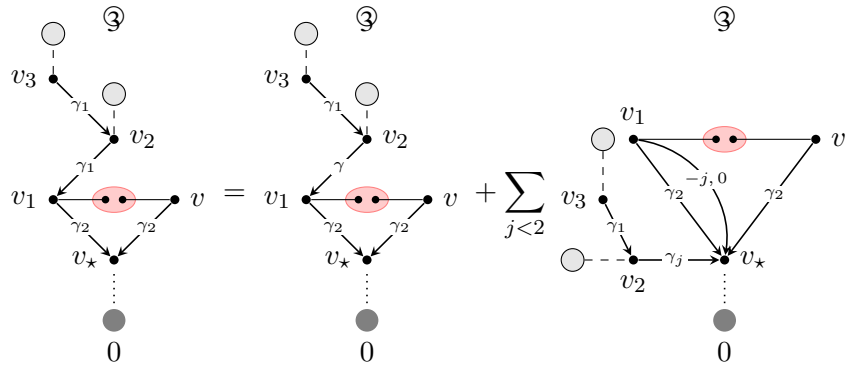
\begin{tikzpicture}[baseline=0cm,scale=0.8]
\node at (0,3) {\textcircled{3}};
\node at (0,-2) [root,label=below:$0$] (a) {};
\node at (0,-1) [dot,label=left:$v_\star$] (b) {};
\node at (1,0) [dot,label=right:$v$] (c) {};
\node at (-1,0) [dot,label=left:$v_1$] (d) {};
\node at (0,1) [dot,label=right:$v_2$] (e) {};
\node at (-1,2) [dot,label=left:$v_3$] (f) {};
\node at (0,1.75) [var] (i) {};
\node at (-1,2.75) [var] (j) {};
\node[cumu2n] (z) at (0,0){};
\draw[cumu2] (z) ellipse (10pt and 6pt);
\draw[kepsus] (c) to node[labl,pos=0.45] {\tiny $\gamma_2$} (b);
\draw[kepsus] (d) to node[labl,pos=0.45] {\tiny $\gamma_2$} (b);
\draw[kepsus] (e) to node[labl,pos=0.45] {\tiny $\gamma_1$} (d);
\draw[kepsus] (f) to node[labl,pos=0.45] {\tiny $\gamma_1$} (e);
\draw[dashed] (e) to (i);
\draw[dashed] (f) to (j);
\draw[testfcn] (a) to (b);
\draw[] (z.west) node[dot1] {} to (d);
\draw[] (z.east) node[dot1] {} to (c);
\end{tikzpicture}=\begin{tikzpicture}[baseline=0cm,scale=0.8]
\node at (0,3) {\textcircled{3}};
\node at (0,-2) [root,label=below:$0$] (a) {};
\node at (0,-1) [dot,label=left:$v_\star$] (b) {};
\node at (1,0) [dot,label=right:$v$] (c) {};
\node at (-1,0) [dot,label=left:$v_1$] (d) {};
\node at (0,1) [dot,label=right:$v_2$] (e) {};
\node at (-1,2) [dot,label=left:$v_3$] (f) {};
\node at (0,1.75) [var] (i) {};
\node at (-1,2.75) [var] (j) {};
\node[cumu2n] (z) at (0,0){};
\draw[cumu2] (z) ellipse (10pt and 6pt);
\draw[kepsus] (c) to node[labl,pos=0.45] {\tiny $\gamma_2$} (b);
\draw[kepsus] (d) to node[labl,pos=0.45] {\tiny $\gamma_2$} (b);
\draw[kepsus] (e) to node[labl,pos=0.45] {\tiny $\gamma$} (d);
\draw[kepsus] (f) to node[labl,pos=0.45] {\tiny $\gamma_1$} (e);
\draw[dashed] (e) to (i);
\draw[dashed] (f) to (j);
\draw[testfcn] (a) to (b);
\draw[] (z.west) node[dot1] {} to (d);
\draw[] (z.east) node[dot1] {} to (c);
\end{tikzpicture} + \sum_{j < 2} \begin{tikzpicture}[baseline=0cm,scale=0.8]
\node at (0,3) {\textcircled{3}};
\node at (0,-2) [root,label=below:$0$] (a) {};
\node at (0,-1) [dot,label=right:$v_\star$] (b) {};
\node at (1.5,1) [dot,label=right:$v$] (c) {};
\node at (-1.5,1) [dot,label=above:$v_1$] (d) {};
\node at (-2,0) [dot,label=left:$v_3$] (f) {};
\node at (-1.5,-1) [dot,label=below:$v_2$] (e) {};
\node at (-2,1) [var] (g) {};
\node at (-2.5,-1) [var] (i) {};
\node[cumu2n] (z) at (0,1){};
\draw[cumu2] (z) ellipse (10pt and 6pt);
\draw[kepsus] (c) to node[labl,pos=0.45] {\tiny $\gamma_2$} (b);
\draw[kepsus] (d) to node[labl,pos=0.45] {\tiny $\gamma_2$} (b);
\draw[kepsus] (f) to node[labl,pos=0.45] {\tiny $\gamma_1$} (e);
\draw[kepsus] (e) to node[labl,pos=0.45] {\tiny $\gamma_j$} (b);
\draw[kepsus] (d) to[bend left=45] node[labl,pos=0.5] {\tiny $-j,0$} (b);
\draw[dashed] (g) to (f);
\draw[dashed] (i) to (e);
\draw[testfcn] (a) to (b);
\draw[] (z.west) node[dot1] {} to (d);
\draw[] (z.east) node[dot1] {} to (c);
\end{tikzpicture}
$$

where $\gamma = (1,\,2,\,v_\star)$ and as before this will cure the subdivergence for the first graph on the right hand side. The treatment for the other terms on the right hand side will be the same as before.

The higher order cumulants will see take the form:

$$
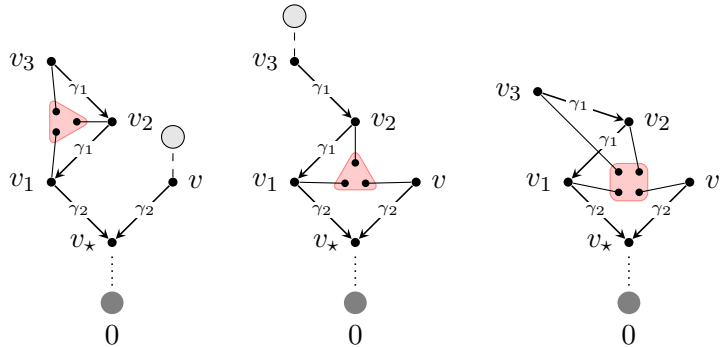
\begin{tikzpicture}[baseline=0cm,scale=0.8]
\node at (0,-2) [root,label=below:$0$] (a) {};
\node at (0,-1) [dot,label=left:$v_\star$] (b) {};
\node at (1,0) [dot,label=right:$v$] (c) {};
\node at (-1,0) [dot,label=left:$v_1$] (d) {};
\node at (0,1) [dot,label=right:$v_2$] (e) {};
\node at (-1,2) [dot,label=left:$v_3$] (f) {};
\node at (1,0.75) [var] (g) {};
\node[] (z) at (-0.75,1) {};
\node[cumu3,rotate=270] (z-) at (-0.8,1) {};
\draw[kepsus] (c) to node[labl,pos=0.45] {\tiny $\gamma_2$} (b);
\draw[kepsus] (d) to node[labl,pos=0.45] {\tiny $\gamma_2$} (b);
\draw[kepsus] (e) to node[labl,pos=0.45] {\tiny $\gamma_1$} (d);
\draw[kepsus] (f) to node[labl,pos=0.45] {\tiny $\gamma_1$} (e);
\draw[] (z.east) node[dot1] {} to (e);
\draw[] (z.north west) node[dot1] {} to (f);
\draw[] (z.south west) node[dot1] {} to (d);
\draw[dashed] (c) to (g);
\draw[testfcn] (a) to (b);
\end{tikzpicture}\quad\begin{tikzpicture}[baseline=0cm,scale=0.8]
\node at (0,-2) [root,label=below:$0$] (a) {};
\node at (0,-1) [dot,label=left:$v_\star$] (b) {};
\node at (1,0) [dot,label=right:$v$] (c) {};
\node at (-1,0) [dot,label=left:$v_1$] (d) {};
\node at (0,1) [dot,label=right:$v_2$] (e) {};
\node at (-1,2) [dot,label=left:$v_3$] (f) {};
\node at (-1,2.75) [var] (j) {};
\node[] (z) at (0,0.15) {};
\node[cumu3] (z-) at (0,0.1) {};
\draw[kepsus] (c) to node[labl,pos=0.45] {\tiny $\gamma_2$} (b);
\draw[kepsus] (d) to node[labl,pos=0.45] {\tiny $\gamma_2$} (b);
\draw[kepsus] (e) to node[labl,pos=0.45] {\tiny $\gamma_1$} (d);
\draw[kepsus] (f) to node[labl,pos=0.45] {\tiny $\gamma_1$} (e);
\draw[] (z.south west) node[dot1] {} to (d);
\draw[] (z.south east) node[dot1] {} to (c);
\draw[] (z.north) node[dot1] {} to (e);  
\draw[dashed] (f) to (j);
\draw[testfcn] (a) to (b);
\end{tikzpicture}\quad\begin{tikzpicture}[baseline=0cm,scale=0.8]
\node at (0,-2) [root,label=below:$0$] (a) {};
\node at (0,-1) [dot,label=left:$v_\star$] (b) {};
\node at (1,0) [dot,label=right:$v$] (c) {};
\node at (-1,0) [dot,label=left:$v_1$] (d) {};
\node at (0,1) [dot,label=right:$v_2$] (e) {};
\node at (-1.5,1.5) [dot,label=left:$v_3$] (f) {};
\node[] (z) at (0,0) {};
\node[cumu4] (e-) at (z) {};
\draw[kepsus] (c) to node[labl,pos=0.45] {\tiny $\gamma_2$} (b);
\draw[kepsus] (d) to node[labl,pos=0.45] {\tiny $\gamma_2$} (b);
\draw[kepsus] (e) to node[labl,pos=0.30] {\tiny $\gamma_1$} (d);
\draw[kepsus] (f) to node[labl,pos=0.45] {\tiny $\gamma_1$} (e);
\draw[testfcn] (a) to (b);
\draw[] (z.north west) node[dot1] {} to (f);
\draw[] (z.south west) node[dot1] {} to (d);
\draw[] (z.north east) node[dot1] {} to (e);
\draw[] (z.south east) node[dot1] {} to (c);
\end{tikzpicture}
$$

The subtree $\{v_1,\,v_2,\,v_3\}$ appears to be problematic $|\s| - 2 + |\s| - 2 + 3/2(3) = 2|\s| + 1/2 \nless 2|\s|$ but this is treated using the excision method we have already used. None of the other trees cause any trouble.

Next we look at the symbol: $\begin{tikzpicture}[baseline=-0.25cm,scale=0.4]
\node at (0,-1) [bar] (a) {};
\node at (-0.5,-0.5) [bar] (b) {};
\node at (-0.5,0.5) [bar] (c) {};
\node at (0.5,0.5) [bar] (d) {};
\node at (0,0) [bnode] (e) {};
\draw[] (a) to (b);
\draw[] (b) to (0,0);
\draw[thick] (0,0) to (c);
\draw[thick] (0,0) to (d);
\end{tikzpicture}$. The relevant tree is of the form:

$$
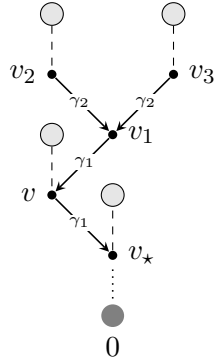
\begin{tikzpicture}[baseline=0cm,scale=0.8]
\node at (0,-2) [root, label=below:$0$] (a) {};
\node at (0,-1) [dot,label=right:$v_\star$] (b) {};
\node at (-1,0) [dot,label=left:$v$] (c) {};
\node at (0,1) [dot,label=right:$v_1$] (d) {};
\node at (-1,2) [dot,label=left:$v_2$] (e) {};
\node at (1,2) [dot,label=right:$v_3$] (f) {};
\node at (-1,3) [var] (g) {};
\node at (1,3) [var] (h) {};
\node at (-1,1) [var] (i) {};
\node at (0,0) [var] (j) {};
\draw[testfcn] (a) to (b);
\draw[dashed] (g) to (e);
\draw[dashed] (h) to (f);
\draw[dashed] (i) to (c);
\draw[dashed] (b) to (j);
\draw[kepsus] (e) to node[labl,pos=0.45] {\tiny $\gamma_2$} (d);
\draw[kepsus] (f) to node[labl,pos=0.45] {\tiny $\gamma_2$} (d);
\draw[kepsus] (d) to node[labl,pos=0.45] {\tiny $\gamma_1$} (c);
\draw[kepsus] (c) to node[labl,pos=0.45] {\tiny $\gamma_1$} (b);
\end{tikzpicture}$$

It is an easy exercise to see that the above tree satisfies all the required conditions. The first order contractions one will see are as follows:

$$
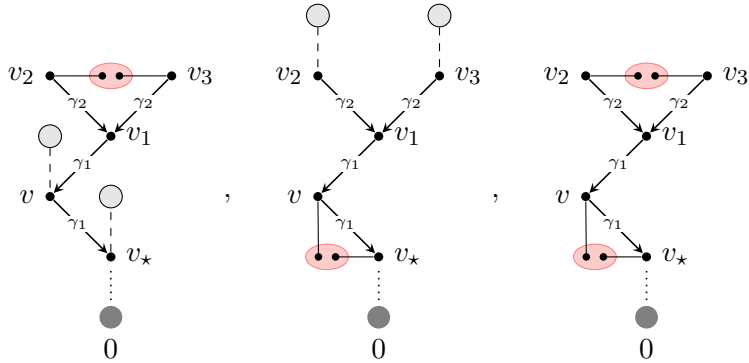
\begin{tikzpicture}[baseline=0cm,scale=0.8]
\node at (0,-2) [root, label=below:$0$] (a) {};
\node at (0,-1) [dot,label=right:$v_\star$] (b) {};
\node at (-1,0) [dot,label=left:$v$] (c) {};
\node at (0,1) [dot,label=right:$v_1$] (d) {};
\node at (-1,2) [dot,label=left:$v_2$] (e) {};
\node at (1,2) [dot,label=right:$v_3$] (f) {};
\node at (-1,1) [var] (i) {};
\node at (0,0) [var] (j) {};
\node[cumu2n] (z) at (0,2){};
\draw[cumu2] (z) ellipse (10pt and 6pt);
\draw[] (z.west) node[dot1] {} to (e);
\draw[] (z.east) node[dot1] {} to (f);
\draw[testfcn] (a) to (b);
\draw[dashed] (i) to (c);
\draw[dashed] (b) to (j);
\draw[kepsus] (e) to node[labl,pos=0.45] {\tiny $\gamma_2$} (d);
\draw[kepsus] (f) to node[labl,pos=0.45] {\tiny $\gamma_2$} (d);
\draw[kepsus] (d) to node[labl,pos=0.45] {\tiny $\gamma_1$} (c);
\draw[kepsus] (c) to node[labl,pos=0.45] {\tiny $\gamma_1$} (b);
\end{tikzpicture},\quad\begin{tikzpicture}[baseline=0cm,scale=0.8]
\node at (0,-2) [root, label=below:$0$] (a) {};
\node at (0,-1) [dot,label=right:$v_\star$] (b) {};
\node at (-1,0) [dot,label=left:$v$] (c) {};
\node at (0,1) [dot,label=right:$v_1$] (d) {};
\node at (-1,2) [dot,label=left:$v_2$] (e) {};
\node at (1,2) [dot,label=right:$v_3$] (f) {};
\node[cumu2n] (z) at (-0.85,-1){};
\draw[cumu2] (z) ellipse (10pt and 6 pt);
\node at (-1,3) [var] (g) {};
\node at (1,3) [var] (h) {};
\draw[] (z.west) node[dot1] {} to (c);
\draw[] (z.east) node[dot1] {} to (b);
\draw[testfcn] (a) to (b);
\draw[dashed] (g) to (e);
\draw[dashed] (h) to (f);
\draw[kepsus] (e) to node[labl,pos=0.45] {\tiny $\gamma_2$} (d);
\draw[kepsus] (f) to node[labl,pos=0.45] {\tiny $\gamma_2$} (d);
\draw[kepsus] (d) to node[labl,pos=0.45] {\tiny $\gamma_1$} (c);
\draw[kepsus] (c) to node[labl,pos=0.45] {\tiny $\gamma_1$} (b);
\end{tikzpicture},\quad\begin{tikzpicture}[baseline=0cm,scale=0.8]
\node at (0,-2) [root, label=below:$0$] (a) {};
\node at (0,-1) [dot,label=right:$v_\star$] (b) {};
\node at (-1,0) [dot,label=left:$v$] (c) {};
\node at (0,1) [dot,label=right:$v_1$] (d) {};
\node at (-1,2) [dot,label=left:$v_2$] (e) {};
\node at (1,2) [dot,label=right:$v_3$] (f) {};
\node[cumu2n] (z) at (-0.85,-1){};
\draw[cumu2] (z) ellipse (10pt and 6 pt);
\node[cumu2n] (y) at (0,2){};
\draw[cumu2] (y) ellipse (10pt and 6pt);
\draw[] (y.west) node[dot1] {} to (e);
\draw[] (y.east) node[dot1] {} to (f);
\draw[] (z.west) node[dot1] {} to (c);
\draw[] (z.east) node[dot1] {} to (b);
\draw[testfcn] (a) to (b);
\draw[kepsus] (e) to node[labl,pos=0.45] {\tiny $\gamma_2$} (d);
\draw[kepsus] (f) to node[labl,pos=0.45] {\tiny $\gamma_2$} (d);
\draw[kepsus] (d) to node[labl,pos=0.45] {\tiny $\gamma_1$} (c);
\draw[kepsus] (c) to node[labl,pos=0.45] {\tiny $\gamma_1$} (b);
\end{tikzpicture}$$

As before we expect the subset $\{v_1,\,v_2,\,v_3\}$ to cause us problems, but this may be treated in our usual excision method, while the subset $\{v,\,v_\star\}$ does not cause any trouble. The constant at the end we do not expect to be zero, and as such will be removed by the renormalisation.

In the higher order cumulants, the problematic one that we will find happens to be:

$$
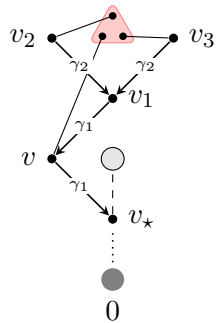
\begin{tikzpicture}[baseline=0cm,scale=0.8]
\node at (0,-2) [root, label=below:$0$] (a) {};
\node at (0,-1) [dot,label=right:$v_\star$] (b) {};
\node at (-1,0) [dot,label=left:$v$] (c) {};
\node at (0,1) [dot,label=right:$v_1$] (d) {};
\node at (-1,2) [dot,label=left:$v_2$] (e) {};
\node at (1,2) [dot,label=right:$v_3$] (f) {};
\node[] (z) at (0,2.2) {};
\node[cumu3] (z-) at (0,2.15) {};
\node at (0,0) [var] (j) {};
\draw[] (z.south west) node[dot1] {} to (c);
\draw[] (z.south east) node[dot1] {} to (f);
\draw[] (z.north) node[dot1] {} to (e);
\draw[testfcn] (a) to (b);
\draw[dashed] (b) to (j);
\draw[kepsus] (e) to node[labl,pos=0.45] {\tiny $\gamma_2$} (d);
\draw[kepsus] (f) to node[labl,pos=0.45] {\tiny $\gamma_2$} (d);
\draw[kepsus] (d) to node[labl,pos=0.45] {\tiny $\gamma_1$} (c);
\draw[kepsus] (c) to node[labl,pos=0.45] {\tiny $\gamma_1$} (b);
\end{tikzpicture}$$

One can calculate on $\{v_2,\,v_3,\,v_1,\,v\}$ leads to $|\s| - 1 + |\s| - 1 + |\s| - 2 + 3/2(3) = 3|\s| + 1/2 \not < (4 - 1)|\s| = 3|\s|$. This is dealt with the usual excision trick. No other combination (of three noises) causes any contravention of the rules. For example if either $v_2$ or $v_3$ is replaced in the previous set by $v_\star$, one cannot expect to see the same divergence for an edge of weight $|\s| - 1$ is replaced by an edge with weight $|\s| - 2$, while the RHS remains the same. If one wants to consider the effect of the cumulant term generated by the noises attached to $v_2,\,v_3\text{, and }v_\star$ we get, the set to be considered would be the entire tree except the origin and for this we will have $B(5) = 4|\s|$.

\subsection{Convergence of the model}\label{sec:convergence}

In this section, we want to talk about how the methods developed so far are to be put together to give the convergence result we are after. With reference to \cite{MH21}, this section is comparable to Section 5 in that article, wherein the authors show that $u^N$ which solves the rescaled, discretised, renormalised PAM converges in probability to the solution of the renormalised PAM driven by the Ornstein-Uhlenbeck process in the space $C^{\bar\eta,T}_N$ for $\bar \eta\in(0,\frac{1}{2}\wedge\eta)$. In particular, we want to say what needs to be changed for our result to work.

Take $\rho$ to be any smooth, compactly supported function defined on $\mb{R}^{d+1}$ that integrates to one, and define for $\delta\in(2^{-N},1]$ the rescaled version $\rho^\delta(t,x) \coloneqq \delta^{-5}\rho(\delta^{-2},\delta^{-1}x)$ and also $\rho^{\delta,N}\coloneqq 2^{-dN}\int\rho^{\delta}(t,y)\mathbbm{1}\{|y-x|\le 2^{-N-1}\}dy$. Then the semi-discrete convolution $\rho^{\delta,N}\star_N\xi^{N}$ defines a smooth discrete noise which we denote by $\xi^{\delta,N}$ and use to introduce the equation:
\begin{equs}\partial_tu^{\delta,N} & = (\Delta u^{\delta,N}) + g(u^{\delta,N})(\nabla u^{\delta,N})^2 \\ & + k(u^{\delta,N})(\nabla u^{\delta,N})+ \bar h(u^{\delta,N}) + \xi^{\delta,N}(u^{\delta,N})
\end{equs}
defined on $\mb{R}_{+}\times 2^{-N}\mb{S}$ where $\bar h$ is as it was before and the constants are the same but with $\xi^N$ replaced by $\xi^{\delta,N}$. The point the authors take in \cite{MH21} is that the setup above as it corresponds to the PAM is amenable to \cite[Theorem 16]{EH17} and hence one is able to argue that $\|u^{N};u^{\delta,N}\|_{C^{\eta,T}_{N}}$ converges to zero in probability provided one takes $N\rightarrow\infty$ before the mollification is removed. We are not able to directly argue like this for our equation because the named theorem pertains to a fixed point with non-linearities only in $u^N$. We however anticipate, that due to the general result in \cite{BCCH}, a similar conclusion for the \eqref{eq:gKPZ} should also hold. The next step in the argument is to look to the equation:
 \begin{equs}
\partial_t u^{\delta} & = (\Delta u^{\delta}) + g(u^{\delta})(\nabla u^{\delta})^2 \\ & + k(u^{\delta})(\nabla u^{\delta})+ h(u^{\delta}) + Y^{\delta}(u^{\delta})
\end{equs}
defined on $\mb{R}_+\times\mb{T}^{d}$, with $\bar h$ again as before, and $Y^{\delta} = \rho^{\delta}\star Y$ with $Y$ as in \eqref{eq:gKPZcon}. Then, with reference to the results in \cite{CH16} and \cite{BHZ}, one has that $u^{\delta}$ converges in probability in the space $\mc{C}^{\bar \eta,T}_N$ to a limit $u$ as $\delta\rightarrow 0$. At this point it only remains to prove the convergence for our result to hold:
\be
\lim_{\delta\rightarrow 0}\lim_{N\rightarrow\infty}\|u^{\delta,N};u^{\delta}\|_{\mc{C}^{\bar\eta,T}_N} = 0
\ee

Due to the smoothness of $u^{\delta}$, this follows as soon as we have:

$$\lim_{N\rightarrow\infty}\sup_{(t,x)\in[0,T]\times 2^{-N}\mb{S}}|\xi^{\delta,N}(t,x)-Y^\delta(t,x)|=0$$

which we have assumed to be true.

\begin{appendix}
\section{Joint Cumulants}
\label{sec:cum}
In this appendix, we recall the definition and basic properties of joint cumulants. Consider a collection of random variables $\{X_a\}_{a\in\CA}$ for some index set $\CA$. For the subsets $B\subseteq\CA$, we write $X_B=\{X_a : a\in B\}\text{ and }X^B=\prod_{a\in B}X_a$. One should be mindful of the fact that $X_B$ is a collection of random variables while $X^B$ is itself a random variable. Further, we write $\CP(B)$ for the set of all partitions of $B$.

\begin{definition}\label{def:cum}
Let $B\subseteq\CA$. The cumulant $\mbbE_c(X_B)$ is defined inductively over $|B|$ by $\mathbb{E}_c(X_B)=\mathbb{E}(X_a)$, if $B$ is the singleton containing $a$ and
\beq
\mbbE(X^B)=\sum_{\pi\in\CP(B)}\prod_{\bar{B}\in\pi}\mbbE_c(X_{\bar{B}}),\qquad\text{if }|B|\ge 2
\eeq
\end{definition}

The usefulness of the notation $X_B$ and $X^B$ is in the fact that cumulants are defined for collections (say vectors) of random variables, while the moments can be taken of at most a product of random variables. Due to this reliance on the moment in its definition the following properties are natural:

\begin{itemize}
\item $\mathbb{E}_c\left[h_1X_1,\cdots, h_k X_k\right] = \prod_{i=1}^k h_i \times\mathbb{E}_c\left[X_1,\cdots,X_k\right]$
\item $\mathbb{E}_c\left[X_B\right] = 0$, if one has $B = B' \cup B''$, $B'\cap B'' = \emptyset$,  and $X_{B'}$ and $X_{B''}$ are independent
\item For a jointly Gaussian collection of random variables $X_B$, $\mathbb{E}_c\left[X_B\right]=0$ if $|B| > 2$.
\end{itemize}

\begin{example}
Consider the case when $|B| = 2$:

$$\mathbb{E}_c\left[X_1,X_2\right] = \mathbb{E}\left[X_1X_2\right] - \mathbb{E}\left[X_1\right]\mathbb{E}\left[X_2\right]$$
that is to say, the cumulant of two random variables is its usual covariance. \\
When $|B|=3$, one is able to calculate:
\be
\begin{split}
\mb{E}_c[X_1,X_2,X_3] &= \mb{E}[X_1X_2X_3] - \mb{E}[X_1X_2]\mb{E}[X_3] - \mb{E}[X_1X_3]\mb[X_2] \\
&- \mb{E}[X_2X_3]\mb{E}[X_1] + 2\mb{E}[X_1]\mb{E}[X_2]\mb{E}[X_3]
\end{split}
\ee
\end{example}

The Wick products, that have a central role in the renormalisation of SPDEs can then be formulated in the sense of definition \ref{def:cum}:
\begin{definition}\label{def:wick}
For $A\subseteq \CA$, the Wick product $:\!X_A\!:$ is defined via $:\!X_{\emptyset}\!:=1$ and then recursively:
\beq
X^A=\sum_{B\subseteq A}:\!X_B\!:\sum_{\pi\in\CP(A\setminus B)}\prod_{\bar{B}\in\pi}\mbbE_c(X_{\bar{B}})
\eeq
\end{definition}

It is obvious that the above definition forces $\mathbb{E}\!:\!X_A\!: = 0$ whenever $A\neq\emptyset$, and moreover taking expectation on either side of the equality in Definition~\ref{def:wick} reduces to the equality in Definition~\ref{def:cum}.
\begin{example}
For a family of centred random variables $X_i$, one has $:\!X_i\!:=X_i$, $:\!X_1X_2\!: = X_1X_2 - \mathbb{E}(X_1X_2)$.
$$:\!X_1X_2X_3\!: = X_1X_2X_3 - \sum_{i\neq j \neq k}X_i\mathbb{E}(X_jX_k) - \mathbb{E}(X_1X_2X_3)$$
\end{example}

We also define the following class of partitions: Let $M$ and $P$ be two sets and fix a subset $D\subseteq M\times P$. Then $\CP_M(D)$ is the set of all partitions of $D$, such that for every $B\in \pi \in \CP_M(D)$, there must exist $(i,k),(i',k)\in B$ such that $k\neq k'$.

\begin{lemma}\label{lem:cyc}
For $m,p\in\mathbb{N}$. Set $M=\{i:1\le i\le m\}$ and $P=\{k:1\le k \le p\}$. Let $\{X_{(i,k)}\}_{i\in M,\,k\in P}$ be a collection of random variables with bounded moments of all orders. Then:
$$
\mbbE\left(\prod_{k=1}^p:\!\prod_{i=1}^m X_{(i,k)}\!:\right)=\sum_{\pi\in\CP_M(M\times P)}\prod_{B\in\pi}\mathbb{E}_c(X_B)
$$
\end{lemma}
\end{appendix}


\begin{thebibliography}{99}
\bibitem{BB21}
I.~Bailleul, Y.~{Bruned}.
 { \em Renormalised singular stochastic PDEs}. 
 \burlalt{arXiv:2101.11949}{http://arxiv.org/abs/2101.11949}. 
 
 \bibitem{BB21b}
I.~Bailleul, Y.~{Bruned}.
 { \em  Locality for singular stochastic PDEs}. 
 \burlalt{arXiv:2109.00399}{https://arxiv.org/abs/2109.00399}.


\bibitem{BB23}
I.~Bailleul, Y.~{Bruned}.
 { \em Random models for singular SPDEs}. 
 \burlalt{arXiv:2301.09596}{https://arxiv.org/abs/2301.09596}.
 


\bibitem{BCCH}
 { \rm Y. Bruned, A. Chandra, I. Chevyrev,
  M. Hairer}.
 {\em Renormalising SPDEs in regularity structures}.
 J. Eur. Math. Soc. (JEMS), \textbf{23}, no.~3, (2021), 869--947.
\burlalt{doi:10.4171/JEMS/1025}{http://dx.doi.org/10.4171/JEMS/1025}.
  
  \bibitem{BG97}
L. ~Bertini, G. ~Giacomin.
 {\em Stochastic Burgers and KPZ equations from particle systems.}
 Comm. Math. Phys. \textbf{183}, no.~3, (1997), 571--607.
 \burlalt{doi:10.1007/s002200050044}{https://doi.org/10.1007/s002200050044}.
 
 \bibitem{BGHZ22}
Y.~{Bruned}, F.~{Gabriel}, M.~{Hairer},
  L.~{Zambotti}.
 \emph{Geometric stochastic heat equations}.
 J. Amer. Math. Soc. (JAMS), \textbf{35}, no.~1, (2022),
  1--80.
 \burlalt{doi:10.1090/jams/977}{http://dx.doi.org/10.1090/jams/977}. 
  
\bibitem{BHZ}
{\rm Y. Bruned, M. Hairer, L. Zambotti}.
 {\em Algebraic renormalisation of regularity structures.}
 Invent. Math. \textbf{215}, no.~3, (2019), 1039--1156.
\burlalt{doi:10.1007/s00222-018-0841-x}{https://dx.doi.org/10.1007/s00222-018-0841-x}.

\bibitem{EMS}
  {\rm Y. Bruned, M. Hairer, L. Zambotti}.
 {\em Renormalisation of Stochastic Partial Differential Equations.}
 EMS Newsletter \textbf{115}, no.~3, (2020), 7--11.
\burlalt{doi: 10.4171/NEWS/115/3}{http://dx.doi.org/10.4171/NEWS/115/3}.


\bibitem{BN23}
Y. ~Bruned, U. ~Nadeem.
 {\em Diagram-free approach for convergence of tree-based models in Regularity Structures}. To appear in Journal of the Mathematical Society of Japan.
 \burlalt{arXiv:2211.11428}{https://arxiv.org/abs/2211.11428}.



   
\bibitem{BP57}
N.~N. Bogoliubow, O.~S. Parasiuk.
 { \em \"{U}ber die {M}ultiplikation der {K}ausalfunktionen in der
  {Q}uantentheorie der {F}elder.}
 Acta Math. \textbf{97}, (1957), 227--266.
\burlalt{doi:10.1007/BF02392399}{http://dx.doi.org/10.1007/BF02392399}.

\bibitem{BR18} 
{ \rm Y. Bruned}.
 {\em Recursive formulae in regularity structures.}
 Stoch. Partial Differ. Equ. Anal. and Comput. \textbf{6},
  no.~4, (2018), 525--564.
 \burlalt{doi:10.1007/s40072-018-0115-z}{http://dx.doi.org/10.1007/s40072-018-0115-z}.   
  

\bibitem{CCHS20}
A. ~Chandra, I. ~Chevyrev, M. ~Hairer, H. ~Shen.
 {\em Langevin dynamic for the 2D Yang-Mills measure}.
Publ. math. IHES, (2022).
 \burlalt{doi:10.1007/s10240-022-00132-0}{https://doi.org/10.1007/s10240-022-00132-0}.

\bibitem{CCHS22}
A. ~Chandra, I. ~Chevyrev, M. ~Hairer, H. ~Shen.
 {\em Stochastic quantisation of Yang-Mills-Higgs in 3D}.
 \burlalt{arXiv:2201.03487v1}{https://arxiv.org/abs/2201.03487v1}.

\bibitem{CGP16}
K.~Chouk, J.~Gairing, N. ~Perkowski.
{\em An invariance principle
for the two-dimensional parabolic Anderson model with small potential}.
 {Stoch. Partial Differ. Equ. Anal. and Comput. \textbf{5}, (2017).}
 \burlalt{doi:10.1007/s40072-017-0096-3}{https://doi.org/10.1007/s40072-017-0096-3}.

\bibitem{CH16}
A.~Chandra, M.~Hairer.
 {\em An analytic {BPHZ} theorem for regularity structures}.
 \burlalt{arXiv:1612.08138}{http://arxiv.org/abs/1612.08138}. 





\bibitem{C22}
I. ~Chevyrev.
 {\em Stochastic quantisation of Yang-Mills}. Journal of Mathematical Physics. Volume \textbf{63}, Issue 9, (2022).
 \burlalt{doi:10.1063/5.0089431}{https://doi.org/10.1063/5.0089431}

 \bibitem{CS23}
I. ~Chevyrev, H. Shen.
 {\em Invariant measure and universality of the 2D Yang–Mills Langevin dynamic}. \burlalt{arXiv:2302.12160}{http://arxiv.org/abs/2302.12160}. 

\bibitem{CM18}
G.~Cannizzaro, K.~Matetski.
 {\em Space-time discrete KPZ equation}.
Comm. Math. Phys. \textbf{358},
(2018), 521–588. 
 \burlalt{doi:10.1007/s00220-018-3089-9}{https://doi.org/10.1007/s00220-018-3089-9}.

 

\bibitem{DGP17}
J.~Diehl, M.~Gubinelli, N.~Perkowski. 
 {\em The Kardar–Parisi–Zhang Equation as Scaling Limit of Weakly Asymmetric Interacting Brownian Motions}.
Comm. Math. Phys. \textbf{354},
(2017), 549–589. 
 \burlalt{doi:10.1007/s00220-017-2918-6}{https://doi.org/10.1007/s00220-017-2918-6}.

\bibitem{EH17}
{\rm D. Erhard, M. Hairer}
 {\em Discretisation of regularity structures.}
 {Ann. Inst. H. Poincaré Probab. Statist. \textbf{55}, no.~4, (2019), 2209--2248.}
 \burlalt{doi:10.1214/18-AIHP947}{https://doi.org/10.1214/18-AIHP947}.

\bibitem{MH21}
D.~Erhard, M.~Hairer.
 {\em A scaling limit of the parabolic Anderson model with exclusion interaction}.
 \burlalt{arXiv:2103.13479}{https://arxiv.org/abs/2103.13479}.

\bibitem{GIP13}
M. ~Gubinelli, P. ~Imkeller, N. ~Perkowski.
 {\em Paracontrolled distributions and singular PDEs}.
{Forum of Mathematics Pi, \textbf{3}, e6, (2015)}.
 \burlalt{doi:1017/fmp.2015.2}{https://doi.org/10.1017/fmp.2015.2}.

\bibitem{GJ10}
P. ~Gon\c{c}alves, M. ~Jara.
 {\em Universality of KPZ equation}.
 arXiv preprint, (2010).
 \burlalt{arxiv:1003.
4478}{http://arxiv.org/abs/1003.4478}.

\bibitem{GJ14}
P. ~Gon\c{c}alves, M. ~Jara.
 {\em Nonlinear fluctuations of weakly asymmetric interacting particle systems}.
 Arch. Ration. Mech. Anal. \textbf{212}, no.~2, (2014), 597–-644.
 \burlalt{doi:10.1007/
s00205-013-0693-x.
58}{http://dx.doi.org/10.1007/
s00205-013-0693-x.
58}.



\bibitem{GP16}
M. ~Gubinelli, N. ~Perkowski.
 {\em The Hairer-Quastel universality result at stationarity}.
  Stochastic Analysis on Large Scale Interacting Systems, RIMS Kôkyûroku Bessatsu, B59. Res. Inst. Math. Sci. (RIMS), Kyoto, pp. 101–115.
 \burlalt{arXiv:1602.02428}{http://arxiv.org/abs/1602.02428}.

\bibitem{GP17}
M. ~Gubinelli, N. ~Perkowski.
 {\em KPZ reloaded}.
  Comm. Math. Phys. \textbf{394}, no.~1, (2017), 165--269.
 \burlalt{doi:10.1007/s00220-016-2788-3}{https://doi.org/10.1007/s00220-016-2788-3}.

\bibitem{GP18}
M.~Gubinelli, N.~Perkowski.
 \emph{Energy solutions of KPZ are unique.} J. Amer. Math. Soc. (JAMS), \textbf{31}, no.~2, (2018),
  427-471.
 \burlalt{doi:10.1090/jams/889}{http://dx.doi.org/10.1090/jams/889}. 

\bibitem{reg}
M. ~Hairer.
 {\em A theory of regularity structures}.
 Invent. Math. \textbf{198}, no.~2, (2014), 269--504.
\burlalt{doi:10.1007/s00222-014-0505-4}{https://dx.doi.org/10.1007/s00222-014-0505-4}.
  
  \bibitem{KH69}
K. ~Hepp.
 {\em On the equivalence of additive and analytic renormalization}.
 Comm. Math. Phys. \textbf{14}, (1969), 67--69.
 \burlalt{doi:10.1007/
BF01645456}{http://dx.doi.org/10.1007/
BF01645456}.

\bibitem{HI67}
T. ~Hida and N. ~Ikeda.
 {\em Analysis on Hilbert space with reproducing kernel
arising from multiple Wiener integral.} 
 In Proc. Fifth Berkeley Sympos. Math. Statist. and Probability \textbf{2}, no.~1, (1967), 117-–143.

\bibitem{HM18} 
M. ~Hairer, K. ~Matetski.
 {\em Discretisations of rough stochastic PDEs}.
 {Ann. Probab. \textbf{46}, no.~3, (2018), 1651--1709.}
\burlalt{doi:10.1214/17-AOP1212}{https://doi.org/10.1214/17-AOP1212}.

\bibitem{HP15}
{M. ~Hairer, E. ~Pardoux}.
{\em A Wong-Zakai theorem for stochastic PDEs.}
{J. Math. Soc. Japan, \textbf{67}, no.~4, (2015), 1551--1604.}
 \burlalt{doi:10.2969/jmsj/06741551}{https://doi.org/10.2969/jmsj/06741551}.
 
\bibitem{HQ15}
M. ~Hairer, J. ~Quastel.
 {\em A class of growth models rescaling to KPZ}.
 {Forum of Mathematics Pi, \textbf{6}, e3, (2018).}
\burlalt{doi:10.1017/fmp.2018.2}{https://doi.org/10.1017/fmp.2018.2}.


\bibitem{HS23}
M.~Hairer, R.~Steele.
 {\em The BPHZ Theorem for Regularity Structures via the Spectral Gap Inequality}. 
 \burlalt{arXiv:2301.10081
}{https://arxiv.org/abs/2301.10081}.



\bibitem{K16}
W. ~K\"{o}nig.
 {\em The parabolic Anderson Model}.
 Pathways in Mathematics \textbf{Xi}, (2016), 192.
 \burlalt{doi:10.1007/978-3-319-33596-4}{https://doi.org/10.1007/978-3-319-33596-4}.

 \bibitem{LOT}
P.~Linares, F.~Otto, M.~Tempelmayr.
 {\em The structure group for quasi-linear equations via universal enveloping algebras}. Comm. Amer. Math. Soc. \textbf{3}, (2023), 1-64.
 \burlalt{doi:10.1090/cams/16}{https://doi.org/10.1090/cams/16}

  \bibitem{LOTT}
P.~Linares, F.~Otto, M.~Tempelmayr, P.~Tsatsoulis.
 {\em A diagram-free approach to the stochastic estimates in regularity structures}. 
 \burlalt{arXiv:2112.10739
}{https://arxiv.org/abs/2112.10739}.

\bibitem{Mat18}
P.~Grazieschi, K.~Matetski, H.~Weber.
{\em Martingale-driven integrals and singular SPDEs}.
\burlalt{arXiv:2303.10245}{https://arxiv.org/abs/2303.10245}.

\bibitem{MP19}
J. ~Martin, N. ~Perkowski.
 {\em Paracontrolled distributions on Bravais lattices and weak universality of the 2d parabolic Anderson model}.
 Ann. Inst. H. Poincaré Probab. Statist. \textbf{55}, no.~4, (2019), 2058--2110.
 \burlalt{doi:10.1214/18-AIHP942}{https://doi.org/10.1214/18-AIHP942}.

\bibitem{MW17}
J. C. ~Mourrat, H. ~Weber.
 {\em Convergence of the two-dimensional dynamic Ising-Kac model to $\Phi^4_2$}.
 Comm. Pure Appl. Math. \textbf{70}, no.~4, (2017), 717--812.
 \burlalt{doi:10.1002/cpa.21655}{https://doi.org/10.1002/cpa.21655}.

\bibitem{MY89}
P. A. ~Meyer, J. A. ~Yan.
 {\em Distributions sur l\'{e}space de Wiener (suite) d\'{a}pr\`{e}s
I. Kubo et Y. Yokoi}.
 S\'{e}minaire de Probabilit\'{e}s \textbf{XXIII}, no.~1372, (1989), 382-–392.

\bibitem{OSSW}
 F.~Otto, J.~Sauer, S.~Smith, H.~Weber.
 {\em A priori bounds for quasi-linear SPDEs in the full sub-critical regime}. 
 \burlalt{arXiv:2103.11039
}{https://arxiv.org/abs/2103.11039}.

\bibitem{Wick50}
G. C. ~Wick.
 {\em The evaluation of the collision matrix}.
 Physical Rev. \textbf{2}, no.~80, (1950), 268-–
272.
 \burlalt{doi:10.1103/PhysRev.80.268}{https://doi.org/10.1103/PhysRev.80.268}
 

\bibitem{WZ69}
W. ~Zimmermann.
{\em Convergence of Bogoliubov’s method of renormalization in momentum space.}
 {Comm. Math. Phys. \textbf{15}, (1969), 208--234.}
 \burlalt{doi:10.1007/BF01645676}{http://dx.doi.org/10.1007/BF01645676}.
 

\bibitem{ZZ14}
R. ~Zhu, X. ~Zhu.
 {\em Approximating three-dimensional Navier-Stokes equations driven by space-time white noise}.
 Infinite Dimensional Analysis, Quantum Probability and Related Topics \textbf{20}, no.~4, (2017).
 \burlalt{doi:10.1142/S0219025717500205}{https://doi.org/10.1142/S0219025717500205}.

\bibitem{ZZ15}
R. ~Zhu, X. ~Zhu.
 {\em Piecewise linear approximations for dynamical $\varphi^4_3$ model}.
 {Sci. China Math \textbf{63}, (2020), 381-–410.}
 \burlalt{doi:10.1007/s11425-017-9269-1}{https://doi.org/10.1007/s11425-017-9269-1}.
\end{thebibliography}
\end{document}